\documentclass[color]{tumphthesis}

\usepackage{duckuments}
\usepackage{epigraph}
\usepackage{lipsum}
\usepackage{tcolorbox}
\usepackage{mathtools}
\usepackage{aligned-overset}
\usepackage{hyperref}
\usepackage{dirtytalk}
\usepackage{graphicx, caption}
\usepackage{hyperref}
\usepackage{cleveref}
\usepackage{extarrows}
\usepackage{bbm}
\usepackage[dvipsnames]{xcolor}
\usepackage{pgfplots}

\thesistitle
{Quantitative Mean-Field Limits for Recent Variants of Consensus-Based Optimization}
{Der Titel meiner Masterarbeit}
\thesisauthor{Konstantin Niklas Haug}{Author}
\thesisdate{2026}{3}{30}
\supervisor{Prof. Dr. Massimo Fornasier, Jona Klemenc}
\examiner{Prof. Dr. Massimo Fornasier}
\date{30 March 2026}
\newenvironment{proof}[1][Proof]{ \par\noindent\textit{#1.}\ }{\hfill$\square$\par}

\begin{document}

\maketitle
\vspace*{5cm}
\noindent
I hereby declare that this thesis is my own work and that no other sources have been used except those clearly indicated and referenced.
\\
\\
\\
\\
Munich, \,  March 30, 2026 
\\
\\
\\
\\
Konstantin Haug
\providecommand*{\Covv}[1]{\operatorname{Cov}\left({#1}\right)}   
\providecommand*{\textCovv}[1]{\operatorname{Cov}({#1})}   

\providecommand*{\Covvn}[2]{\operatorname{Cov}_{#2}\left({#1}\right)}   
\providecommand*{\Var}[1]{\operatorname{Var}\left({#1}\right)}   
\providecommand*{\textVar}[1]{\operatorname{Var}({#1})}   

\providecommand*{\Uni}[1]{\operatorname{Uni}({#1})}
\providecommand*{\Span}{\operatorname{span}}     

\providecommand*{\len}[1]{\operatorname{length}({#1})}   
\providecommand{\subG}{\operatorname{subG}}
\providecommand*{\seqInt}[2]{[{#1}:{#2}]}   %

\providecommand{\Bdsymb}{\Gamma}                   
\providecommand{\Div}{\operatorname{div}}          
\providecommand{\bDiv}{{\Div}_{\Bdsymb}}           
\providecommand{\mDiv}{\operatorname{\sf div}}     
\providecommand{\curl}{\operatorname{{\bf curl}}}  
\providecommand{\plcurl}{\operatorname{{\bf curl}}_{2D}} 
\providecommand{\bcurl}{{\curl}_{\Bdsymb}}         
\providecommand{\bscurl}{\operatorname{curl}_{\Bdsymb}} 
\providecommand{\scurl}{\operatorname{curl}_{2D}}       
\providecommand{\mcurl}{\operatorname{{\bf\sf curl}}}   
\providecommand{\grad}{\operatorname{{\bf grad}}}       
\providecommand{\bgrad}{\grad_{\Bdsymb}}                
\providecommand{\mgrad}{\operatorname{{\bf\sf grad}}}   

\providecommand{\Dim}{\operatorname{dim}}            
\providecommand{\dim}{\Dim}
\providecommand{\Card}{\sharp}                       
\providecommand{\Vol}{\operatorname{Vol}}   
\providecommand{\vol}{\Vol}
\providecommand{\ker}{\operatorname{ker}}
\providecommand{\ran}{\operatorname{ran}}
\providecommand*{\Meas}[2][]{\operatorname{meas}_{#1}({#2})} 
\providecommand*{\meas}[2][]{\Meas{#1}{#2}}
\providecommand*{\dist}[2]{\operatorname{dist}({#1};{#2})}   
\providecommand*{\normaldist}[2]{\operatorname{ndist}({#1};{#2})}   
\providecommand*{\Dist}[2]{\dist{#1}{#2}}
\providecommand{\Det}{\operatorname{det}}                    
\providecommand{\Sgn}{\operatorname{sgn}}                    
\providecommand*{\Kern}[1]{{\rm Ker}({#1})}                  
\providecommand*{\Dom}[1]{{\rm Dom}({#1})}                   
\providecommand*{\Span}[1]{\operatorname{Span}\left\{{#1}\right\}}     
\providecommand*{\Convex}[1]{\operatorname{convex}\left\{{#1}\right\}} 
\providecommand{\Supp}{\operatorname{supp}}                            
\providecommand{\supp}{\Supp}
\providecommand{\WF}{\operatorname{WF}}                            

\providecommand{\sing}{\operatorname{sing}}                            
\providecommand*{\Image}[1]{\mathrm{Im}({#1})}                         
\providecommand*{\Range}[1]{\operatorname{Range}({#1})}                
\providecommand{\range}{\operatorname{range}}                
\providecommand{\rank}{\operatorname{rank}}                        
\providecommand{\Rank}{\rank}
\renewcommand{\Re}{\operatorname{Re}}             
\renewcommand{\Im}{\operatorname{Im}}             
\providecommand{\argmin}{\operatorname*{\arg\min}}  
\providecommand{\arginf}{\operatorname*{arginf}}  
\providecommand{\argmax}{\operatorname*{argmax}}  
\providecommand{\nnz}{\operatorname{nnz}}         
\providecommand{\Id}{\Op{Id}}                     
\providecommand{\diam}{\operatorname{diam}}
\providecommand{\diag}{\operatorname{diag}}
\providecommand{\Diag}[1]{\operatorname{Diag}\left({#1}\right)}
\providecommand{\esssup}{\operatorname*{esssup}}
\providecommand{\essinf}{\operatorname*{essinf}}
\providecommand{\limsup}{\operatorname*{limsup}}
\providecommand{\closed}[1]{\operatorname*{cl}\left({#1}\right)}
\providecommand{\Law}{\operatorname*{Law}}


\newcommand{\Va}{{\mathbf{a}}}
\newcommand{\Vb}{{\mathbf{b}}}
\newcommand{\Vc}{{\mathbf{c}}}
\newcommand{\Vd}{{\mathbf{d}}}
\newcommand{\Ve}{{\mathbf{e}}}
\newcommand{\Vf}{{\mathbf{f}}}
\newcommand{\Vg}{{\mathbf{g}}}
\newcommand{\Vh}{{\mathbf{h}}}
\newcommand{\Vi}{{\mathbf{i}}}
\newcommand{\Vj}{{\mathbf{j}}}
\newcommand{\Vk}{{\mathbf{k}}}
\newcommand{\Vl}{{\mathbf{l}}}
\newcommand{\Vm}{{\mathbf{m}}}
\newcommand{\Vn}{{\mathbf{n}}}
\newcommand{\Vo}{{\mathbf{o}}}
\newcommand{\Vp}{{\mathbf{p}}}
\newcommand{\Vq}{{\mathbf{q}}}
\newcommand{\Vr}{{\mathbf{r}}}
\newcommand{\Vs}{{\mathbf{s}}}
\newcommand{\Vt}{{\mathbf{t}}}
\newcommand{\Vu}{{\mathbf{u}}}
\newcommand{\Vv}{{\mathbf{v}}}
\newcommand{\Vw}{{\mathbf{w}}}
\newcommand{\Vx}{{\mathbf{x}}}
\newcommand{\Vy}{{\mathbf{y}}}
\newcommand{\Vz}{{\mathbf{z}}}


\providecommand{\Ba}{{\boldsymbol{a}}}
\providecommand{\Bb}{{\boldsymbol{b}}}
\providecommand{\Bc}{{\boldsymbol{c}}}
\providecommand{\Bd}{{\boldsymbol{d}}}
\providecommand{\Be}{{\boldsymbol{e}}}
\providecommand{\Bf}{{\boldsymbol{f}}}
\providecommand{\Bg}{{\boldsymbol{g}}}
\providecommand{\Bh}{{\boldsymbol{h}}}
\providecommand{\Bi}{{\boldsymbol{i}}}
\providecommand{\Bj}{{\boldsymbol{j}}}
\providecommand{\Bk}{{\boldsymbol{k}}}
\providecommand{\Bl}{{\boldsymbol{l}}}
\providecommand{\Bm}{{\boldsymbol{m}}}
\providecommand{\Bn}{{\boldsymbol{n}}}
\providecommand{\Bo}{{\boldsymbol{o}}}
\providecommand{\Bp}{{\boldsymbol{p}}}
\providecommand{\Bq}{{\boldsymbol{q}}}
\providecommand{\Br}{{\boldsymbol{r}}}
\providecommand{\Bs}{{\boldsymbol{s}}}
\providecommand{\Bt}{{\boldsymbol{t}}}
\providecommand{\Bu}{{\boldsymbol{u}}}
\providecommand{\Bv}{{\boldsymbol{v}}}
\providecommand{\Bw}{{\boldsymbol{w}}}
\providecommand{\Bx}{{\boldsymbol{x}}}
\providecommand{\By}{{\boldsymbol{y}}}
\providecommand{\Bz}{{\boldsymbol{z}}}


\newcommand{\VA}{{\mathbf{A}}}
\newcommand{\VB}{{\mathbf{B}}}
\newcommand{\VC}{{\mathbf{C}}}
\newcommand{\VD}{{\mathbf{D}}}
\newcommand{\VE}{{\mathbf{E}}}
\newcommand{\VF}{{\mathbf{F}}}
\newcommand{\VG}{{\mathbf{G}}}
\newcommand{\VH}{{\mathbf{H}}}
\newcommand{\VI}{{\mathbf{I}}}
\newcommand{\VJ}{{\mathbf{J}}}
\newcommand{\VK}{{\mathbf{K}}}
\newcommand{\VL}{{\mathbf{L}}}
\newcommand{\VM}{{\mathbf{M}}}
\newcommand{\VN}{{\mathbf{N}}}
\newcommand{\VO}{{\mathbf{O}}}
\newcommand{\VP}{{\mathbf{P}}}
\newcommand{\VQ}{{\mathbf{Q}}}
\newcommand{\VR}{{\mathbf{R}}}
\newcommand{\VS}{{\mathbf{S}}}
\newcommand{\VT}{{\mathbf{T}}}
\newcommand{\VU}{{\mathbf{U}}}
\newcommand{\VV}{{\mathbf{V}}}
\newcommand{\VW}{{\mathbf{W}}}
\newcommand{\VX}{{\mathbf{X}}}
\newcommand{\VY}{{\mathbf{Y}}}
\newcommand{\VZ}{{\mathbf{Z}}}


\providecommand{\BA}{{\boldsymbol{A}}}
\providecommand{\BB}{{\boldsymbol{B}}}
\providecommand{\BC}{{\boldsymbol{C}}}
\providecommand{\BD}{{\boldsymbol{D}}}
\providecommand{\BE}{{\boldsymbol{E}}}
\providecommand{\BF}{{\boldsymbol{F}}}
\providecommand{\BG}{{\boldsymbol{G}}}
\providecommand{\BH}{{\boldsymbol{H}}}
\providecommand{\BI}{{\boldsymbol{I}}}
\providecommand{\BJ}{{\boldsymbol{J}}}
\providecommand{\BK}{{\boldsymbol{K}}}
\providecommand{\BL}{{\boldsymbol{L}}}
\providecommand{\BM}{{\boldsymbol{M}}}
\providecommand{\BN}{{\boldsymbol{N}}}
\providecommand{\BO}{{\boldsymbol{O}}}
\providecommand{\BP}{{\boldsymbol{P}}}
\providecommand{\BQ}{{\boldsymbol{Q}}}
\providecommand{\BR}{{\boldsymbol{R}}}
\providecommand{\BS}{{\boldsymbol{S}}}
\providecommand{\BT}{{\boldsymbol{T}}}
\providecommand{\BU}{{\boldsymbol{U}}}
\providecommand{\BV}{{\boldsymbol{V}}}
\providecommand{\BW}{{\boldsymbol{W}}}
\providecommand{\BX}{{\boldsymbol{X}}}
\providecommand{\BY}{{\boldsymbol{Y}}}
\providecommand{\BZ}{{\boldsymbol{Z}}}

\newcommand{\nablabf}{\boldsymbol{\nabla}}
\newcommand{\alphabf}{\boldsymbol{\alpha}}
\newcommand{\betabf}{\boldsymbol{\beta}}
\newcommand{\gammabf}{\boldsymbol{\gamma}}
\newcommand{\deltabf}{\boldsymbol{\delta}}
\newcommand{\epsilonbf}{\boldsymbol{\epsilon}}
\newcommand{\varepsilonbf}{\boldsymbol{\varepsilon}}
\newcommand{\zetabf}{\boldsymbol{\zeta}}
\newcommand{\etabf}{\boldsymbol{\eta}}
\newcommand{\thetabf}{\boldsymbol{\theta}}
\newcommand{\varthetabf}{\boldsymbol{\vartheta}}
\newcommand{\iotabf}{\boldsymbol{\iota}}
\newcommand{\kappabf}{\boldsymbol{\kappa}}
\newcommand{\lambdabf}{\boldsymbol{\lambda}}
\newcommand{\mubf}{\boldsymbol{\mu}}
\newcommand{\nubf}{\boldsymbol{\nu}}
\newcommand{\xibf}{\boldsymbol{\xi}}
\newcommand{\pibf}{\boldsymbol{\pi}}
\newcommand{\varpibf}{\boldsymbol{\varpi}}
\newcommand{\rhobf}{\boldsymbol{\rho}}
\newcommand{\varrhobf}{\boldsymbol{\varrho}}
\newcommand{\sigmabf}{\boldsymbol{\sigma}}
\newcommand{\varsigmabf}{\boldsymbol{\varsigma}}
\newcommand{\taubf}{\boldsymbol{\tau}}
\newcommand{\upsilonbf}{\boldsymbol{\upsilon}}
\newcommand{\phibf}{\boldsymbol{\phi}}
\newcommand{\varphibf}{\boldsymbol{\varphi}}
\newcommand{\chibf}{\boldsymbol{\chi}}
\newcommand{\psibf}{\boldsymbol{\psi}}
\newcommand{\omegabf}{\boldsymbol{\omega}}
\newcommand{\Gammabf}{\boldsymbol{\Gamma}}
\newcommand{\Deltabf}{\boldsymbol{\Delta}}
\newcommand{\Thetabf}{\boldsymbol{\Theta}}
\newcommand{\Lambdabf}{\boldsymbol{\Lambda}}
\newcommand{\Xibf}{\boldsymbol{\Xi}}
\newcommand{\Pibf}{\boldsymbol{\Pi}}
\newcommand{\Psibf}{\boldsymbol{\Psi}}
\newcommand{\Sigmabf}{\boldsymbol{\Sigma}}
\newcommand{\Upsilonbf}{\boldsymbol{\Upsilon}}
\newcommand{\Phibf}{\boldsymbol{\Phi}}
\newcommand{\Omegabf}{\boldsymbol{\Omega}}


\newcommand{\VGamma}  {\mathbf{\Gamma}}
\newcommand{\VDelta}  {\mathbf{\Delta}}
\newcommand{\VTheta}  {\mathbf{\Theta}}
\newcommand{\VLambda} {\mathbf{\Lambda}}
\newcommand{\VXi}     {\mathbf{\Xi}}
\newcommand{\VPi}     {\mathbf{\Pi}}
\newcommand{\VPsi}    {\mathbf{\Psi}}
\newcommand{\VSigma}  {\mathbf{\Sigma}}
\newcommand{\VUpsilon}{\mathbf{\Upsilon}}
\newcommand{\VPhi}    {\mathbf{\Phi}}
\newcommand{\VOmega}  {\mathbf{\Omega}}

\newcommand{\BGamma}  {\boldsymbol{\mathit{\Gamma}}}
\newcommand{\BDelta}  {\boldsymbol{\mathit{\Delta}}}
\newcommand{\BTheta}  {\boldsymbol{\mathit{\Theta}}}
\newcommand{\BLambda} {\boldsymbol{\mathit{\Lambda}}}
\newcommand{\BXi}     {\boldsymbol{\mathit{\Xi}}}
\newcommand{\BPi}     {\boldsymbol{\mathit{\Pi}}}
\newcommand{\BPsi}    {\boldsymbol{\mathit{\Psi}}}
\newcommand{\BSigma}  {\boldsymbol{\mathit{\Sigma}}}
\newcommand{\BUpsilon}{\boldsymbol{\mathit{\Upsilon}}}
\newcommand{\BPhi}    {\boldsymbol{\mathit{\Phi}}}
\newcommand{\BOmega}  {\boldsymbol{\mathit{\Omega}}}

\newcommand{\Valpha}     {{\boldsymbol{\upalpha}}}
\newcommand{\Vbeta}      {{\boldsymbol{\upbeta}}}
\newcommand{\Vgamma}     {{\boldsymbol{\upgamma}}}
\newcommand{\Vdelta}     {\boldsymbol{\updelta}}
\newcommand{\Vepsilon}   {\boldsymbol{\upepsilon}}
\newcommand{\Vvarepsilon}{\boldsymbol{\upvarepsilon}}
\newcommand{\Vzeta}      {\boldsymbol{\upzeta}}
\newcommand{\Veta}       {\boldsymbol{\upeta}} %
\newcommand{\Vtheta}     {\boldsymbol{\uptheta}}
\newcommand{\Vvartheta}  {\boldsymbol{\upvartheta}}
\newcommand{\Viota}      {\boldsymbol{\upiota}}
\newcommand{\Vkappa}     {\boldsymbol{\upkappa}}
\newcommand{\Vlambda}    {\boldsymbol{\uplambda}}
\newcommand{\Vmu}        {\boldsymbol{\upmu}}
\newcommand{\Vnu}        {\boldsymbol{\upnu}}
\newcommand{\Vxi}        {\boldsymbol{\upxi}}
\newcommand{\Vpi}        {\boldsymbol{\uppi}}
\newcommand{\Vvarpi}     {\boldsymbol{\upvarpi}}
\newcommand{\Vrho}       {\boldsymbol{\uprho}}
\newcommand{\Vvarrho}    {\boldsymbol{\upvarrho}}
\newcommand{\Vsigma}     {\boldsymbol{\upsigma}}
\newcommand{\Vvarsigma}  {\boldsymbol{\upvarsigma}}
\newcommand{\Vtau}       {\boldsymbol{\uptau}}
\newcommand{\Vupsilon}   {\boldsymbol{\upupsilon}}
\newcommand{\Vphi}       {\boldsymbol{\upphi}}
\newcommand{\Vvarphi}    {\boldsymbol{\upvarphi}}
\newcommand{\Vchi}       {\boldsymbol{\upchi}}
\newcommand{\Vpsi}       {\boldsymbol{\uppsi}}
\newcommand{\Vomega}     {\boldsymbol{\upomega}}

\newcommand{\Vzero}      {\boldsymbol{0}}
\newcommand{\Vone}       {\boldsymbol{1}}

\newcommand{\Balpha}     {\boldsymbol{\alpha}}
\newcommand{\Bbeta}      {\boldsymbol{\beta}}
\newcommand{\Bgamma}     {\boldsymbol{\gamma}}
\newcommand{\Bdelta}     {\boldsymbol{\delta}}
\newcommand{\Bepsilon}   {\boldsymbol{\epsilon}}
\newcommand{\Bvarepsilon}{\boldsymbol{\varepsilon}}
\newcommand{\Bzeta}      {\boldsymbol{\zeta}}
\newcommand{\Beta}       {\boldsymbol{\eta}}                
\newcommand{\Btheta}     {\boldsymbol{\theta}}
\newcommand{\Bvartheta}  {\boldsymbol{\vartheta}}
\newcommand{\Biota}      {\boldsymbol{\iota}}
\newcommand{\Bkappa}     {\boldsymbol{\kappa}}
\newcommand{\Blambda}    {\boldsymbol{\lambda}}
\newcommand{\Bmu}        {\boldsymbol{\mu}}
\newcommand{\Bnu}        {\boldsymbol{\nu}}
\newcommand{\Bxi}        {\boldsymbol{\xi}}
\newcommand{\Bpi}        {\boldsymbol{\pi}}
\newcommand{\Bvarpi}     {\boldsymbol{\varpi}}
\newcommand{\Brho}       {\boldsymbol{\rho}}
\newcommand{\Bvarrho}    {\boldsymbol{\varrho}}
\newcommand{\Bsigma}     {\boldsymbol{\sigma}}
\newcommand{\Bvarsigma}  {\boldsymbol{\varsigma}}
\newcommand{\Btau}       {\boldsymbol{\tau}}
\newcommand{\Bupsilon}   {\boldsymbol{\upsilon}}
\newcommand{\Bphi}       {\boldsymbol{\phi}}
\newcommand{\Bvarphi}    {\boldsymbol{\phi}}
\newcommand{\Bchi}       {\boldsymbol{\chi}}
\newcommand{\Bpsi}       {\boldsymbol{\psi}}
\newcommand{\Bomega}     {\boldsymbol{\omega}}


\newcommand{\Amat}{{\rm A}}
\newcommand{\Bmat}{{\rm B}}
\newcommand{\Cmat}{{\rm C}}
\newcommand{\Dmat}{{\rm D}}
\newcommand{\Emat}{{\rm E}}
\newcommand{\Fmat}{{\rm F}}
\newcommand{\Gmat}{{\rm G}}
\newcommand{\Hmat}{{\rm H}}
\newcommand{\Imat}{{\rm I}}
\newcommand{\Jmat}{{\rm J}}
\newcommand{\Kmat}{{\rm K}}
\newcommand{\Lmat}{{\rm L}}
\newcommand{\Mmat}{{\rm M}}
\newcommand{\Nmat}{{\rm N}}
\newcommand{\Omat}{{\rm O}}
\newcommand{\Pmat}{{\rm P}}
\newcommand{\Qmat}{{\rm Q}}
\newcommand{\Rmat}{{\rm R}}
\newcommand{\Smat}{{\rm S}}
\newcommand{\Tmat}{{\rm T}}
\newcommand{\Umat}{{\rm U}}
\newcommand{\Vmat}{{\rm V}}
\newcommand{\Wmat}{{\rm W}}
\newcommand{\Xmat}{{\rm X}}
\newcommand{\Ymat}{{\rm Y}}
\newcommand{\Zmat}{{\rm Z}}


\newcommand{\Asf}{\mathsf{A}}
\newcommand{\Bsf}{\mathsf{B}}
\newcommand{\Csf}{\mathsf{C}}
\newcommand{\Dsf}{\mathsf{D}}
\newcommand{\Esf}{\mathsf{E}}
\newcommand{\Fsf}{\mathsf{F}}
\newcommand{\Gsf}{\mathsf{G}}
\newcommand{\Hsf}{\mathsf{H}}
\newcommand{\Isf}{\mathsf{I}}
\newcommand{\Jsf}{\mathsf{J}}
\newcommand{\Ksf}{\mathsf{K}}
\newcommand{\Lsf}{\mathsf{L}}
\newcommand{\Msf}{\mathsf{M}}
\newcommand{\Nsf}{\mathsf{N}}
\newcommand{\Osf}{\mathsf{O}}
\newcommand{\Psf}{\mathsf{P}}
\newcommand{\Qsf}{\mathsf{Q}}
\newcommand{\Rsf}{\mathsf{R}}
\newcommand{\Ssf}{\mathsf{S}}
\newcommand{\Tsf}{\mathsf{T}}
\newcommand{\Usf}{\mathsf{U}}
\newcommand{\Vsf}{\mathsf{V}}
\newcommand{\Wsf}{\mathsf{W}}
\newcommand{\Xsf}{\mathsf{X}}
\newcommand{\Ysf}{\mathsf{Y}}
\newcommand{\Zsf}{\mathsf{Z}}


\newcommand{\asf}{\mathsf{a}}
\newcommand{\bsf}{\mathsf{b}}
\newcommand{\csf}{\mathsf{c}}
\newcommand{\dsf}{\mathsf{d}}
\newcommand{\esf}{\mathsf{e}}
\newcommand{\fsf}{\mathsf{f}}
\newcommand{\gsf}{\mathsf{g}}
\newcommand{\hsf}{\mathsf{h}}
\newcommand{\isf}{\mathsf{i}}
\newcommand{\jsf}{\mathsf{j}}
\newcommand{\ksf}{\mathsf{k}}
\newcommand{\lsf}{\mathsf{l}}
\newcommand{\msf}{\mathsf{m}}
\newcommand{\nsf}{\mathsf{n}}
\newcommand{\osf}{\mathsf{o}}
\newcommand{\psf}{\mathsf{p}}
\newcommand{\qsf}{\mathsf{q}}
\newcommand{\rsf}{\mathsf{r}}
\newcommand{\ssf}{\mathsf{s}}
\newcommand{\tsf}{\mathsf{t}}
\newcommand{\usf}{\mathsf{u}}
\newcommand{\vsf}{\mathsf{v}}
\newcommand{\wsf}{\mathsf{w}}
\newcommand{\xsf}{\mathsf{x}}
\newcommand{\ysf}{\mathsf{y}}
\newcommand{\zsf}{\mathsf{z}}


\newcommand{\bbs}{{\boldsymbol{\mathsf{b}}}}
\newcommand{\cbs}{{\boldsymbol{\mathsf{c}}}}
\newcommand{\dbs}{\mathsf{d}}
\newcommand{\ebs}{{\boldsymbol{\mathsf{e}}}}
\newcommand{\fbs}{\mathsf{f}}
\newcommand{\gbs}{\mathsf{g}}
\newcommand{\hbs}{\mathsf{h}}
\newcommand{\ibs}{\mathsf{i}}
\newcommand{\jbs}{\mathsf{j}}
\newcommand{\kbs}{\mathsf{k}}
\newcommand{\lbs}{\mathsf{l}}
\newcommand{\mbs}{\mathsf{m}}
\newcommand{\nbs}{{\boldsymbol{\mathsf{n}}}}
\newcommand{\obs}{\mathsf{o}}
\newcommand{\pbs}{{\boldsymbol{\mathsf{p}}}}
\newcommand{\qbs}{\mathsf{q}}
\newcommand{\rbs}{\mathsf{r}}
\newcommand{\sbs}{\mathsf{s}}
\newcommand{\tbs}{{\boldsymbol{\mathsf{t}}}}
\newcommand{\ubs}{{\boldsymbol{\mathsf{u}}}}
\newcommand{\vbs}{{\boldsymbol{\mathsf{v}}}}
\newcommand{\wbs}{\mathsf{w}}
\newcommand{\xbs}{{\boldsymbol{\mathsf{x}}}}
\newcommand{\ybs}{{\boldsymbol{\mathsf{y}}}}
\newcommand{\zbs}{\mathsf{z}}


\providecommand{\Ca}{{\cal A}}
\providecommand{\Cb}{{\cal B}}
\providecommand{\Cc}{{\cal C}}
\providecommand{\Cd}{{\cal D}}
\providecommand{\Ce}{{\cal E}}
\providecommand{\Cf}{{\cal F}}
\providecommand{\Cg}{{\cal G}}
\providecommand{\Ch}{{\cal H}}
\providecommand{\Ci}{{\cal I}}
\providecommand{\Cj}{{\cal J}}
\providecommand{\Ck}{{\cal K}}
\providecommand{\Cl}{{\cal L}}
\providecommand{\Cm}{{\cal M}}
\providecommand{\Cn}{{\cal N}}
\providecommand{\Co}{{\cal O}}
\providecommand{\Cp}{{\cal P}}
\providecommand{\Cq}{{\cal Q}}
\providecommand{\Cr}{{\cal R}}
\providecommand{\Cs}{{\cal S}}
\providecommand{\Ct}{{\cal T}}
\providecommand{\Cu}{{\cal U}}
\providecommand{\Cv}{{\cal V}}
\providecommand{\Cw}{{\cal W}}
\providecommand{\Cx}{{\cal X}}
\providecommand{\Cy}{{\cal Y}}
\providecommand{\Cz}{{\cal Z}}

\providecommand{\CA}{{\cal A}}
\providecommand{\CB}{{\cal B}}
\providecommand{\CC}{{\cal C}}
\providecommand{\CD}{{\cal D}}
\providecommand{\CE}{{\cal E}}
\providecommand{\CF}{{\cal F}}
\providecommand{\CG}{{\cal G}}
\providecommand{\CH}{{\cal H}}
\providecommand{\CI}{{\cal I}}
\providecommand{\CJ}{{\cal J}}
\providecommand{\CK}{{\cal K}}
\providecommand{\CL}{{\cal L}}
\providecommand{\CM}{{\cal M}}
\providecommand{\CN}{{\cal N}}
\providecommand{\CO}{{\cal O}}
\providecommand{\CP}{{\cal P}}
\providecommand{\CQ}{{\cal Q}}
\providecommand{\CR}{{\cal R}}
\providecommand{\CS}{{\cal S}}
\providecommand{\CT}{{\cal T}}
\providecommand{\CU}{{\cal U}}
\providecommand{\CV}{{\cal V}}
\providecommand{\CW}{{\cal W}}
\providecommand{\CX}{{\cal X}}
\providecommand{\CY}{{\cal Y}}
\providecommand{\CZ}{{\cal Z}}


\newcommand{\Cav}{\boldsymbol {\cal A}}
\newcommand{\Cbv}{\boldsymbol {\cal B}}
\newcommand{\Ccv}{\boldsymbol {\cal C}}
\newcommand{\Cdv}{\boldsymbol {\cal D}}
\newcommand{\Cev}{\boldsymbol {\cal E}}
\newcommand{\Cfv}{\boldsymbol {\cal F}}
\newcommand{\Cgv}{\boldsymbol {\cal G}}
\newcommand{\Chv}{\boldsymbol {\cal H}}
\newcommand{\Civ}{\boldsymbol {\cal I}}
\newcommand{\Cjv}{\boldsymbol {\cal J}}
\newcommand{\Ckv}{\boldsymbol {\cal K}}
\newcommand{\Clv}{\boldsymbol {\cal L}}
\newcommand{\Cmv}{\boldsymbol {\cal M}}
\newcommand{\Cnv}{\boldsymbol {\cal N}}
\newcommand{\Cov}{\boldsymbol {\cal O}}
\newcommand{\Cpv}{\boldsymbol {\cal P}}
\newcommand{\Cqv}{\boldsymbol {\cal Q}}
\newcommand{\Crv}{\boldsymbol {\cal R}}
\newcommand{\Csv}{\boldsymbol {\cal S}}
\newcommand{\Ctv}{\boldsymbol {\cal T}}
\newcommand{\Cuv}{\boldsymbol {\cal U}}
\newcommand{\Cvv}{\boldsymbol {\cal V}}
\newcommand{\Cwv}{\boldsymbol {\cal W}}
\newcommand{\Cxv}{\boldsymbol {\cal X}}
\newcommand{\Cyv}{\boldsymbol {\cal Y}}
\newcommand{\Czv}{\boldsymbol {\cal Z}}


\providecommand{\bbA}{\mathbb{A}}
\providecommand{\bbB}{\mathbb{B}}
\providecommand{\bbC}{\mathbb{C}}
\providecommand{\bbD}{\mathbb{D}}
\providecommand{\bbE}{\mathbb{E}}
\providecommand{\bbF}{\mathbb{F}}
\providecommand{\bbG}{\mathbb{G}}
\providecommand{\bbH}{\mathbb{H}}
\providecommand{\bbI}{\mathbb{I}}
\providecommand{\bbJ}{\mathbb{J}}
\providecommand{\bbK}{\mathbb{K}}
\providecommand{\bbL}{\mathbb{L}}
\providecommand{\bbM}{\mathbb{M}}
\providecommand{\bbN}{\mathbb{N}}
\providecommand{\bbO}{\mathbb{O}}
\providecommand{\bbP}{\mathbb{P}}
\providecommand{\bbQ}{\mathbb{Q}}
\providecommand{\bbR}{\mathbb{R}}
\providecommand{\bbS}{\mathbb{S}}
\providecommand{\bbT}{\mathbb{T}}
\providecommand{\bbU}{\mathbb{U}}
\providecommand{\bbV}{\mathbb{V}}
\providecommand{\bbW}{\mathbb{W}}
\providecommand{\bbX}{\mathbb{X}}
\providecommand{\bbY}{\mathbb{Y}}
\providecommand{\bbZ}{\mathbb{Z}}

\providecommand{\reel}{\bbR}
\newcommand{\nat}{\bbN}
\newcommand{\ganz}{\bbZ}
\providecommand{\comp}{\bbC}


\providecommand{\FA}{\mathfrak{A}}
\providecommand{\FB}{\mathfrak{B}}
\providecommand{\FC}{\mathfrak{C}}
\providecommand{\FD}{\mathfrak{D}}
\providecommand{\FE}{\mathfrak{E}}
\providecommand{\FF}{\mathfrak{F}}
\providecommand{\FG}{\mathfrak{G}}
\providecommand{\FH}{\mathfrak{H}}
\providecommand{\FI}{\mathfrak{I}}
\providecommand{\FJ}{\mathfrak{J}}
\providecommand{\FK}{\mathfrak{K}}
\providecommand{\FL}{\mathfrak{L}}
\providecommand{\FM}{\mathfrak{M}}
\providecommand{\FN}{\mathfrak{N}}
\providecommand{\FO}{\mathfrak{O}}
\providecommand{\FP}{\mathfrak{P}}
\providecommand{\FQ}{\mathfrak{Q}}
\providecommand{\FR}{\mathfrak{R}}
\providecommand{\FS}{\mathfrak{S}}
\providecommand{\FT}{\mathfrak{T}}
\providecommand{\FU}{\mathfrak{U}}
\providecommand{\FV}{\mathfrak{V}}
\providecommand{\FW}{\mathfrak{W}}
\providecommand{\FX}{\mathfrak{X}}
\providecommand{\FY}{\mathfrak{Y}}
\providecommand{\FZ}{\mathfrak{Z}}

\providecommand{\Fa}{\mathfrak{a}}
\providecommand{\Fb}{\mathfrak{b}}
\providecommand{\Fc}{\mathfrak{c}}
\providecommand{\Fd}{\mathfrak{d}}
\providecommand{\Fe}{\mathfrak{e}}
\providecommand{\Ff}{\mathfrak{f}}
\providecommand{\Fg}{\mathfrak{g}}
\providecommand{\Fh}{\mathfrak{h}}
\providecommand{\Fi}{\mathfrak{i}}
\providecommand{\Fj}{\mathfrak{j}}
\providecommand{\Fk}{\mathfrak{k}}
\providecommand{\Fl}{\mathfrak{l}}
\providecommand{\Fm}{\mathfrak{m}}
\providecommand{\Fn}{\mathfrak{n}}
\providecommand{\Fo}{\mathfrak{o}}
\providecommand{\Fp}{\mathfrak{p}}
\providecommand{\Fq}{\mathfrak{q}}
\providecommand{\Fr}{\mathfrak{r}}
\providecommand{\Fs}{\mathfrak{s}}
\providecommand{\Ft}{\mathfrak{t}}
\providecommand{\Fu}{\mathfrak{u}}
\providecommand{\Fv}{\mathfrak{v}}
\providecommand{\Fw}{\mathfrak{w}}
\providecommand{\Fx}{\mathfrak{x}}
\providecommand{\Fy}{\mathfrak{y}}
\providecommand{\Fz}{\mathfrak{z}}

\newcommand*{\textEUSP}[2]{\langle{#1},{#2}\rangle} 

\newcommand*{\EUSP}[2]{\left<{#1},{#2}\right>} 
\newcommand*{\DP}[2]{\left<{#1},{#2}\right>} 
\newcommand*{\cross}[2]{{#1}\times{#2}}       

\providecommand*{\wt}[1]{\widetilde{#1}}
\providecommand*{\wh}[1]{\widehat{#1}}
\providecommand{\DS}{\displaystyle}
\newcommand{\dis}{\displaystyle}


\providecommand*{\abs}[1]{\left|{#1}\right|} 
\providecommand*{\N}[1]{\left\|{#1}\right\|} 
\providecommand*{\Nnormal}[1]{\|{#1}\|} 
\providecommand*{\Nbig}[1]{\big\|{#1}\big\|} 
\providecommand*{\NBig}[1]{\Big\|{#1}\Big\|} 
\providecommand*{\textN}[1]{\|{#1}\|} 

\newcommand*{\IP}[2]{\left({#1},{#2}\right)} 
\providecommand*{\NN}[1]{\N{#1}_{0}}         
\newcommand*{\SN}[1]{\left|{#1}\right|}      
\newcommand*{\textSN}[1]{|{#1}|}      

\providecommand*{\abs}[1]{\left|{#1}\right|} 
\newcommand*{\SP}[2]{\left({#1},{#2}\right)} 
\newcommand{\xLinf}[1]{L^{\infty}({#1})}
\newcommand{\NxLinf}[2]{\left\|{#1}\right\|_{\xLinf{#2}}}
\newcommand*{\NxLone}[2]{\left\|{#1}\right\|_{L^{1}({#2})}}
\newcommand*{\NxLtwo}[2]{\left\|{#1}\right\|_{L^{2}({#2})}}
\newcommand*{\Pol}[1]{\Cp_{#1}}
\newcommand*{\Spline}[2]{\Cs_{#1,#2}}

\newcommand*{\LRP}[1]{\left(#1\right)}
\newcommand*{\LRSP}[1]{\left[#1\right]}

\newcommand{\DX}{\mathrm{d}\Bx}
\newcommand{\dS}{\mathrm{d}S}
\providecommand{\ds}{\mathrm{d}\vec{\Bs}}

\newcommand{\compemb}{\overset{c}{\hookrightarrow}} 
\newcommand{\contemb}{\hookrightarrow}              


\newcommand*{\mat}[1]{\mathbf{#1}} 
\newcommand*{\Op}[1]{\mathsf{#1}} 
\newcommand*{\blf}[1]{\mathsf{#1}} 
\newcommand*{\cvec}[1]{\vec{#1}} 
\newcommand*{\conj}[1]{\overline{#1}} 

\newcommand*{\rst}[2]{\left.\mathstrut#1\right|_{#2}}

\newcommand{\defaultdomain}{\CX;\rho}
\newcommand{\defaultboundary}{\partial\defaultdomain}



\providecommand*{\Lp}[2][\defaultdomain]{L^{#2}({#1})}
\newcommand*{\Lpv}[2][\defaultdomain]{\BL^{#2}({#1})}
\newcommand*{\NLp}[3][\defaultdomain]{\N{#2}_{\Lp[#1]{#3}}}
\newcommand*{\NLpv}[3][\defaultdomain]{\N{#2}_{\Lpv[#1]{#3}}}

\newcommand*{\Ltwo}[1][\defaultdomain]{\Lp[#1]{2}}
\newcommand*{\Ltwov}[1][\defaultdomain]{\Lpv[#1]{2}}
\newcommand*{\NLtwo}[2][\defaultdomain]{\NLp[#1]{#2}{2}}
\newcommand*{\NLtwov}[2][\defaultdomain]{\NLpv[#1]{#2}{2}}
\newcommand*{\SPLtwo}[3][\defaultdomain]{\SP{#2}{#3}_{\Ltwo[#1]}}
\newcommand*{\Linf}[1][\defaultdomain]{L^{\infty}({#1})}
\newcommand*{\NLinf}[2][\defaultdomain]{\N{#2}_{\Linf[{#1}]}}

\newcommand{\defaultseqdomain}{\bbZ}
\newcommand*{\lp}[2][\defaultseqdomain]{l^{#2}({#1})}
\newcommand*{\lpv}[2][\defaultseqdomain]{\Bl^{#2}({#1})}
\newcommand*{\Nlp}[3][\defaultseqdomain]{\N{#2}_{\lp[#1]{#3}}}
\newcommand*{\Nlpv}[3][\defaultseqdomain]{\N{#2}_{\lpv[#1]{#3}}}
\newcommand*{\ltwo}[1][\defaultseqdomain]{\lp[#1]{2}}
\newcommand*{\ltwov}[1][\defaultseqdomain]{\lpv[#1]{2}}
\newcommand*{\Nltwo}[2][\defaultseqdomain]{\Nlp[#1]{#2}{2}}
\newcommand*{\Nltwov}[2][\defaultseqdomain]{\Nlpv[#1]{#2}{2}}
\newcommand*{\linf}[1][\defaultseqdomain]{l^{\infty}({#1})}
\newcommand*{\Nlinf}[2][\defaultseqdomain]{\N{#2}_{\linf[{#1}]}}

\newcommand*{\Wmp}[3][\defaultdomain]{W^{{#2},{#3}}({#1})}
\newcommand*{\Wmpv}[3][\defaultdomain]{\BW^{{#2},{#3}}({#1})}
\newcommand*{\bWmp}[4][\defaultdomain]{W_{#4}^{{#2},{#3}}({#1})}    
\newcommand*{\bWmpv}[4][\defaultdomain]{\BW_{#4}^{{#2},{#3}}({#1})}
\newcommand*{\zbWmp}[3][\defaultdomain]{\bWmp[#1]{#2}{#3}{0}}       
\newcommand*{\zbWmpv}[3][\defaultdomain]{\bWmpv[#1]{#2}{#3}{0}}
\newcommand*{\NWmp}[4][\defaultdomain]{\N{#2}_{\Wmp[#1]{#3}{#4}}}
\newcommand*{\NWmpv}[4][\defaultdomain]{\N{#2}_{\Wmpv[#1]{#3}{#4}}}
\newcommand*{\SNWmp}[4][\defaultdomain]{\SN{#2}_{\Wmp[#1]{#3}{#4}}}
\newcommand*{\SNWmpv}[4][\defaultdomain]{\SN{#2}_{\Wmp[#1]{#3}{#4}}}

\newcommand*{\Hm}[2][\defaultdomain]{H^{#2}({#1})}
\newcommand*{\Hmv}[2][\defaultdomain]{\BH^{#2}({#1})}
\newcommand*{\bHm}[3][\defaultdomain]{H_{#3}^{#2}({#1})}
\newcommand*{\zbHm}[3][\defaultdomain]{\bHm[#1]{#2}{0}}
\newcommand*{\bHmv}[3][\defaultdomain]{\BH_{#3}^{#2}({#1})}
\newcommand*{\zbHmv}[3][\defaultdomain]{\bHmv[#1]{#2}{0}}
\newcommand*{\NHm}[3][\defaultdomain]{{\N{#2}}_{\Hm[{#1}]{#3}}}
\newcommand*{\NHmv}[3][\defaultdomain]{{\N{#2}}_{\Hmv[{#1}]{#3}}}
\newcommand*{\SNHm}[3][\defaultdomain]{{\SN{#2}}_{\Hm[{#1}]{#3}}}
\newcommand*{\SNHmv}[3][\defaultdomain]{{\SN{#2}}_{\Hmv[{#1}]{#3}}}

\newcommand*{\Hone}[1][\defaultdomain]{\Hm[#1]{1}}
\newcommand*{\Honev}[1][\defaultdomain]{\Hmv[#1]{1}}
\newcommand*{\bHone}[2][\defaultdomain]{\bHm[#1]{1}{#2}}
\newcommand*{\bHonev}[2][\defaultdomain]{\bHmv[#1]{1}{#2}}
\newcommand*{\zbHone}[1][\defaultdomain]{\bHm[#1]{1}{0}}
\newcommand*{\zbHonev}[1][\defaultdomain]{\bHmv[#1]{1}{0}}
\newcommand*{\NHone}[2][\defaultdomain]{{\N{#2}}_{\Hone[{#1}]}}
\newcommand*{\NHonev}[2][\defaultdomain]{{\N{#2}}_{\Honev[{#1}]}}
\newcommand*{\SNHone}[2][\defaultdomain]{{\SN{#2}}_{\Hone[{#1}]}}
\newcommand*{\SNHonev}[2][\defaultdomain]{{\SN{#2}}_{\Honev[{#1}]}}

\newcommand{\hlb}{\frac{1}{2}}
\providecommand*{\Hh}[1][\defaultboundary]{\Hm[#1]{\hlb}}
\newcommand*{\Hhv}[1][\defaultboundary]{\Hmv[#1]{\hlb}}
\newcommand*{\NHh}[2][\defaultboundary]{\N{#2}_{\Hh[#1]}}
\newcommand*{\NHhv}[2][\defaultboundary]{\N{#2}_{\Hhv[#1]}}
\newcommand*{\SNHh}[2][\defaultboundary]{\SN{#2}_{\Hh[#1]}}
\newcommand*{\SNHhv}[2][\defaultboundary]{\SN{#2}_{\Hhv[#1]}}

\newcommand*{\Hdiv}[1][\defaultdomain]{\boldsymbol{H}(\Div,{#1})}
\newcommand*{\bHdiv}[2][\defaultdomain]{\boldsymbol{H}_{#2}(\Div,{#1})}
\newcommand*{\zbHdiv}[1][\defaultdomain]{\bHdiv[#1]{0}}
\newcommand*{\kHdiv}[1][\defaultdomain]{\boldsymbol{H}(\Div0,{#1})}
\newcommand*{\bkHdiv}[2][\defaultdomain]{\boldsymbol{H}_{#2}(\Div0,{#1})}
\newcommand*{\zbkHdiv}[1][\defaultdomain]{\bkHdiv[#1]{0}}
\newcommand*{\NHdiv}[2][\defaultdomain]{\N{#2}_{\Hdiv[#1]}}
\newcommand{\Honediv}[1][\defaultdomain]{\boldsymbol{H}^{1}_{0}(\Div0,{#1})}

\newcommand*{\Hcurl}[1][\defaultdomain]{\boldsymbol{H}(\curl,{#1})}
\newcommand*{\bHcurl}[2][\defaultdomain]{\boldsymbol{H}_{#2}(\curl,{#1})}
\newcommand*{\zbHcurl}[1][\defaultdomain]{\bHcurl[#1]{0}}
\newcommand*{\kHcurl}[1][\defaultdomain]{\boldsymbol{H}(\curl0,{#1})}
\newcommand*{\bkHcurl}[2][\defaultdomain]{\boldsymbol{H}_{#2}(\curl0,{#1})}
\newcommand*{\zbkHcurl}[1][\defaultdomain]{\bkHcurl[#1]{0}}
\newcommand*{\NHcurl}[2][\defaultdomain]{\N{#2}_{\Hcurl[#1]}}

\newcommand*{\Contm}[2][\defaultdomain]{C^{#2}({#1})}
\newcommand*{\Contmv}[2][\defaultdomain]{\BC^{#2}({#1})}
\newcommand*{\NContm}[3][\defaultdomain]{\N{#2}_{\Cm[#1]{#3}}}
\newcommand*{\NContmv}[3][\defaultdomain]{\N{#2}_{\Cmv[#1]{#3}}}

\newcommand{\period}{]0,T[}
\newcommand{\cperiod}{]0,T[}
\newcommand*{\TLtwo}[2][\period]{L^2({#1};{#2})}
\newcommand*{\TLinf}[2][\period]{L^{\infty}({#1};{#2})}
\newcommand*{\TCont}[2][\cperiod]{C^{0}({#1};{#2})}
\newcommand*{\THm}[3][\period]{H^{#2}({#1};{#3})}

\newcommand*{\BdV}[1][\defaultdomain]{BV_{\mathrm{loc}}({#1})}


\newcommand*{\jump}[2][]{\left[{#2}\right]_{#1}}
\newcommand*{\avg}[2][]{\left\{{#2}\right\}_{#1}}

\newcommand*{\Dp}[2]{\frac{\partial{#1}}{\partial{#2}}}
\newcommand*{\DDp}[2]{\frac{\partial^2{#1}}{\partial{#2}^2}}
\newcommand*{\DpDp}[3]{\frac{\partial^2{#1}}{\partial{#2}\partial{#3}}}
\newcommand*{\dtemp}[1]{\frac{\partial {#1}}{\partial t}}
\newcommand*{\ddtemp}[1]{\frac{\partial^2 {#1}}{\partial t^2}}

\newcommand{\MFH}{{\overbar{X}}}
\newcommand{\cons}{{X^\alpha}}
\newcommand{\oa}{\omega_\alpha}
\newcommand{\Sum}[2]{\sum_{#1=1}^{#2}}
\newcommand{\Gnorm}[1]{\N{\omega_\alpha}_{L^1({#1})}}

\newcommand{\supk}{\underset{k \leq K}{\sup}}
\newcommand{\maxk}{\underset{k \leq K}{\max}}
\newcommand{\maxu}[2]{\underset{1 \leq {#1} \leq {#2}}{\max}}
\newcommand{\supu}[2]{\underset{ 0 \leq {#1} \leq {#2}}{\sup}}
\newcommand{\overleq}[1]{\overset{#1}{\leq}}
\newcommand{\eck}[1]{[#1]}
\newcommand{\bigeck}[1]{\big[ #1 \big]}
\newcommand{\Bigeck}[1]{\Big[ #1 \Big]}
\newcommand{\Biggeck}[1]{\Bigg[ #1 \Bigg]}
\newcommand{\biggeck}[1]{\bigg[ #1 \bigg]}
\newcommand{\rou}[1]{( #1 )}
\newcommand{\bigrou}[1]{\big( #1 \big)}
\newcommand{\Bigrou}[1]{\Big( #1 \Big)}
\newcommand{\Biggrou}[1]{\Bigg( #1 \Bigg)}
\newcommand{\biggrou}[1]{\bigg( #1 \bigg)}
\newcommand{\supttj}{\underset{t \in [t_j, t_{j+1})}{\sup}}
\newcommand{\suptT}{\underset{t \in [0, T]}{\sup}}
\newcommand{\supjJ}{\underset{0 \leq j \leq J}{\sup}}
\newcommand{\supr}[3]{\underset{#1 \leq #2 \leq #3}{\sup}}
\newcommand{\empIPSd}{\mu_{\mathcal{X}^N_k}}
\newcommand{\empIPSc}{\mu_{\mathcal{X}^N_t}}
\newcommand{\MempIPSd}{\mu_{\overbar{\mathcal{X}}^N_k}}
\newcommand{\MempIPSc}{\mu_{\overbar{\mathcal{X}}^N_t}}
\newcommand{\IPS}[1]{\mathcal{X}^N_{#1}}
\newcommand{\MIPS}[1]{\overbar{\mathcal{X}}^N_{#1}}
\newcommand{\Emp}[1]{\mu_{\mathcal{X}^N_{#1}}}
\newcommand{\MEmp}[1]{\mu_{\overbar{\mathcal{X}}^N_{#1}}}

\newcommand{\PRd}{\mathcal{P}(\mathbb{R}^d)}
\newcommand{\iid}{\operatorname{i.i.d.}} 
\newcommand{\ProbStd}{\rou{\Omega, \mathcal{F}, \mathbb{P}}}
\newcommand{\standfilt}{\rou{\mathcal{F}_t}_{t\geq 0}}
\newcommand{\bnuak}{\overbar{\nu}^N_{k}}
\newcommand{\brhoak}{\overbar{\rho}^N_{k}}
\newcommand{\nuak}{\nu^{\alpha}_k}
\newcommand{\orho}{\overbar{\rho}}
\newcommand{\mnu}{\overline{\nu}^N}
\newcommand{\mrho}{\overbar{\rho}^N}

\providecommand{\EMP}{\operatorname*{EMP}}

\newcommand{\tiT}{\Tilde{T}}
\newcommand{\suptTit}{\underset{t \in [0, \tiT]}{\sup}}

\newcommand{\overbar}[1]{\makebox[0pt]{$\phantom{#1}\mkern 1.5mu\overline{\mkern-1.5mu\phantom{#1}\mkern-1.5mu}\mkern 0.2mu$}#1}

\acknowledgments
I would like to express my sincere gratitude to Professor Dr. Massimo Fornasier for supervising this thesis and for introducing me to the fascinating field of consensus-based optimization.

I am also very grateful to Jona Klemenc for many helpful discussions, thoughtful advice, and constant support throughout this work.

\abstractpage
Consensus-based optimization (CBO) is a robust and flexible zero-order, multi-particle class of methods devised to provide provable solutions to high-dimensional global optimization problems, including truly nonconvex and/or nonsmooth cases. The algorithmic framework is grounded in a principled trade-off between stochastic exploration of the search space and deterministic contraction toward a consensus point, the latter being defined, via the Laplace principle, as a proxy for the global minimizer. This thesis is concerned with three recently established variants of CBO. The $\delta$-\textbf{CBO} method is similar to the classic CBO method, with the difference that the diffusion function is constant. The \textbf{consensus freezing} (CF) scheme can be understood as a piecewise CBO method. The overall time horizon is divided into intervals in which the consensus point is fixed. The \textbf{consensus hopping} (CH) scheme is related to CF and can be understood as a deterministic iterative scheme that jumps at each iteration to a Gaussian distribution centered at the consensus point of the previous distribution.
The central mathematical contribution of this thesis is the rigorous derivation of propagation of chaos results for all three schemes. For each variant, we establish that the empirical measure of the interacting particle system converges, as the number of particles $N$ approaches infinity, to the solution of the corresponding mean-field equation. Two regimes are treated systematically: first, on a high-probability set on which the empirical measures are uniformly bounded; second, in the unrestricted setting, where moment bounds and a careful treatment of the inter-particle dependence yield convergence without imposing conditioning on favorable events. The propagation of chaos estimates are derived by the well-known coupling method. Together, these results provide quantitative propagation of chaos estimates for all three CBO variants and establish that the particle systems faithfully approximate the idealized mean-field dynamics in the large-particle limit.
Finally, we demonstrate how the propagation of chaos results can be combined with recently established global convergence results for the mean-field dynamics to derive high-probability convergence guarantees for the empirical mean of the interacting particle systems.

\tableofcontents

\mainmatter\setcounter{page}{1}
\chapter{Introduction}
\label{chap:intro}

Consensus-based optimization (CBO) is a class of stochastic particle methods designed to solve optimization problems of the form
\begin{align*}
    \text{(arg)}\underset{x \in \mathbb{R}^d}{\min} \, f(x)
\end{align*}
for a possibly nonconvex or nonsmooth objective function $f$. Introduced in \cite{Pinnau2017ALimit}, the core idea is to model optimization as the collective evolution of an interacting particle system in which particles are driven toward a weighted consensus point — an average of particle positions that puts more weight in regions where the objective function is small — while stochastic noise promotes global exploration of the search space. The analytical properties and first global convergence results, established for example in \cite{Fornasier2024Consensus-BasedGlobally}, have made CBO a subject of growing mathematical interest.
Convergence proofs of CBO are typically done on a mean-field level.
This raises an important question underlying the theoretical validity of any CBO scheme: \textit{Do the finite-particle dynamics faithfully approximate the mean-field limit as the number of particles $N$ grows?} This is the content of
propagation of chaos: the phenomenon, first proposed by Kac \cite{Kac1956FOUNDATIONSTHEORY} and developed intensively by McKean \cite{McKean1967PropagationEquations}, in which particles in a large interacting system become asymptotically independent, with their marginal distributions converging to the law of a single limiting process. Establishing propagation of chaos quantitatively — with an explicit convergence rate in the particle number $N$ — is both of theoretical importance and of direct practical relevance, as it underpins the use of the mean-field process as a tractable surrogate for the finite-particle system and feeds into global convergence guarantees for the optimization algorithm, as can be seen in \cite[Theorem 3.8]{Fornasier2024Consensus-BasedGlobally}.

This thesis is devoted to proving quantitative propagation of chaos for three CBO variants that have recently been proposed: the $\delta$-\textbf{CBO} scheme \cite{Fornasier2026FromConvergence}, the \textbf{consensus freezing scheme} (CF) \cite{Fornasier2026FromConvergence}, and the \textbf{consensus hopping scheme} (CH) \cite{Riedl2023GradientNeed}. These three schemes form a hierarchy of decreasing structural proximity to the classical CBO dynamics
\begin{align}
    dX^i_t = -\lambda (X^i_t - X^\alpha(\rho^N_t))dt + \sigma \N{X^i_t- X^\alpha(\rho^N_t)} dB^i_t, \quad &X^i_0 \sim \orho_0, \quad i = 1, ..., N,  \\
    & X^\alpha(\rho^N_t) = \frac{\sum_{i=1}^N X^i_t \exp(-\alpha f(X^i_t))}{\sum_{i=1}^N \exp(- \alpha f(X^i_t))}, \notag
    \label{Intro:ClassicalCBO} 
\end{align} 
where $\rho^N_t = \frac{1}{N} \sum^N_{i = 1} \delta_{X^i_t}$ denotes the empirical measure of all particles at time $t$. The constants $\lambda >0$ and $\sigma >0$ are the drift and diffusion parameters, respectively.
The $\delta$-CBO scheme is the closest to the classical CBO scheme, differing only in that the diffusion coefficient is a fixed constant rather than depending on the distance from the consensus.
\begin{align}
     dX^i_t = - s \cdot \lambda (X^i_t - X^\alpha(\rho^N_{t}))dt + \sqrt{s}\delta dB^i_t,  \quad X^i_0 \sim \orho_0, \quad i = 1, ..., N.
     \tag{$\delta$-CBO}
\end{align}
The consensus freezing scheme operates in piecewise continuous time, with the consensus point being fixed over uniformly spaced time intervals. 
This reduces the dynamics within each interval to an Ornstein–Uhlenbeck process that admits an explicit strong solution. At each freezing interval, the dynamics are
\begin{align*}
    dX^i_t = - s \cdot \lambda (X^i_t - X^\alpha(\rho^N_{t_j}))dt + \sqrt{s}\delta dB^i_t,  \quad \forall t \in [t_j, t_{j+1}), \quad X^i_0 \sim \orho_0, \quad  i = 1, ..., N. \tag{CF}
\end{align*}
The hopping scheme operates in discrete time: at step 
$k$, the weighted average is computed with respect to the Gaussian measure $\rho_k = \mathcal{N}(x_{k-1}^{HS}, \sigma^2 I_d)$, which is determined by the previous consensus point $x_{k-1}^{HS}$:
\begin{align*}
    & x^{HS}_k = X^\alpha (\rho_{k}), \, \, \rho_{k} = \mathcal{N}(x^{HS}_{k-1}, \sigma^2 I_d), \quad \text{for } 1 \leq k \in \mathbb{N}, \tag{CH}\\
    & x^{HS}_0 = x_0. 
\end{align*}

For all three schemes, we establish convergence of the empirical particle distributions toward their respective mean-field limits in Wasserstein-$p$ distance, uniformly over finite time horizons. Under sufficiently high moment assumptions, this convergence is obtained at the rate $\mathcal{O}(N^{-1/2})$, while in the unrestricted setting weaker moment assumptions lead to slower rates. The analysis proceeds along two different ideas that correspond to the two main chapters of the thesis.

In Chapter \ref{chap:PropGoodSets}, following the ideas presented in \cite{Fornasier2024Consensus-BasedGlobally}, propagation of chaos is established conditionally on a high-probability set $\Omega_M$ on which the empirical measures of both the interacting and mean-field particle systems have uniformly bounded $p$th moments. Although the probability satisfies $\mathbb{P}(\Omega_M)\xrightarrow[]{M \rightarrow \infty} 1$, the restriction to such good trajectories leaves open the question of unconditional estimates. Chapter \ref{chap:PropChaUnrestricted} solves this by developing an unrestricted approach, following the ideas of \cite{Gerber2025Mean-fieldSampling}, in which the expected consensus point distance is split according to a time-dependent set $\Omega_{N,t}$ on which the empirical measure of just the independent mean-field copies has bounded $p$th moments. This allows us to use a law-of-large-numbers-type bound to ensure $\mathbb{P}(\Omega \setminus \Omega_{N,t}) \xrightarrow[]{N \rightarrow \infty} 0$. Together with an improved Wasserstein stability estimate for the consensus point, which only requires one of the two measures to have uniformly bounded moments, this leads to unrestricted propagation of chaos.

For propagation of chaos in Wasserstein-$p$ distance, the unrestricted approach requires bounded moments of order $q \geq p^2 +p$ to achieve a rate $\mathcal{O}(N^{-1/2})$, whereas for $2p \leq q < p^2+p$ it yields the slower rate $\mathcal{O}\rou{N^{-(q-p)/2p^2}}$. By contrast, the restricted approach requires only bounded moments of order $p$ to achieve the rate $\mathcal{O}(N^{-1/2})$. This distinction has concrete consequences: in the relevant case of propagation of chaos in Wasserstein-$2$ distance under bounded fourth moments, the restricted approach achieves the sharp rate $\mathcal{O}(N^{-1/2})$, while the unrestricted approach yields the slower rate $\mathcal{O}(N^{-1/4})$.

As a main application of the propagation of chaos results, in Section \ref{section:globalconvergence} we demonstrate how they combine with the global convergence results of \cite{Fornasier2026FromConvergence} for the mean-field dynamics to yield high-probability convergence guarantees for the empirical mean of the interacting particle systems toward the mean $\boldsymbol\mu_\infty^\alpha$ of the invariant measure $\boldsymbol\rho^\alpha_\infty$, which concentrates near the global minimizer as $\alpha \to \infty$. The core argument is a triangle inequality separating the statistical error — bounded by our propagation of chaos estimates\footnote{We assume the initial distribution to have bounded moments of sufficiently large order, i.e. $q \geq p^2 +p$ for Wasserstein-$p$ distance.} — from the deterministic exponential convergence of the mean-field limit to stationarity, yielding an explicit high-probability bound on the distance of the empirical mean from $\boldsymbol\mu^\alpha_\infty$.

The thesis is organized as follows. Chapter \ref{chap:prelim} collects the necessary mathematical background, including basic martingale theory and the Burkholder–Davis–Gundy inequalities, \\Wasserstein distances and empirical measures, the general theory of interacting particle systems and the coupling method for proving propagation of chaos, and the classical CBO framework together with the key analytic properties of the consensus point. Chapter \ref{chap:PropGoodSets} then treats the three schemes in turn, establishing restricted propagation of chaos for each. Chapter \ref{chap:PropChaUnrestricted} develops the unrestricted theory, refining selected previous results and applying them to obtain unconditional propagation of chaos estimates for all three schemes. Finally, Section \ref{section:globalconvergence} applies these results to derive global convergence guarantees for the freezing and hopping interacting particle systems.

\section{Related Work}
In general, the consensus point $X^\alpha(\rho^N_t)$ is not globally Lipschitz continuous. Consequently, classical propagation of chaos results, such as McKean’s theorem \cite[Theorem 3.1]{Chaintron2022PropagationApplications}, are not directly applicable. 
The seminal papers \cite{Pinnau2017ALimit, Carrillo2018AnMethod} postulated that propagation of chaos holds for \eqref{Intro:ClassicalCBO} and that its mean-field process is governed by the SDE
\begin{align}
    d\overbar{X}_t = -\lambda (\overbar{X}_t - X^\alpha(\orho_t))dt + \sigma \N{\overbar{X}_t- X^\alpha(\orho_t)} dB_t, \quad X_0 \sim \orho_0. \label{Intro:ClassicalCBOMeanFieldLimit}
\end{align}
Furthermore, in both works \cite{Pinnau2017ALimit, Carrillo2018AnMethod}, the authors show that the distribution of the mean-field process \eqref{Intro:ClassicalCBOMeanFieldLimit} $\orho_t$ satisfies the nonlinear and nonlocal Fokker-Planck equation
\begin{align}
    \partial_t \orho_t = \lambda(\nabla \cdot (x-X^\alpha(\orho_t))\orho_t) + \frac{\sigma^2}{2}\Delta\bigrou{\N{x-X^\alpha(\orho_t)}^2 \, \orho_t}. \label{Intro:FPE}
\end{align}
A rigorous proof of propagation of chaos of \eqref{Intro:ClassicalCBO} toward \eqref{Intro:ClassicalCBOMeanFieldLimit} was left as an open problem. In recent years, several proofs establishing propagation of chaos for consensus-based optimization methods have been developed \cite{Fornasier2024Consensus-BasedGlobally, Gerber2025Mean-fieldSampling, Huang2022OnOptimization}.
The work presented in \cite{Huang2022OnOptimization} proves propagation of chaos for the classical CBO \eqref{Intro:ClassicalCBO} over a finite time horizon $[0, T]$ using Prokhorov's theorem. 
The authors write 
\begin{align*}
    (\rho_t^N)_t = \frac{1}{N} \sum^N_{i=1}\delta_{(X^i_t)_t} \quad \text{and} \quad  \mathcal{L}\bigrou{(\rho^N_t)_t} := \Law \bigrou{(\rho^N_t)_t} \in \mathcal{P}\bigrou{\mathcal{P}(C([0, T]; \mathbb{R}^d))}
\end{align*}
and show that the sequence $\big \{ \mathcal{L}\bigrou{(\rho^N_t)_t} \big \}_N$ is tight in $\mathcal{P}\bigrou{\mathcal{P}(C([0, T]; \mathbb{R}^d))}$. By Prokhorov's theorem, it follows that there exists a subsequence of $\big \{ \mathcal{L}\bigrou{(\rho^N_t)_t} \big \}_N$ that converges in law to a random measure $(\rho_t)_t$. Afterwards they show that this measure solves the to the mean-field process \eqref{Intro:ClassicalCBOMeanFieldLimit} corresponding Fokker-Planck equation \eqref{Intro:FPE} in the weak sense.
This first result on propagation of chaos has the problem of not providing a quantitative estimate in terms of the number of particles $N$. Hence, we cannot deduce a quantitative convergence result for the finite particle system to the global minimizer. \newline
A solution for this problem is presented in \cite{Fornasier2024Consensus-BasedGlobally}. Here,  quantitative propagation of chaos in Wasserstein-$2$ distance is proven with the coupling method \cite[Definition 4.1]{Chaintron2022PropagationMethods} using a probabilistic approach. To circumvent the non-Lipschitzness of the consensus point, they introduce the set
\begin{align*}
    \Omega_M := \bigg\{ \omega \in \Omega : \suptT \,  \frac{1}{N} \Sum{i}{N} \max \Big\{ \N{X^i_t}^4 , \N{\overbar{X}^i_t}^4 \Big \} \leq M \bigg\}, 
\end{align*}
for which they show that $\mathbb{P}(\Omega_M)\xrightarrow[]{M \rightarrow \infty} 1$. Upon this result, they show propagation of chaos by coupling restricted to the high-probability set $\Omega_M$.

\begin{align}
    \underset{1 \leq i \leq N}{\max}\, \suptT \,\mathbb{E}  \biggeck{ \N{X^i_t - \overbar{X}^i_t}^2 \Big | \Omega_M} \leq CN^{-1}, \label{Intro:FKREstimate}
\end{align}
for some constant $C>0$ independent of $N$. However, this still left open the question whether one can prove propagation of chaos for CBO without having to make restrictions on a high-probability set.\newline
The answer to this question is presented in \cite{Gerber2025Mean-fieldSampling}, where the authors remove the conditioning on a set $\Omega_M$ and therefore provide an unrestricted propagation of chaos result for \eqref{Intro:ClassicalCBO}. If the initial distribution $\orho_0$ has bounded moments of all orders, they show that for every $p \geq 1$
\begin{align*}
   \underset{1 \leq i \leq N}{\max}\, \mathbb{E}  \biggeck{  \suptT \, \N{X^i_t - \overbar{X}^i_t}^p} \leq CN^{- \frac{p}{2}},
\end{align*}
see \cite[Theorem 2.6]{Gerber2025Mean-fieldSampling}. For initial distributions with only finitely many bounded moments, a slower rate is obtained, see \cite[Remark 2.7]{Gerber2025Mean-fieldSampling}. In particular, the speed of convergence of \eqref{Intro:FKREstimate} is preserved without restricting to $\Omega_M$. Observe that in each of the aforementioned results, the constant $C$ exhibits an exponential dependence on the final time $T$. This exponential growth prevents these estimates from being applicable to propagation of chaos over an infinite time horizon. \newline
Recently, there have been several results on uniform-in-time propagation of chaos for consensus based optimization \cite{Bayraktar2025Uniform-in-timeOptimization, Huang2025Uniform-in-timeOptimization, Gerber2026Uniform-in-timeOptimization}. We postpone the discussion of these results to Section \ref{Sec:Outlook}.

\section*{Notation}

\begin{itemize}

    \item $\|\cdot\|$ denotes the Euclidean norm on 
    $\mathbb{R}^d$, i.e. $\N{x} := \Bigrou{\sum_{l=1}^d |x_l|^2}^{1/2}$ 
    for $x \in \mathbb{R}^d$. This is the only norm used throughout 
    this thesis unless explicitly stated otherwise.

    \item For a matrix $A \in \mathbb{R}^{d \times d}$, $\mathrm{tr}(A)$ 
    denotes its trace.

    \item For two real numbers $a, b$, we may write 
    $a \vee b := \max\{a,b\}$ and $a \wedge b := \min\{a,b\}$.

    \item For a set $A \subseteq \Omega$, $\mathbbm{1}_A$ denotes 
    its indicator function.

    \item $\Gamma(p) := \int_0^\infty t^{p-1}e^{-t}\,dt$ denotes 
    the Gamma function.

    \item $(\Omega, \mathcal{F}, \mathbb{P})$ denotes a complete 
    filtered probability space. Expectations with respect to 
    $\mathbb{P}$ are written $\mathbb{E}[\cdot]$.

    \item $\mathcal{P}(E)$ denotes the set of all Borel probability 
    measures on a measurable space $(E, \mathcal{E})$. For $p > 0$, 
    the space of probability measures on $\mathbb{R}^d$ with finite 
    $p$th moment is
    \begin{align*}
        \mathcal{P}_p(\mathbb{R}^d) := \bigg\{ \mu \in 
        \mathcal{P}(\mathbb{R}^d) : \int_{\mathbb{R}^d} 
        \|x\|^p\, d\mu(x) < \infty \bigg\}.
    \end{align*}

    \item For $\mu \in \mathcal{P}(\mathbb{R}^d)$ and a measurable 
    function $\varphi$, we write 
    $\mathbb{E}_\mu[\varphi(x)] := \int_{\mathbb{R}^d} \varphi(x)\, 
    d\mu(x)$, abbreviated 
    to $\mathbb{E}[\varphi(x)]$ when the measure is clear from context.

    \item The $p$th moment of $\mu \in \mathcal{P}_p(\mathbb{R}^d)$ 
    is denoted $M_p(\mu) := \mathbb{E}_\mu[\|x\|^p]$, abbreviated 
    to $M_p$ when the measure is clear from context.

    \item For interacting particle systems $(\IPS{t
    })_t$, we often denote their empirical measure with $\rho_t^N := \mu_{\IPS{t}} = \frac{1}{N}\sum_{i=1}^N 
    \delta_{X^i_t}$ when $(\IPS{t})_t$ is clear from context

    \item For $(\MIPS{t})_t$ defined by independent and identically distributed copies $\overbar{X}_t^1, \ldots, \overbar{X}_t^N 
    \sim \orho_t$, their empirical measure is written $\orho^N_t := \mu_{\overline{\mathcal{X}}^N_t} =  \frac{1}{N}\sum_{i=1}^N 
    \delta_{\overline{X}^i_t} $ when $(\MIPS{t})_t$ is clear from context

    \item $\mu^{\otimes N}$ denotes the $N$-fold product of 
    $\mu$

    \item $\mathrm{Law}(X)$ denotes the distribution of a random 
    variable $X$. We write $X \sim \mu$ to mean $\mathrm{Law}(X) 
    = \mu$.

     \item $f : \mathbb{R}^d \to \mathbb{R}_+$ denotes the objective function to be minimized, with unique global minimizer $x^* \in \mathbb{R}^d$. Throughout, $f$ is assumed to satisfy \eqref{A1}--\eqref{A3} of Section \ref{Pre:ToolsForMeanField}.

    \item $\omega_\alpha(x) := e^{-\alpha f(x)}$ for $\alpha > 0$.
    
    \item $\|\omega_\alpha\|_{L^1(\mu)} := \int_{\mathbb{R}^d} 
    \omega_\alpha(x)\, d\mu(x)$.

     \item $X^\alpha(\mu) := \mathbb{E}_{\eta^\alpha_\mu}[x] = 
    \frac{\int_{\mathbb{R}^d} x\,\omega_\alpha(x)\,d\mu(x)}
    {\int_{\mathbb{R}^d} \omega_\alpha(x)\,d\mu(x)}$ 
    denotes the consensus point of $\mu$.

    \item $\eta^\alpha_\mu := \frac{\omega_\alpha \mu}
    {\|\omega_\alpha\|_{L^1(\mu)}}$ denotes the consensus measure associated to $\mu$, abbreviated to $\eta^\alpha$ when the measure is clear from context

\end{itemize}
\chapter{Preliminaries}
\label{chap:prelim}
This chapter introduces the mathematical background required for the analysis of the systems carried out in this thesis. Since our main results concern mean-field limits and propagation of chaos for stochastic interacting particle models, we first review several probabilistic tools and then introduce the framework of interacting particle systems and consensus-based optimization methods.

Section \ref{Pre:ProbabilityTheory} recalls fundamental concepts from probability theory that are used throughout the thesis, including filtrations, martingales, stopping times, quadratic variation, and key tools such as the Burkholder–Davis–Gundy inequalities. We also introduce the Wasserstein-$p$ distance, which provides a natural metric for measuring the convergence of sequences of measures, such as the convergence of distributions of interacting particle systems toward their mean-field limits.

In Section \ref{Pre:InteractingParticleSystems}, we present the general theory of interacting particle systems. This gives us a probabilistic framework for studying the particle dynamics arising in CBO. In the continuous time case, they are often described as systems of stochastic differential equations. We introduce the important object of the empirical measure which can be understood as the distribution of an interacting particle system assigning equal mass to each particle.\\
Section \ref{Pre:PropagationofChaos} then introduces the concept of propagation of chaos, which describes how particles become asymptotically independent and identically distributed as their number approaches infinity. We establish the coupling method that will be used repeatedly in Chapters \ref{chap:PropGoodSets} and \ref{chap:PropChaUnrestricted} to prove propagation of chaos.

Section \ref{SectionCBO} provides an introduction to consensus-based optimization (CBO) methods. We describe the classical CBO dynamics, highlighting how proving propagation of chaos usually works alongside the main challenges we face when we try to prove propagation of chaos.

Finally, Section \ref{Pre:ToolsForMeanField} introduces the main CBO-specific analytical tools employed in the proofs of propagation of chaos. In particular, it provides a Wasserstein stability estimate for the consensus points of two probability measures, as well as a quantitative convergence estimate describing how the empirical consensus of i.i.d. samples approaches the consensus of their underlying distribution. 

Together, these preliminaries establish the theoretical foundation needed to perform the analysis carried out in Chapters \ref{chap:PropGoodSets} and \ref{chap:PropChaUnrestricted}.
\newpage
\section{Some Preliminaries from Probability Theory}
\label{Pre:ProbabilityTheory}
In this section, we collect several basic notions and results from probability theory that are crucial to the concepts introduced and the analysis carried out in this thesis. These tools are standard in the study of stochastic differential equations, interacting particle systems, and mean-field limits, and will be used repeatedly in the proofs of moment bounds, coupling arguments, and propagation-of-chaos results presented in later chapters. Most of the material is drawn from \cite{LeGall2016BrownianCalculus}, with the introduction of Wasserstein distances following \cite{Villani2009OptimalTransport}.

Throughout this section, we consider the probability space $\ProbStd$.
\begin{definition}
A \textbf{filtration} is a sequence of $\sigma$-Algebras $\standfilt$ such that for any $s < t$, we have $\mathcal{F}_s \subseteq \mathcal{F}_t \subseteq \mathcal{F}$.
\newline
We call a stochastic process $\rou{X_t}_{t\geq 0}$ on $\ProbStd$ \textbf{adapted} to $\rou{\mathcal{F}_t}_{t\geq 0}$ if, for any $t \geq 0$, $X_t$ is $\mathcal{F}_t$-measurable.
\end{definition}
\begin{definition}
    A \textbf{stopping time} of the filtration $\standfilt$ is a random variable $T \colon \Omega \longrightarrow [0, \infty]$ such that at any given time $t\geq0$, we have $\{ T \leq t \} \in \mathcal{F}_t$.
\end{definition}
\begin{definition}
    We call a stochastic process $(X_t)_{t \geq 0}$ \textbf{uniformly integrable} if 
    \begin{align*}
        \underset{t \in [0, \infty)}{\sup} \, \mathbb{E} \Bigeck{|X_t| \mathbbm{1}_{\{|X_t| \geq \varepsilon\}}} \xrightarrow[]{\varepsilon \rightarrow \infty}0.
    \end{align*}
\end{definition}
\begin{definition}
    A real-valued stochastic process $(X_t)_{t \geq 0}$ is called a \textbf{martingale} with respect to the filtration $(\mathcal{F}_t)_{t \geq 0}$ if  $(X_t)_{t \geq 0}$ is adapted to $(\mathcal{F}_t)_{t \geq 0}$ and
    \begin{itemize}
        \item [(i)] $\mathbb{E}\bigeck{| X_t|} < \infty \quad $for every $t \geq 0$ and
        \item [(ii)]  $\mathbb{E} \bigeck{X_t  | \mathcal{F}_s} = X_s \quad \mathbb{P} \text{-a.s.}$ for every $0 \leq s \leq t$.
    \end{itemize}
\end{definition}

\begin{definition}
    A stochastic process $M \equiv (M_t)_{t \geq 0}$ with $M_0 = 0$ a.s. and continuous sample paths that is adapted to the filtration $\standfilt$ is a \textbf{continuous local martingale} if there exists a nondecreasing sequence of stopping times $(T_n)_{n \geq 0}$ such that $T_n(\omega) \nearrow \infty$ for every $\omega \in \Omega$ and for every $n\geq 0$, the stopped process $M_{t \wedge T_n}$ is a uniformly integrable martingale.
\end{definition}
\begin{theorem}
    Let $M \equiv (M_t)_{t \geq 0}$ be a continuous local martingale. Then there exists an increasing process $\left\langle  M\right \rangle \equiv (\left\langle  M\right \rangle_t)_{t \geq 0}$ that is unique up to sets with measure zero such that $(M^2_t - \left\langle  M\right \rangle_t)_{t \geq 0}$ is a continuous local martingale. Furthermore, for fixed $t>0$, let $E_n = \big \{ 0 = t^{(n)}_0 < t^{(n)}_1 <...< t^{(n)}_{m_n} = t \big \}$ be an arbitrary increasing subdivision of $[0, t]$ with
    \begin{align*}
        \underset{t_i \in E_n}{\sup} |t_i - t_{i+1}| \xrightarrow[]{n\rightarrow \infty}0.
    \end{align*}
    Then, we have 
    \begin{align*}
        \left\langle  M\right \rangle_t = \lim_{n \rightarrow \infty} \sum^{m_n-1}_{i = 0} \Bigrou{M_{t^{(n)}_{i+1}} -M_{t^{(n)}_{i}}}^2.
    \end{align*}
    We call $\left\langle  M\right \rangle$ the \textbf{quadratic variation} of $M$.
\end{theorem}
\begin{proof}
    \cite[p. 80-84]{LeGall2016BrownianCalculus}.
\end{proof}
\begin{example}
    The Brownian motion $(B_t)_{t \geq 0}$ is a continuous local martingale and its quadratic variation is $\left \langle B \right \rangle_t = t$ for every $t \geq 0$.
    \label{BrownianExample}
\end{example}
\begin{proof}
    One can see immediately that Brownian motion is a martingale and its sample paths are by definition continuous. Therefore, Brownian motion is a continuous local martingale. For the computation of the quadratic variation, we refer the reader to \cite[Proposition 2.16]{LeGall2016BrownianCalculus}.
\end{proof}
\begin{lemma}[Burkholder-Davis-Gundy inequalities]
For every $p>0$, there exist two constants $c_{\mathrm{BDG}}, C_{\mathrm{BDG}} >0$ depending only on $p$ such that, for every continuous local martingale $M$ that vanishes at $t = 0$ and every stopping time $T$
\begin{align*}
    c_{\mathrm{BDG}} \mathbb{E}\biggeck{ \left\langle  M\right \rangle_T^{p/2}} \leq \mathbb{E} \biggeck{ \suptT \N{M_t}^p} \leq C_{\mathrm{BDG}} \mathbb{E} \biggeck{ \left\langle  M\right \rangle_T^{p/2}}.
\end{align*}
\label{Burkholder}
\end{lemma}
\begin{proof}
    \cite[p. 124-126]{LeGall2016BrownianCalculus}
\end{proof}
\begin{definition}
    For $p \in [1,\infty)$, the \textbf{Wasserstein-$p$ distance} between two probability measures $\mu,\nu \in \mathcal P_p(\mathbb R^d)$ is defined as
    \begin{align*}
        W_p(\mu,\nu)
        :=
        \biggrou{
        \underset{\pi \in \Pi(\mu,\nu)}{\inf}
        \int_{\mathbb R^d \times \mathbb R^d} d(x,y)^p \, d\pi(x,y)
        }^{1/p},
    \end{align*}
    where
    \[
    \Pi(\mu,\nu)
    :=
    \big\{
    \pi \in \mathcal P(\mathbb R^d \times \mathbb R^d)
    :
    M_X\pi=\mu \text{ and } M_Y\pi=\nu
    \big\}.
    \]
    Here, $d(x,y)$ denotes a prescribed metric on $\mathbb R^d$; a common choice is $d(x,y)=\N{x-y}$. The measures
    \begin{align*}
        (M_X\pi)(A) := \int_{A \times \mathbb R^d} d\pi(x,y)
        \quad \text{and} \quad
        (M_Y\pi)(B) := \int_{\mathbb R^d \times B} d\pi(x,y),
    \end{align*}
    for all $A,B \in \mathcal B(\mathbb R^d)$, are called the \textbf{first} and \textbf{second marginal} of $\pi$, respectively.
    \label{Pre:Wasserstein}
\end{definition}
The Wasserstein distance can also be written in terms of an expectation, namely
\begin{align*}
    W_p(\mu,\nu)
    =
    \inf \Bigg\{
    \mathbb E\biggeck{d(X,Y)^p}^{1/p}
    ,
    X\sim\mu,\; Y\sim\nu
    \Bigg\}.
\end{align*}
\begin{remark}
For any $p \in [1,\infty)$, the Wasserstein-$p$ distance is a metric on $\mathcal P_p(\mathbb R^d)$. Moreover, the Wasserstein space $\bigrou{\mathcal P_p(\mathbb R^d),W_p}$ is a complete separable metric space. \cite[Theorem 6.9, Theorem 6.18]{Villani2009OptimalTransport}
\end{remark}
We conclude this section by presenting some general statements about deviations of means of random variables. They will be very helpful throughout this thesis. 
The following lemma provides a polynomial upper bound on the upper tail of the empirical mean of i.i.d. random variables. It will be useful in the proofs of unrestricted propagation of chaos presented in Chapter \ref{chap:PropChaUnrestricted}.
\begin{lemma}[{{\cite[Lemma 2.5]{Gerber2025Mean-fieldSampling}}}]
    Let $(Z_i)_{i \in \mathbb{N}}$ be a collection of independent identically distributed real random variables such that $\mathbb{E} \bigeck{|Z_1|^r} < \infty$ for some $r \geq 2$. Then, for every $M > \mathbb{E}\bigeck{|Z_1|} $ there exists a constant $C > 0$, independent of $N$ such that for every $N \in \mathbb{N}$
    \begin{align*}
        \mathbb{P} \biggrou{\frac{1}{N} \sum^N_{i=1} Z_i \geq M} \leq CN^{-r/2}.
    \end{align*}
    \label{GerberLemma2.5}
\end{lemma}
We now state a bound on the expected maximum of independent and identically distributed, centered Gaussian random vectors. This will help us bound the moments of the consensus freezing scheme in Section \ref{ch:MFF} and the consensus hopping scheme in Section \ref{ch:MFH}. The proof requires some auxiliary results and is therefore put in the appendix. We refer to Section \ref{sec:LemaNormalVectorsProof}.

\begin{lemma}
    Let $N \in \mathbb{N}, d \in \mathbb{N}$ and $Z_1, .., Z_N$ be independent, identically distributed Gaussian random vectors such that $Z_1 \sim \mathcal{N}(0, \sigma^2I_d)$
    Then, for every $p \geq 2$, there exists a constant $C_p >0$ depending only on $p$ such that
    \begin{align*}
          \mathbb{E} \biggeck{\underset{1 \leq i \leq N}{\max} \N{Z_i}^p} \leq \sigma^pC_p \Bigrou{d\log (2Nd)}^{\frac{p}{2}}.
    \end{align*}
\label{Pre:LemmaNormalVectors}
\end{lemma}
\section{Interacting Particle Systems}
\label{Pre:InteractingParticleSystems}
Interacting particle systems constitute a fundamental framework for modeling large collections of stochastic agents whose dynamics evolve under mutual influence. In this section, we introduce the probabilistic setting and basic notation for interacting particle systems in both discrete and continuous time. We then establish the notion of empirical measures, which play a central role in describing the collective behavior of particles and in analyzing convergence toward mean-field limits.\vspace{.3 cm}\newline
We begin by defining interacting particle systems as collections of stochastic processes on a common probability space, first in discrete time and subsequently in continuous time.
\begin{definition}
    Let $\bigrou{\Omega, \mathcal{F}, \mathbb{P}}$ be a probability space. An \textbf{interacting particle system} of size $N\in \mathbb{N}$ \textbf{in discrete time} is a collection of $N$ stochastic processes 
    \begin{align*}
   \big(\mathcal{X}^N_k \big)_k = \big(X^i_k\big)^N_k = \big ( X^1_k, ..., X^N_k  \big)_{k \in \mathbb{N}} ,
\end{align*}
where for each $i \leq N$ at each time point $k$, $X^i_k \colon \Omega \longrightarrow \mathbb{R}^d$ is a random variable in $\bigrou{\Omega, \mathcal{F}, \mathbb{P}}$.
\end{definition}
If we specifically want to observe our particle system over a time horizon $\{0, ..., K\}$, we may write:
\begin{align*}
     \mathcal{X}^N_K = \big( \mathcal{X}^N_k)_{k \in\{ 0, ..., K\}}.
\end{align*}
For a fixed $\omega = (\omega_1, ..., \omega_N) \in \Omega^N$, we call $k \longmapsto \mathcal{X}^N_k(\omega) = ( X^1_k(\omega_1), ..., X^N_k(\omega_N)  \big)$ a \textbf{realization} or \textbf{path} of $\big(\mathcal{X}^N_k \big)_k$. 
\begin{definition}
    Let $\bigrou{\Omega, \mathcal{F}, \mathbb{P}}$ be a probability space. An \textbf{interacting particle system} of size $N\in \mathbb{N}$ \textbf{in continuous time} is a collection of $N$ stochastic processes 
    \begin{align*}
   \big(\mathcal{X}^N_t \big)_t =   \big(X^i_t\big)^N_t = \big ( X^1_t, ..., X^N_t  \big)_{t \geq 0},
\end{align*}
where for each $i \leq N$ at each time $t$, $X^i_t \colon \Omega \longrightarrow \mathbb{R}^d$ is a random variable in $\bigrou{\Omega, \mathcal{F}, \mathbb{P}}$.
\label{Pre:IPScts}
\end{definition}
If we specifically want to observe our particle system over a time horizon $ [0, T]$, we may write:
\begin{align*}
     \ \mathcal{X}^N_T =  \big(\mathcal{X}^N_t \big)_{t \in [0, T]}.
\end{align*}
For a fixed $\omega = (\omega_1, ..., \omega_N) \in \Omega^N$, we call $t \longmapsto \mathcal{X}^N_t(\omega) = ( X^1_t(\omega_1), ..., X^N_t(\omega_N)  \big)$ a \textbf{realization} or \textbf{path} of $\big(\mathcal{X}^N_t \big)_t$. 
\vspace{.3 cm}
\newline 
The stochasticity of those interacting particle systems investigated in this thesis comes from Gaussian random vectors in the discrete time case and Brownian motion in the continuous time case, respectively. In particular, in both cases, the noise increments of each particle are independent. This means that the probabilistic analysis of the paths of the discrete time interacting particle systems can be done on the probability space
\begin{align*}
    \Bigrou{(\mathbb{R}^d)^{N \times \mathbb{N}}, \mathcal{B}\bigrou{(\mathbb{R}^d)^{N \times \mathbb{N}}}, \bigrou{\mathcal{N}(0, \sigma^2I_d)^{\otimes N}}^{\otimes \mathbb{N}}}. 
\end{align*}
The continuous time interacting particle systems that we are investigating in this thesis are all defined by stochastic differential equations (SDE). Furthermore, said SDEs admit unique strong solutions. This allows us to carry out their analysis in the Wiener space 
\begin{align*}
    \Bigrou{\bigrou{C_0\bigrou{[0, \infty};\mathbb{R}^d)}^N, \mathcal{B}\Bigrou{\bigrou{C_0([0, \infty);\mathbb{R}^d)}^N}, \mathbb{W}^{\otimes N}},
\end{align*}
where $\mathbb{W}$ denotes the $d$-dimensional Wiener measure (for reference, see \cite[p. 27 - 29]{LeGall2016BrownianCalculus}). When analyzing a specific interacting particle system, we refer to the probability space used as $\ProbStd$.

Interacting particle systems, in general, depend on the position of the individual particles. A central object that represents the instantaneous average of all particle positions in a distributional sense is the empirical measure.
\begin{definition}
    For a collection of random variables $\mathcal{X}^N = (X^1, ..., X^N)$, we define the \textbf{empirical measure} as
    \begin{align*}
        \mu_{\mathcal{X}^N} \colon \Omega^N &\longrightarrow \mathcal{P}(\mathbb{R}^d) \\
       (\omega_1, ..., \omega_N) &\longmapsto \frac{1}{N} \Sum{i}{N} \delta_{X^i (\omega_i)}.
    \end{align*}
    \label{Pre:EmpMeas}
\end{definition}
For an interacting particle system in the sense of Definition \ref{Pre:IPScts}, the empirical measure $\mu_{\mathcal{X}_t^N}$ encodes the distribution of particle positions at a given time $t$, assigning equal probability mass to each particle. 
Moreover, the empirical measure is a random variable between the event space $\Omega^N$ and the space of probability measures $\PRd$. In order to prove this, we follow the ideas presented in \cite[p. 27 and p. 77, Problem 24]{Schervish1995TheoryStatistics} to construct the Borel $\sigma$-Algebra over $\PRd$ with a suited generator. 
\begin{proposition}
    The empirical measure
    \[
    \mu_{\mathcal X^N}\colon (\Omega^N,\mathcal F^N)\to \bigl(\mathcal P(\mathbb R^d),\mathcal C(\mathcal P(\mathbb R^d))\bigr)
    \]
    is a random variable, where $\mathcal{C}(\PRd)$ denotes the $\sigma$-Algebra over $\PRd$ generated by all sets of the form $\big \{ \mu \in \PRd : \mu(B) \leq y \big \}$, for $B \in \mathcal{B}(\mathbb{R}^d)$ and $y \in [0, 1]$.
    \label{Pre:PropEmpMeas}
\end{proposition}
\begin{proof}
    For $B \in \mathcal{B}(\mathbb{R}^d)$ we define the function 
    \begin{align*}
      f_B\colon \PRd \longrightarrow [0, 1],  \quad f_B(\mu) := \mu(B). 
    \end{align*}
    Then $\mathcal{C}(\PRd)$ is generated by
    \begin{align*}
        \big \{ f_B^{-1}(A) : A \in \mathcal{B}([0, 1]), \, B \in \mathcal{B}(\mathbb{R}^d)\big \}.
    \end{align*}
    Therefore, we fix $B \in \mathcal{B}(\mathbb{R}^d)$ and consider arbitrary $A \in \mathcal{B}([0, 1])$. We can immediately deduce the equality
    \begin{align*}
        \big \{ \omega \in \Omega^N  : \mu_{\mathcal{X}^N}(\omega) \in f^{-1}_B(A)\big \} = \big \{ \omega \in \Omega^N  : \bigrou{\mu_{\mathcal{X}^N}(\omega)} (B) \in A \big \}, 
    \end{align*}
    where we used the fact that $f_B^{-1}(A)$ denotes the set of all probability measures $\mu \in \mathcal{P}(\mathbb{R}^d)$ such that $\mu (B) \in A$. We can use the definition of the empirical measure to further compute 
    \begin{align*}
        \big \{ \omega \in \Omega^N  : \bigrou{\mu_{\mathcal{X}^N}(\omega)} (B) \in A \big \} &= \bigg \{ \omega \in \Omega^N  : \biggrou{\frac{1}{N} \sum^N_{i=1} \delta_{X^i(\omega_i)}} (B) \in A \bigg \} \\
        &= \bigg \{ \omega \in \Omega^N  : \biggrou{\frac{1}{N} \sum^N_{i=1} \mathbbm{1}_B(X^i(\omega_i))} \in A \bigg \}. 
    \end{align*}
    This set belongs to $\mathcal{F}^N$, which follows from the measurability of each random variable $X^i$ and the measurability of the indicator functions and, therefore, their sum. By the fact that we have chosen $A$ and $B$ to be arbitrary and that  $\mathcal{C}(\PRd)$ is generated by \[\big \{ f^{-1}_B(A) : A \in \mathcal{B}([0, 1]), B \in \mathcal{B}(\mathbb{R}^d)\big \},\] we have shown that 
    \begin{align*}
        \big \{ \omega \in \Omega^N : \mu_{\mathcal{X}^N}(\omega) \in C \big \} \in \mathcal{F}^N, \quad \forall C \in \mathcal{C}(\PRd), 
    \end{align*}
    and therefore $\mu_{\mathcal{X}^N}$ is a random variable from $\bigrou{\Omega^N, \mathcal{F}^N}$ to $\bigrou{\PRd, \mathcal{C}(\PRd)}$. 
\end{proof}
The following example shows us that the expected Wasserstein distance between two empirical measures can be estimated with the maximum distance between two particles of the respective particle systems.
\begin{example}
\label{WasserBD}
    For two collections of random variables $\mathcal{X}^N, \mathcal{Y}^N$ of size $N \in \mathbb{N}$, the $p$th power of the Wasserstein-$p$ distance between their empirical measures can be estimated as follows.
    \begin{align*}
   \mathbb{E}\,  \Bigeck{W^p_p(\mu_{\mathcal{X}^N}, \mu_{\mathcal{Y}^N})} = \mathbb{E} \biggeck{\underset{\pi \in S_N}{\min} \frac{1}{N} \Sum{i}{N} \N{X^i - Y^{\pi (i)}}^p} &\leq  \frac{1}{N} \Sum{i}{N} \mathbb{E} \biggeck{\N{X^i - Y^i}^p} \notag \\&\leq \underset{1 \leq i \leq N}{\max} \,  \mathbb{E} \biggeck{\N{X^i- Y^i}^p}. 
\end{align*}
\end{example}
The following lemma gives a quantitative bound on the expected distance between the empirical measure of i.i.d. random variables and the distribution from which they are drawn. It is from \cite[Theorem 1]{Fournier2015OnMeasure}.
\begin{lemma}[{{\cite[Theorem 1]{Fournier2015OnMeasure}}}]
Let $1 \leq p<q$ and $\orho \in \mathcal{P}_q(\mathbb{R}^d)$ with $\mathbb{E}_{\overline{\rho}} \bigeck{\N{x}^q}\leq M_q$. Define the particle system $\overbar{\mathcal{X}}^N$ by random variables $\overbar{X}^1, ..., \overbar{X}^N$ that are drawn independently from $\orho$. Then there exists a constant $C = C(p, q, d)$ such that for every $N \geq 1$
\begin{align*}
    \mathbb{E} \bigeck{W^p_p(\mu_{\overline{\mathcal{X}}^N}, \orho)} \leq C M_q^{p/q} \begin{cases}
             N^{-\frac{1}{2}} + N^{- \frac{q-p}{q}} \quad &\text{if } p> \frac{d}{2} \text{ and } p \neq \frac{q}{2}, \\
             N^{-\frac{1}{2}}\log (1+ N) + N^{- \frac{q-p}{q}} \quad &\text{if } p = \frac{d}{2} \text{ and } p \neq \frac{q}{2}, \\
             N^{-\frac{p}{d}} + N^{- \frac{q-p}{q}} \quad &\text{if } p \in (0, \frac{d}{2}) \text{ and } q \neq \frac{d}{d-p}.
         \end{cases}
\end{align*}
\label{Pre:ConvEmp}
\end{lemma}
\section{Propagation of Chaos}
\label{Pre:PropagationofChaos}
Analyzing interacting particle systems can be very difficult, and many interesting problems such as analyzing their long term behavior directly are intractable in the current state of the art. The concept of propagation of chaos provides a feasible framework for analyzing interacting particle systems: as the number of particles grows, the joint law of the system is expected to approach a product measure corresponding to independent copies of a limiting process. This phenomenon underlies so called mean-field approximations and is fundamental in the analysis of many stochastic particle models.

In this section, we formalize propagation of chaos using the Wasserstein-$p$ distance and introduce the notion of mean-field limits. We then present the coupling approach, a widely used technique for proving propagation of chaos, and outline how it applies to particle systems defined by stochastic differential equations. A detailed description of propagation of chaos in CBO can be found in \cite{Gerber2025Mean-fieldSampling}, while \cite{Chaintron2022PropagationMethods} provides a broader overview of the topic in general. The following two definitions are due to \cite{Chaintron2022PropagationMethods}, whereas the description of proving propagation of chaos by coupling is heavily inspired by \cite[Chapter 3]{Chaintron2022PropagationApplications}.
\begin{definition}
Let $p \in [1, \infty)$ and let $\bigrou{\rho^N}_N$ be a collection of measures on $\bigrou{\mathbb{R}^d}^N$. Then we say that \textbf{propagation of chaos in Wasserstein-$p$ distance} towards the probability measure $\orho \in \mathcal{P}\bigrou{\mathbb{R}^d}$ holds, if
\begin{align*}
    W_p\bigrou{\rho^N, \orho^{\otimes N}} \xrightarrow[]{N \rightarrow \infty}0.
\end{align*}
We call $\orho$ the \textbf{mean-field distribution} or \textbf{mean-field limit} of $\bigrou{\rho^N}_N$.\newline
For flows of probability measures $\rho_{[0, T]}^N := (\rho^N_t)_{t \in [0, T]} \in \mathcal{P}\bigrou{C([0, T];\mathbb{R}^d)^N}$, we say that propagation of chaos uniformly in time over $[0, T]$ towards $\orho_{[0, T]} \in \mathcal{P}\bigrou{C([0, T];\mathbb{R}^d)}$ holds, if
\begin{align*}
    W_p\Bigrou{\rho_{[0, T]}^N, \orho_{[0, T]}^{\otimes N}} \xrightarrow[]{N \rightarrow \infty}0,
\end{align*}
and propagation of chaos pointwise in time over $[0, T]$ holds if
\begin{align*}
  \suptT \, W_p\bigrou{\rho_t^N, \orho_t^{\otimes N}} \xrightarrow[]{N \rightarrow \infty}0.
\end{align*}
\end{definition}
If the collection of measures is the distribution of an interacting particle system, we refer to $\orho_k$ in discrete, and $\orho_t$ in continuous time as the mean-field limit of the respective particle system.
A very common way of proving propagation of chaos for an IPS defined by a system of stochastic differential equations is the so called coupling method, which goes as follows. We sample $N$ copies of the mean-field distribution that have the same initial states as the IPS. We then show that the average of the distances between the particles and their respective mean-field counterparts will converge to 0 uniformly over a finite time horizon as $N$ approaches infinity. We will now state this mathematically.
\begin{definition}
Fix $T \in (0, \infty]$. For every $N \in \mathbb{N}$, let $\mathcal{X}^N_T$ be a $N$-particle system with $\Law\big( X^1_t, \dots, X^N_t\big) \in \mathcal{P}\big((\mathbb{R}^d)^N\big)$ for every $0 \leq t \leq T$. We say that \textbf{propagation of chaos by coupling} holds for this system if:\\
For every $N \in \mathbb{N}$, there exists an independent particle system $\overbar{\mathcal{X}}^N_T$, i.e. a collection of $N$ independent stochastic processes $\overbar{\mathcal{X}}^N_T = (\overbar{X}^1_t, ..., \overbar{X}^N_t)_{t \in [0, T]}$, and some $\orho_t \in \mathcal{P}(\mathbb{R}^d)$ such that $\Law\big( \overbar{X}^1_t, \dots \overbar{X}^N_t\big) = \orho_t^{\otimes N}$ for every $0 \leq t \leq T$ and a $\varepsilon(N, T)>0$ with $\varepsilon(N, T) \xrightarrow[]{N \rightarrow \infty}0$ such that
\begin{align*}
    &\frac{1}{N} \Sum{i}{N} \mathbb{E} \biggeck{ \suptT \N{X^i_t - \overbar{X}^i_t}^p } \leq \varepsilon(N, T) \quad (\text{pathwise}) \\
    \text{or} \quad & \frac{1}{N} \Sum{i}{N}  \suptT \,\mathbb{E}  \biggeck{ \N{X^i_t - \overbar{X}^i_t}^p} \leq \varepsilon(N, T) \quad (\text{pointwise}).
\end{align*}
In this case, we will refer to $\orho_{[0, T]}$ as the \textbf{mean-field law} of $\mathcal{X}^N_T$, where $\orho_{[0,T]}$ denotes the pathwise law of one nonlinear process with time marginals $(\orho_t)_{t\in[0,T]}$.
\label{Pre:ChaosCoup}
\end{definition}

Definition \ref{Pre:ChaosCoup} is exactly Definition 4.1 of \cite{Chaintron2022PropagationMethods} with $d(x, y) = \N{x -y}$ as the chosen metric. We will show below that pathwise propagation of chaos by coupling implies propagation of chaos in Wasserstein-$p$ distance uniformly in time. To this end, we consider the path space $\mathcal{C}([0,T];(\mathbb{R}^d)^N)$ and define the metric
\begin{align*}
    d_{[0,T]}(x,\overbar{x}) := \biggrou{\frac{1}{N} \sum^N_{i=1} \suptT \N{x^i_t - \overbar{x}^i_t}^p}^{1/p},
\end{align*}
for $x=(x^1,\dots,x^N), \overbar{x}=(\overbar{x}^1,\dots,\overbar{x}^N) \in \mathcal{C}([0,T];(\mathbb{R}^d)^N)$. This means that we can estimate the Wasserstein-$p$ distance as follows:
\begin{align*}
    W_p^p\Bigrou{\rho^N_{[0, T]}, \orho^{\otimes N}_{[0, T]}} 
    &= \underset{\pi \in \Pi\bigrou{\rho^N_{[0, T]}, \overline{\rho}^{\otimes N}_{[0, T]}}}{\inf}
    \int_{\mathcal{C}([0,T];(\mathbb{R}^d)^N)\times \mathcal{C}([0,T];(\mathbb{R}^d)^N)}
    d_{[0,T]}(x,\overbar{x})^p \, d \pi(x,\overbar{x}) \\
    &= \underset{\pi \in \Pi\bigrou{\rho^N_{[0, T]}, \overline{\rho}^{\otimes N}_{[0, T]}}}{\inf}
    \int_{\mathcal{C}([0,T];(\mathbb{R}^d)^N)\times \mathcal{C}([0,T];(\mathbb{R}^d)^N)}
    \frac{1}{N} \sum^N_{i=1} \suptT \N{x^i_t - \overbar{x}^i_t}^p \, d \pi(x,\overbar{x}) \\
    &\leq \frac{1}{N} \sum^N_{i =1} \mathbb{E} \biggeck{\suptT \N{X^i_t - \overbar{X}^i_t}^p},
\end{align*}
where $\rho^N_{[0, T]}$ and $\orho^{\otimes N}_{[0, T]}$ denote the pathwise law of the interacting particle system and the mean-field limit over $[0, T]$, respectively.
\begin{remark}
In the last step, we bound the infimum by the specific coupling given by the joint law of $(\mathcal{X}^N_T, \overbar{\mathcal{X}}^N_T)$ from Definition \ref{Pre:ChaosCoup}, which is an admissible coupling since its marginals are $\rho^N_{[0,T]}$ and $\orho^{\otimes N}_{[0,T]}$ by construction. More precisely, under this coupling,
\begin{align*}
    \int_{\mathcal{C}([0,T];(\mathbb{R}^d)^N)\times \mathcal{C}([0,T];(\mathbb{R}^d)^N)}
    \frac{1}{N} \sum^N_{i=1} \suptT \N{x^i_t - \overbar{x}^i_t}^p \, d \pi(x,\overbar{x})
    = \frac{1}{N} \sum^N_{i =1} \mathbb{E} \biggeck{\suptT \N{X^i_t - \overbar{X}^i_t}^p}.
\end{align*}
\end{remark}

    Let us take a look at the expected Wasserstein distance between the empirical measure of an IPS and the assumed mean-field law at time $t \in [0, T]$: 
\begin{align*}
    \mathbb{E} \, \Bigeck{W^p_p(\empIPSc, \orho_t)} \leq 2^{p-1}\mathbb{E} \, \Bigeck{W^p_p(\empIPSc, \mu_{\overline{\mathcal{X}}^N_t})} + 2^{p-1}\mathbb{E} \,\Bigeck{W^p_p(\mu_{\overline{\mathcal{X}}^N_t}, \orho_t)}, 
\end{align*}
with $\mu_{\overline{\mathcal{X}}^N_t}$ being the empirical measure of the mean-field particles $(\overbar{X}^1_t, ..., \overbar{X}^N_t)$. The second part of this bound converges to 0 as $N \rightarrow \infty$ by Lemma \ref{Pre:ConvEmp}. Using Example \ref{WasserBD} yields
\begin{align}
    \mathbb{E} \, \Bigeck{W^p_p(\empIPSc, \mu_{\overline{\mathcal{X}_t}^N})} \leq \frac{1}{N} \sum^N_{i=1} \mathbb{E} \biggeck{\N{X^i_t - \overbar{X}^i_t}^p}. \label{WassersteinEmpiricalToDist}
\end{align}
Hence, pathwise coupling implies in expectation convergence of the empirical measure of the interacting particle system to the mean-field limit in Wasserstein-$p$ distance, uniform in time, while the pointwise version does imply it pointwise in time, respectively. More generally, by \cite[Proposition 2.2]{Sznitman1991TopicsChaos}, a sequence of symmetric laws is $\orho$-chaotic if and only if the associated empirical measures converge in law to the deterministic measure $\orho$.

\subsection{Proving Propagation of Chaos by Coupling}\label{Pre:ProvingPropChaos}
If we look at Definition \ref{Pre:ChaosCoup} we immediately ask ourselves the question of how to prove the existence of the mean-field law $\orho_{[0, T]}$ and the associated mean-field particle system $\MIPS{T}$. In many cases, including those studied in this thesis, the interacting particle system of interest is described by a system of stochastic differential equations of the form
\begin{align}
dX_t^i = b(X_t^i, \empIPSc)\,dt + \sigma(X_t^i, \empIPSc)\,dB_t^i, \quad i = 1, \dots, N,
\label{Pre:SDE}
\end{align}
with initial distribution $X_0^i \sim \orho_0$ for every $i \leq N$, and drift and diffusion coefficients
\begin{align*}
b:\mathbb{R}^d \times \mathcal{P}(\mathbb{R}^d) \to \mathbb{R}^d
\quad \text{and} \quad
\sigma:\mathbb{R}^d \times \mathcal{P}(\mathbb{R}^d) \to \mathbb{R}^{d \times d},
\end{align*}
respectively. In numerous instances, we have the convenience that we can claim that the distribution of the mean-field limit is characterized by the following SDE:
\begin{align}
    d\overbar{X}_t = b(\overbar{X}_t, \orho_t)dt + \sigma(\overbar{X}_t, \orho_t) dB_t,
    \label{Pre:MFSDE}
\end{align}
where we assume the initial distribution $\orho_0 \in \mathcal{P}(\mathbb{R}^d)$ to have sufficiently many moments. An issue to be dealt with is that these approaches require some well posedness properties for the particle SDEs and the mean-field process. In particular, we require the finite particle SDEs (\ref{Pre:SDE}) and the mean-field SDE (\ref{Pre:MFSDE}) to admit a unique strong solution. Based on these properties, we introduce the particle system $\MIPS{T}$, where for each $i \leq N$, the particle $\overbar{X}^i_t$ is defined by the stochastic differential equation:
\begin{align}
    d\overbar{X}^i_t = b(\overbar{X}^i_t, \orho_t)dt + \sigma(\overbar{X}^i_t, \orho_t) dB^i_t,
    \label{Pre:MFSDE2}
\end{align}
with initially $\overbar{X}^i_0 \sim \orho_0, \, \overbar{X}^i_0 = X^i_0$ almost surely. This defines an IPS which is coupled to $\IPS{T}$ by Brownian motion paths and initial condition.
The typical coupling approach is now to show that there exists a  non-negative $ \varepsilon(N, T) \xrightarrow[]{N \rightarrow \infty}0$ such that for all $N \in \mathbb{N}$ and all $i \leq N$  
\begin{align*}
    \mathbb{E} \biggeck{ \suptT \N{X^i_t - \overbar{X}^i_t}^p } \leq \varepsilon(N, T).
\end{align*}
In the fortunate case of having global Lipschitz continuity of the drift and diffusion terms of \eqref{Pre:SDE} and \eqref{Pre:MFSDE2} respectively, we immediately get pathwise propagation of chaos by the famous theorem of McKean (\cite[Theorem 3.1]{Chaintron2022PropagationApplications}). Unfortunately, for the methods we are investigating, this is not the case. However, the proofs we will use will follow many of the steps of the Sznitman proof of McKeans theorem (\cite[p. 26 - 28]{Chaintron2022PropagationApplications}). 

We typically use the integral form of \eqref{Pre:SDE} and \eqref{Pre:MFSDE2} to compute
\begin{align}
    & \mathbb{E}  \biggeck{ \supu{t}{T} \N{X^i_t - \overbar{X}^i_t}^p } \notag \\ = \, &\mathbb{E} \Biggeck{ \supu{t}{T} \N{\int^t_0 b\bigrou{X^i_\tau, \mu_{\mathcal{X}^N_\tau}} - b\bigrou{\overbar{X}^i_\tau, \orho_\tau}d\tau + \int^t_0 \sigma\bigrou{X_\tau^i, \mu_{\mathcal{X}^N_\tau}} - \sigma\bigrou{\overbar{X}^i_\tau, \orho_\tau} dB^i_\tau }^p } \notag \\
   \leq \,  &2^{p-1} \Biggrou{\mathbb{E} \Biggeck{ \biggrou{\int^T_0 \N{b\bigrou{X^i_\tau, \mu_{\mathcal{X}^N_\tau}} - b\bigrou{\overbar{X}^i_\tau, \orho_\tau}}d\tau}^p }\notag \\ & \quad   +   C_{\mathrm{BDG}} \mathbb{E} \Biggeck{ \biggrou{\int^T_0 \N{\sigma\bigrou{X_\tau^i, \mu_{\mathcal{X}^N_\tau}} - \sigma\bigrou{\overbar{X}^i_\tau, \orho_\tau}}^2 d\tau }^{\frac{p}{2}} }}\notag \\
   \leq \, &C_T \Biggrou{  \int^T_0 \mathbb{E} \Bigeck{\N{b\bigrou{X^i_\tau, \mu_{\mathcal{X}^N_\tau}} - b\bigrou{\overbar{X}^i_\tau, \orho_\tau}}^p }d\tau \notag \\ &\quad  +   C_{\mathrm{BDG}} \int^T_0 \mathbb{E} \Bigeck{\N{\sigma\bigrou{X_\tau^i, \mu_{\mathcal{X}^N_\tau}} - \sigma\bigrou{\overbar{X}^i_\tau, \orho_\tau}}^p} d\tau  },
   \label{Pre:ProofPropchaoEst1}
\end{align}
where we used Hölder's inequality and the Fubini-Tonelli theorem in the last line, with $C_T := 2^{p-1}\max\big\{T^{p-1}, T^{\frac{p}{2}-1}\big\}$ for $p \geq 2$. Here $\N{\cdot}$ denotes the Frobenius norm when applied to the diffusion matrices.
We then try to control the differences between the drift terms and between the diffusion terms separately, typically with a Grönwall argument (\cite[p. 310, VI]{Walter1998OrdinaryEquations}). Note that the integral of the diffusion part becomes deterministic due to Lemma \ref{Burkholder} together with the statement of Example \ref{BrownianExample} that a Brownian motion fulfills $\left\langle B \right\rangle_t = t$. In both the drift and diffusion, we expand the inside of the expectation with the respective mean-field term and use the triangle inequality to split the estimate into two more manageable terms:
\vspace{.4 cm}
\newline
\textbf{Drift Term}
\begin{align*}
    &\mathbb{E} \Bigeck{\N{b\bigrou{X^i_\tau, \mu_{\mathcal{X}^N_\tau}} - b\bigrou{\overbar{X}^i_\tau, \orho_\tau}}^p } \\ \leq \, & 2^{p-1} \Biggrou{\mathbb{E} \biggeck{\N{b\bigrou{X^i_\tau, \mu_{\mathcal{X}^N_\tau}} - b\bigrou{\overbar{X}^i_\tau, \mu_{\overline{\mathcal{X}}^N_\tau}}}^p } + \mathbb{E} \biggeck{\N{b\bigrou{\overbar{X}^i_\tau, \mu_{\overline{\mathcal{X}}^N_\tau}} - b\bigrou{\overbar{X}^i_\tau, \orho_\tau}}^p }}. 
\end{align*}
\textbf{Diffusion Term}
\begin{align*}
     &\mathbb{E} \Bigeck{\N{\sigma\bigrou{X^i_\tau, \mu_{\mathcal{X}^N_\tau}} - \sigma\bigrou{\overbar{X}^i_\tau, \orho_\tau}}^p }\\ \leq \, & 2^{p-1} \Biggrou{\mathbb{E} \biggeck{\N{\sigma\bigrou{X^i_\tau, \mu_{\mathcal{X}^N_\tau}} - \sigma\bigrou{\overbar{X}^i_\tau, \mu_{\overline{\mathcal{X}}^N_\tau}}}^p } + \mathbb{E} \biggeck{\N{\sigma\bigrou{\overbar{X}^i_\tau, \mu_{\overline{\mathcal{X}}^N_\tau}} - \sigma\bigrou{\overbar{X}^i_\tau, \orho_\tau}}^p }}.
\end{align*}
The strategy is then to derive for the drift part
\begin{align*}
    \mathbb{E} \biggeck{\N{b\bigrou{X^i_\tau, \mu_{\mathcal{X}^N_\tau}} - b\bigrou{\overbar{X}^i_\tau, \mu_{\overline{\mathcal{X}}^N_\tau}}}^p } &\leq K_{b, 1} \mathbb{E} \biggeck{\N{X^i_\tau - \overbar{X}^i_\tau}^p },\\
     \mathbb{E} \biggeck{\N{b\bigrou{X^i_\tau, \mu_{\overline{\mathcal{X}}^N_\tau}} - b\bigrou{\overbar{X}^i_\tau, \orho_\tau}}^p } &\leq K_{b, 2} \varepsilon(N, T) \xrightarrow[]{N \rightarrow \infty} 0, 
\end{align*}
and analogously for the diffusion part
\begin{align*}
    \mathbb{E} \biggeck{\N{\sigma\bigrou{X^i_\tau, \mu_{\mathcal{X}^N_\tau}} - \sigma\bigrou{\overbar{X}^i_\tau, \mu_{\overline{\mathcal{X}}^N_\tau}}}^p } &\leq K_{\sigma, 1} \mathbb{E} \biggeck{\N{X^i_\tau - \overbar{X}^i_\tau}^p }, \\
     \mathbb{E} \biggeck{\N{\sigma\bigrou{X^i_\tau, \mu_{\overline{\mathcal{X}}^N_\tau}} - \sigma\bigrou{\overbar{X}^i_\tau, \orho_\tau}}^p } &\leq K_{\sigma, 2} \varepsilon(N, T)\xrightarrow[]{N \rightarrow \infty} 0,
\end{align*}
where all constants on the right side are independent of $N$. 
Note that for every $\tau \in [0, T]$
\[
\mathbb{E} \biggeck{\N{X^i_\tau - \overbar{X}^i_\tau}^p} \leq \mathbb{E} \biggeck{\underset{s \in [0, \tau]}{\sup} \N{X^i_s - \overbar{X}^i_s}^p}.
\]
Putting the pieces together and absorbing all the constants would, therefore, yield for two constants $c_1, c_2 > 0$ independent of $N$
\begin{align}
    \mathbb{E} \biggeck{\suptT \N{X^i_t - \overbar{X}^i_t}^p} \leq c_1  \int^T_0  \mathbb{E} \biggeck{\underset{s \in [0, \tau]}{\sup} \N{X^i_s - \overbar{X}^i_s}^p}d\tau + c_2 \varepsilon(N, T), \label{Pre:ProvingPropChaosGronwallBound}
\end{align}
and the statement would follow directly from the below stated Grönwall lemma.
\begin{lemma}
    Let $\phi: [0, T] \rightarrow [0,\infty)$ be bounded and measurable and let there be constants $\alpha \geq 0, \beta >0$ such that
    \begin{align*}
        \phi(t) \leq \alpha + \beta \int^t_0 \phi(\tau) \, d \tau \quad \forall t \in [0, T].
    \end{align*}
Then, for all $t \in [0, T]$
\begin{align*}
    \phi(t) \leq \alpha e^{\beta t}.
\end{align*}
\label{GronwallLemma}
\end{lemma}
\begin{proof}
    \cite[p. 310]{Walter1998OrdinaryEquations}
\end{proof}
\section{Consensus-Based Optimization} \label{SectionCBO}
Consensus-based optimization (CBO) algorithms are stochastic particle methods designed to approximate global minimizers of possibly nonconvex objective functions of the form
\begin{align*}
    \text{(arg)}\underset{x \in \mathbb{R}^d}{\min} \, f(x).
\end{align*}
They model optimization as the evolution of an interacting particle system, where particles are driven toward a consensus point—a weighted average of particle positions that emphasizes regions where the objective function is small—while stochastic noise promotes exploration.\\
In this section, we introduce the classical CBO dynamics and define the consensus point. We discuss its main analytical properties, including the lack of global Lipschitz continuity, and present standard assumptions on the objective function that ensure moment bounds and stability estimates, on which analysis is typically done and also carried out in this thesis.
In Chapters \ref{chap:PropGoodSets} and \ref{chap:PropChaUnrestricted}, we analyze some newer CBO variants—$\delta$-CBO, consensus freezing, and consensus hopping—and establish quantitative convergence toward their mean-field limits.
\vspace{.5 cm}
\newline
As stated above, consensus-based optimization algorithms are a class of optimization methods designed to solve global optimization problems
\begin{align*}
    \text{(arg)} \underset{x \in \mathbb{R}^d}{\min} \, f(x),
\end{align*}
where the objective function $f \colon \mathbb{R}^d \longrightarrow \mathbb{R}$ is possibly nonconvex.
To do that, the classical CBO uses the dynamics of interacting particle systems in the form of the following system of SDEs:
\begin{align}
    dX^i_t = -\lambda (X^i_t - X^\alpha(\rho^N_t))dt + \sigma \N{X^i_t- X^\alpha(\rho^N_t)} dB^i_t, \quad X^i_0 \sim \orho_0, \quad i = 1, ..., N, 
    \label{Pre:CBOSDE}
\end{align}
where $\orho_0 \in \PRd$ is the initial distribution and $\lambda, \sigma >0$ are the drift and diffusion parameters, respectively. The main mechanism of CBO is the consensus point $X^\alpha(\rho_t^N)$ of the empirical measure of all particles at time $t$, which is defined in the following.
\begin{definition}
    Let $f \colon \mathbb{R}^d \longrightarrow \mathbb{R}$ be a proper function. The \textbf{consensus point} of a probability measure $\rho \in \mathcal{P}(\mathbb{R}^d)$ is defined as
    \begin{align*}
        X^\alpha \colon \mathcal{P}(\mathbb{R}^d) &\longrightarrow \mathbb{R}^d, \quad
        \rho \longmapsto \frac{\int_{\mathbb{R}^d} x\oa (x) d\rho(x)}{\Gnorm{\rho}},
    \end{align*}
    where $\oa(x) := \exp(-\alpha f(x))$.
\end{definition}
On an event $\omega = (\omega_1, ..., \omega_N)$ of an empirical measure of $N$ particles, the consensus point becomes 
\begin{align*}
    X^\alpha(\mu_{\mathcal{X}^N_t} (\omega)) = \frac{\sum_{i=1}^N X^i_t(\omega_i) \oa (X^i_t(\omega_i))}{\sum_{i=1}^N\oa (X^i_t(\omega_i))}.
\end{align*}
By the measurability of the empirical measure and the continuity of the consensus point, we can directly deduce that the consensus point of an empirical measure is again measurable. 
In general, the consensus point is not globally Lipschitz continuous.

The consensus point describes the interaction among the particles. One way is to think of it as a weighted average of all points, attributing more weight at the positions where the objective function attains lower values. 
More rigorously, the specific choice of the consensus point is motivated by the so-called Laplace Principle which originated in the theory of large deviations. There are several ways to state this property, a very famous one being Varadhan's Lemma\footnote{for reference, the reader is referred to \cite[p. 137, Theorem 4.3.1]{Dembo2010LargeApplications}}. In this thesis, we characterize the Laplace principle by the following theorem, since it gives the most obvious connection to the consensus point.
\begin{theorem}[{{\cite[Proposition 2.1]{Pinnau2017ALimit}}}]
Let $f \in C_b(\mathbb{R}^d; \mathbb{R})$ nonnegative and attain its unique global minimum at $x^* \in \mathbb{R}^d$. Then, if $\rho \in \mathcal{P}(\mathbb{R}^d)$ is absolutely continuous with respect to the Lebesgue measure with $x^* \in \supp(\rho)$, we have
\begin{align*}
    \underset{\alpha \rightarrow \infty}{\lim} \Biggrou{ - \frac{1}{\alpha} \log \biggrou{\int_{\mathbb{R}^d} e^{-\alpha f} d \rho}} = f(x^*).
\end{align*}
\end{theorem}
So, for sufficiently large $\alpha$, the consensus point approximates, among the particle positions $(X_t^1,\dots,X_t^N)$, a point at which $f$ is close to minimal, i.e.
\begin{align*}
    X^\alpha\bigrou{\empIPSc} \approx \underset{X^1_t, ..., X^N_t}{\arg \min} \Big  \{f(X^1_t), ..., f(X^N_t) \Big \}.
\end{align*}
\subsection{Proving Propagation of Chaos for CBO}
\label{Pre:ProvPropChaosCBO}
In this subsection, we sketch the most common way of using the approach of Section \ref{Pre:ProvingPropChaos} to prove propagation of chaos for the CBO system defined by \eqref{Pre:CBOSDE}. As described in Section \ref{Pre:ProvingPropChaos}, we claim that the mean-field limit is described by the single particle process 
\begin{align}
    d\overbar{X}_t = -\lambda (\overbar{X}_t - X^\alpha(\orho_t))dt + \sigma \N{\overbar{X}_t- X^\alpha(\orho_t)} dB_t, \label{Pre:CBOMeanFieldLimit}
\end{align}
with $\orho_0 \in \PRd$ being the initial distribution. To employ the coupling method, we describe the interacting particle system of $N$ copies of \eqref{Pre:CBOMeanFieldLimit}:
\begin{align}
    d\overbar{X}^i_t = -\lambda (\overbar{X}^i_t - X^\alpha(\orho_t))dt + \sigma \N{\overbar{X}^i_t- X^\alpha(\orho_t)} dB^i_t, \quad \overbar{X}_0^i \sim \orho_0, \quad \overbar X_0^i = X_0^i \, \, \text{  a.s.}, \quad i = 1, ..., N,
\end{align}
In standard CBO, we have the fortunate case that inside a norm, the drift and diffusion functions are equal up to constants. Indeed, with $b(x, \mu) = -\lambda\bigrou{x - X^\alpha(\mu)}$ and $\sigma(x, \mu) = \sigma\N{x - X^\alpha(\mu)} I_d$, the reverse triangle inequality gives
\begin{align*}
    \N{b(X^i_\tau, \rho^N_\tau) - b(\overbar{X}^i_\tau, \orho_\tau)} &= \lambda \N{(X^i_\tau - \overbar{X}^i_\tau) - \bigrou{X^\alpha(\rho^N_\tau) - X^\alpha(\orho_\tau)}},\\
    \N{\sigma(X^i_\tau, \rho^N_\tau) - \sigma(\overbar{X}^i_\tau, \orho_\tau)}_F &\leq \sigma \sqrt{d}\, \N{(X^i_\tau - \overbar{X}^i_\tau) - \bigrou{X^\alpha(\rho^N_\tau) - X^\alpha(\orho_\tau)}}.
\end{align*}
Hence, we do not have to treat them separately. With $K_T := 2^{p-1} C_T \bigrou{\lambda^p + C_{\mathrm{BDG}}\, d^{\frac{p}{2}} \sigma^p}$, estimate \eqref{Pre:ProofPropchaoEst1} becomes
\begin{align}
   \mathbb{E} \biggeck{ \suptT \N{X^i_t - \overbar{X}^i_t}^p }
   &\leq C_T \bigrou{\lambda^p +  C_{\mathrm{BDG}}\, d^{\frac{p}{2}} \sigma^p} \int^T_0 \mathbb{E} \Bigeck{\N{(X^i_\tau - \overbar{X}^i_\tau) - \bigrou{X^\alpha(\rho^N_\tau) - X^\alpha(\orho_\tau)}}^p} d\tau \notag \\
   & \leq K_T \int^T_0 \mathbb{E} \biggeck{ \underset{t \in [0, \tau]}{\sup} \N{X^i_t - \overbar{X}^i_t}^p} + \mathbb{E}\Bigeck{\N{X^\alpha(\rho^N_\tau) - X^\alpha(\orho_\tau)}^p} d\tau.
   \label{CBOPROPIntegral}
\end{align}
We see that this gives us almost the opportunity to use Grönwall's lemma. To get our estimate of the desired form as \eqref{Pre:ProvingPropChaosGronwallBound}, we need to estimate the distance term by a function that vanishes as $N$ approaches infinity. So, what we try to achieve is to find for every $\tau \in [0, T]$ constants $C_1, C_2 >0$, which are independent of $N$ and a function $g(N) \xrightarrow[]{N \rightarrow \infty} 0$ such that 
\begin{align}
    \mathbb{E}\Bigeck{\N{X^\alpha(\rho^N_\tau) - X^\alpha(\orho_\tau)}^p} \leq C_1 \mathbb{E}\biggeck{ \underset{t \in [0, \tau]}{\sup}\N{X^i_t - \overbar{X}^i_t}^p} + C_2 g(N).
    \label{Pre:ConsensusDistProblem}
\end{align}
We use a simple expansion with $\orho^N_\tau$ that denotes the empirical measure of the $N$ i.i.d. random variables $(\overbar{X}^1_\tau, ..., \overbar{X}^N_\tau) \sim \orho_\tau^{\otimes N}$ at time $\tau$:
\begin{align}
    \mathbb{E}\Bigeck{\N{X^\alpha(\rho^N_\tau) - X^\alpha(\orho_\tau)}^p} \leq 2^{p-1} \biggrou{\mathbb{E}\Bigeck{\N{X^\alpha(\rho^N_\tau) - X^\alpha(\orho^N_\tau)}^p} + \mathbb{E}\Bigeck{\N{X^\alpha(\orho^N_\tau) - X^\alpha(\orho_\tau)}^p}}.
    \label{Pre:CBOPropchaosConsEstimate1}
\end{align}
Assuming, for the sake of this sketch, that the empirical measures $\rho^N_\tau$ and $\orho^N_\tau$ have $p$th moments bounded by some $M>0$ almost surely, we can invoke Lemma \ref{CarrL2}, which is stated in Section \ref{Pre:ToolsForMeanField}, on the first term. Chapters \ref{chap:PropGoodSets} and \ref{chap:PropChaUnrestricted} show how this assumption can be removed. We obtain
\begin{align}
    \mathbb{E}\Bigeck{\N{X^\alpha(\rho^N_\tau) - X^\alpha(\orho^N_\tau)}^p} \leq C^p_0\, \mathbb{E}\Bigeck{W_p^p(\rho^N_\tau, \orho^N_\tau)},
    \label{CBOPROP1}
\end{align}
for a constant $C_0 >0$ that is independent of $N$.
We can then use Example \ref{WasserBD} and the invariance under permutation of the particle labels to eventually get
\begin{align}
    \mathbb{E}\Bigeck{W_p^p(\rho^N_\tau, \orho^N_\tau)} \leq  \frac{1}{N} \sum^N_{j=1}\mathbb{E}\biggeck{\N{X^j_\tau - \overbar{X}^j_\tau}^p} = \mathbb{E}\biggeck{\N{X^i_\tau - \overbar{X}^i_\tau}^p} \leq \mathbb{E}\biggeck{ \underset{t \in [0, \tau]}{\sup}\N{X^i_t - \overbar{X}^i_t}^p}.
    \label{CBOPROP2}
\end{align}
The second term requires more care. Lemma \ref{CH:ConsBound}, introduced in Section \ref{Pre:ToolsForMeanField}, only controls it on the event that the empirical $p$th moments of the mean-field particles are bounded by some $M>0$. If $\orho_\tau$ has bounded moments of some order $q>p$, Lemma \ref{GerberLemma3.7} in Chapter \ref{chap:PropChaUnrestricted} removes this restriction and gives a constant $C_{\mathrm{EM}}>0$, independent of $N$, such that
\begin{align}
    \mathbb{E}\Bigeck{\N{X^\alpha(\orho^N_\tau) - X^\alpha(\orho_\tau)}^p} \leq C_{\mathrm{EM}} N^{-\frac{p}{2}}.
    \label{CBOPROP3}
\end{align}
We put the estimates \eqref{CBOPROP1}, \eqref{CBOPROP2}, and \eqref{CBOPROP3} into \eqref{Pre:CBOPropchaosConsEstimate1} and get
\begin{align*}
   \mathbb{E}\Bigeck{\N{X^\alpha(\rho^N_\tau) - X^\alpha(\orho_\tau)}^p} \leq 2^{p-1}\biggrou{C^p_0\, \mathbb{E}\biggeck{ \underset{t \in [0, \tau]}{\sup}\N{X^i_t - \overbar{X}^i_t}^p} + C_{\mathrm{EM}} N^{-\frac{p}{2}}}.
\end{align*}
Combining the latter with our original estimate \eqref{CBOPROPIntegral} yields
\begin{align*}
    \mathbb{E}  \biggeck{ \suptT \N{X^i_t - \overbar{X}^i_t}^p }
    \leq K_T\bigrou{1+2^{p-1}C_0^p} \int^T_0 \mathbb{E} \biggeck{ \underset{t \in [0, \tau]}{\sup} \N{X^i_t - \overbar{X}^i_t}^p} d\tau + 2^{p-1} K_T T\, C_{\mathrm{EM}} N^{-\frac{p}{2}}.
\end{align*}
We can now apply Grönwall's lemma (Lemma \ref{GronwallLemma}) to obtain our final estimate
\begin{align*}
   \mathbb{E}  \biggeck{ \suptT \N{X^i_t - \overbar{X}^i_t}^p } \leq 2^{p-1} K_T T\, C_{\mathrm{EM}} N^{-\frac{p}{2}}\, e^{K_T (1+2^{p-1}C_0^p) T} =: \varepsilon(N, T).
\end{align*}
By Definition \ref{Pre:ChaosCoup}, this implies propagation of chaos in Wasserstein-$p$ distance. 
\begin{remark}
 Note that these estimates are not uniform over infinite time horizons, as the above $\varepsilon(N, T)$ grows exponentially with $T$.
\end{remark}
\newpage
\section{Assumptions and Tools for Mean-Field Approximations in CBO}
\label{Pre:ToolsForMeanField}
In this section, we state the assumptions on the objective function that will be used throughout the mean-field analysis of CBO methods and collect several auxiliary results needed later in the proofs of propagation of chaos.
\vspace{.2 cm}

Throughout this thesis, we impose the following CBO-typical assumptions about the objective function:
\begin{itemize}
    \item[\textbf{A1}] $f : \mathbb{R}^d \rightarrow\mathbb{R}_+$ is bounded from below, attains its global minimum at a unique point $x^*$ and there exists a $C_1 \geq 0$ such that \begin{align*}
         \big | f(x) - f(y)\big | \leq C_1 (1 + \N{x} + \N{y}) \N{x-y} \quad \forall x, y \in \mathbb{R}^d.
         \label{A1} \tag{A1} 
    \end{align*}
    \item[\textbf{A2}] There is a $C_2 \geq 0$ such that \begin{align*}
        f(x) - f(x^*) \leq C_2 (1+ \N{x}^2) \quad \forall x \in \mathbb{R}^d.
        \label{A2} \tag{A2}
    \end{align*}
    \item[\textbf{A3}] $\exists C_3 >0, R>0 :$  \begin{align*}
        f(x) - f(x^*) \geq C_3 \N{x}^2 \quad\forall x \in B_R (x^*)^{\mathrm{c}}.
   \label{A3} \tag{A3}
    \end{align*}
    
\end{itemize}
\begin{remark}
   Throughout this thesis, whenever we assert propagation of chaos in the Wasserstein-$p$ distance, we shall, unless explicitly stated otherwise, assume that the objective function $f$ satisfies conditions \eqref{A1}, \eqref{A2}, and \eqref{A3}. We emphasize that the said assumptions do not depend on the order $p$ of the Wasserstein distance or on the moments we work with. 
\end{remark}

\begin{remark}
The setting \eqref{A1}--\eqref{A3} is the special case $s = 1$, $\ell = u = 2$ of the assumptions in \cite[Section 2.1]{Gerber2025Mean-fieldSampling}, where the local Lipschitz condition is allowed to grow like $(1 + \N{x} + \N{y})^s$ and the objective function is allowed to grow like $\N{x}^\ell$ from below and like $\N{x}^u$ from above. For the sake of simplicity, we restrict ourselves to the case $s = 1$, $\ell = u = 2$ and refer to \cite[Section 2.1]{Gerber2025Mean-fieldSampling} for the general case.
\end{remark}
In general, assumptions \eqref{A1}–\eqref{A3} do not imply that 
$f$ is globally Lipschitz or bounded from above; hence McKean’s theorem \cite[Theorem 3.1]{Chaintron2022PropagationApplications} is not directly applicable to prove propagation of chaos by coupling for \eqref{Pre:CBOSDE}. However, throughout this thesis, we shall demonstrate that these assumptions are sufficient to establish proofs of propagation of chaos by appropriately adapting Sznitman’s original argument \cite[p. 26 - 28]{Chaintron2022PropagationApplications}. To this end, we will require several auxiliary results that follow from the assumptions imposed on the objective functions.

We state some lemmas from \cite{Carrillo2018AnMethod} and \cite{Gerber2025Mean-fieldSampling} that will be of great use to us.
The first lemma is from \cite[Lemma 3.1]{Carrillo2018AnMethod}. The proof follows the same line of argument as in the cited reference and is reproduced here for the sake of completeness.\footnote{In~\cite{Carrillo2018AnMethod}, the lemma is stated under the stronger local Lipschitz condition $\big| f(x) - f(y)\big| \leq L_f(\N{x} + \N{y})\N{x-y}$ in place of \eqref{A1}. The proof, however, uses neither \eqref{A1} nor the stronger Lipschitz condition, so the lemma remains valid under our assumptions.}
\begin{lemma}[{{\cite[Lemma 3.1]{Carrillo2018AnMethod}}}]
\label{CarrL1}
 Let f fulfill (\ref{A1}), (\ref{A2}) and $\rho \in \mathcal{P}_2(\mathbb{R}^d)$ with $\mathbb{E}_\rho \bigeck{\N{x}^2} \leq M_2$. \\
  Then
  \begin{align*}
      \frac{\oa(x^*)}{\Gnorm{\rho}} \leq e^{\alpha C_2(1+M_2)}.
  \end{align*}
\end{lemma}
\begin{proof}
   The proof follows as a direct consequence of Jensen's inequality combined with assumption \eqref{A2} imposed on the objective function $f$. We begin by invoking Jensen's inequality to obtain the estimate
    \begin{align*}
        \frac{\oa(x^*)}{\Gnorm{\rho}} = \frac{e^{-\alpha f(x^*)}}{\int_{\mathbb{R}^d} e^{-\alpha f(x)} d\rho(x)} &= \biggrou{\int_{\mathbb{R}^d} e^{-\alpha (f(x)- f(x^*))} d\rho(x)}^{-1}\\ \text{(Jensen)} \quad &\leq \biggrou{e^{- \alpha \int_{\mathbb{R}^d} f(x) - f(x^*) d\rho(x)}}^{-1}.
    \end{align*}
We now rearrange and use \eqref{A2}:
\begin{align*}
    \biggrou{e^{- \alpha \int_{\mathbb{R}^d} f(x) - f(x^*) d\rho(x)}}^{-1} = e^{ \alpha \int_{\mathbb{R}^d} f(x) - f(x^*) d\rho(x)} &\leq e^{ \alpha \int_{\mathbb{R}^d} C_2(1 + \N{x}^2) d\rho(x)}\\ &= e^{ \alpha C_2 \bigrou{1+ \int_{\mathbb{R}^d} \N{x}^2 d\rho(x)}}.
\end{align*}
Finally, we conclude by using the moment bound of $\rho$:
\begin{align*}
    e^{ \alpha C_2 \bigrou{1+ \int_{\mathbb{R}^d} \N{x}^2 d\rho(x)}} \leq e^{\alpha C_2 (1+M_2)}.
\end{align*}
\end{proof}

\begin{lemma}[{{\cite[Proposition 3.1 (a)]{Gerber2025Mean-fieldSampling}}}]
Let $p \geq 1$, let $f$ fulfill \eqref{A1}, \eqref{A2} and \eqref{A3}, and let $\mu, \nu \in \mathcal{P}_p(\mathbb{R}^d)$ such that there exists a $M > 0$ with $\mathbb{E}_\mu \bigeck{\N{x}^p}\leq M$ and $\mathbb{E}_\nu \bigeck{\N{x}^p} \leq M$. Then there exists a constant $C_0 = C_0(\alpha, p, M, C_1, C_2, C_3, R)$ such that
    \begin{align*}
        \N{X^\alpha(\mu) - X^\alpha(\nu) } \leq C_0W_p(\mu, \nu).
    \end{align*}
    \label{CarrL2}
\end{lemma}
Lemma \ref{CarrL2} provides a stability estimate for the distance between the consensus points of two probability measures with finite $p$th moments. In Chapters \ref{chap:PropGoodSets} and \ref{chap:PropChaUnrestricted}, we will see that such estimates play a vital role in the proofs of propagation of chaos.
The following lemma is a generalization of {\cite[Lemma 3.3]{Carrillo2018AnMethod}}.
\begin{lemma}[{{\cite[Proposition A.4]{Gerber2025Mean-fieldSampling}}}]
    Let $p\geq 1$, let f fulfill (\ref{A1}), (\ref{A2}), (\ref{A3}) and let $\rho \in \mathcal{P}_p(\mathbb{R}^d)$ and define the consensus measure $\eta^\alpha := \frac{\oa \rho}{\Gnorm{\rho}}$\footnote{$\eta^\alpha$ is a probability measure on $\mathbb{R}^d$, since $0 < \oa \leq \oa(x^*)$ and hence $0 < \Gnorm{\rho} < \infty$.}. Then there exists a constant $b_p>0$ depending only on $(\alpha, p, C_2, C_3, R)$ such that
   \begin{align*}
       \mathbb{E}_{\eta^\alpha} \bigeck{\N{x}^p} \leq b_p\,  \mathbb{E}_\rho \bigeck{\N{x}^p}.
   \end{align*}
    \label{GerberPropA4}
\end{lemma}
\begin{corollary}
    Let $p \geq 1$, let $f$ fulfill (\ref{A1}), (\ref{A2}), (\ref{A3}), and let $\rho \in \mathcal{P}_p(\mathbb{R}^d)$ with $\eta^\alpha := \frac{\oa \rho}{\Gnorm{\rho}}$. Then there exists a constant $b_p>0$ depending only on $(\alpha, p, C_2, C_3, R)$ such that
    \begin{align*}
       \N{X^\alpha(\rho)}^p \leq  b_p\,  \mathbb{E}_\rho \bigeck{\N{x}^p}.
    \end{align*}
    \label{GerberPropA4Corollary}
\end{corollary}
\begin{proof}
    First, we use the definition of the consensus point:
    \begin{align*}
        \N{X^\alpha(\rho)}^p
        = \N{\int x \frac{\oa(x)}{\Gnorm{\rho}} d \rho(x)}^p
        = \N{\mathbb{E}_{\eta^\alpha} [x]}^p.
    \end{align*}
    Since $\N{\cdot}^p$ is convex, Jensen's inequality and Lemma \ref{GerberPropA4} with exponent $p$ yield
    \begin{align*}
        \N{\mathbb{E}_{\eta^\alpha} [x]}^p
        \leq \mathbb{E}_{\eta^\alpha} \bigeck{\N{x}^p}
        \leq b_p\, \mathbb{E}_\rho \bigeck{\N{x}^p}.
    \end{align*}
\end{proof}
An important step in proofs of propagation of chaos for CBO methods is to ensure convergence of the consensus point associated with the empirical measure of the mean-field particles to that of the mean-field distribution, namely
\begin{align}
    \mathbb{E} \biggeck{ \N{X^\alpha(\orho_t^N) - X^\alpha(\orho_t)}^p} \xrightarrow[]{N \rightarrow \infty} 0. \label{Pre:EmpiricalConsensusConvergencePostulation}
\end{align}
A qualitative proof of this is fairly straightforward if the mean-field measure has uniformly bounded moments, i.e. $\mathbb{E}_{\overline{\rho}_t}[\N{x}^p] \leq M_p$ for all $t \in [0,T]$. Lemma \ref{GerberStability2} in Chapter \ref{chap:PropChaUnrestricted}, applied with $\mu = \orho_t$, then yields
\begin{align}
    \mathbb{E} \biggeck{ \N{X^\alpha(\orho_t^N) - X^\alpha(\orho_t)}^p} \leq C^p \mathbb{E} \Bigeck{W_p^p(\orho_t^N, \orho_t)}\xrightarrow[]{N \rightarrow \infty} 0, \label{Pre:EmpiricalConsensusConvergencePostulationWasserstein}
\end{align}
where the convergence follows from the almost sure convergence $W_p(\orho_t^N, \orho_t) \to 0$, see \cite[Proposition 2.2.6]{Panaretos2020AnSpace}, together with the uniform integrability of $W_p^p(\orho^N_t, \orho_t) \leq 2^{p-1}\bigrou{\mathbb{E}_{\orho^N_t}[\N{x}^p] + M_p}$, which holds since the empirical moments converge in $L^1$ by the strong law of large numbers. 


However, several important CBO results require a quantitative propagation of chaos estimate. A notable example is \cite[Theorem 3.8]{Fornasier2024Consensus-BasedGlobally}. Here, the authors combine convergence of the mean-field distribution to a global minimizer with a quantitative mean-field estimate in order to prove global convergence in probability of a numerical CBO scheme. We recall that the rate of this quantitative mean-field estimate depends on the rate of convergence in \eqref{Pre:EmpiricalConsensusConvergencePostulation}. At first glance, this does not seem to pose a serious problem if there exists a $q>p$ such that the $q$th moments of the mean-field measure and the empirical measure of the mean-field particles are both uniformly bounded by some $M_q>0$. In that case, we could combine \eqref{Pre:EmpiricalConsensusConvergencePostulationWasserstein} with Lemma \ref{Pre:ConvEmp} to obtain a quantitative bound of the form
\begin{align*}
   \mathbb{E}\biggeck{\N{X^\alpha(\orho^N_t) - X^\alpha(\orho_t)}^p}  \leq  C M_q^{p/q} \begin{cases}
             N^{-\frac{1}{2}} + N^{- \frac{q-p}{q}} \quad &\text{if } p> \frac{d}{2} \text{ and } q \neq 2p, \\
             N^{-\frac{1}{2}}\log (1+ N) + N^{- \frac{q-p}{q}} \quad &\text{if } p = \frac{d}{2} \text{ and } q \neq 2p, \\
             N^{- \frac{p}{d}} + N^{- \frac{q-p}{q}} \quad &\text{if } p \in (0, \frac{d}{2}) \\ & \, \, \, \text{ and } q \neq \frac{d}{d-p},
         \end{cases}
\end{align*}
where $C>0$ depends only on $(p, q, d, \alpha, C_1, M_p, M_q)$. The problem, however, is that these bounds are far from optimal. For example, for $d \gg p$, the convergence rate deteriorates as the dimension $d$ increases. The main issue is that Lemma \ref{Pre:ConvEmp} is a general bound on the Wasserstein distance between an empirical measure of i.i.d.\ samples from a distribution and the distribution itself. Therefore, it does not exploit any CBO-specific assumptions such as \eqref{A1}, \eqref{A2}, or \eqref{A3}. In the next lemma, we do exactly that, which yields a much better quantitative result at rate $\mathcal{O}\bigrou{N^{-p/2}}$.
We will generalize a lemma that was first introduced in \cite[Lemma 3.1]{Fornasier2020Consensus-basedLimit} and was also used in \cite[Proposition 3.11, Lemma A2]{Fornasier2024Consensus-BasedGlobally} to ensure propagation of chaos in Wasserstein-2 distance at rate $\mathcal{O}(N^{-1/2})$. Our lemma will generalize the previous results for the $p$th power of the distance. Due to its length, the proof is postponed to Chapter \ref{cha:app1}, Section \ref{Ch:ConsBoundProof}.
\begin{lemma}
    Let $\nu \in \mathcal{P}_p(\mathbb{R}^d)$ for $p \geq 2$ and let for every $N \in \mathbb{N}$, $\overbar{X}^1, ..., \overbar{X}^N$ be independent, identically distributed random variables such that $\overbar{X}^i \sim \nu$ and let $\overbar{\nu}^N = \frac{1}{N} \sum^N_{i=1} \delta_{\overline{X}^i}$ denote their empirical measure. Furthermore, let there exist a $M_p > 0$ such that $\mathbb{E}_\nu \bigeck{\N{x}^p} \leq M_p$. For some $M>0$, define the event
    \begin{align*}
        \Omega_M := \bigg\{ \omega \in \Omega : \frac{1}{N} \sum_{i=1}^N \N{\overbar{X}^i}^p \leq M \bigg\}.
    \end{align*}
    Then, there exists a constant $C_{\mathrm{EM}} >0$ independent of $N$ such that
    \begin{align*}
        \mathbb{E} \biggeck{ \N{X^\alpha(\overbar{\nu}^N) - X^\alpha(\nu)}^p \mathbbm{1}_{\Omega_M}} \leq C_{\mathrm{EM}} N^{-\frac{p}{2}}.
    \end{align*}
\label{CH:ConsBound}
\end{lemma}

\chapter{Propagation of Chaos Restricted on Good Sets}
\label{chap:PropGoodSets}
This chapter analyzes three interacting particle models that arise in the context of consensus-based optimization: the $\delta$-\textbf{CBO} scheme
\cite{Fornasier2026FromConvergence}, the \textbf{consensus freezing scheme} (CF) \cite{Fornasier2026FromConvergence}, and the \textbf{consensus hopping scheme} (CH) \cite{Riedl2023GradientNeed}. Together, these three schemes form a hierarchy of decreasing structural similarity to the classical CBO dynamics: the $\delta$-CBO scheme operates in continuous time with a live consensus point, differing from classical CBO only in that the diffusion coefficient is a fixed constant rather than depending on the distance to the consensus; the freezing scheme operates in continuous time but holds the consensus point fixed over short intervals, reducing the within-interval dynamics to an Ornstein–Uhlenbeck process; and the hopping scheme operates entirely in discrete time by forming consensus at each step on a Gaussian distribution centered at the previous consensus iterate.

Building on the mean-field theoretic framework established in Chapter \ref{chap:prelim}, the central goal of this chapter is to prove quantitative propagation of chaos for all three schemes in Wasserstein-$p$
distance, uniformly over finite time horizons, on a suitable high-probability set 
$\Omega_M$. Under this restricted formulation, the estimates are obtained at rate $\mathcal O(N^{-1/2})$. Despite their structural differences, the three analyses share a common proof architecture consisting of three steps: uniform-in-$N$ moment bounds for the particle trajectories, a high-probability argument ensuring that moments of the empirical measures remain bounded on a set \begin{align*}
    \Omega_M := \bigg\{ \omega \in \Omega : \suptT \,  \frac{1}{N} \Sum{i}{N} \max \Big\{ \N{X^i_t}^p , \N{\overbar{X}^i_t}^p \Big \} \leq M \bigg\},
\end{align*} with $\mathbb{P}(\Omega_M) \geq 1 - C/M \xrightarrow[]{M \to \infty} 1$
 and a coupling argument that quantifies the pathwise distance between the interacting particle system and its mean-field counterpart, conditioned on $\Omega_M$.
The tools used to execute this program differ across the three schemes, reflecting their distinct structures. For the $\delta
$-CBO scheme, the continuous-time dynamics are handled via Grönwall's lemma, in a proof that is the closest in spirit to the classical CBO propagation of chaos arguments of \cite{Fornasier2024Consensus-BasedGlobally, Gerber2025Mean-fieldSampling}, of which it is in essence a simplification. For the freezing scheme, the explicit strong solution of the Ornstein–Uhlenbeck dynamics enables precise moment control at each time interval, while the coupling proof exploits the closed-form distance evolution between particle and mean-field trajectories. For the hopping scheme, the discrete-time setting does not require tools from stochastic analysis; moment estimates instead rely on direct probabilistic inequalities and properties of Gaussian random vectors.

Section \ref{ch:Delta} treats the $\delta$-CBO scheme, Section \ref{ch:MFF} the consensus freezing scheme, and Section \ref{ch:MFH} the consensus hopping scheme.

\section{The $\delta$-CBO Scheme}
\label{ch:Delta}
In this section, we analyze the $\delta$-\textbf{CBO} scheme, a variant of consensus-based optimization introduced in \cite{Fornasier2026FromConvergence}, whose dynamics remain closest in structure to the classical CBO model \eqref{Pre:CBOSDE}. The essential difference is that the diffusion term no longer depends on either the individual particle’s distance to the consensus or on the consensus itself, but is instead governed by a fixed scaling parameter. As a consequence, the $\delta$-CBO system can be viewed as a simplified version of the classical CBO dynamics.

We first introduce the interacting particle system and its corresponding mean-field limit, and construct a coupled system with an independent particle system sharing the same initial condition and the same driving Brownian motions. Due to the similarity with the classical CBO framework, the moment estimates and probabilistic bounds required for the analysis follow from straightforward adaptations of existing arguments as in \cite{Fornasier2024Consensus-BasedGlobally, Gerber2025Mean-fieldSampling}. In particular, we establish uniform-in-finite-time bounds on particle moments and high-probability uniform control of empirical measure moments, which allow us to restrict analysis to high probability trajectories where the system is well-behaved, which allows us to use Lemma \ref{CarrL2}. On these good events, we prove propagation of chaos for the $\delta$-CBO dynamics by coupling, obtaining quantitative convergence rates toward the mean-field limit. More precisely, we show that the expected pathwise error between the interacting particle system and its mean-field counterpart is of order $\mathcal{O}\bigrou{N^{-p/2}}$, which implies convergence in Wasserstein-$p$ distance at rate $\mathcal{O}\bigrou{N^{-1/2}}$.
\vspace{.5 cm}\newline
We consider the $\delta$-CBO dynamics as introduced in \cite{Fornasier2026FromConvergence} for the particle system  $\bigrou{\mathcal{X}_t^N }_t = (X^1_t, ..., X^N_t)_{t \geq 0 }$ by
\begin{align}
    dX^i_t = - s \cdot \lambda (X^i_t - X^\alpha(\rho^N_{t}))dt + \sqrt{s}\delta dB^i_t,  \quad X^i_0 \sim \orho_0, \quad i = 1, ..., N,
    \label{Delta:Scheme}
\end{align}
with parameters $s, \lambda, \delta >0$.
We claim the mean-field limit to be
\begin{align}
    d\overbar{X}_t = - s \cdot \lambda (\overbar{X}_t - X^\alpha(\orho_{t}))dt + \sqrt{s}\delta dB_t,
    \label{Delta:MFSDE}
\end{align}
where $\orho_{t} = \text{Law}(\overbar{X}_{t})$. The system we want to couple with (\ref{Delta:Scheme}) is the independent particle system $\bigrou{\overbar{\mathcal{X}}_t^N }_t = (\overbar{X}^1_t, ..., \overbar{X}^N_t)_{t \geq 0 }$ defined by
\begin{align}
    d\overbar{X}^i_t = - s \cdot \lambda (\overbar{X }^i_t - X^\alpha(\orho_{t}))dt + \sqrt{s}\delta dB^i_t,  \quad \overbar X_0^i = X_0^i \, \, \text{  a.s.}, \quad i = 1, ..., N. 
    \label{Delta:MFScheme}
\end{align}\newline
The existence of strong solutions for \eqref{Delta:Scheme}, \eqref{Delta:MFSDE} and \eqref{Delta:MFScheme} follows directly from \cite[Theorem 2.2]{Gerber2025Mean-fieldSampling} and \cite[Theorem 2.3]{Gerber2025Mean-fieldSampling}, respectively. Because of this, throughout this section, if not otherwise stated, the probabilities and expectations will be taken with respect to the probability space 
\begin{align*}
    \ProbStd
:=
\Bigg(
(\mathbb{R}^d)^N \times \big(C_0([0,\infty);\mathbb{R}^d)\big)^N,\,
\mathcal{B}\Big((\mathbb{R}^d)^N \times \big(C_0([0,\infty);\mathbb{R}^d)\big)^N\Big),\,
\orho_0^{\otimes N}\otimes \mathbb{W}^{\otimes N}
\Bigg).
\end{align*}
\vspace{.3 cm}
\newline
Propagation of chaos by coupling has already been proven (\cite[Lemma 3.10, Proposition 3.11]{Fornasier2024Consensus-BasedGlobally} with a probabilistic argument and \cite[Theorem 2.6]{Gerber2025Mean-fieldSampling} more general) for the CBO method defined by 
\begin{align}
    dX^i_t = -\lambda (X^i_t - X^\alpha(\rho^N_{t}))dt + \sigma \N{X^i_t - X^\alpha(\rho^N_t)} dB^i_t,  \quad X^i_0 \sim \orho_0, \quad i = 1, ..., N.
    \label{CBOGlob}
\end{align}
We note that in terms of analyzing the moments and propagation of chaos, the $\delta$-CBO system (\ref{Delta:Scheme}) is just a simpler version of (\ref{CBOGlob}) as the diffusion depends only on constants. Consequently, the proofs of the boundedness of the moments of the scheme, as well as of the high-probability uniform bounds of the moments of the empirical measures, reduce to straightforward adaptations of the arguments established in \cite[Lemma 3.5]{Gerber2025Mean-fieldSampling} and  \cite[ Lemma 3.10]{Fornasier2024Consensus-BasedGlobally}, respectively. The same goes for the actual proof of propagation of chaos, Theorem \ref{Delta:Thrm} being a simpler version of the proof of \cite[Proposition 3.11]{Fornasier2024Consensus-BasedGlobally}.

\subsection{Propagation of Chaos for the $\delta$-CBO Scheme}
As described above, this section is devoted to proving propagation of chaos for the $\delta$-CBO scheme described in \eqref{Delta:Scheme}. In order to prove propagation of chaos, we need moment bounds, uniformly over finite time horizons, not only of the interacting particle system (\ref{Delta:Scheme}) and the mean-field particle system (\ref{Delta:MFScheme}) but also of their respective empirical measures. We cannot directly prove the boundedness of the moments of the empirical measures uniformly in $N$. This constitutes a great problem for proving propagation of chaos, as we could not use vital results such as Lemma \ref{CarrL2} and Lemma \ref{CH:ConsBound}. A workaround is to rely on a probabilistic argument. More precisely, we want to ensure that the empirical measures have $p$th moments uniformly over $[0, T]$ and uniformly in $N$ with a probability approaching one. We will closely follow the steps used in \cite[Lemma 3.10, Proposition 3.11]{Fornasier2024Consensus-BasedGlobally}. Subsequently, we prove propagation of chaos by coupling, conditioned on trajectories along which the empirical measures have uniformly bounded moments. The proof of propagation of chaos for $\delta$-CBO will be the one closest to the classical approach described in Section \ref{Pre:ProvingPropChaos}.
\begin{lemma}
    Let $T > 0$ be the final time. Consider the interacting particle systems $\IPS{T}$ and $\MIPS{T}$ defined by the SDE systems \eqref{Delta:Scheme} and \eqref{Delta:MFScheme}, respectively, with initial data following the distribution $\orho_0 \in \mathcal{P}_p(\mathbb{R}^d)$. Then there exist constants $C_{\mathrm{D}} > 0$ and $\overbar{C}_{\mathrm{D}} > 0$ independent of $N$ such that for every $N \in \mathbb{N}$ and every $i \leq N$
    \begin{align*}
         \mathbb{E} \biggeck{\suptT \N{X_t^i}^p} \leq C_{\mathrm{D}} \quad \text{and} \quad \mathbb{E} \biggeck{\suptT \N{\overbar{X}_t^i}^p} \leq \overbar{C}_{\mathrm{D}}.
    \end{align*}
    \label{NPDelta:UniformMoment}
\end{lemma}
\begin{proof}
We present the proof only for the interacting particle system $\IPS{T}$. The proof for $\MIPS{T}$ proceeds analogously.
Fix $\tiT \in [0, T]$. We begin by computing for an arbitrary $i \leq N$ on a path $\omega \in \Omega$
\begin{align*}
    \suptTit \N{X_t^i} &= \suptTit \N{X^i_0 + \int^t_0 - s \lambda (X^i_\tau - X^\alpha(\rho^N_\tau))\, d\tau + \sqrt{s} \delta \int^t_0 d B^i_\tau}\\ 
    &\leq \N{X^i_0} + s\lambda \int^{\tiT}_0\N{X^i_\tau} \, d\tau + s\lambda \int^{\tiT}_0 \N{X^\alpha(\rho^N_\tau) } \, d\tau + \sqrt{s}\delta \suptTit \N{B^i_t}.
\end{align*}
We use $(a+b+c+d)^p \leq 4^{p-1} (a^p+b^p+c^p+d^p)$ on the latter to estimate the $p$th moment
\begin{align*}
    &\suptTit \N{X_t^i}^p \\  \leq \quad &4^{p-1} \Biggrou{\N{X^i_0}^p + s^p\lambda^p \biggrou{\int^{\tiT}_0\N{X^i_\tau} \, d\tau}^p + s^p\lambda^p \biggrou{\int^{\tiT}_0 \N{X^\alpha(\rho^N_\tau) } \, d\tau}^p + s^{\frac{p}{2}}\delta^p \suptTit \N{B^i_t}^p}.
\end{align*}
Using Hölders inequality on the two integral terms and taking expectations on both sides leaves us with
\begin{align*}
    \mathbb{E} \biggeck{\suptTit \N{X_t^i}^p} &\leq 4^{p-1} \Biggrou{ \mathbb{E} \biggeck{\N{X^i_0}^p} + s^p\lambda^p {\tiT}^{p-1} \int^{\tiT}_0 \mathbb{E} \Bigeck{\N{X^i_\tau}^p} \, d\tau   \\ & \quad + s^p\lambda^p {\tiT}^{p-1} \int^{\tiT}_0 \mathbb{E} \Bigeck{\N{X^\alpha(\rho^N_\tau) }^p} \, d\tau + s^{\frac{p}{2}}\delta^p \mathbb{E} \Biggeck{\suptTit \N{B^i_t}^p}}.
\end{align*}
We use Lemma \ref{Burkholder} in order to bound the Brownian motion part with a constant $C_{\mathrm{BDG}}$ that depends only on $p$, i.e.
\begin{align*}
    \mathbb{E} \Biggeck{\suptTit \N{B^i_t}^p} \leq C_{\mathrm{BDG}}(d{\tiT})^{\frac{p}{2}}.
    \end{align*}
Together with absorbing the constants into one constant $C = C(s, \lambda, \delta, {\tiT}, d, p)$ yields
\begin{align}
    \mathbb{E} \biggeck{\suptTit \N{X_t^i}^p} \leq C \Biggrou{ \mathbb{E} \biggeck{\N{X^i_0}^p} + \int^{\tiT}_0 \mathbb{E} \Bigeck{\N{X^i_\tau}^p} \, d\tau + \int^{\tiT}_0 \mathbb{E} \Bigeck{\N{X^\alpha(\rho^N_\tau) }^p} \, d\tau + 1}.
    \label{NPDelta:B1}
\end{align}
We can see that we are almost there to use Grönwall's lemma to finish the proof. But first, we need to get a bound on the consensus point part. We use Corollary \ref{GerberPropA4Corollary} to estimate
\begin{align*}
    \int^{\tiT}_0 \mathbb{E} \Bigeck{\N{X^\alpha(\rho^N_\tau) }^p} \, d\tau &\leq \int^{\tiT}_0 b_p \mathbb{E} \Bigeck{ \mathbb{E}_{\rho^N_\tau} \bigeck{\N{x}^p}} \, d \tau.\\
\end{align*}
Combining the latter with (\ref{NPDelta:B1}) gives us for a constant $C = C(s, \lambda, \delta, {\tiT}, d, p, b_p)$\footnote{Note that we abuse the notation here in the way that $C$ displays a constant that just depends on model parameters.}
\begin{align*}
    &\mathbb{E} \biggeck{\suptTit \N{X_t^i}^p} \\ \leq \, & C \Biggrou{ 1+ \mathbb{E} \biggeck{\N{X^i_0}^p} + \int^{\tiT}_0 \mathbb{E} \Bigeck{\N{X^i_\tau}^p} \, d\tau + \int^{\tiT}_0 \mathbb{E} \Bigeck{ \mathbb{E}_{\rho^N_\tau} \bigeck{\N{x}^p}} \, d\tau } \\
    \leq \, &C \Biggrou{ 1+ \mathbb{E} \biggeck{\N{X^i_0}^p} + \int^{\tiT}_0 \mathbb{E} \Biggeck{ \underset{s\in [0, \tau]}{\sup}\N{X^i_s}^p} \, d\tau + \int^{\tiT}_0 \mathbb{E} \Biggeck{ \underset{s\in [0, \tau]}{\sup}\mathbb{E}_{\rho^N_s} \bigeck{\N{x}^p}} \, d\tau }\\
    \leq \, &C \Biggrou{ 1+ \mathbb{E} \biggeck{\N{X^i_0}^p} + 2\int^{\tiT}_0 \mathbb{E} \Biggeck{ \underset{s\in [0, \tau]}{\sup}\N{X^i_s}^p} \, d\tau },
\end{align*}
where we used in the last line that
\begin{align}
     \mathbb{E} \Biggeck{ \underset{s\in [0, \tau]}{\sup}\mathbb{E}_{\rho^N_s} \bigeck{\N{x}^p}} =  \mathbb{E} \Biggeck{ \underset{s\in [0, \tau]}{\sup}\frac{1}{N} \sum^N_{j=1} \N{X^j_s}^p} \leq \mathbb{E} \Biggeck{ \frac{1}{N} \sum^N_{j=1} \underset{s\in [0, \tau]}{\sup} \N{X^j_s}^p} = \mathbb{E} \Biggeck{\underset{s\in [0, \tau]}{\sup} \N{X^i_s}^p}. \label{NPDeltaEmpB1}
\end{align}
Lastly, an application of Grönwall's lemma (Lemma \ref{GronwallLemma}) yields
\begin{align}
     \mathbb{E} \biggeck{\suptT \N{X_t^i}^p} \leq C\biggrou{1+ \mathbb{E} \biggeck{\N{X^i_0}^p}}e^{2CT} = \colon C_{\mathrm{D}},
\end{align}
with $C_{\mathrm{D}} = C(s, \lambda, \delta, T, d, p, b_p)$.
\end{proof}
\begin{remark}
    The auxiliary step (\ref{NPDeltaEmpB1}) immediately gives us a uniform-in-time bound for the expected $p$th moment of the empirical measure of $\IPS{T}$:
    \begin{align}
        \mathbb{E} \Biggeck{ \underset{t\in [0, T]}{\sup}\mathbb{E}_{\rho^N_t} \bigeck{\N{x}^p}} \leq \mathbb{E} \biggeck{\suptT \N{X_t^i}^p} \leq C_{\mathrm{D}}.
    \end{align}
\end{remark}
The results of Lemma \ref{NPDelta:UniformMoment} will help us to prove the high-probability statement on the bounded moments of the empirical measures of the particle systems \eqref{Delta:Scheme} and \eqref{Delta:MFScheme}.
\begin{lemma}
   Let $T >0$ be a fixed final time. Consider the interacting particle systems $\IPS{T}$ and $\MIPS{T}$ defined by the SDE systems \eqref{Delta:Scheme} and \eqref{Delta:MFScheme}, respectively, with initial data following the distribution $\orho_0 \in \mathcal{P}_p(\mathbb{R}^d)$. For a fixed \(M>0\), define the set \begin{align*}
       \Omega_M := \bigg\{ \omega \in \Omega : \suptT \,  \frac{1}{N} \Sum{i}{N} \max \Big\{ \N{X^i_t}^p , \N{\overbar{X}^i_t}^p \Big \} \leq M \bigg\}.
   \end{align*} Then there exists a constant $C>0$ independent of $N$ such that
    \begin{align*}
        \mathbb{P}(\Omega_M) \geq 1 - \frac{C}{M}.
    \end{align*}
    \label{MFD:Moment}
\end{lemma}
\begin{proof}
   We obtain with Markov's inequality 
\begin{align*}
        \mathbb{P}(\Omega^c_M) &\leq \mathbb{P}\Biggrou{\suptT \frac{1}{N}\Sum{i}{N}\Bigrou{\N{X^i_t}^p + \N{\overbar{X}^i_t}^p} > M} \\
        &= \mathbb{P}\Biggrou{\suptT \Sum{i}{N}\Bigrou{\N{X^i_t}^p + \N{\overbar{X}^i_t}^p} > NM}\\
        &\leq \frac{\mathbb{E} \bigg[ \suptT \Sum{i}{N}\N{X^i_t}^p\bigg]}{NM} + \frac{\mathbb{E} \bigg[ \suptT \Sum{i}{N}\N{ \overbar{X}^i_t}^p \bigg]}{NM}\\
        &\leq \frac{\Sum{i}{N}\mathbb{E} \bigg[ \suptT \N{X^i_t}^p \bigg]}{NM} + \frac{\Sum{i}{N}\mathbb{E} \bigg[ \suptT \N{ \overbar{X}^i_t}^p \bigg ]}{NM}\\
        &\leq \frac{\maxu{i}{N}\, \mathbb{E} \bigg[ \suptT \N{X^i_t}^p \bigg ]}{M} + \frac{\maxu{i}{N}\, \mathbb{E} \biggeck{ \suptT \N{ \overbar{X}^i_t}^p }}{M}.
\end{align*}
We use Lemma \ref{NPDelta:UniformMoment} to obtain 
\begin{align*}
   \underset{1 \leq i \leq N}{\max}\, \mathbb{E} \Biggeck{ \suptT \N{X^i_t}^p } \leq C_{\mathrm{D}}, \quad \quad 
   \underset{1 \leq i \leq N}{\max}\, \mathbb{E} \Biggeck{ \suptT \N{\overbar{X}^i_t}^p} \leq \overbar{C}_{\mathrm{D}}
\end{align*}
We set $\frac{C}{2} \coloneqq \max \Big \{ C_{\mathrm{D}}, \overbar{C}_{\mathrm{D}} \Big \}$ and get
\begin{align*}
    \frac{\underset{1 \leq i \leq N}{\max}\, \mathbb{E} \bigg[ \suptT \N{X^i_t}^p \bigg ]}{M} + \frac{\underset{1 \leq i \leq N}{\max}\, \mathbb{E} \biggeck{ \suptT \N{ \overbar{X}^i_t}^p }}{M} \leq \frac{C}{M}.
\end{align*}
Therefore, we have 
\begin{align*}
    \mathbb{P}(\Omega_M) \geq 1 - \frac{C}{M}.
\end{align*}
\end{proof}
Lemma \ref{MFD:Moment} ensures that $\mathbb{P}(\Omega_M) \to 1$ as $M \to \infty$, so restricting to $\Omega_M$ is a mild condition. The key benefit of this restriction is that both empirical measures $\rho^N_t$ and $\orho^N_t$ have $p$th moments deterministically bounded by $M$, uniformly in $t$ and $N$. So, restricting ourselves to $\Omega_M$ allows us to apply Lemma \ref{CarrL2}, which provides the Wasserstein stability estimate needed to prove the main result of this section.
\begin{theorem}
    Let $\mathcal{X}_T^N$ and $\overbar{\mathcal{X}}_T^N$ be defined by the $\delta$-CBO systems (\ref{Delta:Scheme}) and (\ref{Delta:MFScheme}), respectively, both sharing the same initial data $\mathcal{X}^N_0 = \overbar{\mathcal{X}}^N_0$ almost surely, with common distribution $\orho_0 \in \mathcal{P}_p(\mathbb{R}^d)$ for some $p \geq 2$. For a fixed \(M>0\), define the set
    \begin{align*}
       \Omega_M := \bigg\{ \omega \in \Omega : \suptT \,  \frac{1}{N} \Sum{i}{N} \max \Big\{ \N{X^i_t}^p , \N{\overbar{X}^i_t}^p \Big \} \leq M \bigg\}.
   \end{align*} 
   Then there exists a $\varepsilon^{\mathrm{rest}}_{\mathrm{D}}(N, T)>0$ with $\varepsilon^{\mathrm{rest}}_{\mathrm{D}}(N, T) \xrightarrow{N \rightarrow \infty} 0$ such that for every $N \in \mathbb{N}$
    \begin{align*}
        \underset{1 \leq i \leq N}{\max} \,  \mathbb{E} \biggeck{ \suptT \,\N{ X^i_t - \overbar{X}^i_t}^p \mathbbm{1}_{\Omega_M}} \leq \varepsilon^{\mathrm{rest}}_{\mathrm{D}}(N, T).
    \end{align*}
    Consequently, if $\mathbb{P}(\Omega_M)>0$, then
    \begin{align*}
        \underset{1 \leq i \leq N}{\max} \,  \mathbb{E} \biggeck{ \suptT \,\N{ X^i_t - \overbar{X}^i_t}^p \Big | \Omega_M } \leq \frac{\varepsilon^{\mathrm{rest}}_{\mathrm{D}}(N, T)}{\mathbb{P}(\Omega_M)}.
    \end{align*}
    \label{Delta:Thrm}
\end{theorem}
\begin{proof}
We fix some $N \in \mathbb{N}$, $i \leq N$. Since both particle systems share the same Brownian motion, 
the diffusion terms cancel exactly in the difference 
\begin{align*}
dF^i_t := d(X^i_t - \overbar{X}^i_t) = -s \lambda (X^i_t -\overbar{X}^i_t)dt + s\lambda (X^\alpha(\rho^N_t) - X^\alpha(\orho_t))dt, \quad \forall t \in [0, T]
\end{align*}
leaving the deterministic 
linear integral equation
\begin{align*}
    F^i_t = s\lambda \int_0^t e^{-s\lambda(t-\tau)} D_\tau\, d\tau, 
    \quad D_\tau := X^\alpha(\rho^N_\tau) - X^\alpha(\orho_\tau),
\end{align*}
where we used $F^i_0 = X^i_0 - \overbar{X}^i_0 = 0$ almost 
surely by coupling of the initial distributions. Since $e^{-s\lambda(t-\tau)} \leq 1$ 
for all $\tau \leq t$, we can bound on each path
\begin{align*}
    \N{F^i_t} \leq s\lambda \int_0^t \N{D_\tau}\, d\tau 
    \leq s\lambda \int_0^T \N{D_\tau}\, d\tau,
\end{align*}
We fix ${\tiT} \in [0, T]$ and take the $p$th power and supremum over $t \in [0,{\tiT}]$:
\begin{align*}
    \suptTit \N{F^i_t}^p 
    \leq \biggrou{s\lambda \int_0^{\tiT} \N{D_\tau}\, d\tau}^p 
    \leq (s\lambda)^p {\tiT}^{p-1} \int_0^{\tiT} \N{D_\tau}^p\, d\tau.
\end{align*}
where we applied Hölder's inequality in the last step. Multiplying by $\mathbbm{1}_{\Omega_M}$, taking 
expectations and applying the Fubini-Tonelli theorem yields
\begin{align}
    \mathbb{E}\biggeck{\suptTit \N{F^i_t}^p \mathbbm{1}_{\Omega_M}} 
    \leq (s\lambda)^p {\tiT}^{p-1} \int_0^{\tiT} 
    \mathbb{E}\bigeck{\N{D_\tau}^p \mathbbm{1}_{\Omega_M}}\, d\tau. \label{RestrictedDelta Bound1}
\end{align}
We now expand $\N{D_\tau}^p$ by inserting the empirical 
mean-field measure $\overbar{\rho}^N_\tau$ and applying 
the inequality $\N{a + b}^p \leq 2^{p-1}\Biggrou{\N{a}^p + \N{b}^p}$:
\begin{align*}
    \mathbb{E}\bigeck{\N{D_\tau}^p \mathbbm{1}_{\Omega_M}} 
    &= \mathbb{E}\Bigeck{\N{X^\alpha(\rho^N_\tau) 
    - X^\alpha(\orho_\tau)}^p \mathbbm{1}_{\Omega_M}} \\
    &\leq 2^{p-1}\mathbb{E}\Bigeck{\N{X^\alpha(\rho^N_\tau) 
    - X^\alpha(\overbar{\rho}^N_\tau)}^p \mathbbm{1}_{\Omega_M}} 
    + 2^{p-1}\mathbb{E}\Bigeck{\N{X^\alpha(\overbar{\rho}^N_\tau) 
    - X^\alpha(\orho_\tau)}^p \mathbbm{1}_{\Omega_M}}.
\end{align*}
We recognize that on $\Omega_M$, we can apply Lemma \ref{CarrL2} and 
Example \ref{WasserBD} alongside with the exchangeability of the particles to the first term:
\begin{align}
    \mathbb{E}\Bigeck{\N{X^\alpha(\rho^N_\tau) 
    - X^\alpha(\overbar{\rho}^N_\tau)}^p \mathbbm{1}_{\Omega_M}} 
    &\leq C_0^p\, \mathbb{E}\Bigeck{W_p^p(\rho^N_\tau, 
    \overbar{\rho}^N_\tau)\mathbbm{1}_{\Omega_M}} \notag \\
    &\leq C_0^p\, \mathbb{E}\biggeck{ \frac{1}{N} \sum^N_{j=1}\N{X^j_\tau 
    - \overbar{X}^j_\tau}^p \mathbbm{1}_{\Omega_M}} \notag \\
    &= C_0^p\, \mathbb{E}\biggeck{\N{X^i_\tau 
    - \overbar{X}^i_\tau}^p \mathbbm{1}_{\Omega_M}} \notag \\
    &\leq C_0^p\, \mathbb{E}\biggeck{\underset{s \in [0,\tau]}
    {\sup}\N{F^i_s}^p \mathbbm{1}_{\Omega_M}}. \label{Restricted Delta Est1}
\end{align}
For every $\tau \in [0,T]$, define
\begin{align*}
\overbar{\Omega}_{M,\tau}
:=
\left\{
\omega \in \Omega :
\frac{1}{N}\sum_{i=1}^N \N{\overbar{X}^i_\tau}^p \le M
\right\}.
\end{align*}
Since
\begin{align*}
\Omega_M \subseteq
\left\{
\omega \in \Omega :
\sup_{t \in [0,T]} \frac{1}{N}\sum_{i=1}^N \N{\overbar{X}^i_t}^p \le M
\right\}
\subseteq \overbar{\Omega}_{M,\tau},
\end{align*}
we have
\begin{align*}
\mathbbm{1}_{\Omega_M} \le \mathbbm{1}_{\overline{\Omega}_{M,\tau}}.
\end{align*}
Therefore, in the second term, we can invoke Lemma \ref{CH:ConsBound}:
\begin{align}
    \mathbb{E}\Bigeck{\N{X^\alpha(\overbar{\rho}^N_\tau) 
    - X^\alpha(\orho_\tau)}^p \mathbbm{1}_{\Omega_M}} \leq \mathbb{E}\Bigeck{\N{X^\alpha(\overbar{\rho}^N_\tau) 
    - X^\alpha(\orho_\tau)}^p \mathbbm{1}_{\overline{\Omega}_{M, \tau}}}
    \leq C_{\mathrm{EM}} N^{-\frac{p}{2}}. \label{Restricted Delta Est2}
\end{align}
Combining \eqref{Restricted Delta Est1} and \eqref{Restricted Delta Est2} gives:
\begin{align*}
    \mathbb{E}\bigeck{\N{D_\tau}^p \mathbbm{1}_{\Omega_M}} 
    \leq 2^{p-1} C_0^p\, 
    \mathbb{E}\biggeck{\underset{s \in [0,\tau]}{\sup}
    \N{F^i_s}^p \mathbbm{1}_{\Omega_M}} 
    + 2^{p-1} C_{\mathrm{EM}} N^{-\frac{p}{2}}.
\end{align*}
Putting the latter into \eqref{RestrictedDelta Bound1} yields
\begin{align*}
    \mathbb{E}\biggeck{\suptTit \N{F^i_t}^p \mathbbm{1}_{\Omega_M}} 
    &\leq (s\lambda)^p {\tiT}^{p-1} \int_0^{\tiT} 
    \biggrou{2^{p-1} C_0^p\, 
    \mathbb{E}\biggeck{\underset{s \in [0,\tau]}{\sup}
    \N{F^i_s}^p \mathbbm{1}_{\Omega_M}} 
    + 2^{p-1} C_{\mathrm{EM}} N^{-\frac{p}{2}}} d\tau \\
    &= (2{\tiT})^{p-1}(s\lambda C_0)^p  
    \int_0^{\tiT} \mathbb{E}\biggeck{\underset{s \in [0,\tau]}{\sup}
    \N{F^i_s}^p \mathbbm{1}_{\Omega_M}}\, d\tau 
    + 2^{p-1}(s\lambda {\tiT})^p C_{\mathrm{EM}} N^{-\frac{p}{2}}.
\end{align*}
The above bound allows us to apply Grönwall's 
Lemma (Lemma \ref{GronwallLemma}) to ultimately get
\begin{align*}
    \mathbb{E}\biggeck{\suptT\N{X^i_t - \overbar{X}^i_t}^p \mathbbm{1}_{\Omega_M}} \leq 2^{p-1}(s\lambda T)^p C_{\mathrm{EM}} N^{-\frac{p}{2}} e^{2^{p-1}(s\lambda T C_0)^p} =: \varepsilon^{\mathrm{rest}}_{\mathrm{D}}(N, T).
\end{align*}
Finally, if $\mathbb{P}(\Omega_M)>0$, then
\begin{align*}
    \mathbb{E}\biggeck{\suptT\N{X^i_t - \overbar{X}^i_t}^p \Big| \Omega_M}
    =
    \frac{\mathbb{E}\biggeck{\suptT\N{X^i_t - \overbar{X}^i_t}^p \mathbbm{1}_{\Omega_M}}}{\mathbb{P}(\Omega_M)}
    \leq \frac{\varepsilon^{\mathrm{rest}}_{\mathrm{D}}(N,T)}{\mathbb{P}(\Omega_M)}.
\end{align*}
\end{proof}
Restricted to $\Omega_M$, Theorem \ref{Delta:Thrm} implies pathwise propagation of chaos by coupling in the sense of Definition \ref{Pre:ChaosCoup}, with expectations replaced by expectations on $\Omega_M$. In particular, by Example \ref{WasserBD},
\begin{align*}
    \mathbb{E}\Bigeck{W_p^p\bigrou{\rho^N_t, \orho^N_t}\mathbbm{1}_{\Omega_M}} \leq \varepsilon^{\mathrm{rest}}_{\mathrm{D}}(N,T) \quad \forall t \in [0,T],
\end{align*}
so the empirical measures of the interacting and the mean-field particle systems are close on $\Omega_M$, at rate $\mathcal{O}\bigrou{N^{-1/2}}$ in Wasserstein-$p$ distance.

\begin{remark}
    Note that this rate is independent of the exponent $p$. As long as the initial distribution has finite $p$th moments, i.e. $\orho_0 \in \mathcal P_p(\mathbb R^d)$ with $p\ge 2$, we obtain propagation of chaos in Wasserstein-$p$ distance at rate $\mathcal O(N^{-1/2})$.
\end{remark}

\section{The Consensus Freezing Scheme}
\label{ch:MFF}
In this section, we study the \textbf{consensus freezing scheme} (CF), a CBO variant recently proposed in \cite{Fornasier2026FromConvergence} in which the consensus point is held fixed over short time intervals. More precisely, the time horizon $[0, T]$ is partitioned into intervals $[t_j, t_{j+1})$ of equal length $\Delta t$. Within each interval, the particles evolve toward a frozen consensus determined at the beginning of that interval. This modification yields a dynamics that is particularly suited to theoretical analysis: on each interval, the particle evolution is governed by an Ornstein–Uhlenbeck process, which admits an explicit strong solution and allows for direct control of moments and distances between particle trajectories.

We show that the consensus freezing scheme admits a natural mean-field limit. Using a coupling argument and exploiting the closed-form structure of the Ornstein–Uhlenbeck dynamics, we establish propagation of chaos uniformly over finite time horizons. In particular, we obtain quantitative convergence bounds of order $\mathcal{O}(N^{- 1/2})$ in Wasserstein-$p$ distance, restricted to high-probability good sets where suitable moment bounds for the empirical measures hold.
To prove this result, we first derive uniform moment estimates for the particle system and its mean-field counterpart, then establish probabilistic bounds ensuring that empirical measures have sufficient uniform moment bounds, and finally use these estimates within a coupling framework to conclude propagation of chaos for the consensus freezing dynamics.
\vspace{.5 cm}
 
Let $[0, T]$ be the time interval of interest being partitioned into $J$ sub-intervals of equal length $\Delta t$, i.e. 
\begin{align*} 
    \bigcup_{0 \leq j \leq J-1} [t_{j}, t_{j+1}] = [0, T], \quad t_{j+1} - t_j = \Delta t \quad \forall j \in [J-1].
\end{align*}
We define the $N$-particle system $\bigrou{\mathcal{X}_t^N }_t$ by the following system of stochastic differential equations for fixed $0 \leq j \leq J-1$.
\begin{align}
     dX^i_t = - s \cdot \lambda (X^i_t - X^\alpha(\rho^N_{t_j}))dt + \sqrt{s}\delta dB^i_t,  \quad \forall t \in [t_j, t_{j+1}), \quad X^i_0 \sim \orho_0, \quad  i = 1, ..., N.
     \label{CF:CFscheme}
\end{align}
We claim the mean-field limit to be described by the natural choice 
\begin{align}
     d\overbar{X}_t = - s \cdot \lambda (\overbar{X}_t - X^\alpha(\orho_{t_j}))dt + \sqrt{s}\delta dB_t \quad \forall t \in [t_j, t_{j+1}),
     \label{CF:MFscheme}
\end{align}
where $\orho_t = \text{Law}(\overbar{X}_t) \in \PRd$ and an initial $\overbar{X}_0 \sim \orho_0$.
One can easily see that for every $i$, at each time interval $[t_j, t_{j+1})$, the stochastic differential equation \eqref{CF:CFscheme} is an Ornstein-Uhlenbeck process and, as such, admits a unique strong solution.
\newline
Again, we want to show propagation of chaos by coupling. We define the $N$-particle system $\bigrou{\overbar{\mathcal{X}}^N_t}_t$ by the following system of stochastic differential equations: 
\begin{align}
     d\overbar{X}^i_t = - s \cdot \lambda (\overbar{X}^i_t - X^\alpha(\orho_{t_j}))dt + \sqrt{s}\delta dB^i_t \quad \forall t \in [t_j, t_{j+1}), \quad \overbar{X}^i_0 = X^i_0 \, \, \text{ a.s.}, \quad  i = 1, ..., N, 
     \label{CF:MFIPS}
\end{align}
where $\orho_t = \text{Law}(\overbar{X}_t)$ is the distribution of (\ref{CF:MFscheme}). 
\vspace{.5 cm}
\newline
Throughout this section, if not otherwise stated, the probabilities and expectations will be taken with respect to the probability space 
\begin{align*}
    \ProbStd
:=
\Bigg(
(\mathbb{R}^d)^N \times \big(C_0([0,\infty);\mathbb{R}^d)\big)^N,\,
\mathcal{B}\Big((\mathbb{R}^d)^N \times \big(C_0([0,\infty);\mathbb{R}^d)\big)^N\Big),\,
\orho_0^{\otimes N}\otimes \mathbb{W}^{\otimes N}
\Bigg),
\end{align*}
which is a valid choice given that the above SDEs admit strong solutions.
 \subsection{Propagation of Chaos for the Consensus Freezing Scheme}
 As for the $\delta$-CBO scheme, we will show propagation of chaos using a coupling argument. This also means that we once again need uniformly bounded moments of the empirical measures of the interacting particle systems. Analogously to the $\delta$-CBO case, our analysis will be based on a probabilistic argument. However, as in the $\delta$-CBO case, we must first establish appropriate moment bounds for the interacting particle systems. For both the proof of boundedness of the moments of the stochastic process induced by the scheme and the proof of propagation of chaos, we can leverage the fact that the SDEs underlying the consensus-freezing scheme are Ornstein–Uhlenbeck processes. Consequently, they admit explicit strong solutions in closed form, which can be computed and analyzed directly. Furthermore, within each freezing interval $[t_j, t_{j+1})$, the process is well behaved and has uniformly bounded moments. Moreover, we will see in Theorem \ref{MFF: Thrm} that the estimation for propagation of chaos by coupling will rely just on the differences at the points where the consensus changes, i.e.
 \begin{align*}
      \mathbb{E} \biggeck{ \supttj \,\N{ X^i_t - \overbar{X}^i_t}^p} \leq \mathbb{E} \biggeck{ \N{ X^i_{t_{j}} - \overbar{X}^i_{t_{j}}}^p + \N{ X^i_{t_{j+1}} - \overbar{X}^i_{t_{j+1}}}^p}.
 \end{align*}
 We are therefore left with handling the moments of the empirical measures where the consensus point changes. This will leave us with considering the event
 \begin{align*}
     \Omega_M := \bigg\{ \omega \in \Omega : \supjJ \,  \frac{1}{N} \Sum{i}{N} \max \Big\{ \N{X^i_{t_j}}^p , \N{\overbar{X}^i_{t_j}}^p \Big \} \leq M \bigg\}.
 \end{align*}
 \begin{lemma}
    Consider the particle system $\IPS{T}$ and $\MIPS{T}$ defined by \eqref{CF:CFscheme} and \eqref{CF:MFIPS}, respectively, with initial data following the distribution $\rho_0 \in \mathcal{P}_p(\mathbb{R}^d)$ and let $J$ be the number of freezing intervals such that $[0, T] = \bigcup_{0 \leq j \leq J-1} [t_j, t_{j+1}]$. Then there exist constants $C_{\mathrm{F}} >0 $ and $\overbar{C}_{\mathrm{F}} >0 $ independent of $N$ such that for every $N \in \mathbb{N}$ and every $i \leq N$
    \begin{align*}
        \mathbb{E} \biggeck{ \supjJ \N{X^i_{t_j}}^p} \leq C_{\mathrm{F}} \quad \text{and} \quad \mathbb{E} \biggeck{ \supjJ \N{\overbar{X}^i_{t_j}}^p} \leq \overbar{C}_{\mathrm{F}}.
    \end{align*}
    \label{NP:CFLemma}
\end{lemma}
\begin{proof}
We will prove the statement just for the particle system $\IPS{T}$. The corresponding result for $\MIPS{T}$ follows the same line of argument.

    Fix $N \in \mathbb{N}$, $i \leq N, j < J$ and consider the time interval $[t_j, t_{j+1})$ and assume an initial condition $X^i_{t_j}$. Then we can clearly see that 
    \begin{align*}
         dX^i_t = - s \cdot \lambda (X^i_t - X^\alpha(\rho^N_{t_j}))dt + \sqrt{s}\delta dB^i_t,  \quad \forall t \in [t_j, t_{j+1})
    \end{align*}
is an Ornstein-Uhlenbeck process with the strong solution
\begin{align}
    X^i_t = X_{t_j}^i e^{- s \lambda (t-t_j)} + X^\alpha(\rho^N_{t_j})(1- e^{-s \lambda (t-t_j)}) +  \sqrt{s}\delta\int^t_{t_j} e^{-s \lambda (t- \tau)} dB^i_\tau.
    \label{OU:Strong}
\end{align}
Therefore, at each interval endpoint $t_j$, the process is determined as follows
\begin{align}
    X^i_{t_j} &= X^i_{t_{j-1}} e^{-s\lambda (t_j - t_{j-1})} + X^{\alpha}(\rho^N_{t_{j-1}})(1- e^{-s\lambda (t_j - t_{j-1})}) + \sqrt{s} \delta \int^{t_j}_{t_{j-1}} e^{-s\lambda (t_j - \tau)}dB_\tau^i.
    \label{OUStrong2}
\end{align}
For every $j \in \{1, ..., J \}$, we define 
\begin{align}
    Z^i_j := \int^{t_j}_{t_{j-1}} e^{-s\lambda (t_j - \tau)}dB_\tau^i.
    \label{CFMomentZ_jDefinition}
\end{align}
Note that $Z^i_j$ is a Brownian motion integral, with a square integrable integrand over its time horizon. This means that $Z^i_j$ is a Gaussian random vector with zero mean and variance
\begin{align*}
     \mathrm{Cov}(Z_j^i)
    = \mathbb{E}\big[ Z_j^i (Z_j^i)^\top \big] 
    &= \mathbb{E} \Biggeck{
    \biggrou{\int_{t_{j-1}}^{t_j} e^{-s\lambda (t_j-\tau)}\, dB_\tau^i}
    \biggrou{\int_{t_{j-1}}^{t_j} e^{-s\lambda (t_j-\tau)}\, dB_\tau^i}^\top
    }\\
    &= \mathbb{E} \Biggeck{
    \biggrou{\int_{0}^{\Delta t} e^{-s\lambda (\Delta t -\tau)}\, dB_\tau^i}
    \biggrou{\int_{0}^{\Delta t} e^{-s\lambda (\Delta t -\tau)}\, dB_\tau^i}^\top
    }.
\end{align*}
By time-homogeneity and Itô isometry, this becomes
\begin{align*}
    \mathbb{E} \Biggeck{
    \biggrou{\int_{0}^{\Delta t} e^{-s\lambda (\Delta t -\tau)}\, dB_\tau^i}
    \biggrou{\int_{0}^{\Delta t} e^{-s\lambda (\Delta t -\tau)}\, dB_\tau^i}^\top
    }
    = \int_0^{\Delta t} e^{-2s\lambda(\Delta t-\tau)}\, d\tau \, I_d
    = \frac{1-e^{-2s\lambda \Delta t}}{2s\lambda}\, I_d.
\end{align*}
Therefore,
\[
Z_j^i \sim \mathcal N\bigg(0,\frac{1-e^{-2s\lambda \Delta t}}{2s\lambda}\, I_d\bigg).
\]
Let us now shift our attention to our actual goal which is bounding the moments uniformly. We begin by rewriting \eqref{OUStrong2} using the definition of $Z^i_j$ and the fact that $t_j - t_{j-1} = \Delta t$:
\begin{align}
    X^i_{t_j} = X^i_{t_{j-1}} e^{-s\lambda \Delta t} + X^{\alpha}(\rho^N_{t_{j-1}})(1- e^{-s\lambda \Delta t}) + \sqrt{s} \delta Z^i_j.
    \label{OUStrong3}
\end{align}
We use \eqref{OUStrong3}, which holds for $j \geq 1$, to derive a straightforward first bound, treating the initial time separately:
\begin{align}
    \mathbb{E} \biggeck{ \supjJ \N{X^i_{t_j}}^p} &\leq \mathbb{E} \Bigeck{\N{X^i_0}^p} + \mathbb{E} \biggeck{ \supr{1}{j}{J} \N{X^i_{t_{j-1}} e^{-s\lambda \Delta t} + X^{\alpha}(\rho^N_{t_{j-1}})(1- e^{-s\lambda \Delta t}) + \sqrt{s} \delta Z^i_j}^p} \notag \\
    &\leq \mathbb{E} \Bigeck{\N{X^i_0}^p} + 2^{p-1} \Biggrou{\mathbb{E} \biggeck{ \supr{1}{j}{J} \N{X^i_{t_{j-1}} e^{-s\lambda \Delta t} + X^{\alpha}(\rho^N_{t_{j-1}})(1- e^{-s\lambda \Delta t})}^p }\notag  \\ & \quad + \mathbb{E} \biggeck{ \supr{1}{j}{J} \N{\sqrt{s} \delta Z^i_j}^p}}.
    \label{CF:NPBoundHopping}
\end{align}
We begin by treating the first expectation. Since $e^{-s\lambda \Delta t} \in (0,1)$ and the map
$x \mapsto \N{x}^p$ is convex, we can estimate
\begin{align}
    &\mathbb{E} \biggeck{ \supr{1}{j}{J} \N{X^i_{t_{j-1}} e^{-s\lambda \Delta t} + X^{\alpha}(\rho^N_{t_{j-1}})(1- e^{-s\lambda \Delta t})}^p } \notag \\
    &\leq
    \mathbb{E} \Biggeck{
    \supr{1}{j}{J}
    \biggrou{
    e^{-s\lambda \Delta t}\N{X^i_{t_{j-1}}}^p
    +
    (1- e^{-s\lambda \Delta t})\N{X^\alpha(\rho^N_{t_{j-1}})}^p
    }}
    \notag \\
    &\leq
    e^{-s\lambda \Delta t}
    \mathbb{E} \biggeck{ \supr{1}{j}{J} \N{X^i_{t_{j-1}}}^p}
    +
    (1- e^{-s\lambda \Delta t})
    \mathbb{E} \biggeck{\supr{1}{j}{J} \N{X^{\alpha}(\rho^N_{t_{j-1}})}^p}.
    \label{FrMoment:1stExpbound1}
\end{align}
We now write $\mathbb{E} \biggeck{\supr{1}{j}{J} \N{X^{\alpha}(\rho^N_{t_{j-1}})}^p}$ in terms of $\mathbb{E} \biggeck{ \supr{1}{j}{J} \N{X^i_{t_{j-1}}}^p}$. \newline
We start by computing
\begin{align}
   \mathbb{E} \biggeck{ \supr{1}{j}{J} \N{X^\alpha(\rho^N_{t_{j-1}})}^p} = \mathbb{E} \Biggeck{\supr{1}{j}{J} \N{\int \frac{x \oa(x)}{\Gnorm{\rho^N_{t_{j-1}}}} d\rho^N_{t_{j-1}}(x)}^p}. \label{FrMomentConsensus}
\end{align}
We define the random measure $\eta^\alpha_{t_{j-1}} := \frac{\oa \rho^N_{t_{j-1}}}{\Gnorm{\rho^N_{t_{j-1}}}}$\footnote{$\eta^\alpha_{t_{j-1}}$ is measurable in the sense of Proposition \ref{Pre:PropEmpMeas}} which is a probability measure on the paths on which $\rho^N_{t_{j-1}}$ is a probability measure. Therefore, with Jensen's inequality, we can estimate:
\begin{align}
    \mathbb{E} \Biggeck{\supr{1}{j}{J} \N{\int \frac{x \oa(x)}{\Gnorm{\rho^N_{t_{j-1}}}} d\rho^N_{t_{j-1}}(x)}^p} &= \mathbb{E} \Biggeck{\supr{1}{j}{J} \N{\int x \,  d\eta^\alpha_{t_{j-1}}(x)}^p} \notag \\ &\leq \mathbb{E} \Biggeck{\supr{1}{j}{J} \int \N{x}^p \,  d\eta^\alpha_{t_{j-1}}(x)}. \label{FrMomentBound1}
\end{align}
With Lemma \ref{GerberPropA4}, we get a constant $b_p>0$ such that
\begin{align}
    \mathbb{E} \Biggeck{\supr{1}{j}{J} \int \N{x}^p \,  d\eta^\alpha_{t_{j-1}}(x)} =  \mathbb{E} \Biggeck{\supr{1}{j}{J} \,  \mathbb{E}_{\eta^\alpha_{t_{j-1}}} \bigeck{\N{x}^p}} \leq \mathbb{E} \biggeck{b_p \, \supr{1}{j}{J}\, \mathbb{E}_{\rho^N_{t_{j-1}}} \bigeck{\N{x}^p}}. \label{FrMomentBound2}
\end{align}
We use \eqref{FrMomentBound1} and \eqref{FrMomentBound2} to obtain for \eqref{FrMomentConsensus}
\begin{align}
    \mathbb{E} \biggeck{ \supjJ \N{X^\alpha(\rho^N_{t_{j-1}})}^p} \leq \mathbb{E} \biggeck{b_p\,  \supjJ\, \mathbb{E}_{\rho^N_{t_{j-1}}} \bigeck{\N{x}^p}} = b_p \, \mathbb{E} \biggeck{\supjJ\, \mathbb{E}_{\rho^N_{t_{j-1}}} \bigeck{\N{x}^p}}.
    \label{NP:FrMoment2}
\end{align}
Writing out the expectation with respect to $\rho^N_{t_{j-1}}$ yields
\begin{align*}
    \mathbb{E} \biggeck{\supr{1}{j}{J}\, \mathbb{E}_{\rho^N_{t_{j-1}}} \bigeck{\N{x}^p}}
    = \mathbb{E} \biggeck{\supr{1}{j}{J}\, \frac{1}{N}\sum^N_{m=1} \N{X^m_{t_{j-1}}}^p}
    &\leq \mathbb{E} \biggeck{ \frac{1}{N}\sum^N_{m=1} \supr{1}{j}{J} \,  \N{X^m_{t_{j-1}}}^p} \\
    &= \mathbb{E} \biggeck{\underset{0 \leq j \leq J-1}{\sup} \,  \N{X^i_{t_{j}}}^p}, 
\end{align*}
where we used the invariance of the expectation under permutations of the particle index in the last equality.
We combine the latter with \eqref{NP:FrMoment2} to get our estimate for the expected $p$th moment of the consensus point.
\begin{align*}
    \mathbb{E} \biggeck{ \supr{1}{j}{J} \N{X^\alpha(\rho^N_{t_{j-1}})}^p} \leq b_p \mathbb{E} \biggeck{\underset{0 \leq j \leq J-1}{\sup} \,  \N{X^i_{t_{j}}}^p}.
\end{align*}
Now, we can use the latter on our estimate \eqref{FrMoment:1stExpbound1} to further compute
\begin{align}
    &\mathbb{E} \biggeck{ \supr{1}{j}{J} \N{X^i_{t_{j-1}} e^{-s\lambda \Delta t} + X^{\alpha}(\rho^N_{t_{j-1}})(1- e^{-s\lambda \Delta t})}^p } \notag \\
    &\leq e^{-s\lambda \Delta t}
    \mathbb{E} \biggeck{ \supr{1}{j}{J} \N{X^i_{t_{j-1}}}^p}+(1- e^{-s\lambda \Delta t}) \mathbb{E} \biggeck{\supr{1}{j}{J} \N{X^{\alpha}(\rho^N_{t_{j-1}})}^p} \notag \\
    &\leq \biggrou{e^{-s\lambda \Delta t}+ b_p(1- e^{-s\lambda \Delta t})} \mathbb{E} \biggeck{  \underset{0 \leq j \leq J-1}{\sup} \N{X^i_{t_{j}}}^p}.
\end{align}
We combine the latter with \eqref{CF:NPBoundHopping} and set
\begin{align*}
    a_s := 2^{p-1}\Bigrou{e^{-s\lambda \Delta t}+ b_p(1- e^{-s\lambda \Delta t})}
\end{align*}
to get the estimate
\begin{align*}
    \mathbb{E} \biggeck{\supjJ \N{X^i_{t_j}}^p}
    \leq \mathbb{E} \Bigeck{\N{X^i_0}^p} + a_s\, \mathbb{E} \biggeck{  \underset{0 \leq j \leq J-1}{\sup} \N{X^i_{t_{j}}}^p} + 2^{p-1} s^{\frac{p}{2}}\delta^p\, \mathbb{E} \biggeck{ \supr{1}{j}{J} \N{Z^i_j}^p}.
\end{align*}
Since this estimate holds for every number of freezing intervals, we can iterate it and obtain
\begin{align*}
    \mathbb{E} \biggeck{\supjJ \N{X^i_{t_j}}^p}
    \leq \mathbb{E} \Bigeck{\N{X^i_0}^p} \sum^J_{l=0} a_s^l
    + 2^{p-1} s^{\frac{p}{2}}\delta^p \sum^{J}_{l = 1} a_s^{J-l}\, \mathbb{E} \biggeck{ \underset{1 \leq j \leq l}{\sup} \N{Z^i_j}^p}.
\end{align*}
Since $Z^i_1, ..., Z^i_J$ are i.i.d.\ normal random vectors with zero mean, we can use Lemma \ref{Pre:LemmaNormalVectors} to get a constant $C_p >0$ depending only on $p$ such that for every $l \leq J$
\begin{align*}
    \mathbb{E} \biggeck{ \underset{1 \leq j \leq l}{\sup} \N{Z^i_j}^p} \leq \biggrou{\frac{1- e^{- 2 s \lambda \Delta t}}{2s \lambda}}^{\frac{p}{2}} C_p \Bigrou{d\log (2ld)}^{\frac{p}{2}}.
\end{align*}
This leads to our final estimate
\begin{align*}
    \mathbb{E} \biggeck{\supjJ \N{X^i_{t_j}}^p}
    &\leq \mathbb{E} \Bigeck{ \N{X^i_0}^p} \sum^J_{l=0} a_s^l
    + 2^{p-1}\delta^p \biggrou{\frac{1- e^{- 2s\lambda \Delta t}}{2 \lambda}}^{\frac{p}{2}} C_p \sum^{J}_{l = 1} \Bigrou{d\log (2ld)}^{\frac{p}{2}} a_s^{J-l}\\
    &=: C_{\mathrm{F}}(s).
\end{align*}
\end{proof}
Analogously to Section \ref{ch:Delta}, we will now proceed to prove the high-probability uniform boundedness of the empirical moments of the systems \eqref{CF:CFscheme} and \eqref{CF:MFscheme} using the results of Lemma \ref{NP:CFLemma}.
\begin{lemma}
   Let $T >0$ be a fixed final time, and let the interval $[0, T]$ be divided by $0= t_0 < t_1 <...<t_J = T$ with uniform step size $\Delta t$. Let $\big( \mathcal{X}^N_t \big)_{t \in [0, T]} \text{ and } \big( \overbar{\mathcal{X}}^N_t \big)_{t \in [0, T]}$ be defined by \eqref{CF:CFscheme} and \eqref{CF:MFIPS}, respectively, with initial data following the distribution $\rho_0 \in \mathcal{P}_p(\mathbb{R}^d)$. For a fixed \(M>0\), define the set
   \begin{align*}
       \Omega_M := \bigg\{ \omega \in \Omega : \supjJ \,  \frac{1}{N} \Sum{i}{N} \max \Big\{ \N{X^i_{t_j}}^p , \N{\overbar{X}^i_{t_j}}^p \Big \} \leq M \bigg\}.
   \end{align*} Then there exists a constant $C>0$ independent of $N$ such that
    \begin{align*}
        \mathbb{P}(\Omega_M) \geq 1 - \frac{C}{M}.
    \end{align*}
    \label{MFF:Moment}
\end{lemma}
\begin{proof}
In the same way as in Lemma \ref{MFD:Moment}, we obtain with Markov's inequality 
\begin{align*}
    \mathbb{P}(\Omega^c_M) \leq \frac{\maxu{i}{N}\, \mathbb{E} \biggeck{ \supjJ \N{X^i_{t_j}}^p }}{M} + \frac{\maxu{i}{N}\, \mathbb{E} \biggeck{ \supjJ \N{ \overbar{X}^i_{t_j}}^p} }{M}.
\end{align*}
We use Lemma \ref{NP:CFLemma} to obtain 
\begin{align*}
    \maxu{i}{N}\, \mathbb{E} \biggeck{ \supjJ \N{X^i_{t_j}}^p } \leq C_{\mathrm{F}} \quad \text{and} \quad \maxu{i}{N}\, \mathbb{E} \biggeck{ \supjJ \N{\overbar{X}^i_{t_j}}^p } \leq \overbar{C}_{\mathrm{F}}.
\end{align*}
We set $\frac{C}{2} \coloneqq \max \Big \{ C_{\mathrm{F}}, \overbar{C}_{\mathrm{F}} \Big \}$ and get
\begin{align*}
    \frac{\maxu{i}{N}\, \mathbb{E} \biggeck{ \supjJ \N{X^i_{t_j}}^p }}{M} + \frac{\maxu{i}{N}\, \mathbb{E} \biggeck{ \supjJ \N{ \overbar{X}^i_{t_j}}^p }}{M} \leq \frac{C}{M}.
\end{align*}
\end{proof}
\vspace{.5 cm}
As in Section \ref{ch:MFH}, we now proceed to prove propagation of chaos via the coupling method restricted to $\Omega_M$.  
Before we state and prove our main statement, we use some basic ODE theory to get our problem in a more convenient form.
Let $t \in [t_j, t_{j+1})$ for some $j$ and $i \leq N$. Computing the distances between the $i$th particle and the $i$th mean-field sample yields
\begin{align}
    d(X_t^i - \overbar{X}^i_t) = -s \lambda  ( X^i_t - \overbar{X}^i_t) dt + s \lambda  ( X^\alpha(\rho^N_{t_j}) - X^\alpha(\orho_{t_j})) dt.
    \label{Fr:dist}
\end{align}
By defining $F^{(i)}(t) := X^i_t - \overbar{X}^i_t$ and $D_j := X^\alpha(\rho^N_{t_j}) - X^\alpha(\orho_{t_j})$ we can rewrite (\ref{Fr:dist}) as follows:
\begin{align*}
    \frac{d}{dt}F^{(i)}(t) = -s \lambda F^{(i)}(t) + s \lambda D_j,
\end{align*}
which can be solved very simply as an ODE over the time horizon $[t_j, t_{j+1})$ with the solution
\begin{align}
    F^{(i)}(t) = D_j + (F^{(i)}(t_j) - D_j)e^{-s \lambda(t-t_j)} = e^{-s \lambda(t-t_j)}F^{(i)}(t_j) + D_j(1-e^{-s \lambda(t-t_j)}).
    \label{CF:DistFunction}
\end{align}
Now we want to use this to construct an iterative scheme that represents the distance at each freezing time point:
\begin{align*}
    E^{(i)}_{j+1}:=X^i_{t_{j+1}} - \overbar{X}^i_{t_{j+1}} = F^{(i)}(t_{j+1}).
\end{align*}
 We now substitute our original variables in this and noting that $\Delta t  = t_{j+1} - t_j$ for every $j$, we get the following: 
 \begin{align}
     E^{(i)}_{j+1} = F^{(i)}(t_j + \Delta t) = e^{-\lambda s \Delta t}E^{(i)}_j + D_j(1-e^{-\lambda s \Delta t}) = gE_j^{(i)} + D_j(1-g),
     \label{MFF:DistanceScheme}
 \end{align}
 with $g \coloneqq e^{-s \lambda  \Delta t}$.
 \vspace{.2 cm}

 With these tools in the bag, we are now ready to prove the main result of this section, the propagation of chaos for the consensus freezing scheme.
\begin{theorem}
    Let $T >0$ be a fixed final time, and let the interval $[0, T]$ be divided by $0= t_0 < t_1 <...<t_J = T$ with uniform step size $\Delta t$. Let $\mathcal{X}^N_T \text{ and } \overbar{\mathcal{X}}^N_T$ be defined by \eqref{CF:CFscheme} and \eqref{CF:MFIPS}, respectively, both sharing the same initial data $\mathcal{X}^N_0 = \overbar{\mathcal{X}}^N_0$ almost surely, with common distribution $\orho_0 \in \mathcal{P}_p(\mathbb{R}^d)$ for some $p \geq 2$. For a fixed \(M>0\), define the set
    \begin{align*}
     \Omega_M := \bigg\{ \omega \in \Omega : \supjJ \,  \frac{1}{N} \Sum{i}{N} \max \Big\{ \N{X^i_{t_j}}^p , \N{\overbar{X}^i_{t_j}}^p \Big \} \leq M \bigg\}.
 \end{align*}
    Then there exists a $\varepsilon^{\mathrm{rest}}_{\mathrm{F}}(N, T) >0$ with $\varepsilon^{\mathrm{rest}}_{\mathrm{F}}(N, T) \xrightarrow[]{N \rightarrow \infty}0$ such that for every $N \in \mathbb{N}$ and every $i \leq N$
    \begin{align*}
           \mathbb{E} \biggeck{ \suptT \,\N{ X^i_t - \overbar{X}^i_t}^p \mathbbm{1}_{\Omega_M} } \leq \varepsilon^{\mathrm{rest}}_{\mathrm{F}}(N, T).
    \end{align*}
    Consequently, if $\mathbb{P}(\Omega_M)>0$, then for every $N \in \mathbb{N}$ and every $i \leq N$
    \begin{align*}
           \mathbb{E} \biggeck{ \suptT \,\N{ X^i_t - \overbar{X}^i_t}^p \Big| \Omega_M } \leq \frac{\varepsilon^{\mathrm{rest}}_{\mathrm{F}}(N,T)}{\mathbb{P}(\Omega_M)}.
    \end{align*}
    \label{MFF: Thrm}
\end{theorem}
\begin{proof}
Let $j \in \{0, ..., J-1 \}$ be an arbitrary index and fix $i \leq N$. Our strategy is to first find an estimate for the term 
\begin{align*}
    \mathbb{E} \biggeck{ \supttj \N{X^i_{t} - \overbar{X}^i_{t} }^p \mathbbm{1}_{\Omega_M} },
\end{align*}
deduce from that a bound over the whole time horizon for
\begin{align*}
    \mathbb{E} \biggeck{ \suptT \N{X^i_{t} - \overbar{X}^i_{t} }^p \mathbbm{1}_{\Omega_M} }.
\end{align*}
Before we prove the actual statement, we need to clarify some things about the supremum of the difference term. The crucial simplification in the freezing scheme is that, within each interval $[t_j,t_{j+1})$, the two coupled particles evolve with frozen consensus points. Hence, the pathwise error on the whole interval is reduced to the error at the interval endpoints. We recall the function $F^{(i)}(t)$ \eqref{CF:DistFunction} that satisfies
\begin{align*}
   X^i_t - \overbar{X}^i_t =  F^{(i)}(t) = D_j + (F^{(i)}(t_j) - D_j)e^{-s \lambda(t-t_j)} \quad \forall t \in [t_j, t_{j+1}). 
\end{align*}
With slight abuse of the notation, we compute without writing expectations down:
\begin{align*}
\supttj \N{F^{(i)}(t)} &= \supttj \N{D_j + (F^{(i)}(t_j) - D_j)e^{-s \lambda(t-t_j)}}\\ &= \underset{t \in [0, \Delta t)}{\sup}  \N{D_j + (F^{(i)}(t_j) - D_j)e^{-s \lambda t}} \\
&= \underset{t \in [e^{- s \lambda\Delta t}, 1]}{\sup}  \N{D_j + (F^{(i)}(t_j) - D_j)t} \\ &= \underset{t \in [e^{- s \lambda \Delta t}, 1]}{\sup}  \N{tF^{(i)}(t_j) + (1-t) D_j}.
\end{align*}
We define $h \colon [e^{- s \lambda\Delta t}, 1] \rightarrow \mathbb{R}^d, \quad t \mapsto tF^{(i)}(t_j) + (1-t) D_j$. This is a line segment between $F^{(i)}(t_j)$ and $D_j$. Together with the convexity of the norms, this results in $t \mapsto \N{h(t)}$ being convex. Therefore, on the closed interval $[e^{-s \lambda\Delta t}, 1]$, as a continuous convex function, $t \mapsto \N{h(t)}$ attains its maximum at the endpoints, i.e.
\begin{align*}
    \underset{t \in [e^{- s \lambda\Delta t}, 1]}{\sup}  \N{tF^{(i)}(t_j) + (1-t) D_j}  = \max \Big \{ \N{F^{(i)}(t_j)}, \N{e^{-s \lambda \Delta t}F^{(i)}(t_j) + (1-e^{-s \lambda\Delta t}) D_j} \Big \}
\end{align*}
This means that
\begin{align}
    &\mathbb{E} \biggeck{ \supttj \N{X^i_{t} - \overbar{X}^i_{t} }^p \mathbbm{1}_{\Omega_M} } \notag \\ = \, &   \mathbb{E} \biggeck{\max \Big \{ \N{F^{(i)}(t_j)}^p, \N{e^{-s \lambda\Delta t}F^{(i)}(t_j) + (1-e^{-s \lambda\Delta t}) D_j}^p \Big \} \mathbbm{1}_{\Omega_M} } \notag \\
     = \, &\mathbb{E} \biggeck{\max \Big \{ \N{E^{(i)}_j}^p, \N{E^{(i)}_{j+1}}^p \Big \} \mathbbm{1}_{\Omega_M} }, \notag \\
     \label{CF:2Cases}
\end{align}
where $E^{(i)}_j$ is defined in (\ref{MFF:DistanceScheme}).
We deduce for every $j \in \{0,\dots,J-1\}$ that
\begin{align}
    \mathbb{E} \biggeck{ \supttj \N{X^i_{t} - \overbar{X}^i_{t} }^p \mathbbm{1}_{\Omega_M} }
    &=
    \mathbb{E} \biggeck{\max \Big \{ \N{E^{(i)}_j}^p, \N{E^{(i)}_{j+1}}^p \Big \} \mathbbm{1}_{\Omega_M} } \notag\\
    &\leq
    \mathbb{E} \biggeck{\Big(\N{E^{(i)}_j}^p + \N{E^{(i)}_{j+1}}^p\Big)\mathbbm{1}_{\Omega_M}}.
    \label{FreezingRestrIntervalSupBound}
\end{align}
Hence it is enough to estimate the quantities
\begin{align*}
A_j := \mathbb{E}\biggeck{\N{E^{(i)}_j}^p \mathbbm{1}_{\Omega_M}}, \qquad j=0,\dots,J.
\end{align*}
Since $E^{(i)}_0 = 0$, we have
\begin{align*}
A_0 = 0.
\end{align*}
We notice that $e^{-s\lambda \Delta t} \in (0, 1)$ and that $x \mapsto \N{x}^p$ is convex.
So for fixed $j \leq J-1$ and $i \leq N$, 
\begin{align}
    A_{j+1}
    &= \mathbb{E} \biggeck{\N{gE^{(i)}_j + D_j(1-g)}^p \mathbbm{1}_{\Omega_M} } \notag\\
    &\leq g \mathbb{E} \biggeck{ \N{E_j^{(i)}}^p \mathbbm{1}_{\Omega_M} }
    + (1-g) \mathbb{E} \Bigeck{ \N{D_j}^p \mathbbm{1}_{\Omega_M} }.
    \label{FreezingRestrTheoremBound1'}
\end{align}
We now proceed to make a suitable bound for $\mathbb{E} \Bigeck{ \N{D_j}^p \mathbbm{1}_{\Omega_M} }$.
As before,
\begin{align}
    \mathbb{E} \Bigeck{ \N{D_j}^p \mathbbm{1}_{\Omega_M} }
    &=
    \mathbb{E}\biggeck{ \N{X^\alpha(\rho^N_{t_j}) - X^\alpha(\overbar{\rho}_{t_j})}^p \mathbbm{1}_{\Omega_M}} \notag\\
    &\leq 2^{p-1}\Biggrou{\mathbb{E}\biggeck{ \N{X^\alpha(\rho^N_{t_j}) - X^\alpha(\overbar{\rho}^N_{t_j})}^p \mathbbm{1}_{\Omega_M}}
    +\mathbb{E} \biggeck{ \N{X^\alpha(\overbar{\rho}^N_{t_j}) - X^\alpha (\orho_{t_j})}^p \mathbbm{1}_{\Omega_M}}}.
    \label{FreezingRestrTheoremBound2'}
\end{align}
On the left hand side of the sum, we invoke Lemma \ref{CarrL2}, Example \ref{WasserBD}, and exchangeability:
\begin{align*}
    \mathbb{E}\biggeck{ \N{X^\alpha(\rho^N_{t_j}) - X^\alpha(\overbar{\rho}^N_{t_j})}^p \mathbbm{1}_{\Omega_M}}
    &\leq C_0^p \mathbb{E}\biggeck{ W_p^p(\rho^N_{t_j}, \overbar{\rho}^N_{t_j}) \mathbbm{1}_{\Omega_M}}\\
    &\leq C_0^p \mathbb{E}\biggeck{\frac{1}{N} \sum^N_{i=1}\N{X^i_{t_j} - \overbar{X}^i_{t_j}}^p \mathbbm{1}_{\Omega_M}}\\
    &= C_0^p \mathbb{E}\biggeck{ \N{X^i_{t_j} - \overbar{X}^i_{t_j}}^p \mathbbm{1}_{\Omega_M}}
    \\
    &= C_0^p A_j.
\end{align*}
As before, the right-hand side of \eqref{FreezingRestrTheoremBound2'} is bounded by Lemma \ref{CH:ConsBound}:
\begin{align*}
    \mathbb{E} \biggeck{ \N{X^\alpha(\overbar{\rho}^N_{t_j}) - X^\alpha (\orho_{t_j})}^p \mathbbm{1}_{\Omega_M}}
    \leq C_{\mathrm{EM}} N^{-\frac{p}{2}}.
\end{align*}
Therefore,
\begin{align*}
    \mathbb{E} \Bigeck{ \N{D_j}^p \mathbbm{1}_{\Omega_M} }
    \leq 2^{p-1}\Biggrou{C_0^p A_j + C_{\mathrm{EM}} N^{-\frac{p}{2}}}.
\end{align*}
Putting this into \eqref{FreezingRestrTheoremBound1'} yields
\begin{align}
    A_{j+1}
    &\leq
    \bigrou{g + 2^{p-1}(1-g) C_0^p} A_j
    + 2^{p-1}(1-g) C_{\mathrm{EM}} N^{-\frac{p}{2}}.
    \label{FreezingRestrRecursion}
\end{align}
Define
\begin{align*}
a_F := g + 2^{p-1}(1-g) C_0^p,
\qquad
b_F := 2^{p-1}(1-g) C_{\mathrm{EM}}.
\end{align*}
Then \eqref{FreezingRestrRecursion} becomes
\begin{align*}
A_{j+1} \leq a_F A_j + b_F N^{-p/2}.
\end{align*}
Since $A_0=0$, a straightforward induction gives
\begin{align}
    A_j
    \leq b_F N^{-\frac{p}{2}} \sum_{l=0}^{j-1} a_F^l,
    \qquad j=1,\dots,J.
    \label{FreezingRestrExplicitAj}
\end{align}
Using \eqref{FreezingRestrIntervalSupBound}, we get for each $j=0,\dots,J-1$
\begin{align*}
    \mathbb{E} \biggeck{ \supttj \N{X^i_{t} - \overbar{X}^i_{t} }^p \mathbbm{1}_{\Omega_M} }
    \leq A_j + A_{j+1}.
\end{align*}
Hence
\begin{align*}
    \mathbb{E} \biggeck{\suptT \,\N{ X^i_t - \overbar{X}^i_t}^p \mathbbm{1}_{\Omega_M} }
    &\leq \sum_{j=0}^{J-1} \mathbb{E} \biggeck{ \supttj \N{X^i_{t} - \overbar{X}^i_{t} }^p \mathbbm{1}_{\Omega_M} } \\
    &\leq \sum_{j=0}^{J-1}(A_j + A_{j+1})
    \leq 2\sum_{j=1}^{J} A_j \\
    &\leq 2 b_F N^{-\frac{p}{2}} \sum_{j=1}^{J}\sum_{l=0}^{j-1} a_F^l \\
    &= 2 b_F N^{-\frac{p}{2}} \sum_{l=0}^{J-1} (J-l) a_F^l
    =: \varepsilon^{\mathrm{rest}}_{\mathrm{F}}(N,T).
\end{align*}
Thus
\begin{align}
    \varepsilon^{\mathrm{rest}}_{\mathrm{F}}(N,T)
    =
    2\cdot 2^{p-1}(1-g) C_{\mathrm{EM}}
    N^{-\frac{p}{2}} \sum_{l=0}^{J-1}(J-l)
    \Big(g + 2^{p-1}(1-g) C_0^p\Big)^l
    .
    \label{FreezingRestrExplicitEpsilon}
\end{align}
Finally, if $\mathbb{P}(\Omega_M)>0$, then by conditional expectation,
\begin{align*}
    \mathbb{E} \biggeck{ \suptT \,\N{ X^i_t - \overbar{X}^i_t}^p \Big| \Omega_M }
    =
    \frac{\mathbb{E} \biggeck{ \suptT \,\N{ X^i_t - \overbar{X}^i_t}^p \mathbbm{1}_{\Omega_M} }}{\mathbb{P}(\Omega_M)}
    \leq \frac{\varepsilon^{\mathrm{rest}}_{\mathrm{F}}(N,T)}{\mathbb{P}(\Omega_M)}.
\end{align*}
\end{proof}
Theorem \ref{MFF: Thrm} directly implies pathwise propagation of chaos by coupling as stated in Definition \ref{Pre:ChaosCoup}, with expectations replaced by expectations on $\Omega_M$. In particular, by Example \ref{WasserBD},
\begin{align*}
    \mathbb{E}\Bigeck{W_p^p\bigrou{\rho^N_t, \orho^N_t}\mathbbm{1}_{\Omega_M}} \leq \varepsilon^{\mathrm{rest}}_{\mathrm{F}}(N,T) \quad \forall t \in [0,T].
\end{align*}
The speed of convergence is given by the $p$th root of our estimate $\varepsilon^{\mathrm{rest}}_{\mathrm{F}}(N,T)$ and is therefore $\mathcal{O}\bigrou{N^{-1/2}}$.

\section{The Consensus Hopping Scheme}
\label{ch:MFH}
In this section, we study a discrete-time interacting particle model known as the \textbf{consensus hopping scheme} (CH), originally introduced in \cite{Riedl2023GradientNeed}. This scheme provides a simple mechanism for modeling consensus formation by iteratively updating a state according to the center of a Gaussian distribution around the previous iterate. The scheme admits a natural probabilistic formulation and can be understood as the mean-field limit of a suitably defined interacting particle system.

We first introduce the hopping dynamics and the associated particle system and then establish the uniform moment bounds needed for the propagation of chaos analysis in Lemma \ref{NP_CHLemma}. In Theorem \ref{MFH: Thrm}, using the coupling method adapted to the discrete-time setting, we prove propagation of chaos uniformly over finite time horizons restricted to $\Omega_M$ with a quantitative rate of order $\mathcal{O}\bigrou{N^{-1/2}}$ in Wasserstein-$p$ distance (up to constants depending on the time horizon and model parameters).
\vspace{.2 cm}

The original definition of the hopping scheme is as follows: 
\begin{align}
    & x^{HS}_k = X^\alpha (\rho_{k}), \quad \rho_{k} = \mathcal{N}(x^{HS}_{k-1}, \sigma^2 I_d), \quad \text{for } 1 \leq k \in \mathbb{N},\\
    & x^{HS}_0 = x_0, \notag \label{HoppingSchemeOriginal}
\end{align}
which defines a time-discrete deterministic process on $\mathbb{R}^d$.
We claim that this can be understood as the mean-field limit of the following $N$-particle system $(\mathcal{X}^N_k)_k = (X^1_k, ...X_k^N)_k$:
\begin{align}
   & X^i_k = X^\alpha(\rho^N_{k-1}) + \eta^i_k, \quad \text{for } i \leq N, k \in \mathbb{N}
    \label{CH:Scheme} \\
   &\mathcal{X}^N_0 = (X^1_0,...,X^N_0) \sim \orho_0^{\otimes N}, \notag
\end{align}
where $\eta^i_k\sim \mathcal{N}(0, \sigma^2 I_d)$ are independent, identically distributed random vectors and $\rho^N_{k} := \mu_{\mathcal{X}^N_k}$ for every number of particles $N$, at every time step $k$. The probability space for analyzing the paths of the IPS defined in (\ref{CH:Scheme}) is 
\begin{align*}
\Big(
(\mathbb{R}^d)^N \times ((\mathbb{R}^d)^N)^{\mathbb N},
\mathcal B\big((\mathbb{R}^d)^N \times ((\mathbb{R}^d)^N)^{\mathbb N}\big),
\orho_0^{\otimes N} \otimes \big(\mathcal N(0,\sigma^2I_d)^{\otimes N}\big)^{\otimes \mathbb N}
\Big).
\end{align*}
We use the coupling method introduced in Definition \ref{Pre:ChaosCoup} to prove propagation of chaos. We define the $N$-particle system $\bigrou{\overbar{\mathcal{X}}^N_k}_k = (\overbar{X}^1_k, ..., \overbar{X}^N_k)_k$ by the following scheme:
\begin{align}
    &\overbar{X}^i_k = X^\alpha(\rho_{k-1}) + \eta^i_k \quad \text{for } i \leq N, k \in \mathbb{N} \label{CH:MFScheme} \\ 
    &\overbar{\mathcal{X}}^N_0 = (\overbar{X}^1_0,...,\overbar{X}^N_0) \sim \orho_0^{\otimes N}, \quad \overbar{\mathcal{X}}^N_0 = \mathcal{X}^N_0 \,\, \text{ a.s.}, \notag   
\end{align}
where we point out that $(\mathcal{X}^N_k)_k$ and $(\overbar{\mathcal{X}}^N_k)_k$ share the same initial condition and that for all $i \leq N$ at each time step $k$, $X^i_k$ and $\overbar{X}^i_k$ share the same random vector $\eta^i_k$.
Throughout this section, if not otherwise stated, the probabilities and expectations will be taken with respect to the probability space 
\begin{align*}
    \ProbStd := \Big(
(\mathbb{R}^d)^N \times ((\mathbb{R}^d)^N)^{\mathbb N},
\mathcal B\big((\mathbb{R}^d)^N \times ((\mathbb{R}^d)^N)^{\mathbb N}\big),
\orho_0^{\otimes N} \otimes \big(\mathcal N(0,\sigma^2I_d)^{\otimes N}\big)^{\otimes \mathbb N}
\Big).
\end{align*}
\subsection{Propagation of Chaos for the Hopping Scheme}
For the hopping scheme, the proof of the propagation of chaos will be different from that described in Section \ref{Pre:ProvingPropChaos}. This is because we are dealing with an interacting particle system that is discrete in time. We therefore do not require tools from stochastic analysis, but rather standard inequalities for countable sets of random variables. We see directly from \eqref{CH:Scheme} that the proof for propagation of chaos reduces to the problem of estimating \eqref{Pre:ConsensusDistProblem} described in Section \ref{Pre:ProvPropChaosCBO}, i.e.
\begin{align*}
    \mathbb{E}\biggeck{ \supu{k}{K}\N{X^\alpha(\rho^N_k) - X^\alpha(\orho_k)}^p} \leq \varepsilon(N, K) \xrightarrow[]{N \rightarrow \infty} 0 . 
\end{align*}
In order to prove propagation of chaos, we need moment bounds, uniformly over finite time horizons, not only for the interacting particle system (\ref{CH:Scheme}) and the mean-field process (\ref{CH:MFScheme}) but also for their respective empirical measures. We will first show that the empirical measures of the hopping particle system (\ref{CH:Scheme}) and of the mean-field system (\ref{CH:MFScheme}) have uniformly bounded moments with high probability. We then show that, given the uniform boundedness of the moments of the empirical measures, we can obtain a bound sufficient to verify Definition \ref{Pre:ChaosCoup}.\vspace{.2 cm}

We begin by showing that the interacting particle system \eqref{CH:Scheme} has finite $p$th moments, uniformly in finite time.
\begin{lemma}
     Consider the particle systems $\IPS{K}$ and $\MIPS{K}$ defined by \eqref{CH:Scheme} and \eqref{CH:MFScheme}, respectively, with initial data following the distribution $\orho_0 \in \mathcal{P}_p(\mathbb{R}^d)$ and a final time $K \in \mathbb{N}_{>0}$. Then there exist constants $C_{\mathrm{H}} >0$ and $\overbar{C}_{\mathrm{H}} >0$, both independent of $N$, such that for every $N \in \mathbb{N}$ and every $i \leq N$
    \begin{align*}
        \mathbb{E} \biggeck{ \supu{k}{K} \N{X^i_k}^p} \leq C_{\mathrm{H}} \quad \text{and} \quad \mathbb{E} \biggeck{ \supu{k}{K} \N{\overbar{X}^i_k}^p} \leq \overbar{C}_{\mathrm{H}}.
    \end{align*}
    \label{NP_CHLemma}
\end{lemma}

\begin{proof}
We show the proof only for the hopping particle system \eqref{CH:Scheme}. The proof for \eqref{CH:MFScheme} follows the same line of argument.
We fix $i \leq N$. Since $X^i_k = X^\alpha(\rho^N_{k-1}) + \eta^i_k$ holds only for $k \geq 1$, we separate the initial time and use the definition of the hopping scheme \eqref{CH:Scheme}:
    \begin{align}
        \mathbb{E} \biggeck{ \supu{k}{K} \N{X^i_k}^p}
        &\leq \mathbb{E} \Bigeck{\N{X^i_0}^p} + \mathbb{E} \biggeck{ \sup_{1 \leq k \leq K} \N{X^\alpha(\rho^N_{k-1}) + \eta^i_k}^p} \notag \\
        &\leq \mathbb{E} \Bigeck{\N{X^i_0}^p} + 2^{p-1} \biggrou{ \mathbb{E} \biggeck{ \sup_{1 \leq k \leq K} \N{X^\alpha(\rho^N_{k-1})}^p} + \mathbb{E} \biggeck{ \sup_{1 \leq k \leq K} \N{\eta^i_k}^p}}.
        \label{NP:Hopping_Moment1}
    \end{align}
The second term on the right hand side is straightforward to estimate, as it is the expected supremum of i.i.d. Gaussian random vectors, for which we can use Lemma \ref{Pre:LemmaNormalVectors}. 

We now turn to the first term. We compute
\begin{align}
   \mathbb{E} \biggeck{ \sup_{1 \leq k \leq K} \N{X^\alpha(\rho^N_{k-1})}^p}
   = \mathbb{E} \Biggeck{\sup_{1 \leq k \leq K} \N{\int \frac{x \oa(x)}{\Gnorm{\rho^N_{k-1}}} \, d\rho^N_{k-1}(x)}^p}.
   \label{HoppingMomentConsensus}
\end{align}
We define the random measure
\begin{align*}
\eta^\alpha_{k-1} := \frac{\oa \rho^N_{k-1}}{\Gnorm{\rho^N_{k-1}}},
\end{align*}
which is measurable by Proposition \ref{Pre:PropEmpMeas}. Since $\eta^\alpha_{k-1}$ is a probability measure whenever $\rho^N_{k-1}$ is a probability measure, Jensen's inequality yields
\begin{align}
    \mathbb{E} \Biggeck{\sup_{1 \leq k \leq K} \N{\int \frac{x \oa(x)}{\Gnorm{\rho^N_{k-1}}} d\rho^N_{k-1}(x)}^p}
    &= \mathbb{E} \Biggeck{\sup_{1 \leq k \leq K} \N{\int x \, d\eta^\alpha_{k-1}(x)}^p} \notag \\
    &\leq \mathbb{E} \Biggeck{\sup_{1 \leq k \leq K} \int \N{x}^p \, d\eta^\alpha_{k-1}(x)}.
    \label{HoppingMomentBound1}
\end{align}
Using Lemma \ref{GerberPropA4} with exponent $p$, we obtain
\begin{align}
    \mathbb{E} \Biggeck{\sup_{1 \leq k \leq K} \int \N{x}^p \, d\eta^\alpha_{k-1}(x)}
    = \mathbb{E} \Biggeck{\sup_{1 \leq k \leq K} \mathbb{E}_{\eta^\alpha_{k-1}} \bigeck{\N{x}^p}}
    \leq \mathbb{E} \biggeck{b_p \, \sup_{1 \leq k \leq K} \mathbb{E}_{\rho^N_{k-1}} \bigeck{\N{x}^p}}.
    \label{HoppingMomentBound2}
\end{align}
Combining \eqref{HoppingMomentBound1} and \eqref{HoppingMomentBound2}, we get from \eqref{HoppingMomentConsensus}
\begin{align}
    \mathbb{E} \biggeck{ \sup_{1 \leq k \leq K} \N{X^\alpha(\rho^N_{k-1})}^p}
    \leq b_p \mathbb{E} \biggeck{\sup_{1 \leq k \leq K}\, \mathbb{E}_{\rho^N_{k-1}} \bigeck{\N{x}^p}}.
    \label{NP:HoppingMoment2}
\end{align}
Writing out the expectation with respect to $\rho^N_{k-1}$ gives
\begin{align*}
    \mathbb{E} \biggeck{\sup_{1 \leq k \leq K}\, \mathbb{E}_{\rho^N_{k-1}} \bigeck{\N{x}^p}}
    &= \mathbb{E} \biggeck{\sup_{1 \leq k \leq K}\, \frac{1}{N}\sum^N_{j=1} \N{X^j_{k-1}}^p} \\
    &\leq \mathbb{E} \biggeck{ \frac{1}{N}\sum^N_{j=1} \sup_{1 \leq k \leq K} \N{X^j_{k-1}}^p} \\
    &= \frac{1}{N}\sum_{j=1}^N \mathbb{E} \biggeck{\supu{k}{K-1} \N{X^j_k}^p} \\
    &= \mathbb{E} \biggeck{\supu{k}{K-1} \N{X^i_k}^p},
\end{align*}
where we used exchangeability of the particles in the last step.
Hence
\begin{align*}
    \mathbb{E} \biggeck{ \sup_{1 \leq k \leq K} \N{X^\alpha(\rho^N_{k-1})}^p}
    \leq b_p \mathbb{E} \biggeck{\supu{k}{K-1} \N{X^i_k}^p}.
\end{align*}
Inserting this into \eqref{NP:Hopping_Moment1}, we arrive at
\begin{align*}
     \mathbb{E} \biggeck{ \supu{k}{K} \N{X^i_k}^p}
     \leq \mathbb{E} \Bigeck{\N{X^i_0}^p} + 2^{p-1} \biggrou{b_p \mathbb{E} \biggeck{\supu{k}{K-1} \N{X^i_k}^p} + \mathbb{E} \biggeck{ \sup_{1 \leq k \leq K} \N{\eta^i_k}^p}}.
\end{align*}
Now define
\begin{align*}
A_K := \mathbb{E} \biggeck{ \supu{k}{K} \N{X^i_k}^p},
\qquad
B_K := \mathbb{E} \biggeck{ \sup_{1 \leq k \leq K} \N{\eta^i_k}^p},
\end{align*}
so that $A_0 = \mathbb{E}\bigeck{\N{X^i_0}^p}$. Then the previous estimate reads
\begin{align*}
A_K \leq A_0 + 2^{p-1} b_p\, A_{K-1} + 2^{p-1} B_K.
\end{align*}
Iterating this recursion, or equivalently arguing by induction on $K$, yields
\begin{align}
     A_K
     \leq \mathbb{E} \Bigeck{\N{X^i_0}^p} \sum^K_{l=0} (2^{p-1}b_p)^l
     + 2^{p-1} \sum^K_{l = 1} B_l (2^{p-1}b_p)^{K-l}.
\end{align}
By Lemma \ref{Pre:LemmaNormalVectors} there exists a constant $C_p >0$ depending only on $p$ such that for every $l \leq K$
\begin{align*}
    B_l
    = \mathbb{E} \biggeck{ \sup_{1 \leq k \leq l} \N{\eta^i_k}^p}
    \leq \sigma^pC_p \Bigrou{d\log (2ld)}^{\frac{p}{2}}.
\end{align*}
Therefore,
\begin{align}
     \mathbb{E} \biggeck{ \supu{k}{K} \N{X^i_k}^p}
     \leq \mathbb{E} \Bigeck{\N{X^i_0}^p} \sum^K_{l=0} (2^{p-1}b_p)^l
     + 2^{p-1} \sigma^pC_p \sum^K_{l = 1} \Bigrou{d\log (2ld)}^{\frac{p}{2}} (2^{p-1}b_p)^{K-l}
     =:C_{\mathrm{H}}.
     \label{HoppingMomentFinal estiamte}
\end{align}
We note that $C_{\mathrm{H}}$ is independent of $N$, which concludes the proof.
\end{proof}
\begin{remark}
    Recall the consensus freezing moment estimate of Lemma \ref{NP:CFLemma} for $J = K$, with $a_s = 2^{p-1}\bigrou{e^{-s\lambda \Delta t}+ b_p(1- e^{-s\lambda \Delta t})}$, and pass to the limit $s \to \infty$. Since $a_s \to 2^{p-1}b_p$ and $e^{-2s\lambda\Delta t} \to 0$, we obtain
    \begin{align*}
        &\mathbb{E} \biggeck{\supjJ \N{X^i_{t_j}}^p}\\
        \leq  \, & \mathbb{E} \Bigeck{ \N{X^i_0}^p} \sum^K_{l=0} a_s^l
        + 2^{p-1}\delta^p \biggrou{\frac{1- e^{- 2s\lambda \Delta t}}{2 \lambda}}^{\frac{p}{2}} C_p \sum^{K}_{l = 1} \Bigrou{d\log (2ld)}^{\frac{p}{2}} a_s^{K-l} \\
        \xrightarrow[]{s \to \infty }\, & \mathbb{E} \Bigeck{\N{X^i_0}^p} \sum^K_{l=0} (2^{p-1}b_p)^l
        + 2^{p-1}  C_p \biggrou{\frac{\delta}{\sqrt{2 \lambda}}}^{p}\sum^K_{l = 1}  \Bigrou{d\log (2ld)}^{\frac{p}{2}} (2^{p-1}b_p)^{K-l}.
    \end{align*}
This is exactly the moment estimate \eqref{HoppingMomentFinal estiamte} of the consensus hopping scheme with variance $\sigma^2 = \frac{\delta^2}{2 \lambda}$. This connection aligns with the convergence of the freezing scheme to the hopping scheme for $s \to \infty$, proven in Theorem \ref{thm:FStoHS}, which is stated in Section \ref{section:globalconvergence}.
\end{remark}
Now we can use the uniform in time moment bounds of the particles to construct a lemma that ensures us that the probability of the moments of the empirical measures of the systems \eqref{CH:Scheme} and \eqref{CH:MFScheme} being bounded by $M$ converges to one as $M$ approaches infinity.
\begin{lemma}
   Let $K \in \mathbb{N}$ be the final time. Let $\mathcal{X}^N_K \text{ and } \overbar{\mathcal{X}}^N_K$ be the interacting particle systems of the finite particle system (\ref{CH:Scheme}) and the mean-field particle system (\ref{CH:MFScheme}), respectively, with initial data following the distribution $\orho_0 \in \mathcal{P}_p(\mathbb{R}^d)$. For a fixed \(M>0\), define the set \begin{align*}
       \Omega_M := \bigg\{ \omega \in \Omega : \supu{k}{K} \frac{1}{N} \Sum{i}{N} \max \Big\{ \N{X^i_k}^p , \N{\overbar{X}^i_k}^p \Big \} \leq M \bigg\}.
   \end{align*}Then there exists a constant $C >0$ independent of $N$ such that
    \begin{align*}
        \mathbb{P}(\Omega_M) \geq 1 - \frac{C}{M}.
    \end{align*}
    \label{CH:Moment}
\end{lemma}

\begin{proof}
    As in Lemma \ref{MFD:Moment} and Lemma \ref{MFF:Moment}, we compute with Markov's inequality: 
    \begin{align*}
        \mathbb{P}(\Omega^c_M) \leq \frac{\maxu{i}{N}\, \mathbb{E} \bigg[ \supu{k}{K} \N{X^i_k}^p \bigg ]}{M} + \frac{\maxu{i}{N}\, \mathbb{E} \biggeck{ \supu{k}{K} \N{ \overbar{X}^i_k}^p }}{M}.
    \end{align*}
Lemma \ref{NP_CHLemma} immediately gives us the bounds we need:
\begin{align*}
    \maxu{i}{N}\, \mathbb{E} \biggeck{ \supu{k}{K} \N{ X^i_k}^p } \leq C_{\mathrm{H}} , \quad \quad \maxu{i}{N}\, \mathbb{E} \biggeck{ \supu{k}{K} \N{ \overbar{X}^i_k}^p } \leq \overbar{C}_{\mathrm{H}}.
\end{align*}
We set $\frac{C}{2} \coloneqq \max \Big \{ C_{\mathrm{H}}, \overbar{C}_{\mathrm{H}} \Big \}$ and get
\begin{align*}
    \frac{\maxu{i}{N}\, \mathbb{E} \bigg[ \supu{k}{K} \N{X^i_k}^p \bigg ]}{M} + \frac{\maxu{i}{N}\, \mathbb{E} \bigg[ \supu{k}{K} \N{ \overbar{X}^i_k}^p \bigg]}{M} \leq \frac{C}{M}.
\end{align*}
\end{proof}
\vspace{0.5 cm}

Lemma \ref{CH:Moment} ensures that the following main statement on propagation of chaos for the hopping scheme is meaningful, since we condition it on $\Omega_M$. As stated above, we want to show propagation of chaos by coupling uniformly over a finite time horizon, i.e. the pathwise version of Definition \ref{Pre:ChaosCoup}.
\begin{theorem}
    Let $K \in \mathbb{N}, N\in \mathbb{N}$. Let $\mathcal{X}^N_K \text{ and } \overbar{\mathcal{X}}^N_K$ be the interacting particle systems of the finite particle system (\ref{CH:Scheme}) and the mean-field particle system (\ref{CH:MFScheme}), respectively, both sharing the same initial data $\mathcal{X}^N_0 = \overbar{\mathcal{X}}^N_0$ almost surely, with common distribution $\orho_0 \in \mathcal{P}_p(\mathbb{R}^d)$ for some $p \geq 2$. For a fixed \(M>0\), define
    \begin{align*}
       \Omega_M := \bigg\{ \omega \in \Omega : \supu{k}{K} \frac{1}{N} \Sum{i}{N} \max \Big\{ \N{X^i_k}^p , \N{\overbar{X}^i_k}^p \Big \} \leq M \bigg\}.
   \end{align*}
    Then there exists a $\varepsilon^{\mathrm{rest}}_{\mathrm{H}}(N, K)>0$ with $\varepsilon^{\mathrm{rest}}_{\mathrm{H}}(N, K) \xrightarrow[]{N \rightarrow \infty}0$ such that for every $N \in \mathbb{N}$ and every $i \leq N$
    \begin{align*}
         \mathbb{E} \biggeck{ \supu{k}{K} \,\N{ X^i_k - \overbar{X}^i_k}^p \mathbbm{1}_{\Omega_M} } \leq \varepsilon^{\mathrm{rest}}_{\mathrm{H}}(N, K).
    \end{align*}
    Consequently, if $\mathbb{P}(\Omega_M)>0$, then for every $N \in \mathbb{N}$ and every $i \leq N$
    \begin{align*}
         \mathbb{E} \biggeck{ \supu{k}{K} \,\N{ X^i_k - \overbar{X}^i_k}^p \Big| \Omega_M } \leq \frac{\varepsilon^{\mathrm{rest}}_{\mathrm{H}}(N, K)}{\mathbb{P}(\Omega_M)}.
    \end{align*}
    \label{MFH: Thrm}
\end{theorem}
\begin{proof}
To start, we bound the expected distance for fixed $i\leq N$ and $k \leq K$:
\begin{align}
   &\mathbb{E} \biggeck{ \N{X^i_{k+1} - \overbar{X}^i_{k+1}}^p \mathbbm{1}_{\Omega_M}} \notag \\
   = \, &\mathbb{E} \biggeck{ \N{X^\alpha(\rho^N_k) - X^\alpha(\overbar{\rho}_k)}^p \mathbbm{1}_{\Omega_M} } \notag\\
   \leq \, & 2^{p-1} \biggrou{\mathbb{E} \biggeck{\N{X^\alpha(\rho^N_k) - X^\alpha(\overbar{\rho}^N_k)}^p \mathbbm{1}_{\Omega_M}} + \mathbb{E} \biggeck{\N{X^\alpha(\overbar{\rho}^N_k) - X^\alpha(\overbar{\rho}_k)}^p \mathbbm{1}_{\Omega_M}}}.
     \label{HoppingDistEst1}
\end{align}
Here, $\overbar{\rho}^N_k$ denotes the empirical measure of the mean-field finite particle system at time $k$.
On the first part of this sum, we can invoke Lemma \ref{CarrL2} and then Example \ref{WasserBD} together with the exchangeability of the particles:
\begin{align*}
    \mathbb{E} \biggeck{\N{X^\alpha(\rho^N_k) - X^\alpha(\overbar{\rho}^N_k)}^p \mathbbm{1}_{\Omega_M}}
    \leq C_0^p  \mathbb{E} \Bigeck{W_p^p(\rho^N_k, \overbar{\rho}^N_k) \mathbbm{1}_{\Omega_M}}
    \leq C_0^p \mathbb{E} \biggeck{ \N{X^i_{k} - \overbar{X}^i_{k}}^p \mathbbm{1}_{\Omega_M} }.
\end{align*}
Thus, we can rewrite \eqref{HoppingDistEst1} as
\begin{align*}
    & \mathbb{E} \biggeck{ \N{X^i_{k+1} - \overbar{X}^i_{k+1}}^p \mathbbm{1}_{\Omega_M}}
    \\ \leq \,  &2^{p-1} \biggrou{ C_0^p \mathbb{E} \biggeck{ \N{X^i_{k} - \overbar{X}^i_{k}}^p \mathbbm{1}_{\Omega_M} } + \mathbb{E} \biggeck{\N{X^\alpha(\overbar{\rho}^N_k) - X^\alpha(\overbar{\rho}_k)}^p \mathbbm{1}_{\Omega_M}}}.
\end{align*}
We note that the constant $C_0$ arising from Lemma \ref{CarrL2} depends only on $(\alpha, p, M, C_1, C_2, C_3, R)$, and in particular not on $N$.
Now, define
\begin{align*}
A_k := \mathbb{E} \biggeck{ \N{X^i_{k} - \overbar{X}^i_{k}}^p \mathbbm{1}_{\Omega_M} }.
\end{align*}
Then the previous estimate becomes
\begin{align*}
A_{k+1}
\leq 2^{p-1} C_0^p A_k
+ 2^{p-1}\mathbb{E} \biggeck{\N{X^\alpha(\overbar{\rho}^N_k) - X^\alpha(\overbar{\rho}_k)}^p \mathbbm{1}_{\Omega_M}}.
\end{align*}
Iterating this recursion down to time $0$ yields
\begin{align}
    A_{k+1}
    &\leq \bigrou{2^{p-1}C_0^p}^{k+1} A_0
    + 2^{p-1}\sum_{l=0}^{k}  \bigrou{2^{p-1}C_0^p}^{k-l} \mathbb{E} \biggeck{\N{X^\alpha(\overbar{\rho}^N_l) - X^\alpha(\overbar{\rho}_l)}^p \mathbbm{1}_{\Omega_M}} \notag\\
    &= 2^{p-1}\sum_{l=0}^{k}  \bigrou{2^{p-1}C_0^p}^{k-l} \mathbb{E} \biggeck{\N{X^\alpha(\overbar{\rho}^N_l) - X^\alpha(\overbar{\rho}_l)}^p \mathbbm{1}_{\Omega_M}},
    \label{HoppQMF}
\end{align}
where we used $A_0=0$ in the last step, since both systems share the same initial condition. It remains to bound the expectation involving the empirical mean-field measure $\overbar{\rho}^N_k$ and the mean-field law $\overbar{\rho}_k$.
For every $k \leq K$, define
\begin{align*}
\overbar{\Omega}_{M,k}
:=
\left\{
\omega \in \Omega :
\frac{1}{N}\sum_{i=1}^N \N{\overbar{X}^i_k}^p \le M
\right\}.
\end{align*}
Since
\begin{align*}
\Omega_M \subseteq
\left\{
\omega \in \Omega :
\sup_{0 \leq l \leq K} \frac{1}{N}\sum_{i=1}^N \N{\overbar{X}^i_l}^p \le M
\right\}
\subseteq \overbar{\Omega}_{M,k},
\end{align*}
we have
\begin{align*}
\mathbbm{1}_{\Omega_M} \le \mathbbm{1}_{\overline{\Omega}_{M,k}}.
\end{align*}
Hence we can use Lemma \ref{CH:ConsBound} in \eqref{HoppQMF} to obtain
\begin{align*}
    \mathbb{E} \biggeck{\N{X^\alpha(\overbar{\rho}^N_k) - X^\alpha(\overbar{\rho}_k)}^p \mathbbm{1}_{\Omega_M}}
    \leq
    \mathbb{E} \biggeck{\N{X^\alpha(\overbar{\rho}^N_k) - X^\alpha(\overbar{\rho}_k)}^p \mathbbm{1}_{\overline{\Omega}_{M,k}}}
    \leq C_{\mathrm{EM}} N^{-\frac{p}{2}}.
\end{align*}
Inserting this estimate into \eqref{HoppQMF} yields for every $k \leq K$:
\begin{align*}
    \mathbb{E} \biggeck{ \N{X^i_{k+1} - \overbar{X}^i_{k+1}}^p \mathbbm{1}_{\Omega_M} }
    &\leq 2^{p-1}\sum_{l=0}^{k}  \bigrou{2^{p-1}C_0^p}^{k-l}
    \mathbb{E} \biggeck{\N{X^\alpha(\overbar{\rho}^N_l) - X^\alpha(\overbar{\rho}_l)}^p \mathbbm{1}_{\Omega_M}} \\
    &\leq 2^{p-1} C_{\mathrm{EM}} N^{-\frac{p}{2}}
    \sum_{l=0}^{k}  \bigrou{2^{p-1}C_0^p}^{k-l}.
\end{align*}
Define
\begin{align*}
a_H := 2^{p-1}C_0^p,
\qquad
b_H := 2^{p-1}C_{\mathrm{EM}}.
\end{align*}
Then the previous estimate reads
\begin{align*}
    \mathbb{E} \biggeck{ \N{X^i_{k+1} - \overbar{X}^i_{k+1}}^p \mathbbm{1}_{\Omega_M} }
    \leq b_H N^{-\frac{p}{2}} \sum_{l=0}^{k} a_H^{\,k-l}
    = b_H N^{-\frac{p}{2}} \sum_{l=0}^{k} a_H^{\,l}.
\end{align*}
Since $X_0^i$ and $ \overbar{X}_0^i$ share the same initial condition, we have
\begin{align*}
\mathbb{E} \Bigeck{\N{X_0^i-\overbar{X}_0^i}^p \mathbbm{1}_{\Omega_M}}=0,
\end{align*}
and therefore
\begin{align*}
    \mathbb{E} \biggeck{ \supu{k}{K}
    \N{ X^i_k - \overbar{X}^i_k}^p \mathbbm{1}_{\Omega_M}}
    &\leq \sum_{k=1}^{K}\mathbb{E} \biggeck{
    \N{ X^i_k - \overbar{X}^i_k}^p \mathbbm{1}_{\Omega_M}} \\
    &\leq b_H N^{-\frac{p}{2}}
    \sum_{k=1}^{K}\sum_{l=0}^{k-1} a_H^{\,l} \\
    &= b_H N^{-\frac{p}{2}}
    \sum_{l=0}^{K-1}(K-l)a_H^{\,l}.
\end{align*}
Hence, we take
\begin{align}
    \varepsilon^{\mathrm{rest}}_{\mathrm{H}}(N,K)
    :=
    2^{p-1}C_{\mathrm{EM}} N^{-\frac{p}{2}}
    \Bigg(\sum_{l=0}^{K-1}(K-l)(2^{p-1}C_0^p)^l\Bigg).
    \label{MFH:RestPoCBound}
\end{align}
Finally, if $\mathbb{P}(\Omega_M)>0$, then by the definition of conditional expectation,
\begin{align*}
    \mathbb{E} \biggeck{ \supu{k}{K}
    \N{ X^i_k - \overbar{X}^i_k}^p \Big| \Omega_M}
    =
    \frac{\mathbb{E} \biggeck{ \supu{k}{K}
    \N{ X^i_k - \overbar{X}^i_k}^p \mathbbm{1}_{\Omega_M}}}{\mathbb{P}(\Omega_M)}
    \leq \frac{\varepsilon^{\mathrm{rest}}_{\mathrm{H}}(N,K)}{\mathbb{P}(\Omega_M)}.
\end{align*}
This concludes the proof of Theorem \ref{MFH: Thrm}.
\end{proof}
The conclusion of Theorem \ref{MFH: Thrm} proves propagation of chaos, in the sense of Definition \ref{Pre:ChaosCoup}, restricted to the event set $\Omega_M$, for the consensus hopping scheme, with convergence rate of order $\mathcal{O}\bigrou{N^{-1/2}}$.

\begin{remark}
Recall the propagation-of-chaos estimate for the freezing scheme \eqref{FreezingRestrExplicitEpsilon} with $J=K$ freezing intervals and pass to the limit $s \longrightarrow \infty$ (which in this case means $g \longrightarrow 0$):
   \begin{align*}
    \mathbb{E} \biggeck{\suptT \,\N{ X^i_t - \overbar{X}^i_t}^p  \mathbbm{1}_{\Omega_M}} 
    &\leq 2\cdot 2^{p-1}(1-g) C_{\mathrm{EM}}N^{-\frac{p}{2}}
    \sum_{l=0}^{K-1}(K-l)\bigrou{g+ 2^{p-1}(1-g)C_0^p}^{l}\\
    &\xrightarrow[]{s \to \infty}
    2\cdot 2^{p-1} C_{\mathrm{EM}}N^{-\frac{p}{2}}
    \sum_{l=0}^{K-1}(K-l)\bigrou{2^{p-1}C_0^p}^{l}.
    \end{align*}
By \eqref{MFH:RestPoCBound}, the corresponding estimate for the hopping scheme is
\begin{align*}
    \mathbb{E} \biggeck{\supu{k}{K} \N{ X^i_k - \overbar{X}^i_k}^p \mathbbm{1}_{\Omega_M}}
    \leq 2^{p-1} C_{\mathrm{EM}} N^{-\frac{p}{2}}
    \sum_{l=0}^{K-1}(K-l)\bigrou{2^{p-1}C_0^p}^{l}.
\end{align*}
Thus, the two estimates have the same structure, differing only in the appearance of the factor $2$ in the freezing case. This reflects the convergence of the freezing scheme to the hopping scheme as $s \to \infty$, proven in Theorem \ref{thm:FStoHS}.
\end{remark}

\chapter{Unrestricted Propagation of Chaos} \label{chap:PropChaUnrestricted}
In Chapter \ref{chap:PropGoodSets}, we could show propagation of chaos for all three schemes at rate $\mathcal{O}\bigrou{N^{-1/2}}$. However, we had to restrict ourselves to a set of good events of the form
\begin{align*}
    \Omega_M := \bigg\{ \omega \in \Omega : \suptT \,  \frac{1}{N} \Sum{i}{N} \max \Big\{ \N{X^i_t}^p , \N{\overbar{X}^i_t}^p \Big \} \leq M \bigg\}.
\end{align*}
Although we could prove that $\mathbb{P}(\Omega_M) \xrightarrow[]{M \rightarrow \infty} 1$, it would be more satisfying to have an unrestricted estimate. In this chapter, we demonstrate that achieving this objective is indeed possible. We will prove that for bounded moments of sufficiently large order, we can prove propagation of chaos for all three schemes at rate $\mathcal{O}(N^{-1/2})$. At first glance, the results of this chapter are strictly better than those of Chapter \ref{chap:PropGoodSets}. However, this improvement comes with a price: we can only achieve $\mathcal{O}(N^{-1/2})$ in Wasserstein-$p$ distance if we have bounded moments of order $q \geq p^2 + p$ as opposed to only needing bounded moments of order $p$ as in the restricted case. We will see that this already has consequences in widely regarded cases, such as propagation of chaos in Wasserstein-$2$ distance with bounded moments of order $4$. Here, the restricted approach achieves $\mathcal{O}(N^{-1/2})$, while the unrestricted rate is slower with $\mathcal{O}(N^{-1/4})$.

To achieve unrestricted propagation of chaos, we use a technique from \cite{Gerber2025Mean-fieldSampling}, where the authors split the expected distance between the consensus points of the empirical measures as follows.
\begin{align}
    &\mathbb{E} \Bigeck{ \N{X^\alpha(\rho^N_\tau) - X^\alpha (\overbar{\rho}^N_\tau)}^p} \notag \\
    =\, &\mathbb{E} \Bigeck{ \N{X^\alpha(\rho^N_\tau) - X^\alpha (\overbar{\rho}^N_\tau)}^p \mathbbm{1}_{\Omega_{N, \tau}}}
    + \mathbb{E} \Bigeck{ \N{X^\alpha(\rho^N_\tau) - X^\alpha (\overbar{\rho}^N_\tau) }^p  \mathbbm{1}_{\Omega \setminus \Omega_{N, \tau}}}, \label{NP:Split1}
\end{align}
where
\begin{align*}
    \Omega_{N, t} := \bigg \{ \omega \in \Omega: \frac{1}{N} \sum^N_{i=1} \N{\overbar{X}^i_t}^p < M \bigg \},
\end{align*}
for $M>0$ being chosen so that $\mathbb{E}\biggeck{\suptT \N{\overbar{X}^i_t}^p}<M$ for every $i \leq N$.
Introducing this set gives us the advantage that we can now exploit the improved stability estimate on the good event $\Omega_{N,t}$, while its complement $\Omega \setminus \Omega_{N,t}$ is a rare event. Using the independence of the mean-field particles together with Lemma \ref{GerberLemma2.5}, we obtain $\mathbb{P}(\Omega \setminus \Omega_{N, t}) \xrightarrow[]{N \rightarrow \infty}0 $. Combined with Hölder's inequality and Lemma \ref{GerberPropA4}, we can estimate the second term on the right-hand side of \eqref{NP:Split1} by
\begin{align*}
    \mathbb{E} \Bigeck{ \N{X^\alpha(\rho^N_\tau) - X^\alpha (\overbar{\rho}^N_\tau) }^p  \mathbbm{1}_{\Omega \setminus \Omega_{N, \tau}}} \leq C \, \mathbb{P}(\Omega \setminus \Omega_{N, \tau})^{\frac{q-p}{q}} \xrightarrow[]{N \rightarrow \infty}0,
\end{align*}
as $C$ is independent of $N$. 
\newline
Thus, we are left with bounding the distance between the consensus points on
\begin{align*}
    \Omega_{N, t} = \bigg \{ \omega \in \Omega: \frac{1}{N} \sum^N_{i=1} \N{\overbar{X}^i_t}^p < M \bigg \},
\end{align*}
which is a slightly worse set to deal with than
\begin{align*}
    \Omega_M = \bigg\{ \omega \in \Omega : \suptT \,  \frac{1}{N} \Sum{i}{N} \max \Big\{ \N{X^i_t}^p , \N{\overbar{X}^i_t}^p \Big \} \leq M \bigg\},
\end{align*}
as we do not have the assurance that the moments of the empirical measures $\rho^N_t$ are uniformly bounded by $M$. We see that being able to make an estimate with  $\mathbb{P}(\Omega \setminus \Omega_{N, t}) \xrightarrow[]{N \rightarrow \infty}0 $ comes at the price of not being able to use Lemma \ref{CarrL2} as we cannot ensure the uniform boundedness of the empirical $p$th moments of the system $\IPS{T}$. A workaround for this is provided by Lemma \ref{GerberStability2}, which gives the same statement as Lemma \ref{CarrL2}, but relies only on the uniform boundedness of one measure. This means that we are able to control the distance between the consensus points of the empirical measure by the distance of two particles as we did in Chapter \ref{chap:PropGoodSets}:
\begin{align*}
    \mathbb{E} \Bigeck{ \N{X^\alpha(\rho^N_\tau) - X^\alpha (\overbar{\rho}^N_\tau)}^p \mathbbm{1}_{\Omega_{N, \tau}}} &\leq C^p_0 \mathbb{E} \Bigeck{W_p^p (\rho^N_\tau, \orho^N_\tau )} \leq C_0^p\frac{1}{N} \Sum{i}{N} \mathbb{E} \biggeck{\N{X^i_\tau - \overbar{X}^i_\tau}^p} \notag \\ &\leq C_0^p\maxu{i}{N} \,  \mathbb{E} \biggeck{\N{X^i_\tau - \overbar{X}^i_\tau}^p} = C_0^p\mathbb{E} \biggeck{\N{X^i_\tau - \overbar{X}^i_\tau}^p},
\end{align*}
where we used the exchangeability of the particle labels in the last equality. From this point onward, we are able to use the same lines of arguments as in Chapter \ref{chap:PropGoodSets} to achieve a propagation of chaos estimate. We see that the limiting factor that slows down our convergence rate is in
\begin{align}
    \mathbb{E} \Bigeck{ \N{X^\alpha(\rho^N_\tau) - X^\alpha (\overbar{\rho}^N_\tau) }^p  \mathbbm{1}_{\Omega \setminus \Omega_{N, \tau}}} \leq C \, \mathbb{P}(\Omega \setminus \Omega_{N, \tau})^{\frac{q-p}{q}} \leq C_1 N^{-\frac{q-p}{2p}}
\end{align}
The explicit convergence rate for $\mathbb{P}(\Omega \setminus \Omega_{N, \tau})$ is given by Lemma \ref{GerberLemma2.5} and is $\mathcal{O}\bigrou{N^{-q/2p}}$. Therefore, the convergence rate for $\mathbb{P}(\Omega \setminus \Omega_{N, \tau})^{\frac{q-p}{q}}$ is $\mathcal{O}\Bigrou{N^{-\frac{q-p}{2p}}}$. Inserting the above estimates into \eqref{NP:Split1} gives us the following estimate for the distance between the consensus points of the finite particle systems:
\begin{align}
    &\mathbb{E} \Bigeck{ \N{X^\alpha(\rho^N_\tau) - X^\alpha (\overbar{\rho}^N_\tau)}^p} \notag \\
    =\, &\mathbb{E} \Bigeck{ \N{X^\alpha(\rho^N_\tau) - X^\alpha (\overbar{\rho}^N_\tau)}^p \mathbbm{1}_{\Omega_{N, \tau}}} + \mathbb{E} \Bigeck{ \N{X^\alpha(\rho^N_\tau) - X^\alpha (\overbar{\rho}^N_\tau) }^p  \mathbbm{1}_{\Omega \setminus \Omega_{N, \tau}}} \notag \\
    \leq \,  &C_0^p\mathbb{E} \biggeck{\N{X^i_\tau - \overbar{X}^i_\tau}^p} + C_1 N^{-\frac{q-p}{2p}} \label{NP:Split2}
\end{align}
Note that we can, as in the previous sections, estimate the convergence of the mean-field approximations of $\orho^N_\tau$ towards $\orho_\tau$ with the improved Lemma \ref{GerberLemma3.7}:
\begin{align*}
    \mathbb{E} \biggeck{\N{X^\alpha(\orho^N_\tau) - X^\alpha(\orho_\tau)}^p} \leq C_2 N^{-\frac{p}{2}}.
\end{align*}
Combined with \eqref{NP:Split2}, this yields the following for the distance estimate between the consensus of the interacting particle system and the mean-field limit:
\begin{align*}
    \mathbb{E} \biggeck{\N{X^\alpha(\rho^N_\tau) - X^\alpha(\orho_\tau)}^p} &\leq 2^{p-1}\biggrou{\mathbb{E} \biggeck{\N{X^\alpha(\rho^N_\tau) - X^\alpha(\orho^N_\tau)}^p}+ \mathbb{E} \biggeck{\N{X^\alpha(\orho^N_\tau) - X^\alpha(\orho_\tau)}^p}}\\
    &\leq 2^{p-1} \Bigrou{C_2 N^{-\frac{p}{2}} + C_0^p\mathbb{E} \biggeck{\N{X^i_\tau - \overbar{X}^i_\tau}^p} + C_1 N^{-\frac{q-p}{2p}}}\\
    &\lesssim \mathbb{E} \biggeck{\N{X^i_\tau - \overbar{X}^i_\tau}^p} + N^{- \min \{ \frac{p}{2}, \frac{q-p}{2p}\}}.
\end{align*}
For the remaining estimation parts, we can use the same lines of reasoning as in the propagation of chaos results of Chapter \ref{chap:PropGoodSets}, which ultimately yields unrestricted propagation of chaos at rate $\mathcal{O} \bigrou{ N^{- \min \{ \frac{p}{2}, \frac{q-p}{2p}\}}}$.

\section{Preliminaries Revisited}
We now present some refinements of selected results from Section \ref{Pre:ToolsForMeanField} that we will need to prove unrestricted propagation of chaos. 
We begin with the most novel statement of those added in this section. As previously hinted, this is an improvement of the original Wasserstein stability estimate for the consensus point, Lemma \ref{CarrL2}, which only relies on one of the two measures to have an explicit bound on its $p$th moment. We also include the proof for the sake of clarity.
\begin{lemma}[{{\cite[Corollary 3.3]{Gerber2025Mean-fieldSampling}}}]
  Let $p \geq 1$, let $f$ fulfill \eqref{A1}, \eqref{A2} and \eqref{A3}, and let $\mu, \nu \in \mathcal{P}_p(\mathbb{R}^d)$ be such that there exists $M > 0$ with
    \begin{align*}
    \mathbb{E}_\mu \bigeck{\N{x}^p} \leq M.
    \end{align*}
    Then there exists a constant $C = C(\alpha, p, M, C_1, C_2, C_3, R)$ such that
    \begin{align*}
        \N{X^\alpha(\mu) - X^\alpha(\nu)} \leq C W_p(\mu, \nu).
    \end{align*}
\label{GerberStability2}
\end{lemma}

\begin{proof}
Define $\eta_\rho^\alpha := \frac{\oa \rho}{\Gnorm{\rho}}$ for $\rho \in \PRd$. By Lemma \ref{GerberPropA4}, there exists a constant $K>0$ such that for any $\rho \in \mathcal{P}_p(\mathbb{R}^d)$,
\begin{align}
    \N{X^\alpha(\rho)}
    = \N{\mathbb{E}_{\eta^\alpha_\rho}[x]}
    \leq \mathbb{E}_{\eta^\alpha_\rho}[\N{x}]
    \leq K \mathbb{E}_\rho \bigeck{\N{x}}
    \leq K \mathbb{E}_\rho \bigeck{\N{x}^p}^{\frac{1}{p}}
    = K W_p(\rho, \delta_0).
    \label{Corollary33Stability}
\end{align}
We assume that the claim is false. Then for every $i \in \mathbb{N}$ there exist measures $\mu_i, \nu_i \in \mathcal{P}_p(\mathbb{R}^d)$ such that
\begin{align*}
\mathbb{E}_{\mu_i} \bigeck{\N{x}^p} \leq M
\end{align*}
and
\begin{align}
    \frac{\N{X^\alpha(\mu_i) - X^\alpha(\nu_i)}}{W_p(\mu_i, \nu_i)} \geq i.
    \label{Corollary33Contra}
\end{align}
In particular, this implies that
$
W_p(\mu_i,\nu_i) \xrightarrow[]{i\to\infty} 0
$
cannot hold. Note that if $(\mathbb{E}_{\nu_i} \bigeck{\N{x}^p})_i$ were bounded, then both $(\mu_i)_i$ and $(\nu_i)_i$ would have uniformly bounded $p$th moments, and Lemma \ref{CarrL2} would yield
\begin{align*}
\N{X^\alpha(\mu_i)-X^\alpha(\nu_i)} \leq C' W_p(\mu_i,\nu_i)
\end{align*}
for some constant $C' >0$ independent of $i$, contradicting \eqref{Corollary33Contra}. Hence
\begin{align*}
\mathbb{E}_{\nu_i} \bigeck{\N{x}^p} \xrightarrow[]{i \rightarrow \infty} \infty.
\end{align*}
Since
\begin{align*}
W_p^p(\nu_i,\delta_0)=\mathbb{E}_{\nu_i}\bigeck{\N{x}^p},
\end{align*}
it follows that
\begin{align*}
W_p(\nu_i,\delta_0)\xrightarrow[]{i\to\infty}\infty.
\end{align*}
By the inverse triangle inequality, for every $i \in \mathbb{N}$,
\begin{align}
    W_p(\mu_i, \nu_i)
    \geq W_p(\nu_i, \delta_0) - W_p(\mu_i, \delta_0)
    \geq W_p(\nu_i, \delta_0) - M^{1/p}.
    \label{Corollary33Est1}
\end{align}
Moreover, for every $i$,
\begin{align}
    \N{X^\alpha(\mu_i) - X^\alpha(\nu_i)}
    &\leq \N{X^\alpha(\mu_i)} + \N{X^\alpha(\nu_i)} \notag\\
    &\leq K\big(M^{1/p} + W_p(\nu_i, \delta_0)\big),
    \label{Corollary33Est2}
\end{align}
where we used \eqref{Corollary33Stability} in the last inequality. Combining \eqref{Corollary33Est1} and \eqref{Corollary33Est2}, we obtain
\begin{align*}
    \underset{i \rightarrow \infty}{\limsup}
    \frac{\N{X^\alpha(\mu_i) - X^\alpha(\nu_i)}}{W_p(\mu_i, \nu_i)}
    &\leq
    \underset{i \rightarrow \infty}{\limsup}
    \frac{\N{X^\alpha(\mu_i) - X^\alpha(\nu_i)}}{W_p(\nu_i, \delta_0) - M^{1/p}} \\
    &\leq
    K \underset{i \rightarrow \infty}{\limsup}
    \frac{W_p(\nu_i, \delta_0) + M^{1/p}}{W_p(\nu_i, \delta_0) - M^{1/p}}
    = K,
\end{align*}
which contradicts \eqref{Corollary33Contra}. This completes the proof.
\end{proof}
\vspace{.3 cm}
The following Lemma provides a generalization of Lemma \ref{CH:ConsBound} where we do not need an explicit bound on the moments of the empirical measure.
\begin{lemma}[{{\cite[Lemma 3.7]{Gerber2025Mean-fieldSampling}}}]
    Let $0< p < q$ and  $\nu \in \mathcal{P}_q(\mathbb{R}^d)$ with $\mathbb{E}_\nu \bigeck{\N{x}^q} \leq M_q$. Let $\overbar{\nu}^N$ denote the empirical measure of $N$ independent random variables $\overbar{X}_1, ..., \overbar{X}_N \sim \nu$. Then there exists a constant $C_{\mathrm{EM}}>0$, that only depends on $(f, \alpha, p, q, M_q)$ such that for every $N \in \mathbb{N}$
    \begin{align*}
         \mathbb{E} \biggeck{ \N{X^\alpha(\overbar{\nu}^N) - X^\alpha(\nu)}^p} \leq C_{\mathrm{EM}} N^{-\frac{p}{2}}.
    \end{align*}
\label{GerberLemma3.7}
\end{lemma}
With the above refinements, we now have the tools we need to prove unrestricted propagation of chaos. Following the same order as in Chapter \ref{chap:PropGoodSets}, we start with the $\delta$-CBO scheme.

\section{$\delta$-CBO}\label{Section:Unrestricted:Delta}
We recall the $\delta$-CBO dynamics for the particle system $\bigrou{\mathcal{X}_t^N }_t = (X^1_t, ..., X^N_t)_{t \geq 0 }$:
\begin{align*}
    dX^i_t = - s \cdot \lambda (X^i_t - X^\alpha(\rho^N_{t}))dt + \sqrt{s}\delta dB^i_t,  \quad X^i_0 \sim \orho_0, \quad i = 1, ..., N,
\end{align*}
and the coupled mean-field particle system
 $\bigrou{\overbar{\mathcal{X}}_t^N }_t = (\overbar{X}^1_t, ..., \overbar{X}^N_t)_{t \geq 0 }$ defined by
\begin{align*}
    d\overbar{X}^i_t = - s \cdot \lambda (\overbar{X }^i_t - X^\alpha(\orho_{t}))dt + \sqrt{s}\delta dB^i_t, \quad \overbar{X}^i_0 = X^i_0 \, \, \text{ a.s.},  \quad i = 1, ...,N,
\end{align*}
where $\orho_t$ is the distribution of the mean-field process
\begin{align*}
    d\overbar{X}_t = - s \cdot \lambda (\overbar{X}_t - X^\alpha(\orho_{t}))dt + \sqrt{s}\delta dB_t, \quad X_0 \sim \orho_0.
\end{align*}
The following theorem follows the steps from \cite[Theorem 2.6] {Gerber2025Mean-fieldSampling}. We use the split stated at the beginning of this chapter alongside the original argument from Sznitman's proof of McKean's theorem \cite[Theorem 3.1]{Chaintron2022PropagationApplications} to derive a propagation of chaos result at rate $\mathcal{O}\Bigrou{N^{- \min \big \{ \frac{1}{2}, \frac{q-p}{2p^2} \big \} }}$.
\begin{theorem}
Let $T>0$ be the final time. Consider the interacting particle system $\IPS{T}$ defined by the SDE system (\ref{Delta:Scheme}) and the coupled mean-field system $\MIPS{T}$ defined by (\ref{Delta:MFScheme}), both sharing the same initial data $\mathcal{X}^N_0 = \overbar{\mathcal{X}}^N_0$ almost surely, with common distribution $\orho_0 \in \mathcal{P}_q(\mathbb{R}^d)$ for some $q \geq 2$. Then, for every $ 1 \leq p \leq q/2$, there exists a $\varepsilon_{\mathrm{D}}^{\mathrm{unrest}}(N, T)>0$ with $\varepsilon_{\mathrm{D}}^{\mathrm{unrest}}(N, T) \xrightarrow[]{N \rightarrow \infty}0$ such that for every $N \in \mathbb{N}$ and every $i \leq N$
    \begin{align*}
          \mathbb{E} \biggeck{\suptT \,\N{ X^i_t - \overbar{X}^i_t}^p} \leq \varepsilon_{\mathrm{D}}^{\mathrm{unrest}}(N, T).
    \end{align*}
    \label{NPDelta:Thrm}
\end{theorem}
\begin{proof}
For fixed ${\tiT} \in [0, T], i \leq N$, with the exact same reasoning as in Theorem \ref{Delta:Thrm}, we get:
\begin{align}
    \mathbb{E}\biggeck{\suptTit \N{F^i_t}^p} 
    \leq (s\lambda)^p {\tiT}^{p-1} \int_0^{\tiT} 
    \mathbb{E}\bigeck{\N{D_\tau}^p }\, d\tau,
    \label{NPDeltaThrm:Bound1}
\end{align}
where $F^i_t = X^i_t - \overbar{X}^i_t$ and $D_t = X^\alpha(\rho^N_t) -X^\alpha (\orho_t)$ for every $t \in [0, {\tiT}]$.   
We now expand $\N{D_\tau}^p$ by inserting the empirical mean-field measure $\overbar{\rho}^N_\tau$ and applying the triangle inequality:
\begin{align}
    \mathbb{E}\bigeck{\N{D_\tau }^p} 
    &= \mathbb{E}\Bigeck{\N{X^\alpha(\rho^N_\tau) 
    - X^\alpha(\orho_\tau) }^p} \notag \\
    &\leq 2^{p-1}\mathbb{E}\Bigeck{\N{X^\alpha(\rho^N_\tau) 
    - X^\alpha(\overbar{\rho}^N_\tau) }^p} 
    + 2^{p-1}\mathbb{E}\Bigeck{\N{X^\alpha(\overbar{\rho}^N_\tau) 
    - X^\alpha(\orho_\tau)}^p}.
    \label{NPDeltaThrm:Bound2}
\end{align}
In order to bound the expectation of the two empirical measures in (\ref{NPDeltaThrm:Bound2}), we introduce 
\begin{align*}
    \Omega_{N, t} := \bigg \{ \omega \in \Omega: \frac{1}{N} \sum^N_{i=1} \N{\overbar{X}^i_t}^p < M \bigg \}
    \quad \text{and} \quad
    Z_i := \suptT \, \N{\overbar{X}^i_t}^p,
\end{align*}
where $M> \mathbb{E}\bigeck{Z_i}$ is fixed. The uniform boundedness of expectations of $Z_i$ follows from Lemma \ref{NPDelta:UniformMoment}. This allows us to split as follows: 
\begin{align}
    &\mathbb{E} \Bigeck{ \N{X^\alpha(\rho^N_\tau) - X^\alpha (\overbar{\rho}^N_\tau)}^p} \notag \\
    = \, &\mathbb{E} \Bigeck{ \N{X^\alpha(\rho^N_\tau) - X^\alpha (\overbar{\rho}^N_\tau)}^p  \mathbbm{1}_{\Omega_{N, \tau}}}
    + \mathbb{E} \Bigeck{ \N{X^\alpha(\rho^N_\tau) - X^\alpha (\overbar{\rho}^N_\tau)}^p  \mathbbm{1}_{\Omega \setminus \Omega_{N, \tau}}}.
\end{align}
The estimation of the first part is straightforward, as we can directly use the improved stability estimate of Lemma \ref{GerberStability2} and Example \ref{WasserBD} to get
\begin{align}
    \mathbb{E} \Bigeck{ \N{X^\alpha(\rho^N_\tau) - X^\alpha (\overbar{\rho}^N_\tau)}^p \, \mathbbm{1}_{\Omega_{N, \tau}}}
    \leq C^p_0 \maxu{i}{N} \, \mathbb{E} \biggeck{\N{X^i_\tau - \overbar{X}^i_\tau}^p}
    = C^p_0 \mathbb{E} \biggeck{\N{X^i_\tau - \overbar{X}^i_\tau}^p}
    \notag \\
    \leq C^p_0 \mathbb{E} \biggeck{ \underset{s \in [0, \tau]}{\sup}\N{X^i_s - \overbar{X}^i_s}^p}.
    \label{NPDeltaThrm:Bound3}
\end{align}
To the second summand, we apply Hölder's inequality with conjugate exponents $\frac{q}{p}$ and $\frac{q}{q-p}$:
\begin{align}
    \mathbb{E} \Bigeck{ \N{X^\alpha(\rho^N_\tau) - X^\alpha (\overbar{\rho}^N_\tau) }^p\mathbbm{1}_{\Omega \setminus \Omega_{N, \tau}}}
    \leq \mathbb{E} \Bigeck{ \N{X^\alpha(\rho^N_\tau) - X^\alpha (\overbar{\rho}^N_\tau)}^q}^{\frac{p}{q}} 
    \, \mathbb{P}(\Omega \setminus \Omega_{N, \tau})^{\frac{q-p}{q}}.
    \label{NPDeltaProbabilitysplit}
\end{align}
We note that almost surely
\begin{align*}
\frac{1}{N}\sum_{i=1}^N \N{\overbar{X}^i_\tau}^p \leq \frac{1}{N}\sum_{i=1}^N Z_i,
\end{align*}
and therefore
\begin{align*}
\mathbb{P}(\Omega \setminus \Omega_{N,\tau})
=
\mathbb{P}\bigg(\frac{1}{N}\sum_{i=1}^N \N{\overbar{X}^i_\tau}^p \ge M\bigg)
\leq
\mathbb{P}\biggrou{\frac{1}{N}\sum_{i=1}^N Z_i \ge M}.
\end{align*}
By Lemma \ref{GerberLemma2.5} with $r = q/p$ and the fact that $\mathbb{E}\Bigeck{Z_1^{q/p}}<\infty$, there exists $K>0$ independent of $N$ such that
\begin{align*}
    \mathbb{P}(\Omega \setminus \Omega_{N, \tau})
    \leq K N^{- \frac{q}{2p}}.
\end{align*}
It remains to bound
\begin{align}
    \mathbb{E} \Bigeck{ \N{X^\alpha(\rho^N_\tau) - X^\alpha (\overbar{\rho}^N_\tau)}^q}
    \leq 2^{q-1}\biggrou{\mathbb{E} \Bigeck{ \N{X^\alpha(\rho^N_\tau)}^q} + \mathbb{E} \Bigeck{\N{ X^\alpha (\overbar{\rho}^N_\tau)}^q}}.
    \label{NPDeltaTwoCOnsensus Bound 1}
\end{align}
Using Corollary \ref{GerberPropA4Corollary} with exponent $q$, we get
\begin{align*}
    \mathbb{E} \Bigeck{ \N{X^\alpha(\rho^N_\tau)}^q}
    \leq b_q \, \mathbb{E} \Bigeck{\mathbb{E}_{\rho^N_\tau} \bigeck{\N{x}^q}}
    = b_q \, \mathbb{E} \Biggeck{\frac{1}{N} \sum^N_{i=1} \N{X^i_\tau}^q}
    = b_q \, \mathbb{E} \Bigeck{\N{X^i_\tau}^q}
    \leq b_q \, C_\mathrm{D},
\end{align*}
and analogously
\begin{align*}
    \mathbb{E} \Bigeck{ \N{X^\alpha(\overbar{\rho}^N_\tau)}^q}
    \leq b_q \, \overbar{C}_\mathrm{D},
\end{align*}
where the constants $C_D$ and $\overbar{C}_D$ arise from Lemma \ref{NPDelta:UniformMoment} with $q$ as the order of moment.
Thus,
\begin{align*}
    \mathbb{E} \Bigeck{ \N{X^\alpha(\rho^N_\tau) - X^\alpha (\overbar{\rho}^N_\tau)}^q}
    \leq 2^{q-1} b_q (C_\mathrm{D} + \overbar{C}_\mathrm{D}).
\end{align*}
Hence \eqref{NPDeltaProbabilitysplit} yields
\begin{align}
     \mathbb{E} \Bigeck{ \N{X^\alpha(\rho^N_\tau) - X^\alpha (\overbar{\rho}^N_\tau) }^p \mathbbm{1}_{\Omega \setminus \Omega_{N, \tau}}}
     \leq \Big(2^{q-1} b_q (C_\mathrm{D} + \overbar{C}_\mathrm{D})\Big)^{\frac{p}{q}} K^{\frac{q-p}{q}} N^{-\frac{q-p}{2p}}
     =: C N^{-\frac{q-p}{2p}},
     \label{NPDeltaThrm:Bound4}
\end{align}
for a constant $C>0$ independent of $N$.
Combining (\ref{NPDeltaThrm:Bound3}) and (\ref{NPDeltaThrm:Bound4}) yields
\begin{align}
    \mathbb{E} \Bigeck{ \N{X^\alpha(\rho^N_\tau) - X^\alpha (\overbar{\rho}^N_\tau)}^p}
    \leq C^p_0 \mathbb{E} \biggeck{ \underset{s \in [0, \tau]}{\sup}\N{X^i_s - \overbar{X}^i_s}^p}
    + C N^{-\frac{q-p}{2p}}.
    \label{NPDeltaThrm:T1Bound}
\end{align}
For the second term of \eqref{NPDeltaThrm:Bound2}, we use Lemma \ref{GerberLemma3.7}:
\begin{align}
    \mathbb{E} \biggeck{ \N{X^\alpha (\overbar{\rho}^N_\tau) - X^\alpha (\orho_\tau)}^p}
    \leq C_{\mathrm{EM}} N^{-\frac{p}{2}}.
    \label{NPDelta BoundEM}
\end{align}
Hence
\begin{align*}
  \mathbb{E}\bigeck{\N{D_\tau }^p}
  \leq 2^{p-1} \Biggrou{
  C^p_0 \mathbb{E} \biggeck{ \sup_{s \in [0, \tau]}\N{X^i_s - \overbar{X}^i_s}^p}
  + C N^{-\frac{q-p}{2p}}
  + C_{\mathrm{EM}} N^{-\frac{p}{2}}
  }.
\end{align*}
Define
\begin{align*}
r_* := \min \bigg \{ \frac{p}{2}, \frac{q-p}{2p} \bigg \}.
\end{align*}
Then
\begin{align*}
  \mathbb{E}\bigeck{\N{D_\tau }^p}
  \leq 2^{p-1} \Biggrou{
  C^p_0 \mathbb{E} \biggeck{ \sup_{s \in [0, \tau]}\N{X^i_s - \overbar{X}^i_s}^p}
  + (C+C_{\mathrm{EM}}) N^{-r_*}
  }.
\end{align*}
Putting this into \eqref{NPDeltaThrm:Bound1}:
\begin{align*}
    \mathbb{E}\biggeck{\suptTit \N{F^i_t}^p} 
    &\leq (s\lambda)^p {\tiT}^{p-1} \int_0^{\tiT}
    \mathbb{E}\bigeck{\N{D_\tau}^p }\, d\tau \\
    &\leq (2\tiT)^{p-1}(s\lambda  C_0)^p \int_0^{\tiT} 
    \mathbb{E} \biggeck{ \sup_{s \in [0, \tau]}\N{X^i_s - \overbar{X}^i_s}^p}\, d\tau \\
    & \quad + 2^{p-1} (s\lambda {\tiT})^p  (C + C_{\mathrm{EM}}) N^{-r_*}.
\end{align*}
Applying Grönwall’s lemma (Lemma \ref{GronwallLemma}) yields
\begin{align*}
     \mathbb{E} \biggeck{\suptT \,\N{ X^i_t - \overbar{X}^i_t}^p}
     \leq 2^{p-1} (s\lambda T)^p  (C + C_{\mathrm{EM}})
     e^{2^{p-1}(s\lambda T C_0)^p}
     N^{-r_*}.
\end{align*}
Defining
\begin{align*}
\varepsilon^{\mathrm{unrest}}_{\mathrm{D}}(N,T)
:=
2^{p-1} (s\lambda T)^p (C + C_{\mathrm{EM}}) e^{2^{p-1}(s\lambda T C_0)^p}
N^{- r_* },
\end{align*}
concludes the proof.
\end{proof}
\begin{remark}
    If $\orho_0$ has bounded moments of order $q \geq p^2 +p$, this constitutes for every $p \geq 1$ propagation of chaos at rate $\mathcal{O}\bigrou{N^{-1/2}}$. This is the rate stated in \cite[Theorem 2.6]{Gerber2025Mean-fieldSampling} for classical CBO under the assumption of bounded moments of all orders.
\end{remark}

\section{Freezing}
We recall from Section \ref{ch:MFF} the $N$-particle system $\bigrou{\mathcal{X}_t^N }_t$ by the following system of stochastic differential equations for fixed $j$.
\begin{align*}
     dX^i_t = - s \cdot \lambda (X^i_t - X^\alpha(\rho^N_{t_j}))dt + \sqrt{s}\delta dB^i_t,  \quad \forall t \in [t_j, t_{j+1}), \quad X^i_0 \sim \orho_0, \quad  i = 1, ..., N,
\end{align*}
and the coupled mean-field particle system $\bigrou{\overbar{\mathcal{X}}^N_t}_t$ defined by: 
\begin{align*}
     d\overbar{X}^i_t = - s \cdot \lambda (\overbar{X}^i_t - X^\alpha(\orho_{t_j}))dt + \sqrt{s}\delta dB^i_t \quad \forall t \in [t_j, t_{j+1}), \quad \overbar{X}^i_0 = X^i_0 \, \, \text{ a.s.}, \quad  i = 1, ..., N, 
\end{align*}
where $\orho_t$ is the distribution of the mean-field limit process 
\begin{align*}
     d\overbar{X}_t = - s \cdot \lambda (\overbar{X}_t - X^\alpha(\orho_{t_j}))dt + \sqrt{s}\delta dB_t \quad \forall t \in [t_j, t_{j+1}).
\end{align*}
Also, we want to recall the distance function and scheme that we derived in \eqref{CF:DistFunction} and \eqref{MFF:DistanceScheme}, respectively.
We have $F^{(i)}(t) = X^i_t - \overbar{X}^i_t$ and $D_j = X^\alpha(\rho^N_{t_j}) - X^\alpha(\orho_{t_j})$ and therefore:
\begin{align*}
    F^{(i)}(t) = e^{-s \lambda(t-t_j)}F^{(i)}(t_j) + D_j(1-e^{-s \lambda(t-t_j)}),
\end{align*}
and
\begin{align*}
    E^{(i)}_{j+1} = F^{(i)}(t_j + \Delta t) = gE_j^{(i)} + D_j(1-g),
\end{align*}
with $g = e^{-s \lambda \Delta t}$.
\begin{theorem}
Let $T>0$ be the final time. Consider the interacting particle system $\IPS{T}$ defined by the SDE system (\ref{CF:CFscheme}) and the coupled mean-field system $\MIPS{T}$ defined by (\ref{CF:MFIPS}), both sharing the same initial data $\mathcal{X}^N_0 = \overbar{\mathcal{X}}^N_0$ almost surely, with common distribution $\orho_0 \in \mathcal{P}_q(\mathbb{R}^d)$ for some $q \geq 2$. Then, for every $1 \leq p \leq q/2$, there exists a $\varepsilon^{\mathrm{unrest}}_{\mathrm{F}}(N, T) >0$ with $\varepsilon^{\mathrm{unrest}}_{\mathrm{F}}(N, T) \xrightarrow[]{N \rightarrow \infty}0$ such that for every $N \in \mathbb{N}$ and every $i \leq N$
    \begin{align*}
          \mathbb{E} \biggeck{\suptT \,\N{ X^i_t - \overbar{X}^i_t}^p} \leq \varepsilon^{\mathrm{unrest}}_{\mathrm{F}}(N, T).
    \end{align*}
    \label{NPFreezing:Thrm}
\end{theorem}
\begin{proof}
    We let $j \in \{ 0, ..., J-1 \}$ be an arbitrary index and fix an $i \leq N$. With the same reasoning as in Theorem \ref{MFF: Thrm}, for every $j \in \{0,\dots,J-1\}$ we have
\begin{align}
    \mathbb{E} \biggeck{ \supttj \N{X^i_{t} - \overbar{X}^i_{t} }^p}
    &=
    \mathbb{E} \biggeck{\max \Big \{ \N{E^{(i)}_j}^p, \N{E^{(i)}_{j+1}}^p \Big \}} \notag\\
    &\leq
    \mathbb{E} \biggeck{\N{E^{(i)}_j}^p + \N{E^{(i)}_{j+1}}^p }.
    \label{NP:CFIntervalSupBound}
\end{align}
Hence it is enough to estimate
\begin{align*}
A_j := \mathbb{E}\biggeck{\N{E^{(i)}_j}^p}, \qquad j=0,\dots,J.
\end{align*}
Since $E^{(i)}_0=0$ almost surely, we have $A_0=0$.

Using \eqref{MFF:DistanceScheme} alongside with the convexity of $x \mapsto \N{x}^p$, we estimate
\begin{align}
     A_{j+1}
     &= \mathbb{E} \biggeck{\N{gE^{(i)}_{j} + D_{j}(1-g)}^p} \notag\\
    &\leq 
    g A_j + (1-g)\mathbb{E}\bigeck{\N{D_{j}}^p}.
    \label{NP:CFBound1}
\end{align}
To bound the distance between the two consensus points
\begin{align*}
    \mathbb{E}\bigeck{\N{D_{j}}^p} = \mathbb{E}\biggeck{\N{X^\alpha(\rho^N_{t_j}) - X^\alpha(\orho_{t_j})}^p},
\end{align*}
we use once again the same triangle inequality argument by employing the empirical measure of the mean-field particle system:
\begin{align}
    \mathbb{E}\biggeck{\N{X^\alpha(\rho^N_{t_j}) - X^\alpha(\orho_{t_j})}^p}
    \leq 2^{p-1} \Biggrou{\mathbb{E}\biggeck{\N{X^\alpha(\rho^N_{t_j}) - X^\alpha(\overbar{\rho}^N_{t_j})}^p} + \mathbb{E}\biggeck{\N{X^\alpha(\overbar{\rho}^N_{t_j}) - X^\alpha(\orho_{t_j})}^p}}.
    \label{NP:CFConsDecomp1}
\end{align}
For the distance between the consensus points of the two empirical measures, we use the decomposition described in the introduction of this chapter:
\begin{align}
    &\mathbb{E} \Bigeck{ \N{X^\alpha(\rho^N_{t_j}) - X^\alpha (\overbar{\rho}^N_{t_j})}^p} \notag \\
    =\, &\mathbb{E} \Bigeck{ \N{X^\alpha(\rho^N_{t_j}) - X^\alpha (\overbar{\rho}^N_{t_j})}^p \mathbbm{1}_{\Omega_{N, {t_j}}}} + \mathbb{E} \Bigeck{ \N{X^\alpha(\rho^N_{t_j}) - X^\alpha (\overbar{\rho}^N_{t_j})}^p  \mathbbm{1}_{\Omega \setminus \Omega_{N, {t_j}}}},
    \label{NP:CFConsSplit2}
\end{align}
where
\begin{align*}
    \Omega_{N, t} := \bigg \{ \omega \in \Omega: \frac{1}{N} \sum^N_{i=1} \N{\overbar{X}^i_t}^p < M \bigg \}.
\end{align*}
On $\Omega_{N,t_j}$, the empirical mean-field measure $\overbar{\rho}^N_{t_j}$ satisfies $\mathbb{E}_{\overline{\rho}^N_{t_j}}\bigeck{\N{x}^p} \le M$, 
so Lemma \ref{GerberStability2} applies and yields
\begin{align}
    \mathbb{E} \Bigeck{ \N{X^\alpha(\rho^N_{t_j}) - X^\alpha (\overbar{\rho}^N_{t_j})}^p \, \mathbbm{1}_{\Omega_{N, {t_j}}}}
    &\leq C^p_0  \mathbb{E} \biggeck{ \frac{1}{N} \sum^N_{i=1}\N{X^i_{t_j} - \overbar{X}^i_{t_j}}^p} \notag \\ 
    &= C^p_0  \mathbb{E} \biggeck{\N{X^i_{t_j} - \overbar{X}^i_{t_j}}^p} \notag \\
    &= C^p_0 A_j.
    \label{NP:CFConsBound1}
\end{align}
On $\Omega \setminus \Omega_{N, t_j}$, we make use of Hölder's inequality with conjugate exponents $\frac{q}{p}$ and $\frac{q}{q-p}$ to separate the indicator function: 
\begin{align}
    \mathbb{E} \Bigeck{ \N{X^\alpha(\rho^N_{t_j}) - X^\alpha (\overbar{\rho}^N_{t_j}) }^p \mathbbm{1}_{\Omega \setminus \Omega_{N, {t_j}}}} 
    &\leq \mathbb{E} \Bigeck{ \N{X^\alpha(\rho^N_{t_j}) - X^\alpha (\overbar{\rho}^N_{t_j})}^q}^{\frac{p}{q}} \, \mathbb{P}(\Omega \setminus \Omega_{N, {t_j}})^{\frac{q-p}{q}}.
    \label{NPFreezingProbability split}
\end{align}
We define the random variables $Z_1, ..., Z_N$ by $Z_i := \supjJ \, \N{\overbar{X}^i_{t_j}}^p$ for every $i \leq N$.
Note that $\mathbb{E}\Bigeck{\big |Z_1 \big |^{q/p}} < \infty$ by the regularity of the solution of (\ref{CF:MFIPS}) introduced by Lemma \ref{NP:CFLemma}. Therefore, by Lemma \ref{GerberLemma2.5} with $r = q/p$, for $M> 0$ such that $\mathbb{E}\bigeck{ |Z_1 |} < M$, there exists a constant $K > 0$ independent of $N$ such that for every $N \in \mathbb{N}$
\begin{align*}
    \mathbb{P}(\Omega \setminus \Omega_{N, {t_j}})
    = \mathbb{P}\biggrou{\frac{1}{N} \sum^N_{i=1} \N{\overbar{X}^i_{t_j}}^p \geq M}
    \leq \mathbb{P} \biggrou{\frac{1}{N} \sum^N_{i=1} Z_i \geq M}
    \leq K N^{- \frac{q}{2p}}.
\end{align*}
This means that we only have to control $\mathbb{E} \Bigeck{ \N{X^\alpha(\rho^N_{t_j}) - X^\alpha (\overbar{\rho}^N_{t_j})}^q}$ by a constant independent of $N$. Therefore, we can use the very straightforward bound
\begin{align}
    \mathbb{E} \Bigeck{ \N{X^\alpha(\rho^N_{t_j}) - X^\alpha (\overbar{\rho}^N_{t_j})}^q}
    \leq 2^{q-1}\biggrou{\mathbb{E} \Bigeck{ \N{X^\alpha(\rho^N_{t_j})}^q} + \mathbb{E} \Bigeck{\N{ X^\alpha (\overbar{\rho}^N_{t_j})}^q}}.
    \label{NPFRTwoCOnsensus Bound 1}
\end{align}
We can use Corollary \ref{GerberPropA4Corollary} with exponent $q$ to control
\begin{align*}
    \mathbb{E} \Bigeck{ \N{X^\alpha(\rho^N_{t_j})}^q}
    \leq b_q \, \mathbb{E} \Bigeck{\mathbb{E}_{\rho^N_{t_j}} \bigeck{\N{x}^q}}
    = b_q \, \mathbb{E} \Biggeck{\frac{1}{N} \sum^N_{i=1} \N{X^i_{t_j}}^q}
    = b_q \, \mathbb{E} \Bigeck{\N{X^i_{t_j}}^q}
    \leq b_q \, C_\mathrm{F},
\end{align*}
where the last inequality follows from Lemma \ref{NP:CFLemma}. Analogously, we get
\begin{align*}
    \mathbb{E} \Bigeck{ \N{X^\alpha(\orho^N_{t_j})}^q} \leq b_q \,  \overbar{C}_\mathrm{F}.
\end{align*}
We stress that both bounds are independent of $N$. 
We move on to combine the latter estimates with \eqref{NPFRTwoCOnsensus Bound 1}:
\begin{align*}
    \mathbb{E} \Bigeck{ \N{X^\alpha(\rho^N_{t_j}) - X^\alpha (\overbar{\rho}^N_{t_j})}^q}
    \leq 2^{q-1} b_q (C_\mathrm{F} + \overbar{C}_\mathrm{F}).
\end{align*}
Thus, we can write estimate \eqref{NPFreezingProbability split} as follows:
 \begin{align}
     \mathbb{E} \Bigeck{ \N{X^\alpha(\rho^N_{t_j}) - X^\alpha (\overbar{\rho}^N_{t_j}) }^p\mathbbm{1}_{\Omega \setminus \Omega_{N, {t_j}}}}
     &\leq \mathbb{E} \Bigeck{ \N{X^\alpha(\rho^N_{t_j}) - X^\alpha (\overbar{\rho}^N_{t_j})}^q}^{\frac{p}{q}} \, \mathbb{P}(\Omega \setminus \Omega_{N, {t_j}})^{\frac{q-p}{q}} \notag \\
     &\leq \Big(2^{q-1} b_q (C_F + \overbar{C}_F)\Big)^{\frac{p}{q}} K^{\frac{q-p}{q}} N^{-\frac{q-p}{2p}}
     =: C N^{-\frac{q-p}{2p}}.
     \label{NP:CFConsBound2}
 \end{align}
Putting \eqref{NP:CFConsBound1} and \eqref{NP:CFConsBound2} into \eqref{NP:CFConsSplit2} yields
\begin{align}
    \mathbb{E} \Bigeck{ \N{X^\alpha(\rho^N_{t_j}) - X^\alpha (\overbar{\rho}^N_{t_j})}^p}
    \leq C^p_0 A_j + C N^{-\frac{q-p}{2p}}.
    \label{NP:CFConsBound3}
\end{align}
We now move on to making an effort to obtain a useful estimate of the distance between the consensus of the empirical mean-field measure $\overbar{\rho}^N_{t_j}$ and the actual mean-field measure $\orho_{t_j}$.
We can directly apply Lemma \ref{GerberLemma3.7} to obtain
\begin{align*}
    \mathbb{E} \biggeck{ \N{X^\alpha (\overbar{\rho}^N_{t_j}) - X^\alpha (\orho_{{t_j}})}^p} \leq C_{\mathrm{EM}} N^{-\frac{p}{2}}.
\end{align*}
As in the proof of Theorem \ref{MFF: Thrm}, we point out that the above bound is independent of $j$.
Finally, we can combine the latter with \eqref{NP:CFConsBound3} to get for \eqref{NP:CFConsDecomp1}
\begin{align*}
     \mathbb{E}\biggeck{\N{X^\alpha(\rho^N_{t_j}) - X^\alpha(\orho_{t_j})}^p}
     &\leq 2^{p-1} \Biggrou{\mathbb{E}\biggeck{\N{X^\alpha(\rho^N_{t_j}) - X^\alpha(\overbar{\rho}^N_{t_j})}^p} + \mathbb{E}\biggeck{\N{X^\alpha(\overbar{\rho}^N_{t_j}) - X^\alpha(\orho_{t_j})}^p}}\\
     &\leq 2^{p-1} \Biggrou{C^p_0 A_j + C N^{-\frac{q-p}{2p}} + C_{\mathrm{EM}} N^{-\frac{p}{2}}}\\
     &\leq 2^{p-1} \Biggrou{C^p_0 A_j + (C+C_{\mathrm{EM}}) N^{- \min \big \{ \frac{p}{2}, \frac{q-p}{2p} \big \} }}.
\end{align*}
We insert the above in \eqref{NP:CFBound1} and obtain
\begin{align}
    A_{j+1}
    &\leq g A_j + (1-g) 2^{p-1} \Biggrou{C^p_0 A_j + (C + C_{\mathrm{EM}}) N^{- \min \big \{ \frac{p}{2}, \frac{q-p}{2p} \big \} }} \notag \\
    &= \Bigrou{g + (1-g)2^{p-1}C^p_0}A_j + 2^{p-1} (1-g) (C+ C_{\mathrm{EM}}) N^{- \min \big \{ \frac{p}{2}, \frac{q-p}{2p} \big \} }.
    \label{NP:CFRecursion}
\end{align}
Define
\begin{align*}
    a_F &:= g + 2^{p-1}(1-g) C_0^p,\\
    b_F &:= 2^{p-1}(1-g) (C + C_{\mathrm{EM}}),\\
    r_* &:= \min \bigg \{ \frac{p}{2}, \frac{q-p}{2p} \bigg \}.
\end{align*}
Then \eqref{NP:CFRecursion} becomes
\begin{align*}
A_{j+1} \leq a_F A_j + b_F N^{-r_*}.
\end{align*}
Since $A_0=0$, we obtain by induction
\begin{align}
    A_{j+1}
    \leq b_F N^{-r_*} \sum^j_{l = 0} a_F^l.
    \label{NP:CF:Thrm:ItEstimate}
\end{align}
We note that by \eqref{NP:CFIntervalSupBound},
\begin{align*}
    \mathbb{E} \biggeck{\suptT \,\N{ X^i_t - \overbar{X}^i_t}^p}
    &= \mathbb{E} \biggeck{ \supu{j}{J-1} \,  \supttj \,\N{ X^i_t - \overbar{X}^i_t}^p} \notag\\
    &\leq \sum_{j=0}^{J-1} \mathbb{E} \biggeck{\supttj \,\N{ X^i_t - \overbar{X}^i_t}^p} \notag\\
    &\leq \sum_{j=0}^{J-1} \big(A_j + A_{j+1}\big) \notag\\
    &\leq 2\sum_{j=1}^{J} A_j.
\end{align*}
Finally, combining the latter together with \eqref{NP:CF:Thrm:ItEstimate} gives us
\begin{align}
    & \mathbb{E} \biggeck{\suptT \,\N{ X^i_t - \overbar{X}^i_t}^p} \notag \\
    \leq \,  &2\sum_{j=1}^{J} A_j \notag\\
    \leq \,  &2 b_F N^{-r_*} \sum_{j=1}^{J}\sum^ {j-1}_{l = 0} a_F^l \notag\\
    = \,  &2 b_F N^{-r_*} \sum_{l = 0}^{J-1} (J-l)a_F^l \notag\\
    = \,  &2 \cdot 2^{p-1}(1-g) (C + C_{\mathrm{EM}}) N^{-r_*} \sum_{l = 0}^{J-1} (J-l)(g + 2^{p-1}(1-g) C_0^p)^l \notag\\
    =: \,  &\varepsilon^{\mathrm{unrest}}_\mathrm{F}(N, T).
    \label{NPFreezing:PoCBound}
\end{align}
Since $i \leq N$ was chosen to be arbitrary, the proof is complete.
\end{proof}
Theorem \ref{NPFreezing:Thrm} yields unrestricted propagation of chaos at rate $\mathcal{O}\bigrou{N^{- \min  \{ 1/2, (q-p)/(2p^2) \} }}$ for every $1 \leq p \leq q/2$. Again, $q \geq p^2 + p$ implies propagation of chaos at the rate $\mathcal{O}(N^{-1/2})$.
\section{Hopping}
We turn to proving unrestricted propagation of chaos for the consensus hopping scheme. We recall the interacting particle system  $(\mathcal{X}^N_k)_k$ from \eqref{CH:Scheme} which is defined by
\begin{align*}
   & X^i_k = X^\alpha(\rho^N_{k-1}) + \eta^i_k, \quad \text{for } i \leq N, \quad k \in \mathbb{N}, \\
   &\mathcal{X}^N_0 = (X^1_0,...,X^N_0) \sim \orho_0^{\otimes N}, \notag
\end{align*}
and the corresponding coupled mean-field system $(\overbar{\mathcal{X}}^N_k)_k$ from \eqref{CH:MFScheme}:
\begin{align*}
   & \overbar{X}^i_k = X^\alpha(\orho_{k-1}) + \eta^i_k, \quad \text{for } i \leq N, \quad k \in \mathbb{N}, \\
   &\overbar{\mathcal{X}}^N_0 = (\overbar{X}^1_0,...,\overbar{X}^N_0) \sim \orho_0^{\otimes N},
   \quad \overbar{\mathcal{X}}^N_0 = \mathcal{X}^N_0 \,\, \text{ a.s. }.
\end{align*}
\begin{theorem}
 Let $K \in \mathbb{N}_{>0}$ be the final time. Consider the interacting particle system $\IPS{K}$ defined by the hopping scheme \eqref{CH:Scheme} and the coupled mean-field system $\MIPS{K}$ defined by \eqref{CH:MFScheme}, both sharing the same initial data $\mathcal{X}^N_0 = \overbar{\mathcal{X}}^N_0$ almost surely, with common distribution $\orho_0 \in \mathcal{P}_q(\mathbb{R}^d)$ for some $q \geq 2$. Then for every $1 \leq p \leq q/2$, there exists a $\varepsilon_\mathrm{H}^{\mathrm{unrest}}(N, K)>0$ with $\varepsilon_\mathrm{H}^{\mathrm{unrest}}(N, K) \xrightarrow[]{N \rightarrow \infty}0$ such that for every $N \in \mathbb{N}$ and every $i \leq N$
    \begin{align*}
          \mathbb{E} \biggeck{\supu{k}{K} \,\N{ X^i_k - \overbar{X}^i_k}^p} \leq \varepsilon_\mathrm{H}^{\mathrm{unrest}}(N, K).
    \end{align*}
    \label{NPHopping:Thrm}   
\end{theorem}
\begin{proof}
We fix $i \leq N$ and $k \leq K$ and use the definition of \eqref{CH:Scheme} and \eqref{CH:MFScheme} together with a basic inequality to obtain
\begin{align}
    \mathbb{E} \biggeck{\N{ X^i_{k+1} - \overbar{X}^i_{k+1}}^p}
    &= \mathbb{E} \biggeck{\N{X^\alpha(\rho^N_k) - X^\alpha(\orho_k)}^p}\notag \\
    &\leq 2^{p-1} \biggrou{\underbrace{\mathbb{E} \biggeck{\N{X^\alpha(\rho^N_k) - X^\alpha(\overbar{\rho}^N_k)}^p}}_{=: S_1} + \underbrace{\mathbb{E} \biggeck{\N{X^\alpha(\overbar{\rho}^N_k) - X^\alpha(\orho_k)}^p}}_{=: S_2}}.
\end{align}
$S_2$ can be estimated by exploiting the relation between the mean-field empirical measure $\overbar{\rho}^N_k$ and $\orho_k$. 
The estimation of $S_1$ is the crucial part that requires the decomposition described in the introduction of this chapter. We will use the same line of argument as in Theorems \ref{NPDelta:Thrm} and \ref{NPFreezing:Thrm}. We start by introducing
\begin{align*}
    \Omega_{N, k} := \bigg \{ \omega \in \Omega: \frac{1}{N} \sum^N_{i=1} \N{\overbar{X}^i_k}^p < M \bigg \}
\end{align*}
in order to split 
\begin{align}
    &\mathbb{E} \biggeck{\N{X^\alpha(\rho^N_k) - X^\alpha(\overbar{\rho}^N_k)}^p} \notag \\
    = \, &\mathbb{E} \biggeck{\N{X^\alpha(\rho^N_k) - X^\alpha(\overbar{\rho}^N_k)}^p \mathbbm{1}_{\Omega_{N, k}}} + \mathbb{E} \biggeck{\N{X^\alpha(\rho^N_k) - X^\alpha(\overbar{\rho}^N_k)}^p \mathbbm{1}_{\Omega \setminus \Omega_{N, k}}}. 
\end{align}
In the first summand, we can directly apply Lemma \ref{GerberStability2} and Example \ref{WasserBD} together with the exchangeability of the particles to get the estimate
\begin{align}
    \mathbb{E} \biggeck{\N{X^\alpha(\rho^N_k) - X^\alpha(\overbar{\rho}^N_k)}^p \mathbbm{1}_{\Omega_{N, k}}}
    \leq C^p_0 \mathbb{E} \Bigeck{W^p_p (\rho^N_k, \overbar{\rho}^N_k)}
    &\leq C^p_0 \underset{1 \leq i \leq N}{\max}\mathbb{E} \Bigeck{\N{X^i_k - \overbar{X}^i_k}^p} \notag \\
    &\leq  C^p_0 \mathbb{E} \Bigeck{\N{X^i_k - \overbar{X}^i_k}^p}.
    \label{NP:CH:BoundComplement}
\end{align}
To the second summand, we apply Hölder's inequality with conjugate exponents $\frac{q}{p}$ and $\frac{q}{q-p}$:
\begin{align}
    \mathbb{E} \Bigeck{ \N{X^\alpha(\rho^N_k) - X^\alpha (\overbar{\rho}^N_k)}^p\mathbbm{1}_{\Omega \setminus \Omega_{N, k}}}
    \leq \mathbb{E} \Bigeck{ \N{X^\alpha(\rho^N_k) - X^\alpha (\overbar{\rho}^N_k)}^q}^{\frac{p}{q}} \, \mathbb{P}(\Omega \setminus \Omega_{N, k})^{\frac{q-p}{q}}.
    \label{NPHoppingProbability split}
\end{align}
For the probability part, we can make use of Lemma \ref{GerberLemma2.5}. We note that by defining $Z_i := \supu{k}{K} \, \N{\overbar{X}^i_k}^p$, it immediately follows that $\mathbb{E}\Bigeck{\big |Z_1 \big |^{q/p}} < \infty$ by the regularity of (\ref{CH:MFScheme}) justified by Lemma \ref{NP_CHLemma}. Therefore, by Lemma \ref{GerberLemma2.5} with $r = q/p$, for $M> 0$ such that $\mathbb{E}\bigeck{ |Z_1 |} < M$, there exists a constant $K_0 > 0$ independent of $N$ such that for every $N \in \mathbb{N}$
\begin{align*}
    \mathbb{P}(\Omega \setminus \Omega_{N, k})
    = \mathbb{P}\biggrou{\frac{1}{N} \sum^N_{i=1} \N{\overbar{X}^i_k}^p \geq M}
    \leq \mathbb{P} \biggrou{\frac{1}{N} \sum^N_{i=1} Z_i \geq M}
    \leq K_0 N^{- \frac{q}{2p}}.
\end{align*}
It remains to control $\mathbb{E} \Bigeck{ \N{X^\alpha(\rho^N_k) - X^\alpha (\overbar{\rho}^N_k)}^q}$ by a constant independent of $N$. As in the preceding proofs of this chapter, we can use the very straightforward bound
\begin{align}
    \mathbb{E} \Bigeck{ \N{X^\alpha(\rho^N_k) - X^\alpha (\overbar{\rho}^N_k)}^q}
    \leq 2^{q-1}\biggrou{\mathbb{E} \Bigeck{ \N{X^\alpha(\rho^N_k)}^q} + \mathbb{E} \Bigeck{\N{ X^\alpha (\overbar{\rho}^N_k)}^q}}.
    \label{NPCHTwoCOnsensus Bound 1}
\end{align}
Corollary \ref{GerberPropA4Corollary} with exponent $q$ allows us to control
\begin{align*}
    \mathbb{E} \Bigeck{ \N{X^\alpha(\rho^N_k)}^q}
    \leq b_q \, \mathbb{E} \Bigeck{\mathbb{E}_{\rho^N_k} \bigeck{\N{x}^q}}
    = b_q \, \mathbb{E} \Biggeck{\frac{1}{N} \sum^N_{i=1} \N{X^i_k}^q}
    = b_q \, \mathbb{E} \Bigeck{\N{X^i_k}^q}
    \leq b_q \, C_\mathrm{H},
\end{align*}
where the last inequality follows from Lemma \ref{NP_CHLemma}. Analogously, we get
\begin{align*}
    \mathbb{E} \Bigeck{ \N{X^\alpha(\orho^N_k)}^q} \leq b_q \,  \overbar{C}_\mathrm{H}.
\end{align*}
We stress that both bounds are independent of $N$. 
We move on to combine the latter estimates with \eqref{NPCHTwoCOnsensus Bound 1}:
\begin{align*}
    \mathbb{E} \Bigeck{ \N{X^\alpha(\rho^N_k) - X^\alpha (\overbar{\rho}^N_k)}^q}
    \leq 2^{q-1} b_q (C_\mathrm{H} + \overbar{C}_\mathrm{H}).
\end{align*}
Thus, we can write estimate \eqref{NPHoppingProbability split} as follows:
 \begin{align}
     \mathbb{E} \Bigeck{ \N{X^\alpha(\rho^N_k) - X^\alpha (\overbar{\rho}^N_k) }^p\mathbbm{1}_{\Omega \setminus \Omega_{N, k}}}
     &\leq \mathbb{E} \Bigeck{ \N{X^\alpha(\rho^N_k) - X^\alpha (\overbar{\rho}^N_k)}^q}^{\frac{p}{q}} \, \mathbb{P}(\Omega \setminus \Omega_{N, k})^{\frac{q-p}{q}} \notag \\
     &\leq \Big(2^{q-1} b_q (C_\mathrm{H} + \overbar{C}_\mathrm{H})\Big)^{\frac{p}{q}} K_0^{\frac{q-p}{q}} N^{-\frac{q-p}{2p}}
     =: C N^{-\frac{q-p}{2p}}.
     \label{NP:CH:BoundOnSet}
 \end{align}
 The combination of \eqref{NP:CH:BoundComplement} and \eqref{NP:CH:BoundOnSet} yields
 \begin{align*}
    S_1 =  \mathbb{E} \Bigeck{ \N{X^\alpha(\rho^N_k) - X^\alpha (\overbar{\rho}^N_k)}^p}
    \leq C^p_0 \mathbb{E} \Bigeck{\N{X^i_k - \overbar{X}^i_k}^p} + C N^{-\frac{q-p}{2p}}.
 \end{align*}
For $S_2$, we can directly use Lemma \ref{GerberLemma3.7} to obtain a constant $C_{\mathrm{EM}}>0$ independent of $N$ such that
\begin{align*}
    S_2 = \mathbb{E} \biggeck{\N{X^\alpha(\overbar{\rho}^N_k) - X^\alpha(\orho_k)}^p} \leq C_{\mathrm{EM}} N^{-\frac{p}{2}}.
\end{align*}
Define
\begin{align*}
r_* := \min \bigg\{ \frac{p}{2}, \frac{q-p}{2p} \bigg\},
\qquad
a_H := 2^{p-1}C_0^p,
\qquad
b_H := 2^{p-1}(C+C_{\mathrm{EM}}).
\end{align*}
Then we obtain
\begin{align*}
    \mathbb{E} \biggeck{\N{ X^i_{k+1} - \overbar{X}^i_{k+1}}^p}
    &\leq 2^{p-1} \Bigrou{S_1 + S_2} \\
    &\leq 2^{p-1} \Biggrou{C^p_0 \mathbb{E} \Bigeck{\N{X^i_k - \overbar{X}^i_k}^p} + C N^{-\frac{q-p}{2p}} + C_{\mathrm{EM}} N^{-\frac{p}{2}}} \\
    &\leq a_H \mathbb{E} \Bigeck{\N{X^i_k - \overbar{X}^i_k}^p}
    + b_H N^{-r_*}.
\end{align*}
Since $X_0^i=\overbar{X}_0^i$ almost surely, iterating this recursion yields
\begin{align}
    \mathbb{E} \biggeck{\N{ X^i_{k+1} - \overbar{X}^i_{k+1}}^p}
    \leq b_H N^{-r_*}\sum^{k}_{l = 0} a_H^l.
    \label{NP:CH:Thrm:FinalEstimate}
\end{align}
Hence
\begin{align*}
    \mathbb{E} \biggeck{\supu{k}{K} \,\N{ X^i_k - \overbar{X}^i_k}^p}
    &\leq \sum_{k=1}^{K}\mathbb{E} \biggeck{\N{ X^i_k - \overbar{X}^i_k}^p} \\
    &\leq b_H N^{-r_*}\sum_{k=1}^{K}\sum^{k-1}_{l = 0} a_H^l \\
    &= b_H N^{-r_*}\sum_{l=0}^{K-1}(K-l)a_H^l \\
    &= 2^{p-1}(C+C_{\mathrm{EM}})N^{-r_*}
    \sum_{l=0}^{K-1}(K-l)(2^{p-1}C_0^p)^l
    =: \varepsilon_\mathrm{H}^{\mathrm{unrest}}(N,K).
\end{align*}
This concludes the proof.
\end{proof}
\begin{remark}
    We want to compare our mean-field estimate for hopping with the one obtained in the freezing case \eqref{NPFreezing:PoCBound}. Recall that for the freezing scheme we proved
    \begin{align*}
        &\mathbb{E} \biggeck{\suptT \,\N{ X^i_t - \overbar{X}^i_t}^p}\\
        \leq \, 
        &2 \cdot 2^{p-1}(1-g) (C + C_{\mathrm{EM}}) N^{-r_*}
        \sum_{l = 0}^{J-1} (J-l)\bigrou{g + 2^{p-1}(1-g) C_0^p}^l.
    \end{align*}
    If we now take $J=K$ and pass to the limit $s \longrightarrow \infty$ (which means $g \longrightarrow 0$), we obtain
    \begin{align*}
        \mathbb{E} \biggeck{\suptT \,\N{ X^i_t - \overbar{X}^i_t}^p}
        \xrightarrow[]{s \to \infty}
        2 \cdot 2^{p-1}(C + C_{\mathrm{EM}}) N^{-r_*}
        \sum_{l = 0}^{K-1} (K-l)\bigrou{2^{p-1} C_0^p}^l.
    \end{align*}
    By the hopping estimate, we have
    \begin{align*}
        \mathbb{E} \biggeck{\supu{k}{K} \,\N{ X^i_k - \overbar{X}^i_k}^p}
        \leq
        2^{p-1}(C + C_{\mathrm{EM}}) N^{-r_*}
        \sum_{l = 0}^{K-1} (K-l)\bigrou{2^{p-1} C_0^p}^l.
    \end{align*}
    Thus, the two estimates have almost the same bound, differing only by the appearance of the factor $2$ in the freezing case. This reflects the convergence of the freezing scheme to the hopping scheme as $s \to \infty$, proven in Theorem \ref{thm:FStoHS}.
\end{remark}
\begin{remark}
    We observe that in the case where the initial distribution $\orho_0$ has finite moments of order $q \geq p^2+p$, Theorems \ref{NPDelta:Thrm}, \ref{NPFreezing:Thrm}, and \ref{NPHopping:Thrm} yield propagation of chaos in Wasserstein-$p$ distance at rate $\mathcal{O}(N^{-1/2})$, which is consistent with the results in \cite{Gerber2025Mean-fieldSampling, Fornasier2024Consensus-BasedGlobally}. If $2p \leq q < p^2 +p$, the rate of convergence reduces to $\mathcal{O}(N^{-(q-p)/2p^2})$, which deteriorates as $p$ increases. Note that $q \geq 2p$ implies $\frac{q-p}{2p^2} \geq \frac{1}{2p}$, so the unrestricted rate is never worse than $\mathcal{O}(N^{-1/(2p)})$.
\end{remark}

\section{Comparison to the Restricted Case of Chapter \ref{chap:PropGoodSets}}
In this section, we briefly discuss the main differences between the propagation of chaos results in Chapter \ref{chap:PropGoodSets} and those in Chapter \ref{chap:PropChaUnrestricted}. 
We could see in Chapter \ref{chap:PropChaUnrestricted} that we are able to avoid restricting the analysis to a fixed good set by instead splitting according to bounded moments of the empirical measure of the i.i.d. mean-field samples. This gave us the opportunity to invoke Lemma \ref{GerberLemma2.5} on a mean of i.i.d. random variables, bounding the probability of the event of not having bounded $p$th moments by a term that vanishes as $N$ goes to infinity. The problem of having to deal with a possibly non-uniformly bounded sequence of empirical measures for the particle system $\IPS{T}$ is resolved by replacing the stability estimate of Lemma \ref{CarrL2} with Lemma \ref{GerberStability2}, which only requires one of the two measures to have uniformly bounded $p$th moments. But apart from that, we are still left with two estimates that differ slightly in their speed of convergence.
The theorems of Chapter \ref{chap:PropGoodSets} all provide us with a convergence speed at rate $\mathcal{O}\bigrou{N^{-1/2}}$ while those in Chapter \ref{chap:PropChaUnrestricted} give us $\mathcal{O}\bigrou{N^{-\min\left\{1/2,(q-p)/(2p^2)\right\}}}$. We note that in the original statement \cite[Theorem 2.6]{Gerber2025Mean-fieldSampling}, the authors require the initial distribution $\orho_0 \in \PRd$ to have bounded moments of all orders. This has the consequence that for any $p\geq 1$, we have the ability to choose a sufficiently large $q$ such that $\min \big \{ \frac{1}{2}, \frac{q-p}{2p^2} \big \} = \frac{1}{2}$ and therefore Theorems \ref{NPDelta:Thrm}, \ref{NPFreezing:Thrm} and \ref{NPHopping:Thrm} yield a rate of $\mathcal{O}\bigrou{N^{-1/2}}$. In this case the results of Chapter \ref{chap:PropChaUnrestricted} are strictly better than the ones in Chapter \ref{chap:PropGoodSets} as we obtain the same result without restricting ourselves to a set of specific events. But now the question about dealing with initial distributions that do not have finite moments of all orders arises. We point out that $\min \big \{ \frac{1}{2}, \frac{q-p}{2p^2} \big \} = \frac{1}{2}$ exactly requires that $q \geq p^2 + p$. Notice that this is already violated in common scenarios as those discussed in \cite{Fornasier2024Consensus-BasedGlobally, Fornasier2026FromConvergence, Carrillo2018AnMethod} to just name a few. The authors of \cite{Fornasier2024Consensus-BasedGlobally} prove propagation of chaos in Wasserstein-2 distance for the classic CBO \eqref{Pre:CBOSDE} assuming the initial distribution $\orho_0$ to have finite fourth moments. In our framework, this would constitute the case of $q=4$ and $p=2$. For the method of Chapter \ref{chap:PropGoodSets}, this would result in propagation of chaos at rate $\mathcal{O}\bigrou{N^{-1/2}}$ while for Chapter \ref{chap:PropChaUnrestricted}, this would constitute propagation of chaos at rate $\mathcal{O}\bigrou{N^{-1/4}}$. \footnote{That is, if one only assumes bounded fourth moments and no higher-order moment bounds.}
\begin{table}[h]
\centering
\begin{tabular}{c|c|c}
\hline
 Bounded Moments of Order & Restricted (Chapter \ref{chap:PropGoodSets}) & Unrestricted (Chapter \ref{chap:PropChaUnrestricted}) \\
\hline
 $p + p^2 \leq q$ & $\mathcal{O}(N^{-\frac{1}{2}})$ & $\mathcal{O}(N^{-\frac{1}{2}})$ \\
 $ 2p \leq q < p^2 +p$ & $\mathcal{O}(N^{-\frac{1}{2}})$ & $\mathcal{O} \Bigrou{N^{- \min \left \{ \frac{1}{2}, \frac{q-p}{2p^2} \right \} }}$ \\
\hline
\end{tabular}
\caption{Comparison of the convergence rates for both results.}
\end{table}
\newpage

\section{Towards Global Convergence: Freezing and Hopping}
\label{section:globalconvergence}

Let us go back to the reason why we established the propagation of chaos results in the preceding sections. Their primary application is to derive quantitative convergence estimates for the interacting particle systems toward the mean of the invariant measure of the mean-field dynamics. The global convergence of the mean-field dynamics of the $\delta$-CBO scheme \cite[Theorem 2.5, Theorem 3.15]{Fornasier2026FromConvergence}, the consensus freezing scheme \cite[Theorem 3.10]{Fornasier2026FromConvergence}, and the consensus hopping scheme \cite[Corollary 4.2]{Fornasier2026FromConvergence} have recently been established in \cite{Fornasier2026FromConvergence}, showing in the freezing case the convergence of the mean-field limit $(\orho_t)_t$ to the invariant measure $\boldsymbol\rho^\alpha_\infty = \mathcal{N}(\boldsymbol{\mu}^\alpha_\infty, \frac{\delta^2}{2\lambda}I_d)$ at a rate $\gamma^j$ for $\gamma \in (e^{-s\lambda\Delta t}, 1)$
\begin{align*}
     W_2(\orho_{t},\boldsymbol\rho_\infty^\alpha)\leq C_\alpha \gamma^j, \quad \forall t \in [t_j, t_{j+1}), \quad j \in \mathbb{N}_0, 
\end{align*}
while
\begin{align*}
     W_2(\boldsymbol\rho_\infty^\alpha, \boldsymbol{\rho}_\infty) \xrightarrow[]{\alpha \to \infty} 0 \,, 
\end{align*}
where $\boldsymbol{\rho}_\infty= \mathcal{N}\left(x^*,\frac{\delta^2}{2\lambda}I_d\right)$.

In this section, we establish a connection between the above global convergence results and the propagation of chaos results obtained in this thesis to derive meaningful convergence estimates for the interacting particle systems toward $\boldsymbol\mu_\infty^\alpha$. We do this by proving a statement in the spirit of \cite[Theorem 3.8]{Fornasier2024Consensus-BasedGlobally}. Since $\boldsymbol\rho_\infty^\alpha \to \boldsymbol\rho_\infty$ as $\alpha\to\infty$, this also links the particle system to the global minimizer $x^*$.
The key idea is simple: We use the inequality $\N{a+b+c}^2 \le 3(\N{a}^2+\N{b}^2+\N{c}^2)$ to estimate
\begin{align*}
\N{M_T^N-\boldsymbol\mu^\alpha_\infty}^2
    \le
    3\underbrace{\N{M_T^N-\overbar M_T^N}^2}_{\text{propagation of chaos}}
    +3\underbrace{\N{\overbar M_T^N-\mathbb E[\overbar X_T^1]}^2}_{\text{variance}}
    +3\underbrace{\N{\mathbb E[\overbar X_T^1]-\boldsymbol\mu^\alpha_\infty}^2}_{\text{global convergence}},
\end{align*}
where 
\begin{align*}
     M_T^N := \frac1N \sum_{i=1}^N  X_T^i, \qquad \overbar M_T^N := \frac1N \sum_{i=1}^N \overbar X_T^i
\end{align*}
are the empirical means of the finite particle system and the mean-field finite particle system, respectively.
Furthermore $\boldsymbol\mu^\alpha_\infty$ is the mean of the invariant measure $\boldsymbol\rho^\alpha_\infty$ of the mean-field dynamics. On the propagation of chaos term, if $\orho_0$ has bounded moments of order $6$, we can invoke the results from Chapter \ref{chap:PropChaUnrestricted} to bound it at rate $\mathcal{O}(N^{-1})$ over a finite time interval $[0, T]$. Using the finite second moments of the mean-field limit process $(\overbar{X}_t)_{t \in [0, T]}$, the variance term can be bounded by $S/N$ with a constant $S$ independent of $T$ and $N$. The global convergence term is entirely independent of $N$ and measures how far the mean-field dynamics has yet to converge to its invariant measure.
Combining this and our propagation of chaos results with the ideas of \cite[Theorem 3.8]{Fornasier2024Consensus-BasedGlobally} yields for the freezing scheme a bound of the form
\begin{align}
        \mathbb P\bigg(
    \N{M_T^N-\boldsymbol\mu^\alpha_\infty}^2
    \ge 3C_\alpha^2\gamma^{2J}+\varepsilon
    \bigg)
    \le \frac{C(T)}{\varepsilon N}, \label{probabibound1}
\end{align}
showing that for sufficiently large $\alpha$ and large $N$, the empirical mean of the particle system is close to $\boldsymbol{\mu}^\alpha_\infty$ with high probability.
However, this result comes with an important downside. The propagation of chaos estimates in Chapters \ref{chap:PropGoodSets} and \ref{chap:PropChaUnrestricted} are established over a finite time horizon $[0,T]$. In particular, we allow the constant $C(T)$ to grow exponentially in $T$. This poses a problem: as $T$ grows, the mean-field dynamics converges exponentially toward $\boldsymbol\rho^\alpha_\infty$, by driving $C_\alpha\gamma^J \to 0$, but simultaneously the estimate \eqref{probabibound1} requires $N \geq C(T)/(\varepsilon \xi)$ particles to maintain a probability of at most $\xi$, and this threshold grows exponentially in $T$. The result is therefore only meaningful over a fixed finite time horizon, and the number of particles required to achieve a prescribed accuracy grows exponentially with that horizon. Overcoming this limitation is precisely the motivation for uniform-in-time propagation of chaos results, as briefly discussed in Section \ref{Sec:Outlook}. We focus here on the consensus freezing scheme and present the analogous result for the hopping scheme via the connection established in Theorem \ref{thm:FStoHS}. For the $\delta$-CBO scheme, the global convergence result \cite[Theorem 2.5]{Fornasier2026FromConvergence} is formulated in relative entropy, holds up to a finite time horizon, and requires an additional assumption on $f$ near $x^*$; we therefore do not pursue the analogous statement here.

\subsection{Global Convergence in Probability: Freezing}
Before we can turn to the central result of this section, we need to state some auxiliary results. We introduce the main results from \cite{Fornasier2026FromConvergence} that lead to the global convergence of the mean-field dynamics of the consensus freezing scheme. For further details, we refer to \cite[Section 3.1]{Fornasier2026FromConvergence}.

\vspace{.2 cm}

We begin by defining the map $\mathcal{T}^\alpha$ by \begin{equation}\label{eq:T_alpha}
    \begin{split}
        \mathcal{T}^\alpha\colon\mathbb{R}^d&\to \mathbb{R}^d \\
        \mu &\mapsto X^\alpha\Big(\mathcal{N}\big(\mu,\frac{\delta^2}{2\lambda}I_d\big)\Big).
    \end{split}
\end{equation} 
The following theorem establishes the existence and uniqueness of an invariant measure for the mean-field dynamics \eqref{CF:MFscheme}. The key insight is that the map $\mathcal{T}^\alpha$ becomes a contraction on a ball around $x^*$ for sufficiently large $\alpha$. By the Banach fixed point theorem, this contraction has a unique fixed point — a Gaussian invariant measure $\boldsymbol\rho^\alpha_\infty$ centered near $x^*$. As $\alpha \to \infty$, $\boldsymbol\rho^\alpha_\infty$ converges in Wasserstein-2 distance to the measure $\boldsymbol\rho_\infty = \mathcal{N}(x^*, \frac{\delta^2}{2\lambda}I_d)$.
\begin{theorem}[{{\cite[Theorem 3.5]{Fornasier2026FromConvergence}}}]\label{thm:invariant}
    For any $r>0$ and 
    $\alpha$ sufficiently large depending on $r$,
    the map $\mathcal{T}^\alpha$ defined by \cref{eq:T_alpha} restricted to $B_r(x^*)$ is a contraction mapping with Lipschitz constant $L_\alpha$, where $L_\alpha\to 0$ as $\alpha\to\infty$.
    Furthermore, for $\alpha$ sufficiently large, the dynamics \eqref{CF:MFscheme} has a unique invariant measure $\boldsymbol\rho_\infty^\alpha$ with the property that $\boldsymbol\mu_\infty^\alpha \coloneqq \mathbb{E}_{\boldsymbol{\rho}_\infty^\alpha}[x] \in B_r(x^*)$, and $\boldsymbol{\rho}_\infty^\alpha$ is of the form
    \begin{equation}\label{eq:invariant}
        \boldsymbol\rho_\infty^\alpha=\mathcal{N}\left(\boldsymbol\mu_\infty^\alpha,\frac{\delta^2}{2\lambda}I_d\right)=\mathcal{N}\left(\mathcal{T}^\alpha(\boldsymbol\mu_\infty^\alpha),\frac{\delta^2}{2\lambda}I_d\right)=\mathcal{N}\left(X^\alpha(\boldsymbol\rho_\infty^\alpha),\frac{\delta^2}{2\lambda}I_d\right).
    \end{equation}
Furthermore it holds that
 \begin{equation}
     W_2(\boldsymbol\rho_\infty^\alpha, \boldsymbol{\rho}_\infty)\to 0\quad \mbox{ as }\alpha\to\infty \,,
 \end{equation}
 with $\boldsymbol{\rho}_\infty= \mathcal{N}\left(x^*,\frac{\delta^2}{2\lambda}I_d\right)$\,.
\end{theorem}
The following theorem provides global-in-time bounds for the mean-field dynamics of the consensus freezing scheme \eqref{CF:MFscheme}.
\begin{theorem}[{{\cite[Theorem 3.9]{Fornasier2026FromConvergence}}}]\label{thm:global_bound}
   Let $r \in \mathbb{N}_{\geq 2}$ and let $\orho_0 \in \mathcal{P}_r(\mathbb{R}^d)$ be an arbitrary initial measure. Let $\orho_t$ be the solution to the consensus freezing scheme \eqref{CF:MFscheme}. Then, there exist $\alpha_0 > 0$ and $S > 0$ such that for all $\alpha \geq \alpha_0$, the solution satisfies
    \begin{equation}\label{eq:second_moment_bound}
        \mathbb{E}_{\overline{\rho}_t} \big[ \N{x}^q \big] \leq S, \quad \forall t \geq 0, \quad \forall q \in \mathbb{N} \text{ with } q \leq r,
    \end{equation}
    and
    \begin{equation}\label{eq:gamma_t_alpha_bound}
        \N{X^\alpha(\orho_t)} \leq S, \quad \forall t \geq 0.
    \end{equation}
\end{theorem}
The above results now make it possible to derive a global convergence result for the mean-field dynamics of the consensus freezing scheme to the invariant measure $\boldsymbol{\rho}^\alpha_\infty$ which by Theorem \ref{thm:invariant} converges to $\boldsymbol{\rho}_\infty$ in Wasserstein-2 distance as $\alpha \to \infty$.
\begin{theorem}[{{\cite[Theorem 3.10]{Fornasier2026FromConvergence}}}]\label{thm:fsconv}
    Let $\orho_{t}$ be the unique solution to the dynamics \eqref{CF:MFscheme} with the initial data $X_0\sim \orho_0 {\in \mathcal{P}_4(\mathbb{R}^d)}$ (not necessarily a Gaussian) and any $\Delta t>0$. {Let $r=2S$}, {where $S$ is a constant fulfilling \cref{eq:second_moment_bound,eq:gamma_t_alpha_bound} such that $x^* \in B_S(0)$ and $\mathbb{E}_{\rho_0}[x] \in B_S(0)$}, and assume  $\boldsymbol\rho_\infty^\alpha=\mathcal{N}(\boldsymbol\mu_\infty^\alpha,\frac{\delta^2}{2\lambda}I_d)$ is the invariant measure found in Theorem \ref{thm:invariant} {with $\boldsymbol\mu_\infty^\alpha\in B_r(x^*)$}. Then, for {any $  e^{-s \lambda \Delta t}<\gamma<1$} and for $\alpha>0$ sufficiently large it holds that
    \begin{equation}
       W_2(\orho_{t},\boldsymbol\rho_\infty^\alpha)\leq C_\alpha \gamma^j\ \quad \mbox{ for } t \in [t_j, t_{j+1}), \ j \in \mathbb{N}_{0} \,,
    \end{equation}
    for some $C_\alpha>0$. Furthermore, for any $s^\star > 0$, we may choose $C_\alpha$ and $\gamma$ uniformly over all $s\in [s^\star, +\infty)$.
\end{theorem}
We can combine the above result of global convergence of the consensus freezing scheme with our mean-field approximations to derive a probabilistic global convergence result over a finite time horizon $[0, T]$.

\begin{theorem}
\label{thm:global_conv_mean_T}
Let $T > 0$, $\Delta t > 0$, $s > 0$, $\lambda > 0$, and let $J = T/\Delta t \in \mathbb{N}$ be the number of freezing steps. Let $\orho_0 \in \mathcal{P}_6(\mathbb{R}^d)$ be the common initial distribution, not necessarily Gaussian. Let $(\mathcal{X}^N_t)_t$ be the interacting particle system defined by \eqref{CF:CFscheme}, and define
\begin{align*}
M_T^N := \frac1N \sum_{i=1}^N X_T^i.
\end{align*}
Let $\boldsymbol\rho_\infty^\alpha=\mathcal{N}(\boldsymbol\mu_\infty^\alpha,\frac{\delta^2}{2\lambda}I_d)$ be the unique invariant measure from Theorem \ref{thm:invariant}. For any $\gamma \in (e^{-s\lambda\Delta t},1)$ and $\alpha$ sufficiently large, there exists a constant $C_\alpha>0$ and a constant $C(T)>0$, independent of $N$, such that for every $\varepsilon>0$ and every $N\in\mathbb N$,
\begin{align}
    \mathbb P\bigg(
    \N{M_T^N-\boldsymbol\mu^\alpha_\infty}^2
    \ge 3C_\alpha^2\gamma^{2J}+\varepsilon
    \bigg)
    \le \frac{C(T)}{\varepsilon N}.
    \label{eq:global_mean_T}
\end{align}
\end{theorem}

\begin{proof}
Let $(\overbar X_T^1, ..., \overbar{X}^N_T)$ be the coupled mean-field particle system \eqref{CF:MFIPS} at time $T$ and define
\begin{align*}
\overbar M_T^N := \frac1N \sum_{i=1}^N \overbar X_T^i.
\end{align*}
We decompose
\begin{align*}
M_T^N-\boldsymbol\mu^\alpha_\infty
=
\big(M_T^N-\overbar M_T^N\big)
+
\big(\overbar M_T^N-\mathbb E[\overbar X_T^1]\big)
+
\big(\mathbb E[\overbar X_T^1]-\boldsymbol\mu^\alpha_\infty\big).
\end{align*}
Hence, by the inequality $\N{a+b+c}^2 \le 3(\N{a}^2+\N{b}^2+\N{c}^2)$ and taking expectations,
\begin{align*}
    \mathbb{E} \Bigeck{\N{M_T^N-\boldsymbol\mu^\alpha_\infty}^2}
    \le
    3 \mathbb{E} \Bigeck{\N{M_T^N-\overbar M_T^N}^2}
    +3\mathbb{E} \Bigeck{\N{\overbar M_T^N-\mathbb E[\overbar X_T^1]}^2}
    +3\N{\mathbb E[\overbar X_T^1]-\boldsymbol\mu^\alpha_\infty}^2.
\end{align*}
We begin with the deterministic term. Since $\orho_T=\mathrm{Law}(\overbar X_T^1)$ and $\boldsymbol\rho_\infty^\alpha=\mathcal N\bigrou{\boldsymbol\mu_\infty^\alpha,\frac{\delta^2}{2\lambda}I_d}$, Jensen's inequality yields
\begin{align*}
    \N{\mathbb E[\overbar X_T^1]-\boldsymbol\mu^\alpha_\infty}^2
    =
    \N{\int x\,d\orho_T(x)-\int y\,d \boldsymbol \rho_\infty^\alpha(y)}^2
    &\le \inf_{\pi\in\Pi(\orho_T, \boldsymbol\rho_\infty^\alpha)}
\int \N{x-y}^2\,d\pi(x,y) \\ &=
    W_2^2(\orho_T, \boldsymbol\rho_\infty^\alpha).
\end{align*}
By Theorem \ref{thm:fsconv}, for $\alpha$ sufficiently large and $\gamma\in(e^{-s\lambda\Delta t},1)$,
\begin{align*}
W_2(\orho_T,\boldsymbol \rho_\infty^\alpha)\le C_\alpha\gamma^J,
\end{align*}
and therefore
\begin{align*}
\N{\mathbb E[\overbar X_T^1]-\boldsymbol\mu^\alpha_\infty}^2
\le C_\alpha^2\gamma^{2J}.
\end{align*}
It remains to control the two random terms.
For the first term, by the convexity of $x \mapsto \N{x}^2$, Jensen's inequality gives almost surely
\begin{align*}
    \N{M_T^N-\overbar M_T^N}^2
    &=
    \N{
    \frac1N\sum_{i=1}^N (X_T^i-\overbar X_T^i)
    }^2
    \le
    \frac1N\sum_{i=1}^N \N{X_T^i-\overbar X_T^i}^2.
\end{align*}
Taking expectations and using Theorem \ref{NPFreezing:Thrm} with $p=2$ and $q=6$, we obtain
\begin{align*}
    \mathbb E \biggeck{\N{M_T^N-\overbar M_T^N}^2}
    &\le
    \frac1N\sum_{i=1}^N \mathbb E\biggeck{\N{X_T^i-\overbar X_T^i}^2}
    \le
    \varepsilon_{\mathrm{F}}^{\mathrm{unrest}}(N,T)
    =: \frac{C_{\mathrm{PoC}}(T)}{N}.
\end{align*}
Hence, by Markov's inequality,
\begin{align*}
    \mathbb P\bigg(
    3\N{M_T^N-\overbar M_T^N}^2 \ge \frac{\varepsilon}{2}
    \bigg)
    \le
    \frac{6\,C_{\mathrm{PoC}}(T)}{\varepsilon N}.
\end{align*}
For the second term, since $\overbar X_T^1,\dots,\overbar X_T^N$ are i.i.d. with mean $\mathbb E[\overbar X_T^1]$, we have

\begin{align*}
    \mathbb E\biggeck{\N{\overbar M_T^N-\mathbb E[\overbar X_T^1]}^2}
    &=
    \mathbb E\Biggeck{\N{\frac1N\sum_{i=1}^N \big(\overbar X_T^i-\mathbb E[\overbar X_T^1]\big)}^2} \\
    &=
    \frac1{N^2}\sum_{i=1}^N
    \mathbb E\biggeck{\N{\overbar X_T^i-\mathbb E[\overbar X_T^1]}^2} \\
    &=
    \frac1N\,\mathbb E\biggeck{\N{\overbar X_T^1-\mathbb E[\overbar X_T^1]}^2} \\
    &=
    \frac1N\Biggrou{\mathbb E\biggeck{\N{\overbar X_T^1}^2}
    - \N{\mathbb E[\overbar X_T^1]}^2} \\
    &\le
    \frac1N\,\mathbb E\biggeck{\N{\overbar X_T^1}^2}
    =
    \frac1N\,\mathbb E_{\overline{\rho}_T}\biggeck{\N{x}^2}.
\end{align*}
By Theorem \ref{thm:global_bound}, the second moment of $\orho_T$ is bounded uniformly in time, and therefore
\begin{align*}
\mathbb E\biggeck{\N{\overbar M_T^N-\mathbb E[\overbar X_T^1]}^2}
\le \frac{S}{N}.
\end{align*}
Again by Markov's inequality,
\begin{align*}
    \mathbb P\bigg(
    3\N{\overbar M_T^N-\mathbb E[\overbar X_T^1]}^2 \ge \frac{\varepsilon}{2}
    \bigg)
    \le
    \frac{6S}{\varepsilon N}.
\end{align*}
Combining the previous estimates, we get
\begin{align*}
    \mathbb P\bigg(
    \N{M_T^N-\boldsymbol\mu^\alpha_\infty}^2
    \ge 3C_\alpha^2\gamma^{2J}+\varepsilon
    \bigg)
    &\le
    \mathbb P\bigg(
    3\N{M_T^N-\overbar M_T^N}^2
    +
    3\N{\overbar M_T^N-\mathbb E[\overbar X_T^1]}^2
    \ge \varepsilon
    \bigg) \\
    &\le
    \mathbb P\bigg(
    3\N{M_T^N-\overbar M_T^N}^2 \ge \frac{\varepsilon}{2}
    \bigg)
    +
    \mathbb P\bigg(
    3\N{\overbar M_T^N-\mathbb E[\overbar X_T^1]}^2 \ge \frac{\varepsilon}{2}
    \bigg) \\
    &\le
    \frac{6(C_{\mathrm{PoC}}(T)+S)}{\varepsilon N}.
\end{align*}
Thus the claim follows with
\begin{align}
C(T):=6\big(C_{\mathrm{PoC}}(T)+S\big).\label{eq:global_IPS_TT}
\end{align}
\end{proof}
\begin{remark}
The estimate in Theorem \ref{thm:global_conv_mean_T} is meaningful for every fixed time horizon $T>0$, but it is not uniform in time. The bottleneck is the constant
\begin{align*}
C(T)=6\big(C_{\mathrm{PoC}}(T)+S\big),
\end{align*}
and more precisely the propagation of chaos contribution $C_{\mathrm{PoC}}(T)$. The propagation of chaos constant $C_{\mathrm{PoC}}(T)$ grows exponentially in $T$, since it is inherited from the finite-horizon propagation of chaos estimates established earlier in this thesis. Hence, although the deterministic convergence term $C_\alpha^2\gamma^{2J}$ decays exponentially fast as $T\to\infty$, the prefactor in the probabilistic bound deteriorates at the same time.
In order to guarantee
\begin{align*}
\mathbb P\bigg(
    \N{M_T^N-\boldsymbol\mu^\alpha_\infty}^2
    \le 3C_\alpha^2\gamma^{2J}+\varepsilon
\bigg)\ge 1-\xi,
\end{align*}
it is sufficient to choose
\begin{align*}
N \ge \frac{C(T)}{\varepsilon \xi}.
\end{align*}
Thus, the number of particles required to maintain a confidence level of at least $1-\xi$ grows proportionally to $C(T)$. Since $C_{\mathrm{PoC}}(T)$ grows exponentially in $T$, the same is true for the required particle number. Therefore, Theorem \ref{thm:global_conv_mean_T} does not yield a global-in-time convergence result for the interacting particle system with a fixed number of particles. If $C(T)$ were bounded uniformly in $T$, then one could choose $N=\mathcal O((\varepsilon\xi)^{-1})$ independently of the time horizon and obtain a genuine uniform-in-time high-probability estimate. Thus, the main obstruction to a full global convergence result for the interacting particle system is the lack of a propagation of chaos estimate that remains uniform in time.
\end{remark}
\begin{remark}
We might ask ourselves whether, instead of proving convergence at the level of the empirical mean, one could aim directly for a bound on
\begin{align*}
\mathbb{E} \Big[W^2_2\big(\rho_T^N,\boldsymbol\rho_\infty^\alpha\big)\Big].
\end{align*}
The triangle inequality yields
\begin{align*}
\mathbb{E} \Big[W^2_2\big(\rho_T^N,\boldsymbol\rho_\infty^\alpha\big)\Big]
\leq
2\mathbb{E} \Big[W^2_2\big(\rho_T^N,\orho_T\big)\Big]
+
2W^2_2\big(\orho_T,\boldsymbol\rho_\infty^\alpha\big),
\end{align*}
where $\orho_T$ denotes the law of the mean-field dynamics at time $T$. To estimate the first term, one may further decompose
\begin{align*}
\mathbb{E}\Big[W^2_2\big(\rho_T^N,\orho_T\big)\Big]
\leq
2\mathbb{E}\Big[W^2_2\big(\rho_T^N,\overline{\rho}_T^N\big)\Big]
+
2\mathbb{E}\Big[W^2_2\big(\overline{\rho}_T^N,\orho_T\big)\Big],
\end{align*}
where
\begin{align*}
\overline{\rho}_T^N:=\frac1N\sum_{i=1}^N\delta_{\overline X_T^i}
\end{align*}
is the empirical measure of the coupled i.i.d. mean-field particles.

The first term can be controlled by Example \ref{WasserBD} and Theorem \ref{NPFreezing:Thrm} to get a similar estimate to that of Theorem \ref{thm:global_conv_mean_T} at rate $\mathcal{O}(N^{-1/2})$. To bound the second term, we need to use Lemma \ref{Pre:ConvEmp}. In particular, for $p=2, q = 6$,
\begin{align*}
    \mathbb{E} \bigeck{W^2_2(\orho^N_T, \orho_T)} \leq C M_q^{1/3} \begin{cases}
             N^{-1/2} + N^{- 2/3} \quad &\text{if } d<4 \\
             N^{-1/2}\log (1+ N) + N^{- 2/3} \quad &\text{if } d=4, \\
             N^{-2/d} + N^{- 2/3} \quad &\text{if } d>4 \text{ and } \frac{d}{d-2} \neq 6,
         \end{cases}
\end{align*}
where $C> 0$ is a constant depending on the dimension $d$.
Hence, while we still recover the rate $\mathcal{O}(N^{-1/2})$ in low dimension, the estimate deteriorates as soon as $d>4$, where the leading term becomes $\mathcal{O}(N^{-2/d})$. Hence, the convergence rate deteriorates as the dimension $d$ increases.
For this reason, the global convergence results in this section are formulated at the level of the empirical mean rather than directly in Wasserstein distance for the empirical measure.
\end{remark}
\subsection{Global Convergence in Probability: Hopping}
As noted in the remark after the proof of Lemma \ref{NP_CHLemma} and after Theorems \ref{MFH: Thrm} and \ref{NPHopping:Thrm}, there is a close connection between the hopping and freezing schemes. Moreover, as the scaling parameter $s$ tends to infinity, the freezing scheme converges to the hopping scheme. We now show that this allows us to prove for the hopping scheme the same type of result as in Theorem \ref{thm:global_conv_mean_T}.

We begin by stating the result that proves the convergence of the mean-field dynamics of the consensus freezing scheme to the consensus hopping scheme as $s \to \infty$.
\begin{theorem}[{{\cite[Theorem 4.1]{Fornasier2026FromConvergence}}}]\label{thm:FStoHS}
    Let $x_0 \in \mathbb{R}^d$ be an arbitrary initial point. Let $\{x_k^{CH}\}_{k \in \mathbb{N}}$ be the solution to the Consensus Hopping scheme \eqref{HoppingSchemeOriginal} with initial data $x_0$ and $\tilde{\sigma} = \frac{\delta}{\sqrt{2\lambda}}$. Furthermore, for $s > 0$, let $\orho_t$ be the solution to the Consensus Freezing scheme \eqref{CF:MFscheme} with initial data $\orho_0 = \mathcal{N}(x_0, \Sigma)$, where $\Sigma = \frac{\delta^2}{2\lambda} I_d$, and using the prescribed $s$. The evolution of $\orho_t$ is given by
    \begin{equation}\label{eq:flip_scheme_solution_s}
    \begin{split}
        \orho_t &= \mathcal{N}(\mu^{\alpha}_t, \Sigma), \qquad \text{where} \\
        \mu^{\alpha}_t &= \mathcal{T}^\alpha(\mu^{\alpha}_{t_i})+e^{-s \cdot \lambda(t-t_i)}(\mu^{\alpha}_{t_i}-\mathcal{T}^\alpha(\mu^{\alpha}_{t_i})), \quad t \in (t_i, t_{i+1}].
    \end{split}
\end{equation}
    Then, for large enough $\alpha$, we have that
    \begin{equation}\label{eq:FS_to_HS}
        |\mu^{\alpha}_{t_k} - x_k^{CH}| \leq 4 R \cdot e^{-s \lambda \Delta t} \to 0 \quad \text{as } s \to \infty, \quad \forall k \in \mathbb{N},
    \end{equation}
    where $R > 0$ is a uniform bound on $\mu_t^s$ given by global-in-time bounds of the freezing scheme such that $x^* \in B_R(0)$.
    In particular, we have that
    \begin{equation}\label{eq:W2_FS_to_HS}
        W_2(\orho_{t_k}, \mathcal{N}(x_k^{CH}, \Sigma)) \leq 4 R \cdot e^{-s \lambda \Delta t} \to 0 \quad \text{as } s \to \infty,
    \end{equation}
    {uniformly with respect to $k\in \mathbb N$}.
\end{theorem}
A direct consequence of Theorem \ref{thm:FStoHS} is the global convergence of the consensus hopping scheme. This follows directly from the triangle inequality together with Theorems~\ref{thm:fsconv} and \ref{thm:FStoHS}.
\begin{corollary}[{{\cite[Corollary 4.2]{Fornasier2026FromConvergence}}}]\label{cor:HS_convergence}
 Let $x_0 \in \mathbb{R}^d$ be an arbitrary initial point. Let \newline $\{x_k^{CH}\}_{k \in \mathbb{N}}$ be the solution to the Consensus Hopping scheme \eqref{HoppingSchemeOriginal} with initial data $x_0$ and $\tilde{\sigma} = \frac{\delta}{\sqrt{2\lambda}}$. Moreover, for any $k\in\mathbb{N}$, let $\nu_k^\alpha=\mathcal N(x_k^{CH}, \frac{\delta^2}{2\lambda}I)$. 
 Then, for any $\gamma \in (0, 1)$ and for $\alpha$ sufficiently large, it holds that
 \begin{equation*}
        W_2(\nu_k^\alpha, \boldsymbol\rho_\infty^\alpha) \leq C_\alpha \gamma^k,
 \end{equation*}
 for some $C_\alpha > 0$ depending on $\alpha$.
\end{corollary}
As Corollary \ref{cor:HS_convergence} is the hopping version of Theorem \ref{thm:fsconv}, we can use it along with the same line of argument as for Theorem \ref{thm:global_conv_mean_T} to prove the same statement for the consensus hopping scheme.
\begin{theorem}
\label{cor:global_conv_hopping}
Let $\orho_0 \in \mathcal{P}_6(\mathbb{R}^d)$ be the common initial distribution, not necessarily Gaussian, and let the noise in \eqref{CH:Scheme} have variance $\sigma^2 = \frac{\delta^2}{2\lambda}$. Let $(\mathcal{X}^N_k)_k$ be the interacting particle system defined by \eqref{CH:Scheme}, and define
\begin{align*}
M_K^N := \frac1N \sum_{i=1}^N X_K^i.
\end{align*}
Let $\boldsymbol\rho_\infty^\alpha=\mathcal{N}(\boldsymbol\mu_\infty^\alpha,\frac{\delta^2}{2\lambda}I_d)$ be the unique invariant measure from Theorem \ref{thm:invariant}. For any $\gamma \in (0,1)$ and $\alpha$ sufficiently large, there exist constants $C_\alpha>0$ and $C(K)>0$, independent of $N$, such that for every $\varepsilon>0$, every $K \in \mathbb{N}_{>0}$ and every $N\in\mathbb N$,
\begin{align}
    \mathbb P\bigg(
    \N{M_K^N-\boldsymbol\mu^\alpha_\infty}^2
    \ge 3C_\alpha^2\gamma^{2(K-1)}+\varepsilon
    \bigg)
    \le \frac{C(K)}{\varepsilon N}.
    \label{eq:global_mean_K}
\end{align}
\end{theorem}
 
\begin{proof}
The argument is the same as in the proof of Theorem~\ref{thm:global_conv_mean_T}; the propagation of chaos term is now controlled by Theorem~\ref{NPHopping:Thrm} instead of Theorem~\ref{NPFreezing:Thrm}, and the deterministic mean-field convergence term by Corollary~\ref{cor:HS_convergence} instead of Theorem~\ref{thm:fsconv}.
Let
\begin{align*}
\overbar M_K^N := \frac1N \sum_{i=1}^N \overbar X_K^i
\end{align*}
denote the empirical mean of the coupled mean-field particle system. We decompose
\begin{align*}
M_K^N-\boldsymbol\mu^\alpha_\infty
=
\big(M_K^N-\overbar M_K^N\big)
+
\big(\overbar M_K^N-\mathbb E[\overbar X_K^1]\big)
+
\big(\mathbb E[\overbar X_K^1]-\boldsymbol\mu^\alpha_\infty\big).
\end{align*}
For the first term, Theorem \ref{NPHopping:Thrm} with $p=2$ and $q=6$ yields, exactly as in the proof of Theorem \ref{thm:global_conv_mean_T},
\begin{align*}
    \mathbb E \biggeck{\N{M_K^N-\overbar M_K^N}^2}
    \le
    \frac1N\sum_{i=1}^N \mathbb E\biggeck{\N{X_K^i-\overbar X_K^i}^2}
    \le
    \varepsilon_{\mathrm{H}}^{\mathrm{unrest}}(N,K)
    =: \frac{C_{\mathrm{PoC}}(K)}{N}.
\end{align*}
For the second term, we use that $\overbar X_K^1,\dots,\overbar X_K^N$ are i.i.d.\ with law $\orho_K = \mathcal{N}\bigrou{X^\alpha(\orho_{K-1}), \sigma^2 I_d}$, so that
\begin{align*}
\mathbb E\biggeck{\N{\overbar M_K^N-\mathbb E[\overbar X_K^1]}^2}
= \frac{1}{N}\, \mathbb E\biggeck{\N{\overbar X_K^1-\mathbb E[\overbar X_K^1]}^2}
= \frac{d\sigma^2}{N}.
\end{align*}
It remains to treat the deterministic term. By \eqref{CH:MFScheme}, for every $k \geq 1$ we have $\orho_k = \mathcal{N}(m_k, \sigma^2 I_d)$ with $m_k := X^\alpha(\orho_{k-1})$. In particular, $m_1 = X^\alpha(\orho_0)$, and since $\sigma^2 = \frac{\delta^2}{2\lambda}$, the definition \eqref{eq:T_alpha} of $\mathcal{T}^\alpha$ gives
\begin{align*}
    m_k = X^\alpha\bigrou{\mathcal{N}(m_{k-1}, \sigma^2 I_d)} = \mathcal{T}^\alpha(m_{k-1}), \quad k \geq 2.
\end{align*}
The consensus hopping scheme \eqref{HoppingSchemeOriginal} satisfies $x^{CH}_k = \mathcal{T}^\alpha(x^{CH}_{k-1})$ as well. Hence, choosing the initial point $x_0 := X^\alpha(\orho_0)$ in Corollary \ref{cor:HS_convergence}, we obtain $m_k = x^{CH}_{k-1}$ and therefore $\orho_k = \nu^\alpha_{k-1}$ for every $k \geq 1$. Jensen's inequality and Corollary \ref{cor:HS_convergence} yield
\begin{align*}
\N{\mathbb E[\overbar X_K^1]-\boldsymbol\mu^\alpha_\infty}^2
\le
W_2^2(\orho_K,\boldsymbol\rho_\infty^\alpha)
=
W_2^2(\nu^\alpha_{K-1},\boldsymbol\rho_\infty^\alpha)
\le
C_\alpha^2\gamma^{2(K-1)}.
\end{align*}
Proceeding with Markov's inequality exactly as in the proof of Theorem \ref{thm:global_conv_mean_T}, we obtain \eqref{eq:global_mean_K} with
\begin{align*}
    C(K) := 6\bigrou{C_{\mathrm{PoC}}(K) + d\sigma^2}.
\end{align*}
\end{proof}
\vspace{.3 cm}
We see that for both the freezing and hopping scheme, the convergence results of Theorems \ref{thm:global_conv_mean_T} and \ref{cor:global_conv_hopping}, combined with their respective propagation of chaos results, Theorems \ref{NPFreezing:Thrm} and \ref{NPHopping:Thrm}, yield a quantitative bound ensuring that, for sufficiently many particles, i.e. $N \geq C(T) / (\varepsilon \xi)$,\footnote{For readability, we write the condition in terms of the continuous-time variable $T$. The completely analogous statement in discrete time is obtained by replacing $T$ with $K$ and $C(T)$ with $C(K)$.} the probability of being close to the mean of the invariant measure of the dynamics $\boldsymbol{\mu}^\alpha_\infty$ is at least $1 - \xi$.
As discussed after Theorem \ref{thm:global_conv_mean_T}, the bottleneck of the probability bound is the propagation of chaos constant, namely $C_{\mathrm{PoC}}(T)$ in the freezing case and $C_{\mathrm{PoC}}(K)$ in the hopping case, which grows exponentially in the corresponding time horizon. What is still missing for a full global convergence result for the interacting particle system is a propagation-of-chaos estimate that remains uniform in time over infinite time horizons.
\newpage
\section{Outlook: Uniform-in-time Propagation of Chaos}
\label{Sec:Outlook}
The quantitative estimates derived in this thesis establish propagation of chaos
over a finite time horizon $[0, T]$, but the constants involved grow exponentially
in $T$. This is a significant limitation, since many convergence results for CBO
methods concern the long-time behavior. Key convergence results such as
\cite[Theorem 3.7]{Fornasier2024Consensus-BasedGlobally},
\cite[Theorem 4.2]{Carrillo2018AnMethod} as well as
\cite[Theorems 2.5, 3.10, 3.15 and Corollary 4.2]{Fornasier2026FromConvergence}
rely on exponential decay over time. It is therefore desirable to establish
propagation of chaos estimates that are uniform in time. That this is a
non-trivial task can already be seen from the hopping scheme. Recalling the
distance estimate \eqref{NP:CH:Thrm:FinalEstimate}, for all $i \leq N \in
\mathbb{N}$ and $k \in \mathbb{N}$,
\begin{align*}
    \mathbb{E} \biggeck{\N{ X^i_{k} - \overbar{X}^i_{k}}^p} \leq 2^{p-1}
    (C+C_{\mathrm{EM}}) N^{- \min \{ \frac{p}{2}, \frac{q-p}{2p} \} }
    \sum^{k-1}_{l = 0} (2^{p-1}C^p_0)^{l},
\end{align*}
where $C_0$ denotes the stability constant of the consensus point of Lemma
\ref{GerberStability2}. This bound converges as $k \to \infty$ if and only if
$2^{p-1}C_0^p < 1$; the factor $2^{p-1}$ is an artifact of estimating $p$-th
powers, and running the recursion on $L^p$ norms instead yields the condition
$C_0 < 1$. Uniform-in-time propagation of chaos for the hopping
scheme\footnote{The same condition holds for the freezing scheme. An analogous
condition appears for $\delta$-CBO once the estimate of Theorem \ref{Delta:Thrm}
is rederived from the variation-of-constants formula rather than from
\eqref{RestrictedDelta Bound1}, which discards the contractive factor
$e^{-s\lambda(t-\tau)}$.} therefore requires a contractive property of the
consensus point mapping.

Several recent works have made progress on uniform-in-time propagation of chaos for CBO. The authors of \cite{Huang2025Uniform-in-timeOptimization} establish pointwise uniform-in-time propagation of chaos in Wasserstein-2 distance at rate $\mathcal{O}(N^{-1/2})$ for the rescaled CBO scheme
 \begin{align*}
     dX^i_t = -\lambda (X^i_t - \kappa X^\alpha(\rho^N_t))dt + \sigma \Bigrou{\delta I_d + D\bigrou{X^i_t- \kappa X^\alpha(\rho^N_t)}} dB^i_t, \quad X^i_0 \sim \rho_0 \in \mathcal{P}_{16}(\mathbb{R}^d),
\end{align*}
with $\kappa \in (0, 1)$, $\delta \geq 0$ and the anisotropic diffusion matrix $D(x) := \mathrm{diag}\bigrou{|x_1|, \dots, |x_d|}$ for $x \in \mathbb{R}^d$. They show that for sufficiently small $\kappa$ and sufficiently large $\lambda$, there exists a constant $C>0$ independent of $N$ and $t$ such that
\begin{align*}
  \sup_{t \geq 0} \, \sup_{i \leq N} \, \mathbb{E} \biggeck{\N{X^i_t - \overbar{X}^i_t}^2} \leq C N^{-1}.
\end{align*}
The approach used is very similar to that presented in Section \ref{Section:Unrestricted:Delta}. In particular, the authors obtain the result by a Grönwall argument:
\begin{align}
    \mathbb{E} \biggeck{\N{X^i_t - \overbar{X}^i_t}^2} \leq \mathbb{E} \biggeck{\N{X^i_0 - \overbar{X}^i_0}^2} e^{- \vartheta t} + \frac{\Theta}{N}, \label{HuangBound}
\end{align}
where
\begin{align*}
    \vartheta := 2 \lambda - 2\sigma^2 - \kappa\bigrou{\lambda + L \lambda + 2\sigma^2 \kappa L}
    \quad \text{and} \quad
    \Theta := 4 \bigrou{\lambda \kappa + 2\sigma^2\kappa^2} \bigrou{K_{2}\, K_{4}}^{1/2}
\end{align*}
are constants independent of $N$ and $t$. Here, $L>0$ denotes a stability constant of the consensus point that depends on $\alpha$ and $f$, and $K_{2}, K_{4} > 0$ arise from uniform-in-time bounds on the second and fourth moments of the particle systems.
The rescaling parameter $\kappa$ allows one to choose $\kappa$ sufficiently small and $\lambda$ sufficiently large such that $\vartheta >0$, and therefore the exponential part of \eqref{HuangBound} vanishes as $t \to \infty$. Since $X^i_0 = \overbar{X}^i_0$ almost surely, the first term on the right-hand side of \eqref{HuangBound} is in fact zero.
However, this leaves open the question of uniform-in-time propagation of chaos for the classical CBO dynamics without a rescaling parameter.
A result addressing the more general setting for the classical CBO scheme
\begin{align*}
     dX^i_t = -\lambda (X^i_t - X^\alpha(\rho^N_t))dt + \sigma \N{X^i_t- X^\alpha(\rho^N_t)} dB^i_t, \quad X^i_0 \sim \rho_0, \quad i = 1, ..., N,
 \end{align*}
 was recently proposed in the preprint \cite{Gerber2026Uniform-in-timeOptimization}, where the authors establish
 \begin{align*}
    \forall t \geq 0, \quad \forall N \in \mathbb{N}, \quad \mathbb{E}\biggeck{\N{X^i_t - \overbar{X}^i_t}^2} \leq CN^{-1}.
\end{align*}
This result is of particular interest in the context of this thesis, as the proof strategy follows a Grönwall argument similar to that of \cite[Theorem 2.6]{Gerber2025Mean-fieldSampling}, adapted with several new tools that go beyond the scope of this work. In particular, the result requires the objective function $f$ to be globally Lipschitz continuous and bounded from above and below by $\overline{f}$ and $\underline{f}$, respectively, the initial distribution to have finite moments of all orders, and the noise strength $\sigma$ to lie below an explicit threshold.
\begin{align*}
    \mathbb{E} \biggeck{\frac{1}{N} \sum^N_{i = 1} \N{X^i_t - \overbar{X}^i_t}^2} \leq \Biggrou{\mathbb{E} \biggeck{\frac{1}{N} \sum^N_{i = 1} \N{X^i_0 - \overbar{X}^i_0}^2} + \frac{2c_2}{N}}e^{2c_1} = \frac{2c_2 e^{2c_1}}{N},  
\end{align*}
where $c_1, c_2>0$ are constants independent of $N$ and $t$.\newline
These recent developments show that uniform-in-time propagation of chaos is in fact achievable for classical CBO dynamics under suitable structural assumptions. A natural next step would therefore be to investigate whether similar uniform-in-time estimates can also be established for the $\delta$-CBO, CF and CH schemes considered in this thesis. In particular, it would be interesting to understand whether the piecewise and discrete-time structure of these schemes allows for analogous contractive mechanisms.
In the remainder of this section we argue that the answer is more subtle than it might appear and a naive adaption of the arguments in \cite{Gerber2026Uniform-in-timeOptimization} does in fact not work for the schemes treated in this thesis.

A central mechanism behind the uniform-in-time estimate of \cite{Gerber2026Uniform-in-timeOptimization} is the exponential decay of the centered moments of the mean-field law. This decay is a consequence of the multiplicative noise of the classical CBO dynamics, which vanishes as the particles reach consensus. For all three schemes considered here, no such decay occurs.
 
\begin{lemma}
    Let $(\overbar{X
    }_t)_{t \geq 0 }$ be the $\delta$-CBO mean-field process \eqref{Delta:MFSDE} with initial law $\orho_0 \in \mathcal{P}_2(\mathbb{R}^d)$. Then, for every $t \geq 0$,
    \begin{align}
        \mathbb{E} \biggeck{ \N{\overbar{X}_t - \mathbb{E}\Bigeck{\overbar{X}_t}}^2 } = e^{-2s\lambda t} \, \mathbb{E} \biggeck{ \N{\overbar{X}_0 - \mathbb{E}\Bigeck{\overbar{X}_0}}^2 } + \frac{d \delta^2}{2\lambda} \bigrou{1 - e^{-2 s \lambda t}} \xrightarrow[t \to \infty]{} \frac{d\delta^2}{2\lambda} \, > 0 .
        \label{Outlook:OUIdentity}
    \end{align}
    The same identity holds for the consensus freezing scheme \eqref{CF:MFscheme}, and for the consensus hopping scheme one has $\mathbb{E} \biggeck{ \N{\overbar{X}_k - \mathbb{E}\Bigeck{\overbar{X}_k}}^2 } = d\sigma^2$ for every $k \geq 1$.
    \label{Outlook:LemmaMoment}
\end{lemma}
 
\begin{proof}
    Write $m_t := \mathbb{E}\bigeck{\overbar X_t}$. Since $\orho_t$ is deterministic, so is $t \mapsto X^\alpha(\orho_t)$, and it is locally bounded by Corollary \ref{GerberPropA4Corollary}. Taking expectations in \eqref{Delta:MFSDE} therefore yields $\dot m_t = -s\lambda\bigrou{m_t - X^\alpha(\orho_t)}$, and subtracting this from \eqref{Delta:MFSDE} shows that $Y_t := \overbar X_t - m_t$ satisfies
    \begin{align*}
        dY_t = - s \lambda \, Y_t \, dt + \sqrt{s}\delta \, dB_t .
    \end{align*}
    The consensus point cancels exactly, because it enters $\overbar X_t$ and $m_t$ in the same way. Hence $Y$ is an Ornstein--Uhlenbeck process, and
    \begin{align*}
        Y_t = e^{-s\lambda t} Y_0 + \sqrt{s} \delta \int^t_0 e^{-s\lambda(t-\tau)} dB_\tau .
    \end{align*}
    Since $Y_0$ is independent of $(B_t)_{t \geq 0}$ and centered, and the stochastic integral is a centered martingale, the cross term vanishes, and the It\^o isometry gives
   \begin{align}
       \mathbb{E}\bigeck{\N{Y_t}^2} = e^{-2s\lambda t}\mathbb{E}\bigeck{\N{Y_0}^2} + s\delta^2 d \int^t_0 e^{-2s\lambda u} du,
   \end{align}
which is \eqref{Outlook:OUIdentity}. For the freezing scheme the frozen consensus point is likewise state-independent, so the same cancellation occurs on each interval and the resulting process is continuous across the $t_j$. For the hopping scheme the claim is immediate, since $\orho_k = \mathcal{N}\bigrou{X^\alpha(\orho_{k-1}), \sigma^2 I_d}$ for every $k \geq 1$ by \eqref{CH:MFScheme}.
\end{proof}
The identity \eqref{Outlook:OUIdentity} shows that the limiting spread $d\delta^2/(2\lambda)$ is independent of $f$ and of $\alpha$. Therefore, with additive noise there is no particle collapse, and
\begin{align}
    \int^\infty_0 \bigrou{\mathbb{E}_{\overline\rho_t}\bigeck{\N{x - \mathbb{E}_{\overline\rho_t}[x]}^2}}^{1/2} dt = \infty,
\end{align}
whenever $\delta > 0$. A naive adaption of the arguments of \cite{Gerber2026Uniform-in-timeOptimization} consequently does not work.
 
\subsubsection{The Coupling Error is a Translation}
\label{Sub:Outlook:Translation}
 
All three schemes considered in this thesis are driven by additive noise. The diffusion coefficient does not depend on the state of the system. Under the synchronous coupling employed throughout, the error between a particle and its mean-field counterpart is the same for every index. This is not the case for the classical CBO dynamics \eqref{Pre:CBOSDE}.
 
Consider the $\delta$-CBO scheme \eqref{Delta:Scheme} coupled with \eqref{Delta:MFScheme}, the consensus freezing scheme \eqref{CF:CFscheme} coupled with \eqref{CF:MFIPS}, or the consensus hopping scheme \eqref{CH:Scheme} coupled with \eqref{CH:MFScheme}, in each case with $\overbar X^i_0 = X^i_0$ almost surely. Then, as seen in the proofs of Theorems \ref{Delta:Thrm}, \ref{MFF: Thrm} and \ref{MFH: Thrm} there exists a process $\bigrou{F_t}_{t \geq 0}$, respectively a sequence $\bigrou{F_k}_{k \in \mathbb{N}}$, such that almost surely
\begin{align*}
    X^i_t - \overbar X^i_t = F_t \qquad \text{for all } i \leq N.
\end{align*}
    Consequently $\rho^N_t$ is the translation of $\orho^N_t$ by $F_t$, and for every $p \geq 1$
\begin{align}
    W_p\bigrou{\rho^N_t, \orho^N_t} = \N{F_t}.
    \label{Outlook:TranslateIdentity}
\end{align}
Therefore only the behavior of $X^\alpha$ along translations of a fixed measure is relevant for our analysis. This motivates the following assumption, which could serve as a starting point for future work. For $\nu \in \mathcal{P}_p(\mathbb{R}^d)$ and $z \in \mathbb{R}^d$ we write $\nu_z$ for the push-forward of $\nu$ under $x \mapsto x + z$.
 
\begin{assumption}
    There exist $c \in (0, 1]$ and a class $\mathcal{C} \subseteq \mathcal{P}_p(\mathbb{R}^d)$ containing the empirical measures $\orho^N_t$ for all $t \geq 0$ and all $N \in \mathbb{N}$, such that
    \begin{align*}
        \N{X^\alpha(\nu_z) - X^\alpha(\nu)} \leq (1 - c) \N{z} \qquad \text{for all } \nu \in \mathcal{C}, \ z \in \mathbb{R}^d .
        \label{Outlook:HypH}
    \end{align*}
    \label{Outlook:AssumptionH}
\end{assumption}
Whether Assumption \ref{Outlook:AssumptionH} holds for a sufficiently rich class $\mathcal{C}$, for instance for large $\alpha$ and measures concentrated near $x^*$, is left open for future work.

\chapter{Conclusion}
In this thesis, we have investigated propagation of chaos for three interacting particle systems arising in the context of consensus-based optimization: the consensus hopping scheme, the consensus freezing scheme, and the $\delta$-CBO scheme. For all three schemes, we established quantitative propagation of chaos estimates uniformly over finite time horizons. In the restricted setting, these estimates are obtained at the rate $\mathcal{O}(N^{-1/2})$, while in the unrestricted setting the same rate is recovered under sufficiently high moment assumptions and otherwise becomes slower.

We achieved these results following two different strategies. In Chapter \ref{chap:PropGoodSets}, propagation of chaos was established conditionally on the high-probability set $\Omega_M$, on which the empirical measures of both the interacting and mean-field particle systems have uniformly bounded $p$th moments. Although the probability of $\Omega_M$ was shown to approach one as $M$ goes to infinity, the restriction to such good trajectories left open the question of whether unconditional estimates are achievable. Chapter \ref{chap:PropChaUnrestricted} showed that unconditional estimates are indeed possible under stronger moment assumptions, following the ideas of \cite{Gerber2025Mean-fieldSampling} by splitting the expected consensus point distance across $\Omega_{N,t}$ and its complement, exploiting the independence of the mean-field particles via a law of large numbers argument to handle the ``bad'' set, and replacing the two-sided Wasserstein stability estimate of Lemma \ref{CarrL2} with the one-sided refinement of Lemma \ref{GerberStability2}, which is applied on the set $\Omega_{N,t}$.

However, this improvement comes at a price. In order to achieve convergence at rate $\mathcal{O}(N^{-1/2})$, the unrestricted approach requires bounded moments of order $q \geq p^2 + p$, whereas the restricted approach requires only moments of order $p$. This distinction already has consequences in widely studied cases: for propagation of chaos in Wasserstein-$2$
distance under bounded fourth moments, the restricted estimate achieves the rate $\mathcal{O}(N^{-1/2})$, while the unrestricted estimate degrades to $\mathcal{O}(N^{-1/4})$. 

As an application of the propagation of chaos results derived in this thesis, in Section \ref{section:globalconvergence} we demonstrated how the estimates obtained can be combined with the global convergence results of \cite{Fornasier2026FromConvergence} to derive quantitative convergence guarantees for the empirical mean of the interacting particle systems toward the mean of the invariant measure $\boldsymbol\rho^\alpha_\infty$. This yields a high-probability bound showing that the empirical mean of the particle system approximates the mean of the invariant measure up to a statistical error floor of order $\mathcal{O}(N^{-1})$, plus a deterministic term vanishing exponentially in time. An analogous result for the hopping scheme follows via the connection between the freezing and hopping schemes established in \cite[Theorem 4.1]{Fornasier2026FromConvergence}. The central limitation of both results is that the propagation of chaos constant $C_{\mathrm{PoC}}(T)$ grows exponentially in $T$.

Therefore, a natural direction for future work is to study whether propagation of chaos for the interacting particle systems can be established uniformly over infinite time horizons, possibly under suitable confinement assumptions on the objective function.


\renewcommand*{\thesection}{\Alph{section}}
\setcounter{section}{1}
\chapter{Supporting Material}
\label{cha:app1}
\section{Proof of Lemma \ref{Pre:LemmaNormalVectors}}
\label{sec:LemaNormalVectorsProof}
To prove Lemma \ref{Pre:LemmaNormalVectors}, we require some auxiliary results from the theory of high-dimensional probability for independent and identically distributed Gaussian random vectors.
\begin{lemma}[{{\cite[Proposition 2.1.2]{Vershynin2026High-DimensionalProbability}}}]
    Let $X \sim \mathcal{N}(0, 1)$. Then for every $t>0$, we have
    \begin{align*}
        \mathbb{P}(X \geq t) \leq \frac{1}{t} \cdot \frac{1}{\sqrt{2\pi}} e^{- \frac{t^2}{2}}. 
    \end{align*}
\label{Pre:SimpleGaussBound}
\end{lemma}
\begin{proposition}
    Let $N\geq 2$ and $X_1, ...,X_N$ be independent, identically distributed Gaussian random variables such that $X_1\sim \mathcal{N}(0, 1)$. Then there exists a constant $C_p >0$ that depends only on $p$ such that
    \begin{align*}
        \mathbb{E}\biggeck{\maxu{i}{N} \, |X_i|^p} \leq C_p (\log N)^{\frac{p}{2}}.
    \end{align*}
    \label{Pre:LemmaNormaVariables}
\end{proposition}
\begin{proof}
We can compute straightforwardly:
\begin{align*}
    \mathbb{E}\biggeck{\maxu{i}{N} \, |X_i|^p} &= \int^\infty_0 \mathbb{P} 
    (\maxu{i}{N} \, |X_i|^p \geq x) \, dx \\
    (x = t^p) \quad &= \int^\infty_0 \mathbb{P} (\maxu{i}{N} \, |X_i| \geq t) 
    \, pt^{p-1} \, dt.
\end{align*}
We now choose to split the integral at $t_0 = \sqrt{2 \log N}$:
\begin{align}
    \int^\infty_0  \mathbb{P} (\maxu{i}{N} \, |X_i| \geq t)  pt^{p-1} \, dt 
    &\leq \int_0^{t_0}  pt^{p-1} \, dt + \int_{t_0}^\infty  
    \mathbb{P} (\maxu{i}{N} \, |X_i| \geq t)  pt^{p-1} \, dt, 
    \label{Pre:iidBound2}
\end{align}
where we used $\mathbb{P} (\maxu{i}{N} \, |X_i| \geq t) \leq 1$. 
We evaluate both integrals separately, starting with the first one:
\begin{align}
    \int_0^{t_0}  pt^{p-1} \, dt = t_0^{p} 
    \overset{(t_0 = \sqrt{2 \log N})}{=} \bigrou{2 \log N}^{\frac{p}{2}}.
    \label{Pre:iidFirstIntegralB1}
\end{align}
The analysis of the second integral is somewhat more difficult. 
We begin by applying a union bound in conjunction with the i.i.d. assumption, 
which yields
\begin{align}    
     \int^\infty_{t_0} \mathbb{P} (\maxu{i}{N} \, |X_i| \geq t)  
     pt^{p-1} \, dt \leq N \int^\infty_{t_0} 
     \mathbb{P} ( |X_1| \geq t) \, pt^{p-1} \, dt.
     \label{Pre:iidSeconintegralNBound}
\end{align}
We make use of Lemma \ref{Pre:SimpleGaussBound} to obtain 
\begin{align}
      N\int^\infty_{t_0} \mathbb{P} ( |X_1| \geq t) \, pt^{p-1} \, dt
   &\leq \frac{2N}{\sqrt{2\pi}}\int^\infty_{t_0} \frac{1}{t} \cdot 
   e^{- \frac{t^2}{2}} \cdot pt^{p-1} \, dt \notag \\
   &= \frac{2Np}{\sqrt{2\pi}}\int^\infty_{t_0} 
   e^{- \frac{t^2}{2}} t^{p-2} \, dt.
   \label{Pre:iidSeconintegralB1}
\end{align}
We further estimate the integral by inserting $e^{\frac{t_0^2}{2}}
e^{-\frac{t_0^2}{2}}$:
\begin{align}
    \int^\infty_{t_0} e^{- \frac{t^2}{2}} t^{p-2} \, dt 
    &= \int^\infty_{t_0} t^{p-2}  e^{- \frac{t^2 - t_0^2}{2}} 
    e^{-\frac{t^2_0}{2}}\, dt \notag\\
   \Bigrou{\frac{t_0^2}{2} = \log N} \quad 
   &= \frac{1}{N}\int^\infty_{t_0} t^{p-2}  
   e^{- \frac{t^2 - t_0^2}{2}}\, dt \notag\\
\Bigrou{t^2-t_0^2 \geq 2t_0(t-t_0)} \quad   
   &\leq \frac{1}{N}\int^\infty_{t_0} t^{p-2}  
   e^{- t_0(t-t_0)} \, dt \notag\\
\Bigrou{u = t - t_0} \quad   
   &= \frac{1}{N} \int^\infty_{0} (u + t_0)^{p-2}  
   e^{- t_0u} \, du \notag\\
\Bigrou{(u+t_0)^{p-2} \leq 2^{p-2} (u^{p-2} + t_0^{p-2})} \quad 
   &\leq \frac{2^{p-2}}{N}\Biggrou{ 
   \underbrace{\int^\infty_{0} u^{p-2}  e^{-t_0u} \, du}
   _{ = \frac{\Gamma(p-1)}{t_0^{p-1}}} 
   + t_0^{p-2} \underbrace{\int^\infty_{0} e^{-t_0u} \, du}
   _{= \frac{1}{t_0}}} \notag \\
&= \frac{2^{p-2}}{N} \biggrou{\frac{\Gamma(p-1)}{t_0^{p-1}} 
   + t_0^{p-3}} \notag\\
\Bigrou{t_0 = \sqrt{2 \log N}} \quad &= \frac{2^{p-2}}{N} \Biggrou{\frac{\Gamma(p-1)}{(2 \log N)^{\frac{p-1}{2}}} 
   + (2 \log N)^{\frac{p-3}{2}}}.
\label{Pre:iidSeconintegralB2}
\end{align}
We now justify that \eqref{Pre:iidSeconintegralB2} is bounded by $K_p \frac{(2\log N)^{\frac{p-3}{2}}}{N}$
for a constant $K_p$ depending only on $p$. Dividing \eqref{Pre:iidSeconintegralB2} by $\frac{(2\log N)^{\frac{p-3}{2}}}{N}$, this reduces to
\begin{align*}
     2^{p-2}\left(\frac{\Gamma(p-1)}{(2\log N)^{p-2}}
    + 1\right) \leq K_p .
\end{align*}
For $N \geq 2$ and $p \geq 2$, we have $(2\log N)^{p-2} \geq (2\log 2)^{p-2}$,
so we can choose explicitly
\begin{align*}
    K_p := 2^{p-2}\biggrou{\frac{\Gamma(p-1)}{(2\log 2)^{p-2}}+1},
\end{align*}
which is a positive constant depending only on $p$.
\newline
Inserting \eqref{Pre:iidSeconintegralB2} into \eqref{Pre:iidSeconintegralB1} 
gives us
\begin{align*}
    N\int^\infty_{t_0} \mathbb{P} ( |X_1| \geq t) \, pt^{p-1} \, dt
   \leq \frac{2Np}{\sqrt{2\pi}} \cdot \frac{K_p}{N} 
   \bigrou{2 \log N}^{\frac{p-3}{2}} 
   = \frac{2pK_p}{\sqrt{2\pi}}\bigrou{2 \log N}^{\frac{p-3}{2}}.
\end{align*}
This once again combined with \eqref{Pre:iidFirstIntegralB1} yields 
our final estimate
\begin{align*}
    \mathbb{E}\biggeck{\maxu{i}{N} \, |X_i|^p} 
    &\leq \bigrou{2 \log N}^{\frac{p}{2}} 
    + \frac{2pK_p}{\sqrt{2\pi}}\bigrou{2 \log N}^{\frac{p-3}{2}} 
    \leq C_p (\log N)^{\frac{p}{2}}, 
\end{align*}
for a constant $C_p > 0$ depending only on $p$. 
This concludes the proof.
\end{proof}
\vspace{.5 cm}
With these tools, the proof of Lemma \ref{Pre:LemmaNormalVectors} becomes easy.
\vspace{.2 cm}
\begin{proof}[Proof of Lemma \ref{Pre:LemmaNormalVectors}]
     For simplicity, we first assume that $X_1, ..., X_N$ are i.i.d. standard normal distributed random vectors. Therefore, we can compute straightforwardly
    \begin{align}
       \mathbb{E} \biggeck{\maxu{i}{N} \N{X_i}^p} = \mathbb{E} \Biggeck{\maxu{i}{N} \biggrou{\sum ^d_{l=1} |X_{i, l}|^2 }^{\frac{p}{2}}},
        \label{Pre:Gaussian ineq.}
    \end{align}
where we note that for every $i\leq N$ and $l \leq d$, $X_{i, l} \sim N(0, 1)$. We use standard inequalities and the i.i.d. property to compute further:
\begin{align}
    \mathbb{E} \Biggeck{\maxu{i}{N} \biggrou{\sum ^d_{l=1} |X_{i, l}|^2 }^{\frac{p}{2}}} \leq d^{\frac{p}{2}}\,\mathbb{E} \Biggeck{\maxu{i}{N} \,   \maxu{l}{d} \,  |X_{i, l}|^p} &\leq d^{\frac{p}{2}}\,\mathbb{E} \Biggeck{\maxu{m}{2Nd} \, |\widetilde{X}_m|^p} \notag \\ &\leq C_p \Bigrou{d\log (2Nd)}^{\frac{p}{2}},
\end{align}
where $(\widetilde{X}_m)_{m \leq 2Nd}$ denotes the family $(X_{i,l})_{i \leq N, l \leq d}$ extended by $Nd$ further independent standard normal random variables. Adding random variables can only increase the maximum, and since $2Nd \geq 2$, we may apply Proposition \ref{Pre:LemmaNormaVariables}.
From this we can directly deduce the result for $Z_1, ..., Z_N \sim N(0, \sigma^2 I_d)$:
\begin{align}
    \mathbb{E} \biggeck{\maxu{i}{N} \N{Z_i}^p} \leq \sigma^pC_p \Bigrou{d\log (2Nd)}^{\frac{p}{2}}.
\end{align}
\end{proof}
\section{Proof of Lemma \ref{CH:ConsBound}}
\label{Ch:ConsBoundProof}
 A crucial tool for this will be the following theorem, which is due to Rosenthal \cite{Rosenthal1970OnVariables}. It provides an opportunity to estimate the moments of a sum of random variables in terms of the sum of their moments. 

\begin{theorem}[{{\cite[Theorem 3]{Rosenthal1970OnVariables}}}]
    Let $p \in [2, \infty)$. Then, there exists a constant $C_{p, d}>0$, depending only on $p$ and $d$ such that for any $N\in \mathbb{N}$, for any collection of independent, zero mean random variables $X_1, ..., X_N$ with finite $p$th moments, we have
    \begin{align*}
        \mathbb{E} \Biggeck{\N{\sum^N_{i=1} X_i}^p}^{\frac{1}{p}} \leq C_{p, d} \max \Bigg\{ \biggrou{\sum^N_{i=1} \mathbb{E} \Bigeck{\N{X_i}^p}}^{\frac{1}{p}}, \biggrou{\sum^N_{i=1} \mathbb{E} \Bigeck{\N{X_i}^2}}^{\frac{1}{2}}  \Bigg \}.
    \end{align*}
    \label{Rosenthal}
\end{theorem}
\begin{corollary}
    Let $p \in [2, \infty)$. Then, there exists a constant $C_{p, d} > 0$, depending only on $p$ and $d$ such that for any $N \in \mathbb{N}$, for any collection of independent, identically distributed, zero mean random variables $X_1, ..., X_N$ with finite $p$th moments, we have
    \begin{align*}
        \mathbb{E} \Biggeck{\N{\frac{1}{N} \sum^N_{i= 1} X_i}^p} \leq C_{p, d} \max \Bigg\{ \frac{\mathbb{E}\bigeck{\N{X_1}^p}}{N^{p-1}},  \frac{\mathbb{E}\bigeck{\N{X_1}^2}^{\frac{p}{2}}}{N^{\frac{p}{2}}}  \Bigg \}.
    \end{align*}
    \label{RosenthalCorr}
    \end{corollary}
\begin{proof}
    By Theorem \ref{Rosenthal}, we immediately get
    \begin{align*}
          \mathbb{E} \Biggeck{\N{ \frac{1}{N}\sum^N_{i=1} X_i}^p}^{\frac{1}{p}} =  \mathbb{E} \Biggeck{\N{\sum^N_{i=1} \frac{X_i}{N}}^p}^{\frac{1}{p}} &\leq  K_{p, d} \max \Bigg\{ \Biggrou{\sum^N_{i=1} \mathbb{E} \biggeck{\N{\frac{X_i}{N}}^p}}^{\frac{1}{p}}, \Biggrou{\sum^N_{i=1} \mathbb{E} \biggeck{\N{\frac{X_i}{N}}^2}}^{\frac{1}{2}}  \Bigg \} \\
         \text{(identically distributed)} & =  K_{p, d} \max \Bigg\{ \Biggrou{N \mathbb{E} \biggeck{\N{\frac{X_1}{N}}^p}}^{\frac{1}{p}}, \Biggrou{N \mathbb{E} \biggeck{\N{\frac{X_1}{N}}^2}}^{\frac{1}{2}}  \Bigg \}\\
         &= K_{p, d} \max \Bigg\{ \Biggrou{ \frac{\mathbb{E}\bigeck{\N{X_1}^p}}{N^{p-1}}}^{\frac{1}{p}}, \Biggrou{ \frac{\mathbb{E}\bigeck{\N{X_1}^2}}{N}}^{\frac{1}{2}}  \Bigg \}, 
    \end{align*}
where $K_{p, d}>0$ is the constant arising from Theorem \ref{Rosenthal}. Lastly, raising both sides to the power of $p$ and defining $C_{p, d} := K_{p, d}^p$ gives us
\begin{align*}
    \mathbb{E} \Biggeck{\N{ \frac{1}{N}\sum^N_{i=1} X_i}^p} \leq C_{p, d} \max \Bigg\{ \frac{\mathbb{E}\bigeck{\N{X_1}^p}}{N^{p-1}},  \frac{\mathbb{E}\bigeck{\N{X_1}^2}^{\frac{p}{2}}}{N^{\frac{p}{2}}}  \Bigg \}.
\end{align*}
\end{proof}

\begin{proof}[Proof of Lemma \ref{CH:ConsBound}]
       We fix $N \in \mathbb{N}$ and begin to estimate
\begin{align}
     &\mathbb{E} \Biggeck{ \N{X^\alpha(\mnu) - X^\alpha (\nu)}^p \mathbbm{1}_{\Omega_M}}\\ = \, & \mathbb{E} \Biggeck{ \N{ \frac{1}{N}\sum^N_{i=1} \overbar{X}^i\frac{\oa(\overbar{X}^i)}{\Gnorm{\mnu}}  -  \int_{\mathbb{R}^d} x \frac{\oa(x)}{\Gnorm{\nu}} d\nu(x)}^p \mathbbm{1}_{\Omega_M}} \notag \\
     \leq \, & 2^{p-1}\mathbb{E} \Biggeck{ \N{ \frac{1}{N}\sum^N_{i=1} \overbar{X}^i\frac{\oa(\overbar{X}^i)}{\Gnorm{\mnu}}  -  \int_{\mathbb{R}^d} x \frac{\oa(x)}{\Gnorm{\mnu}} d\nu(x)}^p \mathbbm{1}_{\Omega_M}} \notag \\
     \quad \, & + 2^{p-1}\mathbb{E} \Biggeck{\N{ \int_{\mathbb{R}^d} x \frac{\oa(x)}{\Gnorm{\mnu}} d\nu(x) -\int_{\mathbb{R}^d} x \frac{\oa(x)}{\Gnorm{\nu}} d\nu(x)}^p \mathbbm{1}_{\Omega_M}} \notag \\
     = \, & \colon 2^{p-1} \mathbb{E}\Bigeck{\N{S_1}^p \mathbbm{1}_{\Omega_M}} +  2^{p-1}\mathbb{E}\Bigeck{\N{S_2}^p \mathbbm{1}_{\Omega_M}},
     \label{CB:S1S2}
\end{align}
where, in the second line, we added 
\begin{align*}
   0 = \int_{\mathbb{R}^d} x \frac{\oa(x)}{\Gnorm{\mnu}} d\nu(x) - \int_{\mathbb{R}^d} x \frac{\oa(x)}{\Gnorm{\mnu}} d\nu(x)
\end{align*}
and used $(a+b)^p \leq 2^{p-1}(a^p+b^p)$.\newline
Now, we estimate each term separately, beginning with the first one: 
\begin{align}
   \mathbb{E}\Bigeck{\N{S_1}^p \mathbbm{1}_{\Omega_M}} &= \mathbb{E} \Biggeck{ \N{  \frac{1}{N}\sum^N_{i=1} \overbar{X}^i\frac{\oa(\overbar{X}^i)}{\Gnorm{\mnu}}  -  \int_{\mathbb{R}^d} x \frac{\oa(x)}{\Gnorm{\mnu}} d\nu(x)}^p \mathbbm{1}_{\Omega_M}} \notag \\
   &=  \mathbb{E} \Biggeck{ \biggrou{\frac{1}{\Gnorm{\mnu}}}^p \N{\frac{1}{N} \sum^N_{i=1} \overbar{X}^i\oa(\overbar{X}^i)  -  \int_{\mathbb{R}^d} x \oa(x) d\nu(x)}^p \mathbbm{1}_{\Omega_M}}. \notag
\end{align}
Since $p \geq 2$, Jensen's inequality yields on $\Omega_M$
\begin{align*}
    \mathbb{E}_{\mnu}\bigeck{\N{x}^2} = \frac{1}{N}\sum_{i=1}^N \N{\overbar{X}^i}^2 \leq \biggrou{\frac{1}{N}\sum_{i=1}^N \N{\overbar{X}^i}^p}^{\frac{2}{p}} \leq M^{\frac{2}{p}},
\end{align*}
so Lemma \ref{CarrL1} is applicable to $\mnu$ with $M_2 = M^{2/p}$.
We now expand the fraction that we pulled out by $\frac{\oa(x^*)}{\oa(x^*)}$ in the second line and use Lemma \ref{CarrL1} to further estimate:
\begin{align}
    & \mathbb{E} \Biggeck{ \biggrou{\frac{1}{\Gnorm{\mnu}}}^p \N{\frac{1}{N} \sum^N_{i=1} \overbar{X}^i\oa(\overbar{X}^i)  -  \int_{\mathbb{R}^d} x \oa(x) d\nu(x)}^p \mathbbm{1}_{\Omega_M}} \notag \\ 
    = \, & \mathbb{E} \Biggeck{ \biggrou{\frac{\oa(x^*)}{\Gnorm{\mnu}}\frac{1}{\oa(x^*)}}^p \N{ \frac{1}{N} \sum^N_{i=1} \overbar{X}^i\oa(\overbar{X}^i)  -  \int_{\mathbb{R}^d} x \oa(x) d\nu(x)}^p \mathbbm{1}_{\Omega_M}} \notag \\
  \text{(Lemma \ref{CarrL1})}  \quad  \leq \, & \mathbb{E} \Biggeck{ \biggrou{\frac{e^{\alpha C_2(1+M^{2/p})}}{ \oa(x^*)}}^p \N{ \frac{1}{N} \sum^N_{i=1} \overbar{X}^i\oa(\overbar{X}^i)  -  \int_{\mathbb{R}^d} x \oa(x) d\nu(x)}^p \mathbbm{1}_{\Omega_M} } \notag \\
      \leq \, & e^{ p \alpha (f(x^*)+C_2(1+M^{2/p}))} \mathbb{E}  \Biggeck{\N{ \frac{1}{N} \sum^N_{i=1} \overbar{X}^i\oa(\overbar{X}^i)  -  \int_{\mathbb{R}^d} x \oa(x) d\nu(x)}^p \mathbbm{1}_{\Omega_M}},
      \label{CB:S1Bd}
\end{align}
where we used the second-moment bound on $\Omega_M$ from above.
We are left with bounding the non constant part. To this end, for all $i \leq N$ we introduce the random variables
\begin{align*}
    Y_i := \overbar{X}^i\oa(\overbar{X}^i)  -  \int_{\mathbb{R}^d} x \oa(x) d\nu(x).
\end{align*}
Note that for all $i$, we have
\begin{align*}
    \mathbb{E} \bigeck{Y_i} = \underbrace{\mathbb{E} \Bigeck{\overbar{X}^i\oa(\overbar{X}^i)}}_{= \int_{\mathbb{R}^d} x \oa(x) d\nu(x)}  -  \int_{\mathbb{R}^d} x \oa(x) d\nu(x) = 0.
\end{align*}
This means that $\frac{1}{N}\sum^N_{i = 1} Y_i$ is the sample mean of independent, identically distributed, zero mean random variables with finite $p$th moments and we therefore can apply Corollary \ref{RosenthalCorr} to get the estimate
\begin{align}
    \mathbb{E} \Biggeck{\N{\frac{1}{N} \sum^N_{i= 1} Y_i}^p} \leq C_{p, d} \max \Bigg\{ \frac{\mathbb{E}\bigeck{\N{Y_1}^p}}{N^{p-1}},  \frac{\mathbb{E}\bigeck{\N{Y_1}^2}^{\frac{p}{2}}}{N^{\frac{p}{2}}}  \Bigg \}, \label{CB:RosenthalEstimate} 
\end{align}
for a constant $C_{p, d} > 0$ that depends only on $p$ and $d$.
To further estimate, we use the definition of $Y_1$ to bound its $p$th moment
\begin{align*}
   \mathbb{E} \Bigeck{\N{Y_1}^p} &= \mathbb{E} \biggeck{\N{\overbar{X}^1 \oa(\overbar{X}^1) - \int_{\mathbb{R}^d} x \oa(x) d\nu(x)}^p}\\ &\leq 2^{p-1} \mathbb{E} \biggeck{\N{\overbar{X}^1}^p \oa(\overbar{X}^1)^p + \N{\int_{\mathbb{R}^d} x \oa(x) d\nu(x)}^p} \\
    &\leq 2^{p-1} \Biggrou{\oa(x^*)^p \mathbb{E}\Bigeck{\N{\overbar{X}^1}^p} + \oa(x^*)^p \mathbb{E}\Bigeck{\N{\overbar{X}^1}^p}}.\\
    &=(2 \oa(x^*))^p \mathbb{E}\Bigeck{\N{\overbar{X}^1}^p}\\
    &\leq (2 \oa(x^*))^p M_p.
\end{align*}
Analogously, we get for the second moment of $Y_1$
\begin{align*}
    \mathbb{E} \Bigeck{\N{Y_1}^2} \leq 4 \oa(x^*)^2 M_2
\end{align*}
By Jensen's inequality, we have that 
\begin{align*}
    (M_2)^{\frac{p}{2}} \leq M_p. 
\end{align*}
Combining the latter with our Rosenthal estimate \eqref{CB:RosenthalEstimate} using the monotonicity of expectation yields
\begin{align*}
     \mathbb{E} \Biggeck{\N{\frac{1}{N} \sum^N_{i= 1} Y_i}^p} &\leq C_{p, d} \max \Bigg \{\frac{\bigrou{2\oa(x^*)}^p M_p}{N^{p-1}}, \frac{\bigrou{2 \oa(x^*)}^p (M_2)^{\frac{p}{2}}}{N^{\frac{p}{2}}} \Bigg \} \\ &\leq C_{p, d} M_p\bigrou{2 \oa(x^*)}^p \max \Bigg \{ \frac{1}{N^{p-1}} , \frac{1}{N^{\frac{p}{2}}}\Bigg \}.
\end{align*}
We insert the latter into \eqref{CB:S1Bd} to get our final estimate for the first part:
\begin{align}
    \mathbb{E} \bigeck{\N{S_1}^p \mathbbm{1}_{\Omega_M}} \leq e^{ p \alpha (f(x^*)+C_2(1+M^{2/p}))} C_{p, d} M_p\bigrou{2\oa(x^*)}^p \max \Bigg \{ \frac{1}{N^{p-1}} , \frac{1}{N^{\frac{p}{2}}}\Bigg \}. 
    \label{CB:S1FinalBound}
\end{align}
We now proceed to treat the second part of (\ref{CB:S1S2}). We start by expanding fractions to write the difference in terms of the consensus point of $\nu$ and then pulling it out.
\begin{align}
    \mathbb{E}\bigeck{\N{S_2}^p \mathbbm{1}_{\Omega_M} } &= \mathbb{E} \Biggeck{\N{\int_{\mathbb{R}^d} x \frac{\oa(x)}{\Gnorm{\mnu}} d\nu(x) - \int_{\mathbb{R}^d} x \frac{\oa(x)}{\Gnorm{\nu}} d\nu(x)}^p \mathbbm{1}_{\Omega_M} } \notag \\
  \Biggrou{\text{expand by } \frac{\Gnorm{\nu}}{\Gnorm{\nu}}} \quad  &= \mathbb{E} \Biggeck{\N{\frac{\Gnorm{\nu}}{\Gnorm{\mnu}} X^\alpha(\nu) - X^\alpha(\nu)}^p \mathbbm{1}_{\Omega_M}} \notag \\
  \bigrou{\text{pull out } \cons(\nu)} \quad   &= \mathbb{E} \Biggeck{ \frac{\N{\cons(\nu)}^p}{(\Gnorm{\mnu})^p} \Big|\Gnorm{\nu} - \Gnorm{\mnu} \Big|^p \mathbbm{1}_{\Omega_M}}.
    \label{CB:S2FirstBound}
\end{align}
We do a small auxiliary computation. With slight abuse of the notation and recalling that $\eta_{\nu} := \frac{\oa \nu}{\Gnorm{\nu}}$ defines a probability measure on $\mathbb{R}^d$, we can write $\N{\cons(\nu)}^p$ as
\begin{align*}
    \N{\cons(\nu)}^p = \N{\int_{\mathbb{R}^d} \frac{x \oa(x)}{\Gnorm{\nu}} d\nu}^p = \N{\int_{\mathbb{R}^d} x \,  d\eta_{\nu}(x)}^p = \N{\mathbb{E}_{\eta_{\nu}} \bigeck{x}}^p. 
\end{align*}
We can estimate this further by first using Jensen's inequality and after that Lemma \ref{GerberPropA4}:
\begin{align}
    \N{\mathbb{E}_{\eta_{\nu}} \bigeck{x}}^p \leq \mathbb{E}_{\eta_{\nu}} \bigeck{\N{x}^p} \leq b_p \mathbb{E}_{\nu} \bigeck{\N{x}^p}. 
    \label{CB:etaBound}
\end{align}
We insert (\ref{CB:etaBound}) into \eqref{CB:S2FirstBound}, so we can further compute:
\begin{align}
    & \quad \, \, \mathbb{E} \Biggeck{ \frac{\N{\cons(\nu)}^p}{(\Gnorm{\mnu})^p} \Big|\Gnorm{\nu} - \Gnorm{\mnu} \Big|^p \mathbbm{1}_{\Omega_M} } \notag \\
 (\ref{CB:etaBound}) \quad   &\leq \mathbb{E} \Biggeck{ \frac{b_p \mathbb{E}_{\nu} \bigeck{\N{x}^p}}{(\Gnorm{\mnu})^p} \Big|\Gnorm{\nu} - \Gnorm{\mnu} \Big|^p \mathbbm{1}_{\Omega_M}} \notag \\
  (\mathbb{E}_{\nu} \bigeck{\N{x}^p} \leq M_p) \quad &\leq  \mathbb{E} \Biggeck{ \frac{b_p M_p}{(\Gnorm{\mnu})^p} \Big|\Gnorm{\nu} - \Gnorm{\mnu} \Big|^p \mathbbm{1}_{\Omega_M} }
    \notag \\
\Biggrou{\text{multiply } \frac{\oa(x^*)}{\oa(x^*)}} \, \,   &= \mathbb{E} \Biggeck{b_p M_p \biggrou{\frac{\oa(x^*)}{\Gnorm{\mnu}} \frac{1}{\oa(x^*)}}^p \Big|\Gnorm{\nu} - \Gnorm{\mnu} \Big|^p \mathbbm{1}_{\Omega_M} } \notag \\
 (\text{Lemma \ref{CarrL1}}) \quad  &\leq \mathbb{E} \Biggeck{b_p M_p \biggrou{ \frac{e^{\alpha C_2 (1+ M^{2/p})}}{\oa(x^*)}}^p \Big|\Gnorm{\nu} -\Gnorm{\mnu} \Big|^p \mathbbm{1}_{\Omega_M} } \notag \\
 &\leq b_p M_p e^{ p \alpha (f(x^*)+C_2(1+M^{2/p}))} \mathbb{E} \Biggeck{\Big| \int_{\mathbb{R}^d} \oa(x) d\nu(x) - \frac{1}{N} \sum^N_{i=1} \oa(\overbar{X}^i) \Big|^p \mathbbm{1}_{\Omega_M} } \notag \\
    &\leq b_p M_p e^{ p \alpha (f(x^*)+C_2(1+M^{2/p}))} \mathbb{E} \Biggeck{\Big| \frac{1}{N} \sum^N_{i=1} \oa(\overbar{X}^i) - \int_{\mathbb{R}^d} \oa(x) d\nu(x)  \Big|^p \mathbbm{1}_{\Omega_M}}.
    \label{CB:S2ConsBound}
\end{align}
We want to treat the expectation in \eqref{CB:S2ConsBound} in a similar way as the expectation in \eqref{CB:S1Bd} by applying Corollary \ref{RosenthalCorr}. Therefore, we define for every $i \leq N$
\begin{align*}
Z_i := \oa(\overbar{X}^i) - \int_{\mathbb{R}^d} \oa(x) d\nu(x).
\end{align*}
Note that for all $i$, we have
\begin{align*}
    \mathbb{E} \bigeck{Z_i} = \underbrace{\mathbb{E} \biggeck{\oa(\overbar{X}^i)}}_{= \int_{\mathbb{R}^d} \oa(x) d\nu(x)}  -  \int_{\mathbb{R}^d} \oa(x) d\nu(x) = 0.
\end{align*}
This means that $\frac{1}{N}\sum^N_{i = 1} Z_i$ is the sample mean of independent, identically distributed random variables with finite $p$th moments and therefore we can apply Corollary \ref{RosenthalCorr} to get the estimate
\begin{align}
    \mathbb{E} \Biggeck{\N{\frac{1}{N} \sum^N_{i= 1} Z_i}^p} \leq C_{p, d} \max \Bigg \{\frac{\mathbb{E} \bigeck{\N{Z_1}^p}}{N^{p-1}} , \frac{\mathbb{E} \bigeck{\N{Z_1}^2}^{\frac{p}{2}}}{N^{\frac{p}{2}}}\Bigg \}, \label{CB:RosenthalEstimate2} 
\end{align}
for a constant $C_{p, d} > 0$ that depends only on $p$ and $d$.
To further estimate, we use the definition of $Z_1$ to compute on a fixed path $\omega \in \Omega$
\begin{align*}
    \N{Z_1} = \N{ \oa(\overbar{X}^1) - \int_{\mathbb{R}^d}  \oa(x) d\nu(x)} &\leq  |\oa(\overbar{X}^1)| + \bigg |\int_{\mathbb{R}^d} \oa(x) d\nu(x) \bigg | \\
    &\leq  \oa(x^*) + \oa(x^*) = 2\oa(x^*),
\end{align*}
where we used that $\oa(x) = e^{-\alpha f(x)}$ attains its maximum at $x^*$.
Combining the latter with our Rosenthal estimate \eqref{CB:RosenthalEstimate2} yields
\begin{align*}
     \mathbb{E} \Biggeck{\N{\frac{1}{N} \sum^N_{i= 1} Z_i}^p} &\leq C_{p, d} \max \Bigg \{\frac{\bigrou{2\oa(x^*)}^p}{N^{p-1}} ,\frac{\bigrou{2 \oa(x^*)}^p}{N^{\frac{p}{2}}} \Bigg \} \\ &= C_{p, d} \bigrou{2 \oa(x^*)}^p \max \Bigg \{ \frac{1}{N^{p-1}} , \frac{1}{N^{\frac{p}{2}}}\Bigg \}.
\end{align*}
Combining this with \eqref{CB:S2ConsBound} gives us our final estimate for the second part:
\begin{align}
    \mathbb{E} \bigeck{\N{S_2}^p \mathbbm{1}_{\Omega_M}} \leq b_p M_p e^{ p \alpha (f(x^*)+C_2(1+M^{2/p}))} C_{p, d} \bigrou{2 \oa(x^*)}^p \max \Bigg \{ \frac{1}{N^{p-1}} , \frac{1}{N^{\frac{p}{2}}}\Bigg \}. 
    \label{CB:S2FinalBound}
\end{align}
Finally, we can use (\ref{CB:S1FinalBound}) and (\ref{CB:S2FinalBound}) to get
\begin{align*}
     &\mathbb{E} \Biggeck{ \N{X^\alpha(\bar{\nu}^N) - X^\alpha (\nu)}^p \mathbbm{1}_{\Omega_M} }\\
     \leq \, & 2^{p-1}\Big( e^{ p \alpha (f(x^*)+C_2(1+M^{2/p}))} C_{p, d} M_p \bigrou{2 \oa(x^*)}^p \\
     & \quad + b_p M_p e^{ p \alpha (f(x^*)+C_2(1+M^{2/p}))} C_{p, d} \bigrou{2 \oa(x^*)}^p\Big) \max \Bigg \{ \frac{1}{N^{p-1}} , \frac{1}{N^{\frac{p}{2}}}\Bigg \}\\
     =\, & 2^{p-1} e^{ p \alpha (f(x^*)+C_2(1+M^{2/p}))}C_{p, d}(2\oa(x^*))^p  \bigrou{M_p + b_pM_p } \max \Bigg \{ \frac{1}{N^{p-1}} , \frac{1}{N^{\frac{p}{2}}}\Bigg \}\\
     =: \, &C_{\mathrm{EM}}\max \Bigg \{ \frac{1}{N^{p-1}} , \frac{1}{N^{\frac{p}{2}}}\Bigg \},
\end{align*}
with the constant $C_{\mathrm{EM}} = C_{\mathrm{EM}}(\alpha, C_2, M, M_p, b_p, C_{p, d}, p)$. As we choose $p \geq 2$, we have 
\begin{align*}
   \max \Bigg \{ \frac{1}{N^{p-1}} , \frac{1}{N^{\frac{p}{2}}}\Bigg \} = \frac{1}{N^{\frac{p}{2}}}, 
\end{align*}
and therefore
\begin{align*}
    \mathbb{E} \Biggeck{ \N{X^\alpha(\bar{\nu}^N) - X^\alpha (\nu)}^p \mathbbm{1}_{\Omega_M} } \leq C_{\mathrm{EM}} N^{-\frac{p}{2}}.
\end{align*}
This concludes the proof.
\end{proof}

\addcontentsline{toc}{chapter}{\protect Bibliography}
\bibliography{library}



\end{document}